\documentclass[letterpaper]{amsart}

\usepackage[letterpaper,hmarginratio=1:1]{geometry}
\usepackage{chngcntr}
\usepackage{amsmath, amsthm, amsfonts, amsbsy, thmtools, amssymb,bm,stmaryrd,xfrac,mathrsfs}
\usepackage[sans]{dsfont}

\usepackage{thm-restate}
\usepackage{hyperref}
\usepackage[mathscr]{euscript}
\usepackage{tikz}
\usetikzlibrary{arrows}
\usepackage{cleveref}
\usepackage{subcaption}

\numberwithin{equation}{section}

\DeclareMathOperator{\Real}{Re}
\DeclareMathOperator{\Imaginary}{Im}

\DeclareMathOperator{\TV}{TV}

\DeclareMathOperator{\Adm}{Adm}

\DeclareMathOperator{\diag}{diag}
\DeclareMathOperator{\dist}{dist}
\DeclareMathOperator{\fin}{fin}
\DeclareMathOperator{\supp}{supp}

\DeclareMathOperator{\PV}{PV}
\DeclareMathOperator{\semci}{semci}

\DeclareMathOperator{\Tr}{Tr}
\DeclareMathOperator{\emp}{emp}
\DeclareMathOperator{\dL}{L}
\DeclareMathOperator{\inv}{inv}

\newtheorem{thm}{Theorem}[section]
\newtheorem{prop}[thm]{Proposition}
\newtheorem{lem}[thm]{Lemma}
\newtheorem{cor}[thm]{Corollary}

\theoremstyle{remark}
\newtheorem{rem}[thm]{Remark}
\theoremstyle{definition}
\newtheorem{definition}[thm]{Definition}
\newtheorem{example}[thm]{Example}
\newtheorem{assumption}[thm]{Assumption}

\begin{document}
	
\title[Local Statistics and Concentration for Non-intersecting Brownian Bridges]{Local Statistics and Concentration for Non-intersecting Brownian Bridges With Smooth Boundary Data}

\author{Amol Aggarwal and Jiaoyang Huang}

	\begin{abstract} 
		
		In this paper we consider non-intersecting Brownian bridges, under fairly general upper and lower boundaries, and starting and ending data. Under the assumption that these boundary data induce a smooth limit shape (without empty facets), we establish two results. The first is a nearly optimal concentration bound for the Brownian bridges in this model. The second is that the bulk local statistics of these bridges along any fixed time converge to the sine process. 
	\end{abstract} 
	
	\maketitle 
	
	\tableofcontents

\section{Introduction} 

\label{Introduction} 

\subsection{Preface} 

\label{Model1}

Over the past several decades, a substantial amount of effort has been devoted to the study of non-intersecting Brownian bridges. Since the works of Dyson  \cite{MERM} in 1962 and de Gennes \cite{SMFSSC} in 1968, this model has been found to be ubiquitous across probability theory and mathematical physics, with applications including to the spectra of random matrices \cite{MERM} and their universality (see the book \cite{DRM} of Erd\H{o}s--Yau); directed fibrous structures \cite{SMFSSC}; random surfaces (see the works of Fisher \cite{WWWM} and Huse--Fisher \cite{CMDWD}); stochastic growth models (see the surveys of Spohn \cite{EDLE}, Johansson \cite{RMDP}, and Ferrari--Spohn \cite{RGM}); and certain quantum many-body systems (see the survey of Dean--Le Doussal--Majumdar--Schehr \cite{NFTRM}).

The general form of this model can be described as follows. Let $n \ge 1$ be an integer, and fix  two $n$-tuples of \emph{starting data} $\bm{u} = (u_1, u_2, \ldots , u_n)$ and \emph{ending data} $\bm{v} = (v_1, v_2, \ldots , v_n)$; a \emph{time interval} $[a, b]$; a \emph{lower boundary} $f : [a, b] \rightarrow \mathbb{R} \cup \{ -\infty\}$; and an \emph{upper boundary} $g : [a, b] \rightarrow \mathbb{R} \cup \{ \infty \}$. Then sample a family $\bm{x} = (x_1, x_2, \ldots , x_n)$ of $n$ Brownian bridges on $[a, b]$, starting at $\bm{u}$ and ending at $\bm{v}$ (so that $x_j (a) = u_j$ and $x_j (b) = v_j$ for each $j$), that are conditioned to not intersect, stay above $f$, and stay below $g$ (so that $f(t) < x_n (t) < x_{n-1} (t) < \cdots < x_1 (t) < g(t)$ for each $t \in (a, b)$); see \Cref{f:bridge1}. One is then interested in the asymptotic behavior of this model as $n$ becomes large.

This question is most extensively understood in the absence of an upper and lower boundary, that is, when $f = -\infty$ and $g = \infty$. The first choice of starting and ending data to have been analyzed is the Brownian watermelon, corresponding to when all entries of $\bm{u}$ and $\bm{v}$ are equal to $0$. In this case, the joint law of the Brownian paths at any fixed time coincides with that of the spectrum of a Gaussian Unitary Ensemble (GUE) random matrix \cite{MERM}. The latter has long been known to be exactly solvable through the framework determinatal point processes and orthogonal polynomials; its bulk local statistics converge to the sine processes, due to the works of Gaudin and Mehta \cite{DERM,SE}. The bulk local statistics, jointly over multiple times, of the Brownian watermelon was shown by Prah\"{o}fer--Spohn \cite{SIDP} and Johansson \cite{DPGDP} to converge to the extended sine process.

The Brownian bridge model has since been proven to be amenable to the framework of determinantal point processes (though often requiring a considerably more involved analysis) for other special choices of starting and ending data $(\bm{u}; \bm{v})$, when both consist of a bounded number (often at most two) distinct points. In many of these situations, different limiting statistics were found appear at the edge, such as Airy processes \cite{SIDP,DPGDP} with wanderers by Adler--Ferrari--van Moerbeke \cite{PWNUC}; Pearcey processes by Br\'{e}zin--Hikami \cite{USCG}, Tracy--Widom \cite{TP}, and Bleher--Kuijlaars \cite{LRMES}; and tacnode processes by Delvaux--Kuijlaars--Zhang \cite{CBNMT}, Ferrari--Vet\H{o} \cite{NBATP}, and Johansson \cite{NMETP}. 

Under general starting and ending data (but still with no upper and lower boundary), the physics work of Matytsin \cite{LLI} predicted a limit shape for non-intersecting Brownian bridges, in the form of a complex Burgers equation. This was proven in the papers of Guionnet--Zeitouni \cite{LDASI} and Guionnet \cite{FOAMI} through a large deviations principle for spherical integrals, obtained by combining a stochastic analysis of Dyson Brownian motion with ideas from free probability. More recent work \cite{EUNB} of Huang proved (under additional, mild restrictions on the starting and ending data) that the edge statistics of this model converge to the Airy point process.

The behavior of non-intersecting Brownian bridges is significantly less understood in the presence of (finite) upper and lower boundaries, as in this situation the exact connection to random matrices, Dyson Brownian motion, and free probability appears to be lost. The determinantal structure remains, but is in general quite intricate. Only for special choices of the upper and lower boundary $(f; g)$ (and starting and ending data $(\bm{u}; \bm{v})$) has this structure proven to be tractable for asymptotics.

Three papers in this direction are those of Tracy--Widom \cite{NE}, Ferrari--Vet\H{o} \cite{HTPM}, and Leichty--Wang \cite{NBRAW}, which observed different statistics arising in the presence of upper and lower boundaries. The first \cite{NE} considered non-intersecting Brownian bridges conditioned to be positive, corresponding to the choice $f = 0$ and $g = \infty$, that start and end at $0$; it then showed that the local limit of the bottom curve converges to a Bessel process with parameter $\frac{1}{2}$. The second \cite{HTPM} considered the same model, but where the paths started and ended at the same nonzero point; it then showed that the local limit of the bottom curve is given by the hard edge tacnode process. The third \cite{NBRAW} considered non-intersecting Brownian bridges, conditioned to stay above $f = 0$ and below $g = 1$, that start and end at $0$; it then showed that the limiting statistics at the point where these bridges first meet the upper boundary $g = 1$ are given by certain variants of the Pearcey and tacnode processes. 

In this paper we study non-intersecting Brownian bridges under a fairly wide class of boundary conditions (subject to a regularity constraint to be described below), with the intent of proving various universality phenomena about their statistical behavior. Our results are two-fold; the first is a concentration bound on the Brownian paths, and the second analyzes their bulk local statistics.

To be more specific, we fix as above a time interval, normalized to be $[0, L^{-1}]$; upper and lower boundaries $(f; g)$ on $[0, L^{-1}]$;  and starting and ending data $(\bm{u}; \bm{v})$. Let $G : [0, L^{-1}] \times [0, 1] \rightarrow \mathbb{R}$ denote the solution to the equation
\begin{flalign}
	\label{gequation0} 
	\partial_t^2 G(t, x) + \pi^2(\partial_y G)^{-4} \cdot \partial_y^2 G (t ,x) = 0,
\end{flalign} 

\noindent with boundary data along the south, north, west, and east sides of the rectangle $[0, L^{-1}] \times [0, 1]$ prescribed by $f$, $g$, $\bm{u}$, and $\bm{v}$, respectively. This function $G$ will serve as the limit shape for the non-intersecting Brownian bridges, with $t$ and $y$ serving as the time coordinate and normalized index of the Brownian bridges (so that $G(t, y) \approx x_{\lfloor yn \rfloor} (t)$). While we have not seen \eqref{gequation0} previously appear exactly as written for this model, we explain in \Cref{Limit1} below that it is equivalent (after a change of variables) to equations prescribing the limit shape for non-intersecting Brownian bridges derived in the above-mentioned works \cite{LLI,LDASI,FOAMI}, in the case of no upper and lower boundary.

Given an integer $m \ge 1$, the regularity constraint (see \Cref{fgr} below) we will impose on the boundary data $(f; g; \bm{u}; \bm{v})$ stipulates that all of the first $m+1$ derivatives of $G$ are bounded from above, and that $|\partial_y G|$ is bounded from below. The latter ensures that the equation \eqref{gequation0} is uniformly elliptic; it also precludes the existence of macroscopic ``empty'' facets, through which no Brownian paths likely pass (such regions are, in a sense, ``compressed out'' by the lower and upper boundaries $f$ and $g$). As such, these boundary data are qualitatively different from many of those considered in the previous works. 

Under this setup, our first result (see \Cref{gh} below) states that $x_j (t) = G(t, jn^{-1}) + \mathcal{O} (n^{2/m - 1 + o(1)})$. In particular, if $G$ is smooth (enabling $m$ to be arbitrarily large), then this implies that the Brownian paths $x_j$ concentrate around $G$ with the nearly optimal error $\mathcal{O} (n^{o(1) - 1})$. The only concentration bound for non-intersecting Brownian bridges, applicable to general upper and lower boundary, we knew of prior to this work is one for its normalized height function, which has error $\mathcal{O} (n^{o(1) - 1/2})$ (and does not provide any characterization for its expectation).\footnote{Starting with the work of Cohn--Elkies--Propp \cite[Proposition 22]{LSRT}, such concentration bounds are well-documented in the context of discrete random surfaces, such as dimers. While we have not seen the counterpart for Brownian bridges appear in the literature before, the proof is very similar. Although we will not need it in this paper, we provide the short proof of this estimate in \Cref{ProofCurves0} below, for completeness.}

Our second result (see \Cref{xkernel} below) states that, for a given $(t_0, y_0) \in (0, 1) \times (0, L^{-1})$, the point process $\big\{ n \cdot \big(x_j (t_0) - G(t_0, y_0) \big) \big\}$ converges to the sine process.\footnote{Our method should also apply to give convergence to the extended sine process for the joint correlation functions of this point process across multiple times, but we will not address this here.} The condition that $(t_0, y_0) \in (0, 1) \times (0, L^{-1})$ is in the strict interior of the rectangle $[0, 1] \times [0, L^{-1}]$ corresponds to probing the ``bulk'' statistics of the Brownian bridges $x_j$. For $\beta=2$ Dyson Brownian motion (which is equivalent after changing coordinates to the special case of the Brownian bridge model without upper and lower boundary, with arbitrary starting data $\bm{u}$ and all entries of the ending data $\bm{v}$ equal to $0$), such bulk universality results have appeared in several forms over the past twenty-five years; see the works of Br\'{e}zin--Hikami \cite{LSRMES}, Johansson \cite{ULSDM}, Erd\H{o}s--P\'{e}ch\'{e}--Ram\'{i}rez--Schlein--Yau \cite{BUFM}, Landon--Yau \cite{CLSM}, Erd\H{o}s--Schnelli \cite{URMFD}, Landon--Sosoe--Yau \cite{FEUM}, and Claeys--Neuschel--Venker \cite{BUPM}. However, we are unaware of any previous such universality statements for non-intersecting Brownian bridges in the presence of upper and lower boundaries (and with general ending data).

Probing the ``edge'' statistics of the model instead corresponds to taking $t \in \{ 0, 1 \}$, which amounts to analyzing the Brownian bridges around the upper or lower boundaries. If $f$ and $g$ are sufficiently regular (which is implied by our assumption), then locally they are approximately linear, and so we would expect the $x_j$ to look (up to an affine shift) locally like non-intersecting Brownian excursions. In view of the results in \cite{NE}, we then predict that the edge scaling limit of the extreme path is a Bessel process with parameter $\frac{1}{2}$. We will not pursue this question here.

Next let us briefly describe the proofs underlying our results. Given the concentration bound, we establish universality for the local statistics of the non-intersecting Brownian bridges $\bm{x}$ by sandwiching them between two Dyson Brownian motion models $\bm{w}^-$ and $\bm{w}^+$ with slightly different drifts (and almost the same initial data). The concentration bound not only enables the existence of this sandwiching, but it can also be used to show that the initial data for $\bm{w}^-$ and $\bm{w}^+$ are reasonably regular. The latter point, together with the results of \cite{CLSM,FEUM} proving fast local equilibration of Dyson Brownian motion under regular initial data, indicates that the local statistics for $\bm{w}^-$ and $\bm{w}^+$ models quickly converge to (the same) sine process; by the sandwiching, this means that those for $\bm{x}$ do as well. This comparison is analogous to (and, in a sense, simpler than) similar ones recently used to establish universality results in the context of random tilings and non-intersecting Brownian bridges \cite{ULTS,EUNB}. So, to us, the main novelty in this paper is in the proof of the concentration bound.

To do this, instead of directly studying the trajectories $x_j (t)$ of the Brownian paths, it will be convenient to analyze the associated \emph{height function} $\mathsf{H}^{\bm{x}} (t, x) = \# \big\{ 1 \le j \le n : x_j (t) > x \big\}$, counting how many bridges are above $x$ at time $t$; this enables us to interpret the Brownian bridge model $\bm{x}$ as a random surface. The concentration bound on the $x_j (t)$ then becomes equivalent to one for this height function, namely, $n^{-1} \cdot  \mathsf{H}^{\bm{x}} (t, x) = H(t, x) + \mathcal{O} (n^{2/m-1+o(1)})$, where $H(t, x)$ is a continuum height function (which, for fixed $t$, is essentially the functional inverse of $G(t, y)$). 

To that end, we will define two families of ``barrier functions'' $\varphi_i^-$ and $\varphi_i^+$ to  bound $n^{-1} \cdot \mathsf{H}^{\bm{x}}$ from below and above.\footnote{In the actual proof, these barrier functions will instead be indexed by a pair of integers $(j; k)$, but we will ignore this point in the introductory exposition here.} These functions will be chosen so that they satisfy several properties. First, at $i = 0$, $\varphi_0^-$ and $\varphi_0^+$ will be small and large, respectively, guaranteeing $\varphi_0^- \le \mathsf{H}^{\bm{x}} (t, x) \le \varphi_0^+$ to hold deterministically. Second, the $\varphi_i^+ (t, x)$ will decrease to $H(t, x) + n^{2/m-1+o(1)}$, and the $\varphi_i^-$ will increase to $H(t, x) - n^{2/m-1+o(1)}$, as $i$ increases from $0$ to some integer $I \ge 1$. We will then show that $n^{-1} \cdot \mathsf{H}^{\bm{x}} \le \varphi_{i-1}^+$ likely implies $n^{-1} \cdot \mathsf{H}^{\bm{x}} \le \varphi_i^+$ (and $n^{-1} \cdot \mathsf{H}^{\bm{x}} \ge \varphi_{i-1}^-$ likely implies $n^{-1} \cdot \mathsf{H}^{\bm{x}} \ge \varphi_i^-$), for each $1 \le i \le I$. Inductively applying this estimate will then confirm the concentration bound, as then $H(t, x) + n^{2/m-1+o(1)} = \varphi_I^- \le n^{-1} \cdot \mathsf{H}^{\bm{x}} \le \varphi_I^+ = H(t, x) + n^{2/m-1+o(1)}$ would likely hold.

The proof of the inductive implication stated above makes use of the following third fundamental property satisfied by the $\varphi_i^+$ and $\varphi_i^-$. For any point $(t_0, x_0)$, the function $\varphi_i^+$ is more concave at $(t_0, x_0)$ than is any continuum height function $H_0$, whose gradient at $(t_0, x_0)$ matches that of $\varphi_i^+$ (any such $H_0$ will have a constrained Hessian, since it satisfies an elliptic partial differential equation similar to the one \eqref{gequation0} satisfied by $G$; see \Cref{hequation} below); an analogous statement holds for $\varphi_i^-$. Now suppose that $n^{-1} \cdot \mathsf{H}^{\bm{x}}$ is (at most) equal to $\varphi_{i-1}^+$. Then condition on the paths of $\bm{x}$ outside of some neighborhood $\mathfrak{U}$ of a fixed point $(t_0, x_0)$, and resample the paths of $\bm{x}$ inside $\mathfrak{U}$ as non-intersecting Brownian bridges with boundary data induced by the conditioning. Assuming a sufficiently strong concentration bound for the Brownian bridge height function $\mathsf{H}^{\bm{x}}$ on $\mathfrak{U}$ around its continuum limit, the above concavity property will imply that $\mathsf{H}^{\bm{x}} (t_0, x_0)$ likely ``deflates'' upon this resampling, decreasing below $\varphi_i^+ (t_0, x_0)$, thus establishing the inductive step. 

Verifying this last assumption will amount to finding a ``large'' class of boundary data, for which the associated non-intersecting Brownian bridges satisfy a nearly optimal $\mathcal{O} (n^{o(1) - 1})$ concentration bound around their continuum height function; if this holds, we refer to the continuum height function as a \emph{concentrating function}. In practice, ``large'' means that any possible $m$-th order Taylor expansion at a fixed point of some continuum height function, which we refer to as a \emph{local profile}, must be encompassed by at least one concentrating function in our family.

The above methodology was initially developed Laslier--Toninelli \cite{LTGDMLS,MTTD}, in the seemingly unrelated context of analyzing Glauber dynamics on random tilings. Their sequences $\{ \varphi_i^- \}$ and $\{ \varphi_i^+ \}$ were meant to serve as discrete-time, and more tractable, proxies for the mean curvature flow that is widely predicted (dating back to the physics paper \cite{LODT} of Lifshitz) to govern the macroscopic behavior of random surfaces under reversible stochastic dynamics. These ideas were used to show concentration estimates for random tilings in the paper \cite{ULTS} of Aggarwal, but the bounds there were quite weak, involving an error of $\mathcal{O} (n^{-\alpha})$ for some small constant $\alpha > 0$. The reason is that work took $m=2$, that is, it only produced concentrating functions (coming from limit shapes of random tilings of hexagons, as in \cite{LTGDMLS}) matching the first two derivatives of any continuum height function. 

In this paper, we must take $m$ to be arbitrary and introduce concentrating functions encompassing any local profile of order $m$. We find them as the continuum height functions for $\beta = 2$ Dyson Brownian motion with arbitrary initial data; nearly optimal concentration bounds for the latter were proven by Huang--Landon in \cite{MCTMGP}. By the generality of the initial data, this provides an infinite-dimensional family of concentrating functions. We show as \Cref{hconcentrateq2} (see also \Cref{p:constructrho}) below that they encompass any local profile of order $m$, thus enabling us to implement the above framework to prove the $\mathcal{O} (n^{2/m-1+o(1)})$ concentration bound for the non-intersecting Brownian bridges $\bm{x}$.

\subsection{Non-Intersecting Brownian Bridges and the Sine Kernel}

\label{EnsemblesCurve} 

In this section we discuss non-intersecting Brownian bridges and the sine kernel. In what follows, for any integer $d \ge 1$ and subset $I \subseteq \mathbb{R}^d$, we let $\mathcal{C}(I)$ denote the space of real-valued, continuous functions $f : I \rightarrow \mathbb{R}$. Moreover, for any subset $I_0 \subseteq I$ and measurable function $f : I \rightarrow \mathbb{C}$, let $f |_{I_0}$ denote the restriction of $f$ to $I_0$. Let $\mathcal{S} \subseteq \mathbb{Z}_{\ge 1}$ and $I \subseteq \mathbb{R}$ denote intervals. We may identify the set $\mathcal{S} \times \mathcal{C} (I)$ with families of continuous functions $\bm{x} = (x_j)_{j \in \mathcal{S}}$, where $x_j : I \rightarrow \mathbb{R}$ is a continuous function for each $j \in \mathcal{S}$; we then also set $\bm{x} (t) = \big( x_j (t) \big)_{j \in \mathcal{S}}$ for each $t \in I$. We will frequently use this notation for families of (non-intersecting) Brownian bridges.

We next provide notation for the probability measure of $n$ non-intersecting Brownian bridges with given starting and ending points, and for given upper and lower boundaries. In what follows, we set $\overline{\mathbb{R}} = \mathbb{R} \cup \{ -\infty, \infty \}$; we also let $\mathbb{H} = \{ z \in \mathbb{C} : \Imaginary z > 0\}$ denote the upper half-plane and $\overline{\mathbb{H}}$ denote its closure. For any real numbers $a, b \in \mathbb{R}$ with $a < b$, we set $\llbracket a, b \rrbracket = [a, b] \cap \mathbb{Z}$. For any integer $k \ge 1$, we denote the entries of any $k$-tuple $\bm{y} \in \mathbb{R}^k$ by $\bm{y} = (y_1 ,y_2, \ldots , y_k)$, unless stated otherwise. We also let $\mathbb{W}_k = \{ \bm{y} \in \mathbb{R}^k : y_1 > y_2 > \cdots >  y_k \}$ and let $\overline{\mathbb{W}}_k$ denote its closure. For any integer $k \ge 1$ and $k$-tuples $\bm{x} = (x_1, x_2, \ldots , x_k) \in \mathbb{R}^k$ and $\bm{y} = (y_1, y_2, \ldots , y_k) \in \mathbb{R}^k$, we write $\bm{x} < \bm{y}$ (equivalently, $\bm{y} > \bm{x}$) if $x_j < y_j$ for each $j \in \llbracket 1, k \rrbracket$; we similarly write $\bm{x} \le \bm{y}$ (equivalently, $\bm{y} \ge \bm{x}$) if $x_j \le y_j$ for each $j \in \llbracket 1, k \rrbracket$. For any subset $I \subseteq \mathbb{R}$ and measurable functions $f, g : I \rightarrow \overline{\mathbb{R}}$ we write $f < g$ (equivalently, $g > f$) if $f(t) < g(t)$ for each $t \in I$; we similarly write $f \le g$ (equivalently, $g \ge f$) if $f(t) \le g(t)$ for each $t \in I$. 

\begin{definition}
	
	\label{qxyfg}
	
	Fix an integer $n \ge 1$; a real number $\sigma > 0$; two $n$-tuples $\bm{u}, \bm{v} \in \overline{\mathbb{W}}_n$; an interval $[a, b] \subset \mathbb{R}$; and measurable functions $f, g : [a, b] \rightarrow \overline{\mathbb{R}}$ such that $f < g$, $f < \infty$, and $g > -\infty$. Let $\mathfrak{Q}_{f; g}^{\bm{u}; \bm{v}} (\sigma)$ denote the law of $\bm{x} = (x_1, x_2, \ldots , x_n) \in \llbracket 1, n \rrbracket \times \mathcal{C} \big( [a, b] \big)$, given by $n$ independent Brownian motions of variances $\sigma$ on the time interval $t \in [a, b]$, conditioned on satisfying the following three properties.
	\begin{enumerate} 
		\item The $x_j$ do not intersect, that is, $\bm{x} (t) \in \mathbb{W}_n$ for each $t \in (a, b)$.
		\item The $x_j$ start at $u_j$ and end at $v_j$, that is, $x_j (a) = u_j$ and $x_j (b) = v_j$ for each $j \in \llbracket 1, n \rrbracket$.
		\item The $x_j$ are bounded below by $f$ and above by $g$, that is, $f < x_j < g$ for each $j \in \llbracket 1, n \rrbracket$. 
	\end{enumerate} 
		
	\begin{figure}
	\center
\includegraphics[width=0.5\textwidth]{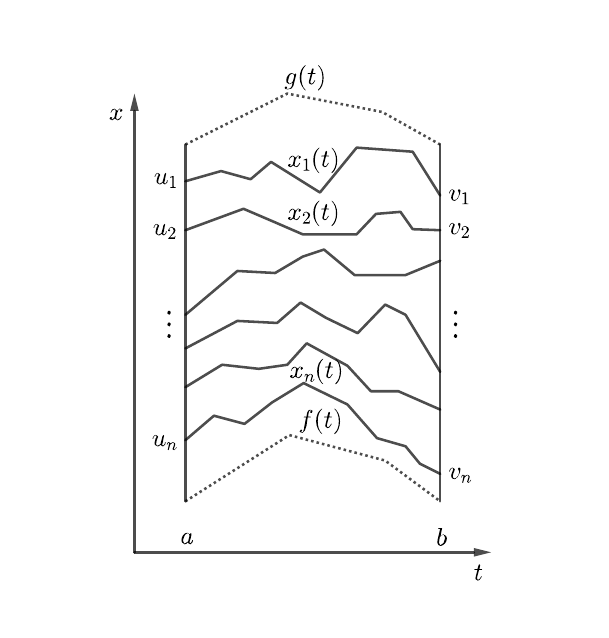}

\caption{An example of non-intersecting Brownian bridges sampled from $\mathfrak{Q}_{f; g}^{\bm{u}; \bm{v}} (\sigma)$ is shown above.}
\label{f:bridge1}
	\end{figure}

	We assume $f(a) \le u_n \le u_1 \le g(a)$ and $f(b) \le v_n \le v_1 \le g(b)$, even when not stated explicitly; see \Cref{f:bridge1} for a depiction. We refer to $\bm{u}$ as \emph{starting data} for $\bm{x}$, and to $\bm{v}$ as its \emph{ending data}. We also refer to $f$ as the \emph{lower boundary} for $\bm{x}$, and to $g$ as its \emph{upper boundary}. If $g = \infty$, then we abbreviate $\mathfrak{Q}_f^{\bm{u}; \bm{v}} (\sigma) = \mathfrak{Q}_{f; \infty}^{\bm{u}; \bm{v}} (\sigma)$; if additionally $f = -\infty$, we further abbreviate $\mathfrak{Q}^{\bm{u}; \bm{v}} (\sigma) = \mathfrak{Q}_{-\infty}^{\bm{u}; \bm{v}} (\sigma) = \mathfrak{Q}_{-\infty; \infty}^{\bm{u}; \bm{v}} (\sigma)$. If $\sigma = n^{-1}$, then we omit the parameter $\sigma$ from the notation, writing $\mathfrak{Q}_{f; g}^{\bm{u}; \bm{v}} = \mathfrak{Q}_{f; g}^{\bm{u}; \bm{v}} (n^{-1})$, $\mathfrak{Q}_f^{\bm{u}; \bm{v}} = \mathfrak{Q}_f^{\bm{u}; \bm{v}} (n^{-1})$, and $\mathfrak{Q}^{\bm{u}; \bm{v}} = \mathfrak{Q}^{\bm{u}; \bm{v}} (n^{-1})$.  
	
\end{definition}

We will show in this paper that (under certain conditions on $\bm{u}$, $\bm{v}$, $f$, and $g$) the local statistics of $\bm{x} (t)$ will converge to a point process described by the sine kernel, defined as follows.

\begin{definition}[{\cite[Equation (3.6.1)]{IRM}}]
	
	\label{kernelcorrelation} 
	
We define the \emph{sine kernel} $\mathcal{K}_{\sin} : \mathbb{R}^2 \rightarrow \mathbb{R}$ by setting
\begin{flalign*}
	\mathcal{K}_{\sin} (x, y) = \displaystyle\frac{\sin \big( \pi(x-y) \big)}{\pi (x-y)}, \qquad \text{for any $x, y \in \mathbb{R}$}.
\end{flalign*}

\noindent For any integer $m \ge 1$ and $m$-tuple $\bm{x} = (x_1, x_2, \ldots , x_m) \in \mathbb{R}^m$, define the $m \times m$ matrix $\bm{\mathcal{K}}_{\sin} (\bm{x})$ whose $(i, j)$ entry is $\mathcal{K}_{\sin} (x_i, x_j)$. Also define the \emph{$m$-point correlation function} $p_{\sin}^{(m)} : \mathbb{R}^m \rightarrow \mathbb{R}$ by 
\begin{flalign*}
	p_{\sin}^{(m)} (\bm{x}) = \det \bm{\mathcal{K}}_{\sin} (\bm{x}), \qquad \text{for any $\bm{x} \in \mathbb{R}^m$}.
\end{flalign*}

\end{definition}

\subsection{Results} 

\label{ConcentrationSmooth0}

	In this section we state two results on the large $n$ behavior of families of non-intersecting Brownian bridges sampled from the measure $\mathfrak{Q}$ of \Cref{qxyfg}; the first is a concentration bound, and the second analyzes convergence of its local statistics. We begin by specifying the assumptions to which we will subject our boundary data, namely, the starting and ending data ($\bm{u}$ and $\bm{v}$) and the lower and upper boundaries ($f$ and $g$). To that end, we require some notation.
	
	 For any set $\mathcal{S} \in \mathbb{C}^d$ and point $z \in \mathbb{C}^d$, we set $\dist (z, \mathcal{S}) = \inf_{w \in \mathcal{S}} |z-w|$. We also let $\partial \mathfrak{S}$ denote the boundary of $\mathfrak{S}$. We further denote the derivative $\partial_{\gamma} = \prod_{i = 1}^d (\partial_{x_i})^{\gamma_i}$ for any $d$-tuple $(\gamma_1, \gamma_2, \ldots , \gamma_d) \in \mathbb{Z}_{\ge 0}^d$; we also set $|\gamma| = \sum_{j=1}^d \gamma_j$. For any open subset $\mathfrak{R} \subseteq \mathbb{R}^d$, let $\mathcal{C}^k (\mathfrak{R})$ denote the set of $f \in \mathcal{C} (\mathfrak{R})$ such that $\partial_{\gamma} f \in \mathcal{C} (\mathfrak{R})$, for each $\gamma \in \mathbb{Z}_{\ge 0}^d$ with $|\gamma| \le k$. Further let $\mathcal{C}^k (\overline{\mathfrak{R}})$ denote the set of functions $f \in \mathcal{C}^k (\mathfrak{R})$ such that $\partial_{\gamma} f \in \mathcal{C} (\mathfrak{R})$ extends continuously to $\overline{\mathfrak{R}}$, for each $\gamma \in \mathbb{Z}_{\ge 0}^d$ with $|\gamma| \le k$. For any function $f \in \mathcal{C} (\mathfrak{R})$ and integer $k \in \mathbb{Z}_{\ge 0}$, we further define the (semi)norms $\| f \|_0 = \| f \|_{0; \mathfrak{R}}$, $[f] = [f]_{k; 0; \mathfrak{R}}$, and $\| f \|_{\mathcal{C}^k (\mathfrak{R})} = \| f \|_{\mathcal{C}^k (\overline{\mathfrak{R}})} = \| f \|_k = \| f \|_{k; 0; \mathfrak{R}}$ on these spaces by
	\begin{flalign*}
		\| f \|_0 = \displaystyle\sup_{z \in \mathfrak{R}} \big| f(z) \big|; \qquad [f]_k = \displaystyle\max_{\substack{\gamma \in \mathbb{Z}_{\ge 0}^d \\ |\gamma| = k}} \big\| \partial_{\gamma} f \|_0; \qquad \| f \|_{\mathcal{C}^k (\mathfrak{R})} = \displaystyle\sum_{j=0}^k [f]_j.
	\end{flalign*}

	For any bounded, open set $\mathfrak{R} \subset \mathbb{R}^2$, let $\Adm (\mathfrak{R})$ be the set of locally Lipschitz functions $F: \overline{\mathfrak{R}} \rightarrow \mathbb{R}$ such that $\partial_y F(z) < 0$ for almost all (with respect to Lebesgue measure)  $z \in \mathfrak{R}$. For any real number $\varepsilon > 0$, further let $\Adm_{\varepsilon} (\mathfrak{R}) \subseteq \Adm (\mathfrak{R})$ denote the set of $F \in \Adm (\mathfrak{R})$ such that $-\varepsilon^{-1} < \partial_y F(z) < -\varepsilon$, for almost all $z \in \mathfrak{R}$ (with respect to Lebesgue measure).

	We can now state the following assumption.

	\begin{assumption}
	
	\label{fgr}
	
	Fix an integer $m \ge 4$ and three real numbers $\varepsilon \in \big( 0, \frac{1}{2} \big)$, $\delta \in \big( 0, \frac{1}{5m^2} \big)$, and $B_0 > 1$. Let $n \ge 1$ be an integer and $L \in (B_0^{-1}, n^{\delta})$ be a real number; define the open rectangle $\mathfrak{R} = (0, L^{-1}) \times (0, 1) \subset \mathbb{R}^2$. Let $G \in \Adm_{\varepsilon} (\mathfrak{R}) \cap \mathcal{C}^{m+1} (\overline{\mathfrak{R}})$ be such that $\big\| G - G(0, 0) \big\|_{\mathcal{C}^{m+1} (\mathfrak{R})} \le B_0$, and such that it solves the equation
	\begin{flalign}
		\label{equationxtd}
		\displaystyle\sum_{j, k \in \{ t, y \}} \mathfrak{d}_{jk} \big( \nabla G(t, y) \big) \partial_j \partial_k G(t, y) = 0, \qquad \text{for each $(t, y) \in \mathfrak{R}$},
	\end{flalign}

	\noindent where the $\mathfrak{d}_{jk} : \mathbb{R}^2 \rightarrow \mathbb{R}$ are defined by
	\begin{flalign}
		\label{uvd} 
		\mathfrak{d}_{tt} (u, v) = 1; \qquad \mathfrak{d}_{ty} (u, v) = 0 = \mathfrak{d}_{yt} (u, v); \qquad \mathfrak{d}_{yy} (u, v) = \pi^2 v^{-4}.
	\end{flalign}

	Define $f, g: [0, L^{-1}] \rightarrow \mathbb{R}$ by setting $f(s) = G(s, 1)$ and $g(s) = G(s, 0)$, for each $s \in [0, L^{-1}]$. Further let $\varkappa > 0$ be a real number, and let $\bm{u}, \bm{v} \in \overline{\mathbb{W}}_n$ be $n$-tuples with  
	\begin{flalign}
		\label{uvg}
		\displaystyle\max_{j \in \llbracket 1, n \rrbracket} \big| u_j - G(0,jn^{-1}) \big| \le \varkappa; \qquad \displaystyle\max_{j \in \llbracket 1, n \rrbracket} \big| v_j - G (L^{-1}, jn^{-1}) \big| \le \varkappa.
	\end{flalign}
	
	\noindent Sample non-intersecting Brownian bridges $\bm{x} = (x_1, x_2, \ldots , x_n) \in \llbracket 1, n \rrbracket \times \mathcal{C} \big( [0,L^{-1}] \big)$ from the measure $\mathfrak{Q}_{f;g}^{\bm{u}; \bm{v}}$. 
\end{assumption}

Let us briefly explain \Cref{fgr}. The function $G$ will eventually be the ``limit shape'' for the family $\bm{x}$ of non-intersecting Brownian bridges, in the sense that we will have $x_j (t) \approx G (t, jn^{-1})$; see \Cref{gh} below. The conditions that $f(s) = G(s, 1)$ and $g(s) = G(s,0)$ ensure that this holds for the upper and lower boundaries of the model (formally, when $j \in \{ 0, n + 1 \}$), and \eqref{uvg} ensures that this holds (up to an error of $\varkappa$) when $t \in \{ 0, L^{-1} \}$. Similarly to how limit shapes for random tilings satisfy a nonlinear partial differential equation \cite[Theorem 12.1]{VPT}, the function $G$ here does as well, given by \eqref{equationxtd}. The constraint that $G \in \Adm_{\varepsilon} (\mathfrak{R})$ ensures that $G$ has no ``frozen facets'' (macroscopic regions containing no curves), and the constraint that $\big\| G - G(0, 0) \big\|_{\mathcal{C}^{m+1} (\mathfrak{R})} \le B_0$ ensures that $G$ has some regularity.

The below result is a concentration bound stating that, with high probability, $x_j (t) = G (t, jn^{-1}) + \mathcal{O} (n^{\delta+2/m-1} + \varkappa)$ (so that the error can be made  smaller by increasing the parameter $m$ that accounts for the regularity of the boundary data); it will be established in \Cref{Prooffgb2} below.

\begin{thm} 
	
	\label{gh} 
	
	Adopt \Cref{fgr}. There exists a constant $c = c (\varepsilon, \delta, B_0, m) > 0$ such that  
	\begin{flalign*} 
		\mathbb{P} \Bigg[ \displaystyle\sup_{s \in [0, L^{-1}]} \bigg( \displaystyle\max_{j \in \llbracket 1, n \rrbracket} \big| x_j (s) - G(s, jn^{-1}) \big| \bigg) > \varkappa + c^{-1} n^{2/m + \delta - 1} \Bigg] \le c^{-1} e^{-c(\log n)^2}.
	\end{flalign*} 
	
\end{thm}

The next result\footnote{Let us clarify that, in \Cref{xkernel}, we view $m$, $\varepsilon$, $\delta$, and $B_0$ as fixed with $n$, while the other parameters $L$, $G$, $\bm{u}$, $\bm{v}$, $\varkappa$, $t_0$, and $y_0$ can vary arbitrarily as $n$ tends to $\infty$ (subject to the conditions of \Cref{fgr} and \Cref{xkernel}).} states that the local statistics of $\bm{x}(t_0)$ at some $t_0 \in (0, L^{-1})$ are prescribed by the sine kernel; it will be established in \Cref{Comparex} below. In what follows, for any integers $n \ge 1$ and $k \in \llbracket 1, n \rrbracket$, and random $n$-tuple $\bm{y} \in \overline{\mathbb{W}}_n$ admitting a density with respect to Lebesgue measure on $\mathbb{R}^n$, define its \emph{$k$-point correlation function} $p^{(k)} = p_{\bm{y}}^{(k)} : \mathbb{R}^k \rightarrow \mathbb{R}$ as follows. Letting $p^{(n)} = p_{\bm{y}}^{(n)} : \mathbb{R}^n \rightarrow \mathbb{R}$ denote the symmetrized density of $\bm{y}$, for any $(x_1, x_2, \ldots , x_k) \in \mathbb{R}^k$, set 
\begin{flalign}
	\label{kp} 
	p^{(k)} (x_1, x_2, \ldots , x_k) = \displaystyle\int_{\mathbb{R}^{n-k}} p^{(n)} (x_1, x_2, \ldots , x_n) \displaystyle\prod_{j=k+1}^n dx_j.
\end{flalign}

In what follows, for any integer $k \ge 1$, complex numbers $a, b \in \mathbb{R}$, vector $\bm{x} = (x_1, x_2, \ldots , x_k) \in \mathbb{C}^k$, and set $\mathcal{S} \subseteq \mathbb{C}$, define $a \cdot \bm{x} + b = (ax_1 + b, ax_2 + b, \ldots , ax_n + b)$ and $a \cdot \mathcal{S} + b = \{ as + b \}_{s \in \mathcal{S}}$. For any additional subset $\mathcal{S}' \subseteq \mathbb{C}$, we denote $\mathcal{S} + \mathcal{S}' = \{ s + s' : s \in \mathcal{S}, s' \in \mathcal{S}' \}$.

\begin{thm}
	
	\label{xkernel}
	
	Adopt \Cref{fgr}; further suppose that $m \ge 2^{25}$ and that $\varkappa \le n^{\delta-1}$. Fix an integer $k \ge 1$; a smooth compactly supported function $F : \mathbb{R}^k \rightarrow \mathbb{R}$; and a point $(t_0, y_0) \in \mathfrak{R}$, such that $\dist \big( (t_0, y_0), \partial \mathfrak{R} \big) \ge n^{-\delta}$. Denote $\theta = -\partial_y G(t_0, y_0)$, and set $\bm{z} = (z_1, z_2, \ldots, z_n) \in \overline{\mathbb{W}}_n$ by $z_j = \theta n \cdot \big( x_j (t_0) - G(t_0, y_0) \big)$ for each $j \in \llbracket 1, n \rrbracket$. Then,
	\begin{flalign*}
		\displaystyle\lim_{n \rightarrow \infty} \displaystyle\int_{\mathbb{R}^k} F(\bm{a}) p_{\bm{z}}^{(k)} (\bm{a}) d\bm{a} = \displaystyle\int_{\mathbb{R}^k} F(\bm{a}) p_{\sin}^{(k)} (\bm{a}) d \bm{a}.
	\end{flalign*}

\end{thm}

\subsection{Outline}

The remainder of this paper is organized as follows. In \Cref{Estimates} we collect miscellaneous (known) facts about non-intersecting Brownian bridges and Dyson Brownian motion, including invariances, H\"{o}lder regularity bounds, and concentration estimates. In \Cref{Limit1} we explain how to recover the equation \eqref{equationxtd}, in the absence of upper and lower boundaries, from the results in \cite{LDASI,FOAMI}. In \Cref{BoundaryGH} we reduce the concentration bound \Cref{gh} to the inductive step, \Cref{gh1}, explained in \Cref{Model1}. In \Cref{HeightLocal} we discuss local profiles and state \Cref{hconcentrateq2}, providing a large class of concentrating functions. Using this, we establish \Cref{gh1} in \Cref{Proof1}. We next prove \Cref{hconcentrateq2} in \Cref{ProofConcentration} and \Cref{xkernel} in \Cref{StatisticsKernel}. We conclude by discussing various regularity properties of the equation \eqref{equationxtd} in \Cref{DerivativesEquation}. While these regularity results are not directly used in the proofs of \Cref{gh} and \Cref{xkernel}, some (such as \Cref{perturbationbdk} and \Cref{f1f2b}) attest to the robustness of \Cref{fgr} under perturbations (and others will also be used in the forthcoming work \cite{U}). The Appendix proves various statements from previous parts of the paper, and it largely constitutes of known (or mild variants of known) results.

\subsection*{Acknowledgements} 

Amol Aggarwal was partially supported by a Packard Fellowship for Science and Engineering, a Clay Research Fellowship, by NSF grant DMS-1926686, and by the IAS School of Mathematics. The research of Jiaoyang Huang is supported by NSF grant DMS-2054835 and DMS-2331096.  The authors also thank Fabio Toninelli and Horng-Tzer Yau for helpful interactions on non-intersecting Brownian bridges, and Camillo De Lellis, L\'{a}szl\'{o} Sz\'{e}kelyhidi Jr., and Vladim\'{i}r Sver\'{a}k for useful comments on elliptic partial differential equations. They wish to acknowledge the NSF grant DMS-1928930, which supported their participation in the Fall 2021 semester program at MSRI in Berkeley, California titled, “Universality and Integrability in Random Matrix Theory and Interacting Particle Systems.

\section{Non-Intersecting Bridges and Dyson Brownian Motion} 

\label{Estimates}

In this section we collect several miscellaneous facts about non-intersecting Brownian bridges and Dyson Brownian motion.

\subsection{Height Monotonicity, H\"{o}lder Estimates, and Invariances}

\label{Coupling2}

In this section we state certain invariances, monotone couplings, and H\"{o}lder estimates for non-intersecting Brownian bridges. We first review two transformations that leave non-intersecting Brownian bridge measures invariant; the first constitutes affine transformations, and the second constitutes diffusive scalings.

\begin{rem}
	
	\label{linear}
	
	Nonintersecting Brownian bridges satisfy the following invariance property under affine transformations. Adopt the notation of \Cref{qxyfg}, and fix real numbers $\alpha, \beta \in \mathbb{R}$. Define the $n$-tuples $\widetilde{\bm{u}}, \widetilde{\bm{v}} \in \mathbb{W}_n$ and functions $\widetilde{f}, \widetilde{g} : [a, b] \rightarrow \overline{\mathbb{R}}$ by setting 
	\begin{flalign*} 
		& \widetilde{u}_j = u_j + \alpha, \quad \text{and} \quad \widetilde{v}_j = v_j + (b-a) \beta + \alpha, \qquad \qquad \qquad \qquad \qquad \text{for each $j \in \llbracket 0, n \rrbracket$}; \\
		& \widetilde{f}(t) = f(t) + (t-a) \beta + \alpha, \quad \text{and} \quad \widetilde{g}(t) = g(t) + (t-a) \beta + \alpha, \qquad \text{for each $t \in [a, b]$}.
	\end{flalign*} 
	
	\noindent Sampling $\widetilde{\bm{x}} = (\widetilde{x}_1, \widetilde{x}_2, \ldots , \widetilde{x}_n )$ under $\mathfrak{Q}_{\tilde{f}, \tilde{g}}^{\tilde{\bm{u}}; \tilde{\bm{v}}}$, there is a coupling between $\widetilde{\bm{x}}$ and $\bm{x}$ such that $\widetilde{x}_j (t) = x_j (t) + (t-a) \beta + \alpha$ for each $(j, t) \in \llbracket 1, n \rrbracket \times [a, b]$. 
	
	Indeed, this follows from the analogous affine invariance of a single Brownian random bridge, together with the fact that affine transformations do not affect the non-intersecting property. More specifically, if $x (t)$, for $t \in [a, b]$, is a Brownian bridge from some $u \in \mathbb{R}$ to some $v \in \mathbb{R}$ then $x(t) + (t-a) \beta + \alpha$ is a Brownian bridge from $u + \alpha$ to $v + (b-a) \beta + \alpha$, and any $\bm{y} (t) \in \mathbb{W}_n$ (namely, is non-intersecting) if and only if $\bm{y}(t) + (t-a) \beta + \alpha \in \mathbb{W}_n$.
	
\end{rem}

\begin{rem}
	
	\label{scale}
	
	Nonintersecting Brownian bridges also satisfy the following invariance property under diffusive scaling. Again adopt the notation of \Cref{qxyfg}; assume that $(a, b) = (0, T)$, for some real number $T > 0$. Further fix a real number $\sigma > 0$, and set $\widetilde{T} = \sigma T$. Define the $n$-tuples $\widetilde{\bm{u}}, \widetilde{\bm{v}} \in \mathbb{W}_n$ and functions $\widetilde{f}, \widetilde{g} : [0, \widetilde{T}] \rightarrow \overline{\mathbb{R}}$ by setting 
	\begin{flalign*} 
		& \widetilde{u}_j = \sigma^{1/2} u_j, \quad \text{and} \quad \widetilde{v}_j = \sigma^{1/2} v_j, \qquad \qquad \qquad \qquad \text{for each $j \in \llbracket 0, n \rrbracket$}; \\
		& \widetilde{f} (t) = \sigma^{1/2} f(\sigma^{-1} t), \quad \text{and} \quad \widetilde{g} (t) = \sigma^{1/2} g(\sigma^{-1} t), \qquad \text{for each $t \in [0, \widetilde{T}]$}.
	\end{flalign*} 
	
	\noindent Sampling $\widetilde{\bm{x}} = ( \widetilde{x}_1, \widetilde{x}_2, \ldots , \widetilde{x}_n)$ under $\mathfrak{Q}_{\tilde{f}, \tilde{g}}^{\tilde{\bm{u}}; \tilde{\bm{v}}}$, there is a coupling between $\widetilde{\bm{x}}$ and $\bm{x}$ such that $\widetilde{x}_j (t) = \sigma^{1/2} x_j (\sigma^{-1} t)$ for each $(j, t) \in \llbracket 1, n \rrbracket \times [0, \widetilde{T}]$. Similarly to \Cref{linear}, this follows from the analogous scaling invariance of a single Brownian bridge.
	
\end{rem}

The next lemma recalls a monotone coupling for non-intersecting Brownian bridges that was shown in \cite{PLE}; we refer to it as \emph{height monotonicity}. 

\begin{lem}[{\cite[Lemmas 2.6 and 2.7]{PLE}}]
	
	\label{monotoneheight}

	Fix an integer $n \ge 1$; four $n$-tuples $\bm{u}, \widetilde{\bm{u}}, \bm{v}, \widetilde{\bm{v}} \in \mathbb{W}_n$; an interval $[a, b] \in \mathbb{R}$; and measurable functions $f, \widetilde{f}, g, \widetilde{g}: [a, b] \rightarrow \overline{\mathbb{R}}$ with $f \le g$ and $\widetilde{f} \le \widetilde{g}$. Sample two families of non-intersecting Brownian bridges $\bm{x} = (x_1, x_2, \ldots , x_n) \in \llbracket 1, n \rrbracket \times \mathcal{C} \big( [a, b] \big)$ and $\widetilde{\bm{x}} = (\widetilde{x}_1, \widetilde{x}_2, \ldots , \widetilde{x}_n) \in \llbracket 1, n \rrbracket \times \mathcal{C} \big( [a, b] \big)$ from the measures $\mathfrak{Q}_{f; g}^{\bm{u}; \bm{v}}$ and $\mathfrak{Q}_{\tilde{f}; \tilde{g}}^{\tilde{\bm{u}}; \tilde{\bm{v}}}$, respectively. If 
	\begin{flalign}
		\label{fgxmonotone} 
		f \le \widetilde{f} ; \quad  g \le \widetilde{g}; \quad \bm{u} \le \widetilde{\bm{u}} ; \quad \bm{v} \le \widetilde{\bm{v}},
	\end{flalign}
	
	\noindent  then there exists a coupling between $\bm{x}$ and $\widetilde{\bm{x}}$ so that $x_j (t) \le \widetilde{x}_j (t)$, for each $(j, t) \in \llbracket 1, n \rrbracket \times [a, b]$.
	
\end{lem}

We next state the following variant of the above coupling, whose second part provides a linear bound on the difference between two families of non-intersecting Brownian bridges, which have the same starting data but different ending data. Its proof will be given in \Cref{DiscreteLinear} below. 

\begin{lem}
	
	\label{uvv}
	
	Fix an integer $n \ge 1$; a real number $B \ge 0$; a finite interval $[a, b] \subset \mathbb{R}$; four $n$-tuples $\bm{u}, \widetilde{\bm{u}}, \bm{v}, \widetilde{\bm{v}} \in \mathbb{W}_n$; and four measurable functions $f, \widetilde{f}, g, \widetilde{g} : [a, b] \rightarrow \overline{\mathbb{R}}$. Assume for each $j \in \llbracket 1, n \rrbracket$ and $t \in [a, b]$ that 
	\begin{flalign}
		\label{buvfg}
		u_j \le \widetilde{u}_j \le u_j + B; \quad v_j \le \widetilde{v}_j \le v_j + B; \quad f(t) \le \widetilde{f} (t) \le f(t) + B; \quad g(t) \le \widetilde{g} (t) \le g(t) + B.
	\end{flalign}
	
	\noindent Sample two families of non-intersecting Brownian bridges $\bm{x} = (x_1, x_2, \ldots , x_n) \in \llbracket 1, n \rrbracket \times \mathcal{C} \big( [a, b] \big)$ and $\widetilde{\bm{x}} = (\widetilde{x}_1, \widetilde{x}_2, \ldots , \widetilde{x}_n) \in \llbracket 1, n \rrbracket \times \mathcal{C} \big( [a, b] \big)$ from the measures $\mathfrak{Q}_{f; g}^{\bm{u}; \bm{v}}$ and $\mathfrak{Q}_{\tilde{f}; \tilde{g}}^{\tilde{\bm{u}}; \tilde{\bm{v}}}$, respectively. 
	
	\begin{enumerate} 
		
		\item  There is a coupling between $\bm{x}$ and $\widetilde{\bm{x}}$ so that $x_j (t) \le \widetilde{x}_j (t) \le x_j (t) + B$ for each $(j, t) \in \llbracket 1, n \rrbracket \times [a, b]$. 
		
		\item Further assume that $\bm{u} = \widetilde{\bm{u}}$ and for each $t \in [a, b]$ that
		\begin{flalign}
			\label{fgb} 
			 f(t) \le \widetilde{f} (t) \le f(t) + \frac{t-a}{b-a} \cdot B; \qquad  g(t) \le \widetilde{g}(t)  \le g(t) + \frac{t-a}{b-a}\cdot B.
		\end{flalign} 
		
		\noindent Then, it is possible to couple $\bm{x}$ and $\widetilde{\bm{x}}$ such that 
		\begin{flalign*} 
			 x_j (t) \le \widetilde{x}_j (t)  \le x_j (t) + \frac{t-a}{b-a} \cdot B, \qquad \text{for each $(j, t) \in \llbracket 1, n \rrbracket \times [a, b]$}.
		\end{flalign*} 
		
	\end{enumerate} 
\end{lem}

\begin{rem} 
	
	Let us mention two quick consequences of \Cref{uvv} that we will often use without comment. First, suppose instead of \eqref{buvfg} that for each $j \in \llbracket 1, n \rrbracket$ and $t \in [a, b]$ we have
	\begin{flalign}
		\label{buvfg2}
		|u_j - \widetilde{u}_j| \le B; \qquad |v_j - \widetilde{v}_j| \le B; \qquad \big| f(t) - \widetilde{f} (t) \big| \le B; \qquad \big| g(t) - \widetilde{g}(t) \big| \le B.
	\end{flalign}

	\noindent Then, it is possible to couple $\bm{x}$ and $\widetilde{\bm{x}}$ such that $\big| x_j (t) - \widetilde{x}_j (t) \big| \le B$ for each $(j, t) \in \llbracket 1, n \rrbracket \times [a, b]$. 
	
	Indeed, define the $n$-tuples $\widehat{\bm{u}}, \widehat{\bm{v}} \in \overline{\mathbb{W}}_n$ by setting $\widehat{u}_j = \widetilde{u}_j + B$ and $\widehat{v}_j = \widetilde{v}_j + B$, for each $j \in \llbracket 1, n \rrbracket$; also define the functions $\widehat{f}, \widehat{g} : [a, b] \rightarrow \mathbb{R}$ by setting $\widehat{f}(t) = \widetilde{f}(t) + B$ and $\widehat{g}(t) = \widetilde{g} (t) + B$, for each $t \in [a, b]$. Then, sample $n$ Brownian bridges $\widehat{\bm{x}} = (\widehat{x}_1, \widehat{x}_2, \ldots , \widehat{x}_n) \in \llbracket 1, n \rrbracket \times \mathcal{C} \big( [a, b] \big)$ from the measure $\mathfrak{Q}_{\hat{f}; \hat{g}}^{\hat{\bm{u}}; \hat{\bm{v}}}$. Observe that we may couple $\widehat{\bm{x}}$ with $\widetilde{\bm{x}}$ such that $\widehat{x}_j (t) = \widetilde{x}_j (t) + B$, for each $(j, t)$. 
	
	Applying the first part of \Cref{uvv} with the $(\bm{x}; \widetilde{\bm{x}}; B)$ there given by $(\bm{x}; \widehat{\bm{x}}; 2B)$ here (using \eqref{buvfg2} to verify \eqref{buvfg}) yields a coupling between $\bm{x}$ and $\widehat{\bm{x}}$ such that for each $(j, t)$ we have $x_j (t) \le \widehat{x}_j (t) \le x_j (t) + 2B$. Since $\widehat{x}_j (t) = \widetilde{x}_j (t) + B$, it follows that $\big| x_j (t) - \widetilde{x}_j (t) \big| \le B$ for each $(j, t)$.

	Second, assume that \eqref{buvfg2} holds; that we have $\bm{u} = \widetilde{\bm{u}}$; and that, instead of \eqref{fgb}, we have  
	\begin{flalign*}
		\displaystyle\sup_{t \in [a, b]} \big| f(t) - \widetilde{f}(t) \big| \le \displaystyle\frac{t-a}{b-a} \cdot B; \qquad \displaystyle\sup_{t \in [a, b]} \big| f(t) - \widetilde{f} (t) \big| \le \displaystyle\frac{t-a}{b-a} \cdot B.
	\end{flalign*}

	\noindent Then, by similar reasoning as used above, it is possible to couple $\bm{x}$ and $\widetilde{\bm{x}}$ such that
	\begin{flalign*}
		\big| x_j (t) - \widetilde{x}_j (t) \big| \le \displaystyle\frac{t-a}{b-a} \cdot B, \qquad \text{for each $(j, t) \in \llbracket 1, n \rrbracket \times [a, b]$}.
	\end{flalign*}  

\end{rem}

We next recall the following H\"{o}lder estimate for non-intersecting Brownian bridges from \cite{BPLE} (scaled by a factor of $n^{-1/2}$, as the variances of the bridges there were $1$ while here they are $n^{-1}$).

\begin{lem}[{\cite[Proposition 3.5]{BPLE}}]
	
	\label{estimatexj2}

	There exist constants $c > 0$ and $C > 1$ such that the following holds. Let $n \ge 1$ be an integer; $B \ge 1$ be a real number; $[a, b] \subset \mathbb{R}$ be an interval; and $\bm{u}, \bm{v} \in \mathbb{W}_n$ be two $n$-tuples. Sampling $\bm{x} = (x_1, x_2, \ldots , x_n) \in \llbracket 1, n \rrbracket \times \mathcal{C} \big( [a, b] \big)$ under $\mathfrak{Q}^{\bm{u}; \bm{v}}$, we have  
	\begin{flalign*}
		\mathbb{P} \Bigg[ \bigcup_{j=1}^n \bigcup_{a \le t < t+s \le b} \bigg\{ \Big| x_j (t+s) - x_j (t) - s \cdot \displaystyle\frac{v_j - u_j}{b-a} \Big| \ge B s^{1/2} \log \big| 2s^{-1} (b-a) \big| \bigg\} \Bigg] \le C e^{Cn - cB^2 n}.
	\end{flalign*} 
\end{lem}

We will also require the following variant of \Cref{estimatexj2} that allows for lower and upper boundaries $f$ and $g$. Its second part provides a H\"{o}lder bound assuming these two boundaries are differentiable. Without at least some continuity constraints on $f$ and $g$, such a bound cannot hold.\footnote{For example, if $x_n (a) = f(a)$ and $f(a^+) > f(a)$, then $x_n$ will be discontinuous at $a$.} However, when $g = \infty$ (in the absence of a continuity condition on $f$), the first part of the below lemma provides an ``interior'' H\"{o}lder bound on these bridges, that is, one that holds away from the boundaries $t \in \{ a, b \}$ of the time interval. The proof of the following result is similar to that of \Cref{estimatexj2} and will be established in \Cref{ProofContinuous} below.

\begin{lem}
	
	\label{estimatexj3}
	
	Let $n$, $B$, $a$, $b$, $\bm{u}$, and $\bm{v}$ be as in \Cref{estimatexj2}. Further let $A \ge 1$ be a real number and $f, g: [a, b] \rightarrow \overline{\mathbb{R}}$ be measurable functions. Sampling $\bm{x} = (x_1, x_2, \ldots , x_n) \in \llbracket 1, n \rrbracket \times \mathcal{C} \big( [a, b] \big)$ under $\mathfrak{Q}_{f; g}^{\bm{u}; \bm{v}}$, the following two statements hold. In what follows, we set $T = b-a$.
	
	\begin{enumerate} 
		\item \label{afg} Assume that $g = \infty$; that $u_n - f(r) \ge -A T^{1/2}$; and that $v_n - f(r) \ge -A T^{1/2}$, for each $r \in [a, b]$. For any real number $0 < \kappa < \min \big\{ \frac{T}{2}, 1 \big\}$, we have  
		\begin{flalign*}
			\mathbb{P} \Bigg[ \bigcup_{j=1}^n \bigcup_{a + \kappa \le t < t+s \le b - \kappa} \bigg\{ \big| x_j (& t+s) - x_j (t) - sT^{-1} (v_j - u_j) \big| \\
			& \ge s^{1/2}  \big( B \log |2s^{-1} T| + \kappa^{-1} (A+B) \big)^2 \bigg\} \Bigg] \le C e^{Cn - cB^2 n}.
		\end{flalign*} 
		
		\item \label{afg2} If instead $f$ and $g$ are differentiable and $\big| \partial_t f(r) \big| + \big| \partial_t g(r) \big| \le A$, for each $r \in [a, b]$, then 
		\begin{flalign*}
			\mathbb{P} \Bigg[ \bigcup_{j=1}^n \bigcup_{a \le t < t+s \le b} \bigg\{ \big| x_j (& t+s) - x_j (t) - s T^{-1} (v_j - u_j) \big| \\
			& \ge s^{1/2}  \big( B\log | 9 s^{-1} T| + 2(A + B) \big)^2 \bigg\} \Bigg]  \le C e^{Cn - cB^2 n}.
		\end{flalign*} 
	\end{enumerate} 
\end{lem}

\subsection{Stieltjes Transforms and Free Convolutions}

\label{TransformConvolution}

In this section we recall various results concerning Stieltjes transforms and free convolution with the semicircle distribution. We first require some notation on measures. In what follows, we let $\mathscr{P}_{\fin} = \mathscr{P}_{\fin} (\mathbb{R})$ denote the set of nonnegative measures $\mu$ on $\mathbb{R}$ with finite total mass, namely, $\mu (\mathbb{R}) < \infty$. Further let $\mathscr{P} = \mathscr{P} (\mathbb{R}) \subset \mathscr{P}_{\fin}$ denote the set of probability measures on $\mathbb{R}$, and let $\mathscr{P}_0 = \mathscr{P}_0 (\mathbb{R}) \subset \mathscr{P}$ denote the set of probability measures that are compactly supported; the support of any measure $\nu \in \mathscr{P}$ is denoted by $\supp \nu$. We say that a probability measure $\mu \in \mathscr{P}$ has density $\varrho$ (with respect to Lebesgue measure) if $\varrho : \mathbb{R} \rightarrow \mathbb{R}$ is a measurable function satisfying $\mu (dx) = \varrho(x) dx$. For any real number $x \in \mathbb{R}$, we let $\delta_x \in \mathscr{P}_0$ denote the delta function at $x$. 

Fix a measure $\mu \in \mathscr{P}_{\fin}$. Its \emph{Stieltjes transform} is the function $m = m_{\mu} : \mathbb{H} \rightarrow \mathbb{H}$ defined by, for any complex number $z \in \mathbb{H}$, setting   
\begin{flalign}
	\label{mz0} 
	m(z) = \displaystyle\int_{-\infty}^{\infty} \displaystyle\frac{\mu (dx)}{x-z}.
\end{flalign}

\noindent If $\mu$ has a density $\varrho \in L^1 (\mathbb{R})$, then $\varrho$ can be recovered from its Stieltjes transform by the identity \cite[Equation (8.14)]{FPRM}, 
\begin{flalign}
	\label{mrho} 
	\pi^{-1} \displaystyle\lim_{y \rightarrow 0} \Imaginary m(x + \mathrm{i} y) = \varrho (x); \qquad \pi^{-1} \displaystyle\lim_{y \rightarrow 0} \Real m(x + \mathrm{i} y) = H \varrho (x),
\end{flalign}

\noindent for any $x \in \mathbb{R}$. In the latter, $Hf$ denotes the Hilbert transform of any function $f \in L^1 (\mathbb{R})$, given by 
\begin{flalign}
	\label{transform2}
	Hf (x) = \pi^{-1} \cdot \PV \displaystyle\int_{-\infty}^{\infty} \displaystyle\frac{f(w) dw}{w-x},
\end{flalign}

\noindent where $\PV$ denotes the Cauchy principal value (assuming the integral exists as a principal value).

The \emph{semicircle distribution} is a measure $\mu_{\semci} \in \mathscr{P} (\mathbb{R})$ whose density $\varrho_{\semci} : \mathbb{R} \rightarrow \mathbb{R}_{\ge 0}$ is  
\begin{flalign}
	\label{rho1} 
	\varrho_{\semci} (x) = \displaystyle\frac{(4-x^2)^{1/2}}{2\pi} \cdot \textbf{1}_{x \in [-2, 2]}, \qquad \text{for all $x \in \mathbb{R}$}. 
\end{flalign}

\noindent For any real $t > 0$, we denote the rescaled semicircle density $\varrho_{\semci}^{(t)}$ and distribution $\mu_{\semci}^{(t)} \in \mathscr{P}$ by 
\begin{flalign}
	\label{rhosct}
	\varrho_{\semci}^{(t)} (x) = t^{-1/2} \varrho_{\semci} (t^{-1/2} x); \qquad \mu_{\semci}^{(t)} (dx) = \varrho_{\semci}^{(t)} (x) dx.
\end{flalign}

 We next discuss the free convolution of a probability measure $\mu \in \mathscr{P}$ with the (rescaled) semicircle distribution $\mu_{\semci}^{(t)}$. First, for any real number $t > 0$, define the function $M = M_{t; \mu} : \mathbb{H} \rightarrow \mathbb{C}$ and the set $\Lambda_t = \Lambda_{t; \mu} \subseteq \mathbb{H}$ by
\begin{flalign}
	\label{mtlambdat} 
	M(z) = z - t m(z); \quad \Lambda_t = \Big\{ z \in \mathbb{H} : \Imaginary \big( z - tm (z) \big) > 0 \Big\} = \Bigg\{ z \in \mathbb{H} : \displaystyle\int_{-\infty}^{\infty} \displaystyle\frac{\mu(dx)}{|z-x|^2} < \displaystyle\frac{1}{t} \Bigg\}.
\end{flalign}


\begin{lem}[{\cite[Lemma 4]{FCSD}}] 
	
	\label{mz} 
	
	The function $M$ is a homeomorphism from $\overline{\Lambda}_t$ to $\overline{\mathbb{H}}$. Moreover, it is a holomorphic map from $\Lambda_t$ to $\mathbb{H}$ and a bijection from $\partial \Lambda_t$ to $\mathbb{R}$. 
	
\end{lem} 

For any real number $t \ge 0$, define $m_t = m_{t; \mu} : \mathbb{H} \rightarrow \mathbb{H}$ as follows. First set $m_0 (z) = m(z)$; for any real number $t > 0$, define $m_t$ so that	
\begin{flalign}
	\label{mt} 
	m_t \big( z - t m_0 (z) \big) = m_0 (z), \qquad \text{for any $z \in \Lambda_t$}.
\end{flalign}

\noindent Since by \Cref{mz} the function $M(z) = z - tm_0 (z)$ is a bijection from $\Lambda_t$ to $\mathbb{H}$, \eqref{mt} defines $m_t$ on $\mathbb{H}$. By \cite[Proposition 2]{FCSD}, $m_t$ is the Stieltjes transform of a probability measure $\mu_t \in \mathscr{P} (\mathbb{R})$. This measure is called the \emph{free convolution} between $\mu$ and $\mu_{\semci}^{(t)}$, and we often write $\mu_t = \mu \boxplus \mu_{\semci}^{(t)}$. By \cite[Corollary 2]{FCSD}, $\mu_t$ has a density $\varrho_t = \varrho_t^{\mu} : \mathbb{R} \rightarrow \mathbb{R}_{\ge 0}$ for any $t > 0$.

The following lemma provides an expression of the $t$-derivative of $\varrho_t$; it is essentially known \cite{ACNRV,FCSD} but we provide the proof of its second part for convenience. 

\begin{lem}[\cite{ACNRV,FCSD}]
	
	\label{yconvolution}
	
	For any real number $t > 0$, both $\varrho_t(y)$ and its Hilbert transform $H \varrho_t (y)$ are smooth in $y \in \big\{ y' \in \mathbb{R} : \varrho_t (y') > 0 \big\}$. Moreover, we have 
	 \begin{flalign}
	 	\label{trhoty}
	 	\partial_t \varrho_t(y)= \pi \cdot \partial_y \big(\varrho_t(y) H \varrho_t (y) \big), \qquad \text{for $y \in \big\{ y' \in \mathbb{R} : \varrho_t (y') > 0 \big\}$}.
	\end{flalign}

\end{lem}

\begin{proof}
	
	The first part follows from \cite[Corollary 3]{FCSD} (together with \cite[Lemma 2]{FCSD}). For the second part, by taking the derivative on both sides of \eqref{mt} with respect to $t$, we get for any $z \in \Lambda_t$ that
	\begin{flalign*} 
		0 = \partial_t m_0 (z) = \partial_t \Big( m_t \big( z - tm_0 (z) \big) \Big) & = \partial_t m_t \big( z - t m_0 (z) \big) - \partial_z m_t \big( z - t m_0 (z) \big) \cdot m_0 (z) \\ 
		& = \partial_t m_t \big( z - tm_0 (z) \big) - \partial_z m_t \big( z - tm_0 (z) \big) \cdot m_t \big( z - tm_0 (z) \big).
	\end{flalign*} 

	\noindent Setting $w = z - tm_0 (z)$ and applying \Cref{mz}, it follows for any $w = y + \mathrm{i} \eta \in \mathbb{H}$ that 
	\begin{flalign*}
		\partial_t m_t (y + \mathrm{i} \eta) = \displaystyle\frac{1}{2} \cdot \partial_z \big( m_t (y + \mathrm{i} \eta) \big)^2.
	\end{flalign*}  
	
	\noindent Sending $\eta$ to zero and applying \eqref{mrho} (with the first part of the lemma), we find
	\begin{align}\label{e:limitzero}
		\pi \cdot \partial_t \big(H \varrho_t (y)+\mathrm{i} \varrho_t(y) \big)=\frac{\pi^2}{2} \cdot \partial_y \Big( \big(H \varrho_t (y)+\mathrm{i} \varrho_t(y) \big)^2 \Big).
	\end{align}

	\noindent Then \eqref{trhoty} follows from comparing the imaginary parts on both sides of \eqref{e:limitzero}.
\end{proof}

The following lemmas indicate that $\varrho_t$ is smooth on its support and bounds its derivatives there; we provide their proofs (which follow similar ones in \cite{FCSD}) in \Cref{ProofRho} below. 

\begin{lem}
	
	\label{rhotestimatek} 
	
	For any real number $t > 0$, we have $\varrho_t^{\mu} (x) \le \pi^{-1} t^{-1/2}$. Moreover, $\varrho_t$ is smooth on its support; in particular, for any integer $k \in \mathbb{Z}_{\ge 1}$ there exists a constant $C = C(k) > 1$ such that the following holds. For any real numbers $t, \delta \in (0, 1)$ and $x_0 \in \mathbb{R}$, if $\varrho_t (x_0) \ge \delta$ then $\big| \partial_x^k \varrho_t (x_0) \big| \le C (t^4 \delta^6)^{-k}$.
	
\end{lem}

\begin{lem} 
	
	\label{derivativetm}	
	
	For any integer $k \ge 2$; real numbers $0 < \delta \le \mathfrak{c} < 1$ and $B > 1$, there exist constants $\mathfrak{t}_0 = \mathfrak{t}_0 (k, \delta, \mathfrak{c}, B) > 0$ and $C = C(k, \delta, \mathfrak{c}, B) > 1$ such that the following holds. Let $\mu_0 \in \mathscr{P}_0$ denote a compactly supported probability measure with a bounded density $\varrho_0 : \mathbb{R} \rightarrow \mathbb{R}_{\ge 0}$ with respect to Lebesgue measure. Assume that $\inf_{|x| \le \mathfrak{c}} \varrho_0(x)\ge \delta$ and $\sup_{|x| \le \mathfrak{c}} \big| \partial_x^m \varrho_0 (x) \big| \le B^{m+1} \delta^{-m}$ for any $m \in \llbracket 0, k \rrbracket$. Then, 
	\begin{flalign}
		\label{rhotxtx}
		\displaystyle\inf_{|x| \le \mathfrak{c}/4} \displaystyle\inf_{t \in [0, \mathfrak{t}_0]} \varrho_t (x) \ge \displaystyle\frac{\delta}{2}; \qquad \displaystyle\max_{k' \in \llbracket 0, k-2 \rrbracket} \displaystyle\sup_{|x| \le \mathfrak{c}/4} \displaystyle\sup_{t \in [0, \mathfrak{t}_0]} \big| \partial_t^{k-k'-2} \partial_x^{k'} \varrho_t (x) \big| \le C.
	\end{flalign}

\end{lem} 

\begin{rem} 
	
	\label{mtscale} 

While free convolutions are typically defined between probability measures, the relation \eqref{mt} also defines the free convolution of any measure $\mu\in \mathscr{P}_{\fin}$, satisfying $A = \mu(\mathbb{R}) < \infty$, with the rescaled semicircle distribution $\mu_{\semci}^{(t)}$. Indeed, define the probability measure $\widetilde{\mu} \in \mathscr{P}$ from $\mu$ by setting $\widetilde{\mu} (I)= A^{-1} \cdot \mu (A^{1/2} I)$, for any interval $I\subseteq \mathbb{R}$. Furthermore, for any real number $s \ge 0$, define the probability measure $\widetilde{\mu}_s = \widetilde{\mu} \boxplus \mu_{\semci}^{(s)}$, and denote its Stieltjes transform by $\widetilde{m}_s$. Then, define the free convolution $\mu_t = \mu \boxplus \mu_{\semci}^{(t)}$ and its Stieltjes transform $m_t = m_{t; \mu}$ by setting
	\begin{flalign*} 
		\mu_t (I) = A \cdot \widetilde{\mu}_t (A^{-1/2} I), \qquad \text{for any interval $I \subseteq \mathbb{R}$, so that} \qquad	m_t (z) = A^{1/2} \cdot \widetilde m_t ( A^{-1/2} z),
	\end{flalign*} 

	\noindent where the second equality follows from the first by \eqref{mz0}. Then, 
\begin{align*}
		m_{t} \big( z-t m_{0}(z) \big) = A^{1/2} \cdot \widetilde{m}_t \Big( A^{-1/2} \big(z- t m_{0}(z) \big) \Big) & = A^{1/2} \cdot \widetilde{m}_t \big( A^{-1/2} z- t \widetilde{m}_{0}(A^{-1/2} z) \big) \\
		& = A^{1/2} \cdot \widetilde{m}_{0}(A^{-1/2} z)= m_{0}(z),
\end{align*}

\noindent so that \eqref{mt} continues to hold for $m_t$. In particular, \Cref{mz} and \Cref{yconvolution} (the latter by its proof) also hold for $\mu$.

\end{rem} 

\begin{rem}
	
	\label{mtscalebeta}
	
	Let us describe a scaling invariance for Stieltjes transforms. Fix a measure $\mu \in \mathscr{P}_{\fin}$ with $\mu (\mathbb{R}) < \infty$, and let $m_s$ denote the Stieltjes transform of $\mu \boxplus \mu_{\semci}^{(s)}$, for any real number $s \ge 0$. Fix a real number $\beta > 0$, and define the measure $\widetilde{\mu} \in \mathscr{P}_{\fin}$ by setting $\widetilde{\mu} (I) = \mu (\beta^{1/2} \cdot I)$, for any interval $I \subseteq \mathbb{R}$. Denote the Stieltjes transform of $\widetilde{\mu}$ by $\widetilde{m} = m_{\tilde{\mu}}$, and let $\widetilde{m}_s$ denote the Stieltjes transform of $\widetilde{\mu} \boxplus \mu_{\semci}^{(s)}$ for any $s > 0$. Then, observe for any real number $t \ge 0$ and complex number $z \in \mathbb{H}$ that 
	\begin{flalign} 
		\label{mtbeta} 
		\widetilde{m}_t (z) = \beta^{1/2} \cdot m_{\beta t} (\beta^{1/2} z).
	\end{flalign} 

	\noindent Indeed, this holds at $t = 0$ by \eqref{mz0}; for $t > 0$, we have
	\begin{flalign*}
		\widetilde{m}_t \big( z - t \widetilde{m}_0 (z) \big) = \widetilde{m}_0 (z) = \beta^{1/2} \cdot m_0 (\beta^{1/2} z) & = \beta^{1/2} \cdot m_{\beta t} \big( \beta^{1/2} z - \beta t m_0 (\beta^{1/2} z) \big) \\
		& = \beta^{1/2} \cdot m_{\beta t} \Big( \beta^{1/2} \big( z - t \widetilde{m}_0 (z) \big) \Big),
	\end{flalign*}

	\noindent which by \Cref{mz} (and \Cref{mtscale}, if $\mu$ is not a probability measure) implies \eqref{mtbeta}. 
\end{rem}

\subsection{Dyson Brownian Motion} 

\label{LambdaEquation}

In this section we recall properties about Dyson Brownian motion. Fix an integer $n \ge 1$ and a sequence $\bm{\lambda} (0) = \big( \lambda_1 (0), \lambda_2 (0), \ldots , \lambda_n (0) \big) \in \overline{\mathbb{W}}_n$. Define $\bm{\lambda} (t) = \big(\lambda_1 (t), \lambda_2 (t), \ldots , \lambda_n (t) \big) \in \overline{\mathbb{W}}_n$, for $t \ge 0$, to be the unique strong solution (see \cite[Proposition 4.3.5]{IRM}) to the stochastic differential equations 
\begin{flalign}
	\label{lambdaequation}
	d\lambda_i (t) = n^{-1} \displaystyle\sum_{\substack{1 \le j \le n \\ j \ne i}}	\displaystyle\frac{dt}{\lambda_i (t) - \lambda_j (t)} + n^{-1/2} dB_i (t).
\end{flalign}

\noindent The system \eqref{lambdaequation} is called \emph{Dyson Brownian motion}, run for time $t$, with \emph{initial data} $\bm{\lambda} (0)$; the $\lambda_i$ are sometimes referred to as \emph{particles}. 

\begin{rem}
	
	\label{scalemotion}
	
	As in \Cref{scale}, Dyson Brownian motion admits the following invariance under diffusive scaling for any real number $\sigma > 0$. If $\bm{\lambda} (t) = \big( \lambda_1 (t), \lambda_2 (t), \ldots , \lambda_n (t) \big) \in \overline{\mathbb{W}}_n$ solves \eqref{lambdaequation} then, denoting $\widetilde{\lambda}_j (t) = \sigma^{-1} \lambda_j (\sigma^2 t)$ for each $(j, t) \in \llbracket 1, n \rrbracket \times \mathbb{R}_{\ge 0}$, the process $\widetilde{\bm{\lambda}} (t) = \big( \widetilde{\lambda}_1 (t), \widetilde{\lambda}_2 (t), \ldots , \widetilde{\lambda}_n (t) \big) \in \overline{\mathbb{W}}_n$ also solves \eqref{lambdaequation}. This again follows from the invariance of the Brownian motions $B_i$ under the same scaling.
	
\end{rem}

We next describe the relation between Dyson Brownian motion, random matrices, and non-intersecting Brownian bridges, to which end we require some additional terminology. A \emph{random matrix} is a matrix whose entries are random variables. The \emph{Gaussian Unitary Ensemble} is an $n \times n$ random Hermitian matrix $\bm{G} = \bm{G}_n$ with random complex entries $\{ w_{ij} \}$ (for $i, j \in \llbracket 1, n \rrbracket$) defined as follows. Its diagonal entries $\{ w_{jj} \}$ are real Gaussian random variables of variance $n^{-1}$; its upper-triangular entries $\{ w_{ij} \}_{i < j}$ are complex random variables, whose real and imaginary parts are independent Gaussian random variables, each of variance $(2n)^{-1}$. These entries are mutually independent, and the lower triangular entries $\{ w_{ij} \}_{i > j}$ are determined from the upper triangular ones by the Hermitian symmetry relation $w_{ij} = \overline{w_{ji}}$.

The \emph{Hermitian Brownian motion} $\bm{G} (t) = \bm{G}_n (t)$ is a stochastic process (over $t \ge 0$) on $n\times n$ random matrices, whose entries $\big\{ w_{ij} (t) \big\}$ are defined as follows. Its diagonal entries $\big\{ w_{jj} (t) \big\}$ are Brownian motions of variance $n^{-1}$, and its upper triangular entries $\big\{ w_{ij} (t) \big\}_{i < j}$ are complex Brownian motions whose real and imaginary parts are independent Brownian motions, each of variance $(2n)^{-1}$. These entries are again mutually independent, and the lower triangular entries $\big\{ w_{ij} (t) \big\}_{i > j}$ are determined from its upper triangular ones by symmetry, $w_{ij} (t) = \overline{w_{ji}} (t)$. Observe that $\bm{G} (1)$ has the same law as a GUE matrix $\bm{G}$.

The following lemma from \cite{MERM} (stated as below in \cite{MCNP}) interprets Dyson Brownian motion in terms of sums of random matrices, and also in terms of non-intersecting Brownian motions conditioned to never intersect (for the definition of the latter in terms of Doob $h$-transforms, see \cite[Section 6.2]{MCNP}).

\begin{lem}[{\cite[Theorems 3 and 4]{MCNP}}]
	
	\label{lambdat}
	
	Fix an integer $n \ge 1$ and a sequence $\bm{\lambda} (0) \in \overline{\mathbb{W}}_n$. For any real number $s > 0$, let $\bm{\lambda} (s) \in \overline{\mathbb{W}}_n$ denote Dyson Brownian motion, run for time $s$, with initial data $\bm{\lambda} (0)$. Further let $\bm{A}$ denote an $n \times n$ diagonal matrix whose eigenvalues are given by $\bm{\lambda} (0)$, and let $\bm{G} (s) = \bm{G}_n (s)$ denote an $n \times n$ Hermitian Brownian motion. 
	
	\begin{enumerate}
		\item The law of the eigenvalues of $\bm{A} + \bm{G} (s)$ coincides with that of $\bm{\lambda} (s)$, jointly over $s \ge 0$. 
		\item Consider $n$ Brownian motions $\bm{x} = (x_1, x_2, \ldots , x_n) \in \llbracket 1, n \rrbracket \times \mathcal{C} (\mathbb{R}_{\ge 0})$, with variances $n^{-1}$ and starting data $\bm{\lambda} (0)$, conditioned to never intersect. Then, $\big( \bm{x}(s) \big)_{s \ge 0}$ has the same law as $\big( \bm{\lambda} (s) \big)_{s \ge 0}$. 
	\end{enumerate}
\end{lem}

We next discuss concentration results from \cite{MCTMGP} for Dyson Brownian motion, to which end we require the notion of a classical location with respect to a measure.

\begin{definition}
	
	\label{gammarho} 
	
	Let $\mu \in \mathscr{P}$ denote a probabilty measure. For any integers $n \ge 1$ and $j \in \mathbb{Z}$, we define the \emph{classical location} (with respect to $\mu$) $\gamma_{j; n} = \gamma_{j; n}^{\mu} \in \mathbb{R}$ by setting
	\begin{flalign*}
		\gamma_{j;n} = \displaystyle\sup \Bigg\{ \gamma \in \mathbb{R} : \displaystyle\int_{\gamma}^{\infty} \mu (dx) \ge  \displaystyle\frac{2j-1}{2n} \Bigg\}, \qquad \text{if $j \in \llbracket 1, n \rrbracket$},
	\end{flalign*}  
	
	\noindent and also setting $\gamma_j = \infty$ if $j \le 0$ and $\gamma_j = -\infty$ if $j > n$.
	
\end{definition}

The following lemma due to\footnote{In \cite{MCTMGP}, the probability on the right side of \eqref{lambdatprobability} was written to be $1 - C n^{-D}$ for any $D > 1$, but it can be seen from the proof (see that of \cite[Proposition 3.8]{MCTMGP}, where $\delta$ there is $\frac{5}{4}$ here) that it can be taken to be $1 - C e^{ - (\log n)^2}$ instead.} \cite{MCTMGP} (together with the scale invariance \Cref{scalemotion}) provides a concentration, or \emph{rigidity}, estimate for the locations of bulk (namely, those sufficiently distant from the first and last) particles under Dyson Brownian motion around the classical locations of a free convolution measure (recall the latter from \Cref{TransformConvolution}).

\begin{lem}[{\cite[Corollary 3.2]{MCTMGP}}]
	
	\label{concentrationequation}
	
	For any real number $A > 1$, there exists a constant $C = C (A) > 1$ such that the following holds. Fix an integer $n \ge 1$ and sequence $\bm{\lambda} (0) \in \overline{\mathbb{W}}_n$ with $-n^A \le \min \bm{\lambda} (0) \le \max \bm{\lambda} (0) \le n^A$. Denote the measure $\mu = n^{-1} \sum_{j=1}^n \delta_{\lambda_j (0)} \in \mathscr{P}$, and set $\mu_t = \mu \boxplus \mu_{\semci}^{(t)}$; also denote the classical locations $\gamma_j (t) = \gamma_{j; n}^{\mu_t} \in \mathbb{R}$. Letting $\bm{\lambda} (s) = \big( \lambda_1 (s), \lambda_2 (s), \ldots , \lambda_n (s) \big) \in \overline{\mathbb{W}}_n$ denote Dyson Brownian motion with initial data $\bm{\lambda} (0)$, we have 
	\begin{flalign}
		\label{lambdatprobability} 
		\mathbb{P} \Bigg[ \bigcap_{j = 1}^n \bigcap_{t \in [0, n^A]} \big\{ \gamma_{j+\lfloor (\log n)^5 \rfloor}(t)-n^{-A} \le \lambda_j (t) \le \gamma_{j- \lfloor (\log n)^5 \rfloor}(t)+n^{-A} \big\} \Bigg] \ge 1 - C e^{ - (\log n)^2}. 
	\end{flalign}
\end{lem}

\subsection{Brownian Watermelon Estimates}

\label{PathsUV0}

In this section we provide estimates for the locations of paths in an ensemble of non-intersecting Brownian bridges conditioned to start and end at $0$. Such ensembles are sometimes referred to as \emph{Brownian watermelons}; see \Cref{f:watermelon}. In what follows, for any integers $n \ge 1$ and $j \in \mathbb{Z}$ we let $\gamma_{\semci; n} (j)$ be the classical location (recall \Cref{gammarho}) with respect to the semicircle law, given by
\begin{flalign}
	\label{gammaj} 
	\gamma_{\semci; n} (j) = \gamma_{j; n}^{\mu_{\semci}}, \quad \text{which satisfies} \quad \displaystyle\frac{1}{2 \pi} \displaystyle\int_{\gamma_{\semci; n} (j)}^2 (4-x^2)^{1/2} dx = \displaystyle\frac{2j-1}{n}, 
	\end{flalign}
whenever $j\in \llbracket 1,n\rrbracket$.

The following lemma then provides a concentration bound for the paths in a Brownian watermelon; its proof is given in \Cref{Proof00} below.

\begin{lem}
	
	\label{estimatexj} 
	
	For any real number $D > 1$, there exist constants $c > 0$ and $C = C (D) > 1$ such that the following holds. Adopt the notation of \Cref{estimatexj2}; assume that $b-a \le n^D$; fix real numbers $u, v \in \mathbb{R}$; and assume that $\bm{u} = (u, u, \ldots , u) \in \overline{\mathbb{W}}_n$ and $\bm{v} = (v, v, \ldots , v) \in \overline{\mathbb{W}}_n$ (where $u$ and $v$ appear with multiplicity $n$).

	\begin{enumerate} 
		
		\item With probability at least $1 - C e^{-c (\log n)^5}$, we have
		\begin{flalign*}
			\displaystyle\max_{j \in \llbracket 1, n \rrbracket}  \displaystyle\sup_{t \in [a, b]} \Bigg|  x_j (t) - \bigg( \displaystyle\frac{(b-t)(t-a)}{(b-a)} & \bigg)^{1/2} \cdot \gamma_{\semci; n} (j) - \displaystyle\frac{b - t}{b-a} \cdot u  - \displaystyle\frac{t-a}{b-a} \cdot v  \Bigg| \\ 
			& \le (\log n)^9 \cdot n^{-2/3} (b-a)^{1/2} \cdot \min \{  j, n-j+1 \}^{-1/3}.
		\end{flalign*} 
		
		\item With probability at least $1 - C e^{-c (\log n)^5}$, we have
		\begin{flalign*}
			& \displaystyle\max_{j \in \llbracket 1, n \rrbracket} \displaystyle\sup_{t \in [a, b]} \Bigg( \bigg| x_j (t) - \displaystyle\frac{b - t}{b-a} \cdot u - \displaystyle\frac{t-a}{b-a} \cdot v \bigg| -   8^{1/2} \bigg( \displaystyle\frac{(b-t)(t-a)}{b-a} \bigg)^{1/2} \Bigg) \le  n^{-D}; \\ 
			&\displaystyle\min_{j \in \llbracket 1, n \rrbracket} \displaystyle\inf_{t \in [a, b]} \Bigg( \bigg| x_j (t) - \displaystyle\frac{b - t}{b-a} \cdot u - \displaystyle\frac{t-a}{b-a} \cdot v \bigg| + 8^{1/2} \bigg( \displaystyle\frac{(b-t)(t-a)}{b-a} \bigg)^{1/2} \Bigg) \ge -n^{-D}.
		\end{flalign*}
	\end{enumerate} 
\end{lem}

\section{Limit Shapes for Non-intersecting Brownian Bridges}

\label{Limit1} 

In this section we consider limit shapes for non-intesection Brownian bridges, in the case without upper and lower boundary, and explain how they solve the partial differential equation \eqref{equationxtd}.

\subsection{Limit Shapes} 

\label{LimitBridges}

In this section we recall results from \cite{LDASI,FOAMI,LDSI} concerning the limiting macroscopic behavior of non-intersecting Brownian bridges under given starting and ending data, with no upper and lower boundaries. Throughout, we use coordinates $(t, x)$, or sometimes $(t, y)$, for $\mathbb{R}^2$. 

For any sequence $\bm{a} = (a_1, a_2, \ldots , a_n) \in \overline{\mathbb{W}}_n$, we denote its \emph{empirical measure} $\emp (\bm{a}) \in \mathscr{P}$ by 
\begin{flalign}
	\label{aemp} 
	\emp (\bm{a}) = \displaystyle\frac{1}{n} \displaystyle\sum_{j=1}^n \delta_{a_j}.
\end{flalign}

\noindent Fix an interval $I \subseteq \mathbb{R}$. A \emph{measure-valued process} (on the time interval $I$) is a family $\bm{\mu} = (\mu_t)_{t \in I}$ of measures $\mu_t \in \mathscr{P}_{\fin}$ for each $t \in I$. Given a real number $A > 0$, we say that $\bm{\mu}$ has \emph{constant total mass $A$} if $\mu_t (\mathbb{R}) = A$, for each $t \in I$. If $\bm{\mu}$ has constant total mass $1$ (so each $\mu_t \in \mathscr{P}$), we call $\bm{\mu}$ a \emph{probability measure-valued process}. Measure-valued processes can be interpreted as elements of $I \times \mathscr{P}_{\fin}$ and probability measure-valued processes as ones of $I \times \mathscr{P}$. We let $\mathcal{C} ( I; \mathscr{P}_{\fin})$ and $\mathcal{C} ( I; \mathscr{P})$ denote the sets of measure-valued processes and probability measure-valued processes that are continuous in $t \in I$, under the topology of weak convergence on $\mathscr{P}_{\fin}$ and $\mathscr{P}$, respectively. 

Given two measures $\mu, \nu \in \mathscr{P}_{\fin}$ of finite total mass, the L\'{e}vy distance between them is 
\begin{flalign}
	\label{munudistance1}
	d_{\dL} (\mu, \nu) = \inf \Bigg\{ a > 0 : \displaystyle\int_{-\infty}^{y-a} \mu (dx) - a \le \displaystyle\int_{-\infty}^y \nu (dx) \le \displaystyle\int_{-\infty}^{y+a} \mu (dx) + a, \quad \text{for all $y \in \mathbb{R}$} \Bigg\}.
\end{flalign}

\noindent Given an interval $I \subseteq \mathbb{R}$ and two measure-valued processes $\bm{\mu} = (\mu_t)_{t \in I} \in I \times \mathscr{P}_{\fin}$ and $\bm{\nu} = (\nu_t)_{t \in I} \in I \times \mathscr{P}_{\fin}$ on the time interval $I$, the L\'{e}vy distance between them is defined to be 
\begin{flalign*} 
	d_{\dL} (\bm{\mu}, \bm{\nu}) = \sup_{t \in I} d_{\dL} (\mu_t, \nu_t).
\end{flalign*}

The following lemma from \cite{LDSI} (based on results from \cite{FOAMI,LDASI}) states that, as $n$ tends to $\infty$, the empirical measure for $n$ non-intersecting Brownian bridges (whose starting and ending data converge in a certain way) has a limit; see \Cref{f:watermelon} in the case of the Brownian watermelon, whose limit shape is shown on the left side of 
\Cref{f:density}. The following lemma was stated in \cite{LDSI,LDASI} in the case when $[a, b] = [0, 1]$ and $A = 1$ but, by the scaling invariance (\Cref{scale}) for non-intersecting Brownian bridges, it also holds for any interval $[a, b]$ and real number $A > 0$, as below. 

\begin{figure}
\centering

  \includegraphics[width=0.6\linewidth]{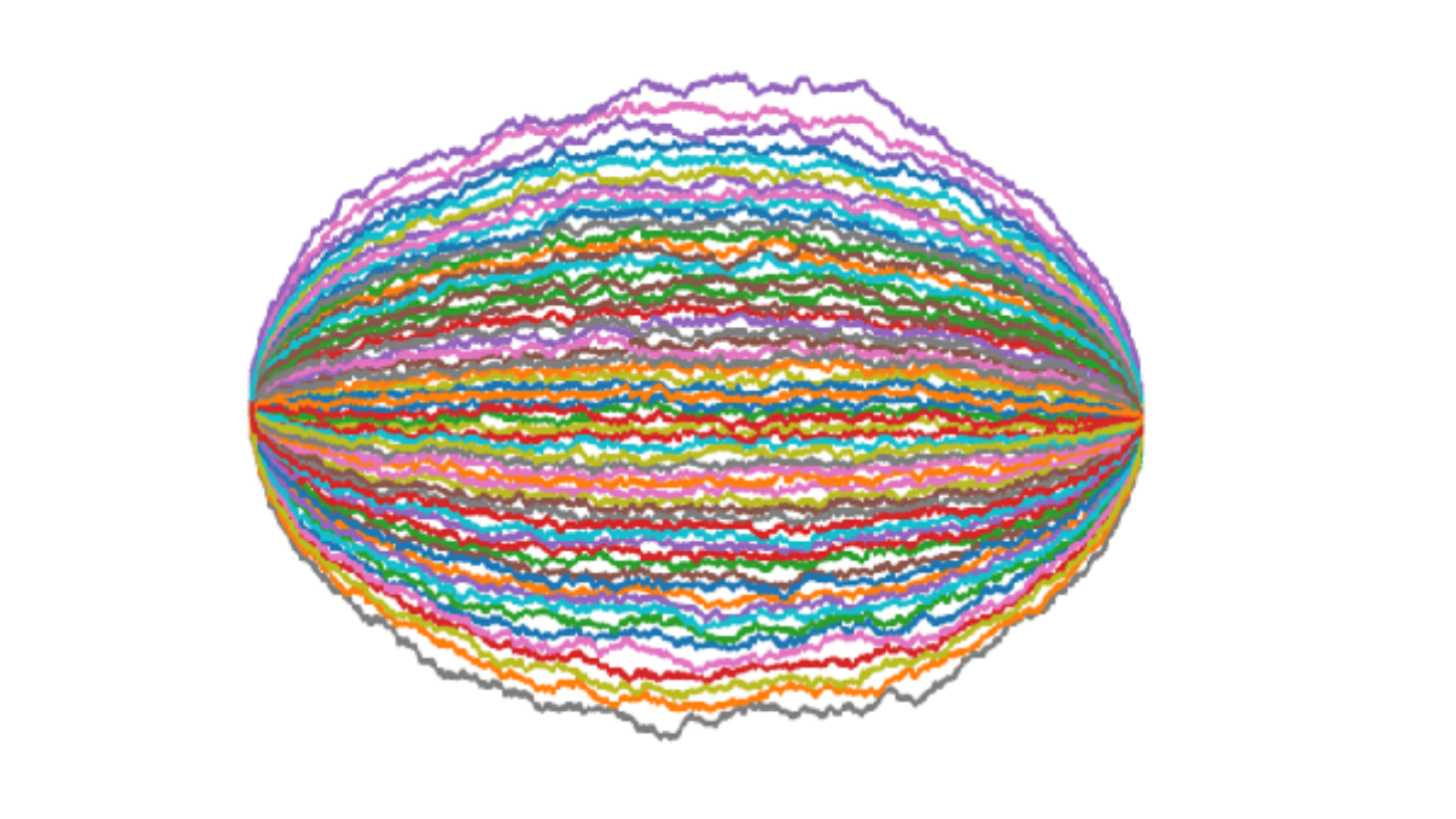}

\caption{Shown above is a Brownian watermelon, consisting of non-intersecting Brownian bridges starting and ending at $0$.}
\label{f:watermelon}
\end{figure}

\begin{lem}[{\cite[Claim 2.13]{LDSI}}]
	
	\label{rhot} 
	
	Fix real numbers $a < b$ and compactly supported measures $\mu_a, \mu_b \in \mathscr{P}_{\fin}$, both of total mass $\mu_a (\mathbb{R}) = A = \mu_b (\mathbb{R})$ for some real number $A > 0$. There is a measure-valued process $\bm{\mu}^{\star} = (\mu_t^{\star})_{t \in [a, b]} \in \mathcal{C} \big( [a, b]; \mathscr{P}_{\fin} \big)$ on $[a, b]$ of constant total mass $A$, which is continuous in the pair $(\mu_a, \mu_b) \in \mathscr{P}^2$ under the L\'{e}vy metric, such that the following holds. For each integer $n \ge 1$, let $\bm{u} = \bm{u}^n \in \overline{\mathbb{W}}_n$ and $\bm{v} = \bm{v}^n \in \overline{\mathbb{W}}_n$ denote sequences such that $A \cdot \emp(\bm{u}^n)$ and $A \cdot \emp(\bm{v}^n)$ converge to $\mu_a$ and $\mu_b$ under the L\'{e}vy metric as $n$ tends to $\infty$, respectively. Sample $n$ non-intersecting Brownian bridges $\bm{x}^n = (x_1^n, x_2^n, \ldots , x_n^n) \in \llbracket 1, n \rrbracket \times \mathcal{C} \big( [a, b] \big)$ from $\mathfrak{Q}^{\bm{u}; \bm{v}} (An^{-1})$; for any $t \in [a, b]$, denote $\nu_t^n = A \cdot \emp \big(\bm{x}^n (t) \big) \in \mathscr{P}$. For any real number $\varepsilon > 0$, we have $\lim_{n \rightarrow \infty} \mathbb{P} \big[ d_{\dL} (\bm{\nu}^n, \bm{\mu}^{\star}) > \varepsilon \big] = 0$.

\end{lem} 

Terminology for the limit shape provided by \Cref{rhot} is given through the following definition.

\begin{definition}
	
	\label{mutmu0mu1}
	
	Adopting the notation of \Cref{rhot}, the measure-valued process $\bm{\mu}^{\star} = (\mu_t^{\star})_{t \in [a, b]}$ is called the \emph{bridge-limiting measure process} (on the interval $[a, b]$) with boundary data $(\mu_a; \mu_b)$.
	
\end{definition} 

The following lemma indicates how bridge-limiting measure processes restrict to others. Its proof is given in \Cref{ProofContinuous1} below.

\begin{lem}
	
	\label{gtabab}
	
	Adopt the notation and assumptions of \Cref{rhot}, and let $\widetilde{a}, \widetilde{b} \in \mathbb{R}$ be real numbers such that $a \le \widetilde{a} < \widetilde{b} \le b$. Then, the bridge-limiting measure process on the interval $[\widetilde{a}, \widetilde{b}]$ with boundary data $(\mu_{\widetilde{a}}^{\star}; \mu_{\tilde{b}}^{\star})$ is given by $(\mu_t^{\star})_{t \in [\tilde{a}, \tilde{b}]}$.
\end{lem}

We will in many cases make use of a height function and inverted height functions associated with a measure-valued process, defined below.

\begin{definition}
	
	\label{hrhot} 
	
	Fix an interval $I = [a, b] \subseteq \mathbb{R}$ and a measure-valued process $\bm{\mu} = (\mu_t)_{t \in I}$ of constant total mass $A > 0$. The \emph{height function associated with $\bm{\mu}$} is defined to be the function $H = H^{\bm{\mu}} : I \times \mathbb{R} \rightarrow [0, A]$ obtained by setting
	\begin{flalign}
		\label{htxintegral}
		H (t, x) = \displaystyle\int_x^{\infty} \mu_t (dw), \qquad \text{for each $(t, x) \in I \times \mathbb{R}$}.
	\end{flalign}
	
	\noindent The \emph{inverted height function} associated with $\bm{\mu}$ is $G = G^{\bm{\mu}}: I \times [0, A] \rightarrow \mathbb{R}$, defined by setting $G(t,0) = \inf \big\{ x : H(t,x) = 0 \big\}$ and 
	\begin{flalign}
		\label{gty}
		G(t, y) = x, \quad \text{where} \quad x = x(t, y) = \sup \big\{ x \in \mathbb{R} : H(t, x) \ge y \big\}, \qquad \text{for $y \in (0, A]$}.
	\end{flalign}

	If $\mu_t = \varrho_t (x) dx$ has a density with respect to Lebesgue measure for each $t \in (a, b)$, then we sometimes associate $H$ (or $G$) with $\varrho = (\varrho_t)$. Moreover, if $\bm{\mu}$ is the bridge-limiting measure process with boundary data $(\mu_a; \mu_b)$ we say that $H$ (or $G$) is associated with boundary data $(\mu_a; \mu_b)$.
\end{definition}

To simplify notation in what follows, it will often be useful to set $(a, b) = (0, 1)$ and $A = 1$. The following lemma indicates a transformation that implements this specialization. Its proof follows from a quick consequence of \Cref{scale} and is provided in \Cref{ProofContinuous1} below.

\begin{lem} 
	
	\label{aab} 
	
	Adopt the notation and assumptions of \Cref{rhot}; let $H^{\star} : [a, b] \times \mathbb{R} \rightarrow \mathbb{R}$ and $G^{\star} : [a, b] \times [0, A] \rightarrow \mathbb{R}$ denote the height function and inverted height function associated with $\bm{\mu}^{\star}$, respectively. Define the probability measures $\widetilde{\mu}_0, \widetilde{\mu}_1$ by for any interval $I \subseteq \mathbb{R}$ setting
	\begin{flalign*}
		\widetilde{\mu}_0 (I) = A^{-1} \cdot \mu_a  \big( A^{1/2} (b-a)^{1/2} \cdot I \big); \qquad \widetilde{\mu}_1 (I) = A^{-1} \cdot \mu_b \big( A^{1/2} (b-a)^{1/2} \cdot I \big)
	\end{flalign*} 
	
	\noindent Further define the probability measure-valued process $\widetilde{\bm{\mu}}^{\star} = (\widetilde{\mu}_t^{\star})_{t \in [0, 1]} \in \mathcal{C} \big( [0, 1]; \mathscr{P} \big)$ and functions $\widetilde{H}^{\star} : [0, 1] \times \mathbb{R} \rightarrow \mathbb{R}$ and $\widetilde{G}^{\star} : [0, 1] \times [0, 1] \rightarrow \mathbb{R}$ by setting
	\begin{flalign}
		\label{muhgaab}
		\begin{aligned}
			& \widetilde{\mu}_t^{\star} (I) = A^{-1} \cdot \mu_{a + t(b-a)}^{\star} \big( A^{1/2} (b-a)^{1/2} \cdot I \big), \qquad \qquad \quad \text{for any interval $I \subseteq \mathbb{R}$}; \\
			& \widetilde{H}^{\star} (t, x) = A^{-1} \cdot H^{\star} \big( a + t(b-a), A^{1/2} (b-a)^{1/2} x \big), \qquad \text{for any $(t, x) \in [0, 1] \times \mathbb{R}$}; \\
			& \widetilde{G}^{\star} (t, y) = A^{-1/2} (b-a)^{-1/2} \cdot G^{\star} \big( a + t(b-a), Ay \big), \qquad \text{for any $(t, y) \in [0, 1] \times [0, 1]$}.
		\end{aligned}	
	\end{flalign}
	
	\noindent Then, $\widetilde{\bm{\mu}}^{\star}$ is the bridge-limiting measure process on $[0, 1]$ with boundary data $(\widetilde{\mu}_0; \widetilde{\mu}_1)$. Moreover, its associated height function and inverted height function are $\widetilde{H}^{\star}$ and $\widetilde{G}^{\star}$, respectively. 
	
\end{lem}

The following corollary reformulates \Cref{rhot} in terms of the associated (inverted) height functions; in particular, it explains that the trajectory of the $\lfloor yn \rfloor$-th path is approximated by the inverted height function $G^{\star}$ associated with the bridge-limiting process. In what follows, given an integer $n \ge 1$, real numbers $a < b$, and a family of non-intersecting curves $\bm{x} = (x_1, x_2, \ldots , x_n) \in \llbracket 1, n \rrbracket \times \mathcal{C} \big( [a, b] \big)$, the associated (discrete) height function $\mathsf{H} = \mathsf{H}^{\bm{x}} : [a, b] \times \mathbb{R} \rightarrow \mathbb{R}$ is defined by 
\begin{flalign}
	\label{htx}
	\mathsf{H} (t, x) = \mathsf{H}^{\bm{x}} (t, x) = \# \big\{ j \in \llbracket 1, n \rrbracket : x_j (t) > x \big\}.
\end{flalign}

\begin{cor}
	
	\label{xghestimate}
	
	Adopt the notation and assumptions of \Cref{rhot}, and fix a real number $\varepsilon > 0$. Let $H^{\star} : [a, b] \times \mathbb{R} \rightarrow [0, A]$ and $G^{\star} : [a, b] \times [0, A] \rightarrow \mathbb{R}$ denote the height and inverted height functions associated with $\bm{\mu}^{\star}$, respectively. We have
	\begin{align}
		\label{munudistancex2}
		\begin{aligned}
			& \displaystyle\lim_{n \rightarrow \infty} \mathbb{P} \Bigg[ \bigcap_{(t, w) \in [a,b] \times \mathbb{R}} \big\{ A^{-1} H^{\star} (t, w + \varepsilon) - \varepsilon \le n^{-1} \mathsf{H}^{\bm{x}} (t, w) \le A^{-1} H^{\star} (t, w - \varepsilon) + \varepsilon \big\} \Bigg] = 1;  \\
			&  \displaystyle\lim_{n \rightarrow \infty} \mathbb{P} \Bigg[ \bigcap_{(t, y) \in [a,b] \times [0, 1]} \big\{ G^{\star} (t, Ay + \varepsilon) - \varepsilon \le x_{\lfloor yn \rfloor} (t) \le G^{\star} (t, Ay - \varepsilon) + \varepsilon \big\} \Bigg] = 1.
		\end{aligned}	
	\end{align}
\end{cor}

\begin{proof}
	By \Cref{rhot}, $d_{\dL}( \bm{\nu}^n, \bm{\mu}^{\star})$ tends to $0$ in probability, as $n$ tends to $\infty$. Combined with the definition \eqref{munudistance1} of L{\'e}vy metric, this yields the first statement of \eqref{munudistancex2}. This, with the definition \eqref{gty} of $G^{\star}$ in terms of $H^{\star}$ and that \eqref{htx} of $\mathsf{H}^{\bm{x}}$ in terms of $\bm{x}$, gives the second statement of \eqref{munudistancex2}.
\end{proof}

The following lemma essentially due to \cite{FOAMI} indicates that the measures $\mu_t^{\star}$ have a density, and it also discusses properties of this density. Its proof is given in \Cref{ProofContinuous1} below. In what follows, we recall the free convolution and semicircle law $\mu_{\semci}^{(t)} \in \mathscr{P}_0$ from \Cref{TransformConvolution}. We also restrict to the case $(a, b) = (0, 1)$ and $A = 1$ to simplify the statement, though an analog of the below lemma can be derived under arbitrary such parameters by scaling (see \Cref{aab1} below).

\begin{lem} 
	
	\label{mutrhot}

	Adopt the notation and assumptions of \Cref{rhot}, and assume $(a, b) = (0, 1)$ and $A = 1$. The following statements hold for each real number $t \in (0, 1)$. 
	
	\begin{enumerate} 
		\item There  exists a measurable function $\varrho_t^{\star} : \mathbb{R} \rightarrow \mathbb{R}_{\ge 0}$ such that $\mu_t^{\star} (dx) = \varrho_t^{\star} (x) dx$. 
		\item There exists some compactly supported probability measure $\nu_t \in \mathscr{P}_0$, dependent on $\mu_0$ and $\varrho_1$ with $\supp \nu_t \subseteq \supp \mu_0 + \supp \mu_1$, such that $\varrho_t^{\star}=\nu_t\boxplus \mu_{\semci}^{(t-t^2)}$.
		\item We have $\varrho_t^{\star} (x) \le (t-t^2)^{-1/2}$, for any $x \in \mathbb{R}$. Moreover, for any integer $k \ge 1$, there exists a constant $C = C (k) > 1$ such that, for any real number $\delta > 0$ and point $(t, x_0) \in \Omega$ with $\varrho_t^{\star} (x_0) \ge \delta$, we have $\big| \partial_x^k \varrho_t^{\star} (x_0) \big| \le C \delta^{-6k} (t-t^2)^{-4k}$.
		\item The function $\varrho_t (x)$ is continuous on $(0, 1) \times \mathbb{R}$.
	\end{enumerate}
	
\end{lem}

Next, fix an interval $(a, b) \subseteq \mathbb{R}$ and a family of measures $\bm{\mu} = (\mu_t)_{t \in (a, b)} \in (a, b) \times \mathscr{P}_{\fin}$ of constant total mass $A > 0$. Assume for each $t \in (a, b)$ that each $\mu_t$ has a density $\varrho_t (x)$,  for some function $\varrho_t : \mathbb{R} \rightarrow \mathbb{R}_{\ge 0}$ that is continuous in $(t, x)$. We define the associated \emph{liquid region} $\Omega \subset (a, b) \times \mathbb{R}$ and \emph{inverted liquid region} $ \Omega^{\inv} \subseteq (a, b) \times (0, 1)$ by
\begin{flalign}
	\label{omega12}
	\Omega= \big\{(t, x)\in (a, b) \times \mathbb{R}: \varrho_t (x)>0 \big\}; \quad \Omega^{\inv} = \big\{ (t, y) \in (a, b) \times [0, A] : y = H^{\bm{\mu}} (t,x), (t, x) \in \Omega \big\}.
\end{flalign} 

\noindent Observe that the map $(t, x) \mapsto \big( t, H(t, x) \big)$ is a bijection to from $\Omega$ to $\Omega^{\inv}$. Moreover, the continuity of $\varrho_t (x)$ implies that the set $\Omega$ is open, which implies that $\Omega^{\inv}$ is also open. See \Cref{f:liquidregion}.

\begin{figure}
\center
\includegraphics[width=0.9\textwidth]{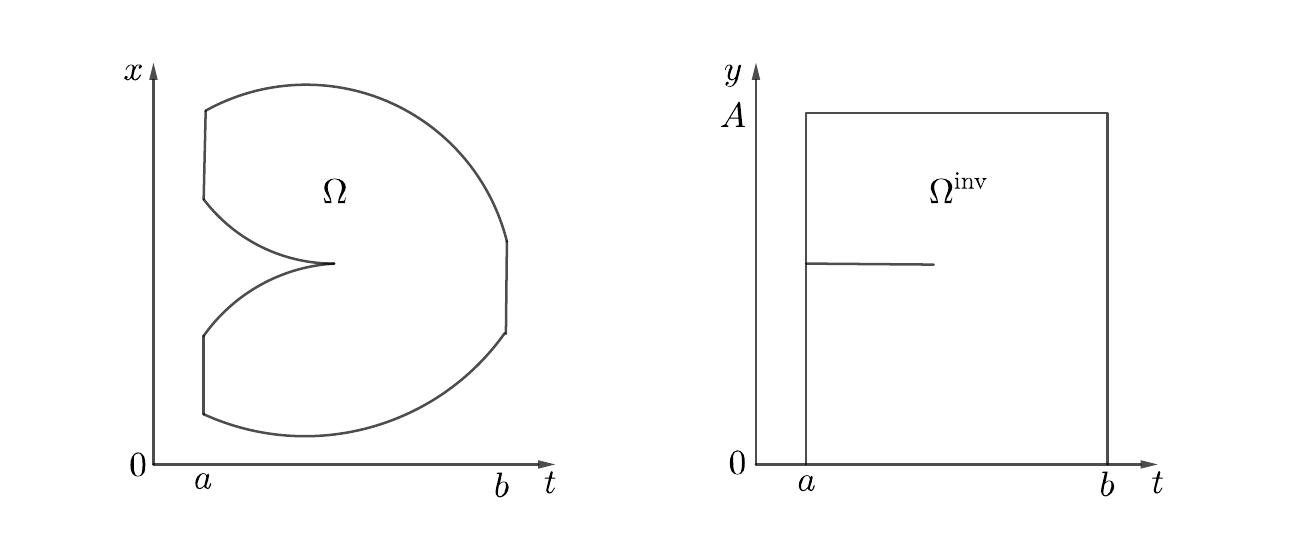}
\caption{The left panel is the liquid region $\Omega$ and the right panel is the inverted liquid region $\Omega^{\inv}$.}
\label{f:liquidregion}
\end{figure}

\begin{cor}
	
	\label{convergepathomega}

	Under the assumptions of \Cref{xghestimate}, for any $y \in (0, 1)$ such that $(t, Ay)\in \Omega^{\inv}$ for each $t \in (a, b)$, we have that $G^{\star} (t, Ay)$ is continuous in $t \in [a, b]$, and
	\begin{align}
		\label{taybg}
		& \displaystyle\lim_{n \rightarrow \infty} \mathbb{P} \Bigg[ \bigcap_{t \in (a, b)} \big\{ G^{\star} (t, Ay) - \varepsilon \le x_{\lfloor yn \rfloor}^n (t) \le G^{\star} (t, Ay) + \varepsilon \big\} \Bigg] = 1.  
	\end{align}
\end{cor}

\begin{proof}
	
	Observe that $G^{\star}$ is continuous on $\Omega^{\inv}$, by \eqref{gty} and the fact that $H(t, x)$ is continuous and strictly increasing in $x$ whenever $(t, x) \in \Omega$. In particular, $G^{\star}$ is uniformly continuous on compact subsets of $\Omega^{\inv}$; this, together with the second statement of \eqref{munudistancex2}, yields for any real number $\delta \in \big( 0, \frac{1}{2} \big)$ that, with probabiliy $1 - o(1)$ (that is, tending to $1$ as $n$ tends to $\infty$), we have 
	\begin{flalign}
		\label{ftxwnt2}
		 \displaystyle\sup_{t \in [\delta, 1-\delta]} \big| x_{\lfloor yn \rfloor} (t) - G^{\star} (t, Ay) \big| < \delta.
	\end{flalign}
	
	Next observe that there is a constant $C_1 = C_1 (y, \mu_0, \mu_1) > 1$ such that, with probability $1 - o(1)$, we have
	\begin{flalign}
		\label{xwnxwn2} 
		& \displaystyle\sup_{t \in [0, \delta]} \big| x_{\lfloor yn \rfloor} (t) - G^{\star} (\delta, Ay) \big| \le C_1 \delta^{1/3}; \qquad \displaystyle\sup_{t \in [1-\delta, 1]} \big| x_{\lfloor yn \rfloor} (b) - G^{\star} (1-\delta, Ay) \big| \le C_1 \delta^{1/3},
	\end{flalign}
	
	\noindent where we used \eqref{ftxwnt2} and \Cref{estimatexj2}. Letting $n$ tend to $\infty$ in \eqref{xwnxwn2} and applying \eqref{ftxwnt2}, we obtain 	
	\begin{flalign*}
		& \displaystyle\sup_{t, t' \in [0, \delta]} \big| G^{\star} (t, Ay) - G^{\star} (t', Ay) \big| \le 2 C_1 \delta^{1/3}; \qquad \displaystyle\sup_{t, t' \in [1-\delta, 1]} \big| G^{\star} (t, Ay) - G^{\star} (t', Ay) \big| \le 2 C_1 \delta^{1/3},
	\end{flalign*}
	
	\noindent This, with the continuity of $G^{\star}$ on compact subsets of $\Omega^{\inv}$, implies that $G^{\star} (t, Ay)$ is continuous in $t \in [a, b]$. Moreover, with \eqref{xwnxwn2} and \eqref{ftxwnt2}, it also yields \eqref{taybg} by taking $\delta$ sufficiently small with respect to $\varepsilon$. 	
\end{proof}

Before proceeding, let us provide two examples of the above notions. The first concerns the case when $\mu_0$ and $\mu_1$ are delta measures, in which the associated non-intersecting Brownian bridges form a Brownian watermelon (recall \Cref{PathsUV0}); see the left side of \Cref{f:density}.

\begin{example}
	
	\label{xuv0} 
	
	Fix real numbers $a < b$; $u, v \in \mathbb{R}$; and $A > 0$. Assume that $(\mu_a, \mu_b) = (A\cdot \delta_u, A \cdot \delta_v)$, where $\delta_x \in \mathscr{P}_0$ denotes the delta measure at $x \in \mathbb{R}$. Then, it follows from \Cref{estimatexj} (multiplying its results by $A^{1/2}$ to account for the fact that the Brownian motions have variance $An^{-1}$ here) and the second statement of \eqref{munudistancex2} (and also the continuity of $\gamma_{\semci} (y)$ below in $y$) that the inverted height function $G^{\star} : [0, 1] \times [0, A] \rightarrow \mathbb{R}$ associated with boundary data $(\mu_0; \mu_1)$ is given by
	\begin{flalign*}
		G^{\star} (t, y) = \bigg( \displaystyle\frac{A (b-t) (t-a)}{b-a} \bigg)^{1/2} \cdot \gamma_{\semci} \Big( \displaystyle\frac{y}{A} \Big) + \displaystyle\frac{b-t}{b-a} \cdot u + \displaystyle\frac{t-a}{b-a} \cdot v, 
	\end{flalign*} 
	
	\noindent  where $\gamma_{\semci} (y)$ is the classical location of the semicircle law, defined to be the 
	\begin{flalign} 
		\label{gammascy} 
		\text{unique} \quad \gamma \in [-2, 2] \quad \text{solving the equation} \quad (2\pi)^{-1} \int_{\gamma}^2 (4-x^2)^{1/2} dx = y.
	\end{flalign} 
	
	\noindent Together with \eqref{htxintegral} and \eqref{gty}, it follows that the associated density process $(\varrho_t^{\star})$ and the height function $H^{\star} : [0, 1] \times \mathbb{R}$ are given by
	\begin{flalign*}
		& \varrho_t^{\star} (x) = A \cdot \varrho_{\semci}^{(\frac{A (b-t)(t-a)}{(b-a)})}  \bigg(x - \displaystyle\frac{b-t}{b-a} \cdot u - \displaystyle\frac{t-a}{b-a} \cdot v \bigg),
	\end{flalign*} 
	
	\noindent and $H^{\star} (t, x) = \int_x^{\infty} \varrho_t^{\star} (y) dy$, where we recall the rescaled semcircle density $\varrho_{\semci}^{(t)}$ from \eqref{rhosct}.
	
\end{example}

The second example concerns the case when $\mu_0$ and $\mu_1$ are rescaled semicircle laws (recall \eqref{rhosct}), which is obtained by restricting a larger watermelon to a smaller interval. See the right side of \Cref{f:density}.

\begin{figure}
\centering
\begin{subfigure}{.5\textwidth}
  \centering
  \includegraphics[width=1\linewidth]{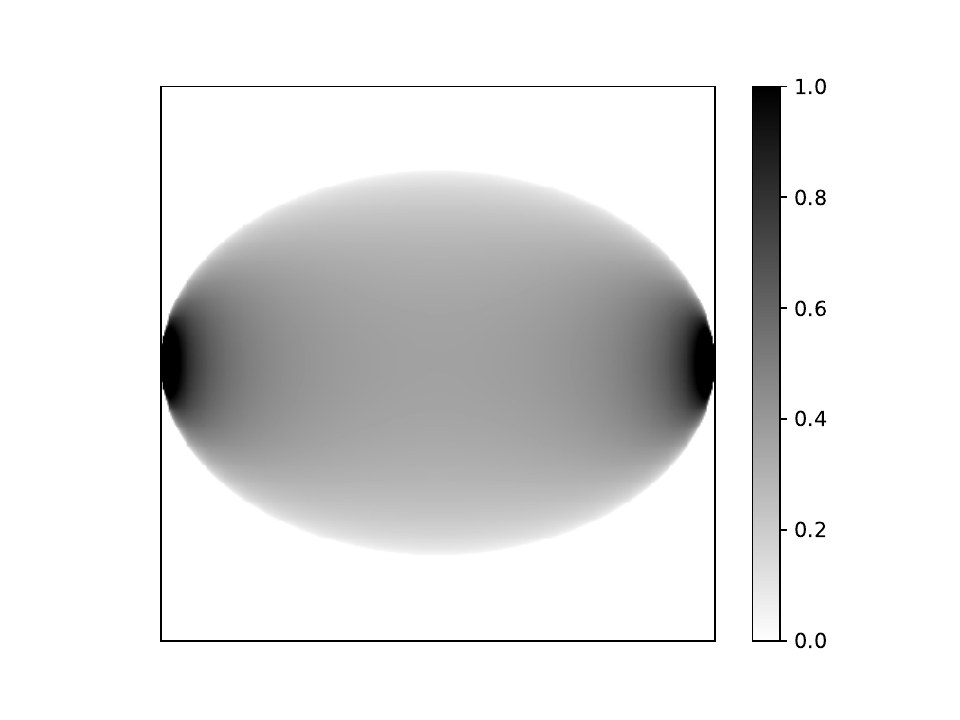}
\end{subfigure}%
\begin{subfigure}{.5\textwidth}
  \centering
  \includegraphics[width=1\linewidth]{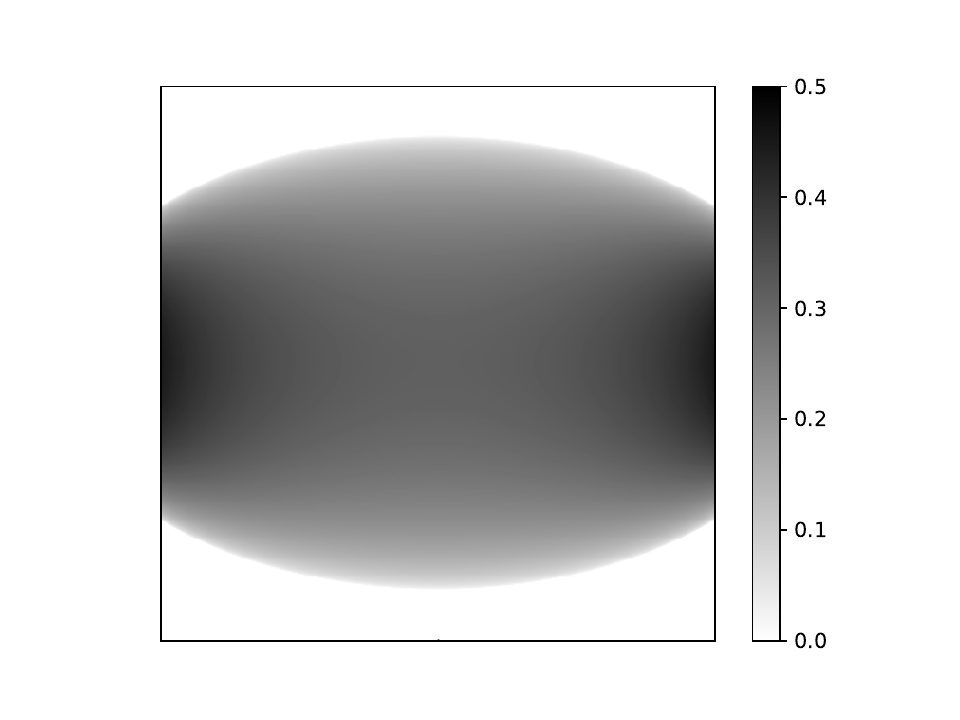}
\end{subfigure}
\caption{The left panel is the density from \Cref{xuv0} with $(a,b,u,v,A)=(0,6,0,0,2)$, and the whole shaded region is the liquid region $\Omega$.  The right panel is the density from \Cref{rhoasct} with $(a,b,d, A)=(0,10,2, 2)$, and the whole shaded region is the liquid region $\Omega$.}
\label{f:density}
\end{figure}

\begin{cor}
	
	\label{rhoasct}
	
	Fix real numbers $a < b$, $d > 0$, and $A > 0$; assume that $\mu_a = A \cdot \mu_{\semci}^{(d)} = \mu_b$. Then, the inverted height function $G^{\star} : [a, b] \times [0,A] \rightarrow \mathbb{R}$ and density process $(\varrho_t^{\star})$ associated with boundary data $(\mu_a; \mu_b)$ are given by 
	\begin{flalign}
		\label{rhoa} 
		\varrho_t^{\star} (x) = A \cdot \varrho_{\semci}^{(d + \frac{A(b-t)(t-a)}{b-a+2\kappa})} (x); \qquad G^{\star} (t, y) = \bigg( d + \displaystyle\frac{A(b-t)(t-a)}{b-a + 2\kappa} \bigg)^{1/2} \cdot \gamma_{\semci} \Big( \displaystyle\frac{y}{A} \Big),
	\end{flalign} 
	
	\noindent where $\kappa = \kappa (a, b, d) > 0$ is defined by
	\begin{flalign}
		\label{kappa} 
		\kappa = \displaystyle\frac{d}{A} + \displaystyle\frac{a-b}{2} + \bigg( \Big( \displaystyle\frac{b-a}{2} \Big)^2 + \Big( \displaystyle\frac{d}{A} \Big)^2 \bigg)^{1/2}.
	\end{flalign}
\end{cor}

\begin{proof}
	
	This will follow from using \Cref{gtabab} to restrict the limit shape from \Cref{xuv0} to a smaller interval. To implement this, let $(\widetilde{\mu}_t^{\star})_{t \in [a-\kappa, b+\kappa]}$ denote the bridge-limiting measure process on $[a-\kappa, b+\kappa]$ with boundary data $(\delta_0, \delta_0)$. Then, for any $t \in [a-\kappa, b+\kappa]$, we have
	\begin{flalign} 
		\label{mutkappa} 
		\widetilde{\mu}_t^{\star} = A \cdot \mu_{\semci}^{(\frac{A (b+\kappa-t)(t-a+\kappa)}{b-a+2\kappa})}; \qquad \displaystyle\frac{A (b+\kappa -t)(t-a+\kappa)}{b-a+2\kappa} = d + \displaystyle\frac{A (b-t)(t-a)}{b-a+2\kappa},
	\end{flalign}
	
	\noindent where the first statement follows from \Cref{xuv0} and the second from the choice \eqref{kappa} of $\kappa$. It follows that $\widetilde{\mu}_a^{\star} = \mu_{\semci}^{(d)} = \widetilde{\mu}_b^{\star}$, and so from \Cref{gtabab} we find $\mu_t^{\star} = \widetilde{\mu}_t^{\star}$ for each $t \in [a, b]$. Applied with \eqref{mutkappa}, this yields the first statement of \eqref{rhoa}, which (together with \eqref{htxintegral}, \eqref{gty}, and \eqref{gammascy}) implies its second statement.
\end{proof}

\subsection{The Functions $(u^{\star}, \varrho^{\star})$}

\label{EquationLimitrhou}

In this section we recall results from \cite{FOAMI,LDASI} stating that, under conditions on $\mu_a$ and $\mu_b$, the associated bridge-limiting density process $\mu_t$ gives rise to a pair of functions $(u_t^{\star}, \varrho_t^{\star})$ that satisfy certain equations (which will later be interpreted in terms of elliptic partial differential equations for the associated height and inverted height functions in \Cref{EquationHG} below). For notational convenience, we at first restrict to the case of \Cref{rhot} when $(a, b) = (0, 1)$ and $A = 1$ (explaining how the more general case follows from it through scaling in \Cref{aab1} below).

The following result appears as stated below\footnote{Let us mention that, as stated, those results apply to Dyson Brownian motion \eqref{lambdaequation} started at some $\bm{u} \in \overline{\mathbb{W}}_n$, conditioned to be $\bm{v} \in \overline{\mathbb{W}}_n$ at time $1$. By \Cref{lambdat}, this is equivalent to non-intersecting Brownian bridges of variance $n^{-1}$, starting at $\bm{u}$ and ending at $\bm{v}$.} in \cite{LDSI} and was originally established in \cite{FOAMI,LDASI} (see\footnote{An alternative proof for the smoothness of $u^{\star}$ and $\varrho^{\star}$ (the second part of \Cref{t:LDP}) would proceed by first using \Cref{hequation} (see also \Cref{rhouderivative}) below to show that the associated height function $H^{\star}$ weakly solves the divergence-form elliptic partial differential equation \eqref{tth}. The bounds of \cite[Corollary 2.8(c)]{FOAMI} verify that this equation is elliptic (using \Cref{eigenvalues2} below), uniformly on compact subsets of $\Omega$. Together with regularity estimates for weak solutions to such elliptic differential equations given by \cite[Lemma 6.27, Theorem 8.20, Equation (13.41)]{EDSO} and \cite[Chapter 3.6]{EDE}, this yields the smoothness of $H^{\star}$, implying that of $u^{\star}$ and $\varrho^{\star}$ by \Cref{hurho} below.} \cite[Theorem 2.1]{FOAMI}). Here, we recall from \Cref{mutrhot} that, at a given time $t \in (0, 1)$, any bridge-limiting measure process $\mu^{\star}$ has a density $\varrho_t^{\star} : \mathbb{R} \rightarrow \mathbb{R}_{\ge 0}$; we also recall from \eqref{omega12} the liquid region $\Omega$ associated with any $\varrho \in (0, 1) \times \mathbb{R} \rightarrow \mathbb{R}_{\ge 0}$.

\begin{lem}[{\cite[Theorems 2.3 and 2.11]{LDSI}}]
	
	\label{t:LDP}
	
	Fix probability measures $\mu_0, \mu_1 \in \mathscr{P}$ with densities $\varrho_0, \varrho_1$, respectively, such that 
	\begin{flalign} 
		\label{rho1rho2} 
		\text{$\varrho_0$ and $\varrho_1$ are compactly supported}; \qquad \sup_{x \in \mathbb{R}} \big| \varrho_0(x) \big| < \infty; \qquad \sup_{x \in \mathbb{R}} \big| \varrho_1 (x) \big| < \infty.
	\end{flalign} 
	
	\noindent Let $(\mu_t^{\star})_{t \in [0, 1]} \in \mathcal{C} \big( [0, 1]; \mathscr{P} \big)$ denote the bridge-limiting measure process on $[0, 1]$ with boundary data $(\mu_0; \mu_1)$. For each $t \in [0, 1]$, denote the density of $\mu_t^{\star} \in \mathscr{P}$ by $\varrho_t^{\star} : \mathbb{R} \rightarrow \mathbb{R}_{\ge 0}$. There exists a measurable function $u^{\star} : [0, 1] \times \mathbb{R} \rightarrow \mathbb{R}$, denoted by $u_t^{\star} (x) = u^{\star} (t, x)$, satisfying the following three properties.
	
	\begin{enumerate} 
		\item We have $\partial_t \varrho_t^{\star}(x) + \partial_x \big( \varrho_t^{\star}(x) u_t^{\star} (x) \big) = 0$ weakly (in the sense of distributions) on $(0, 1) \times \mathbb{R}$. More specifically, for any smooth, compactly supported function $\varphi : (0, 1) \times \mathbb{R} \rightarrow \mathbb{R}$, we have
		\begin{flalign*} 
			\displaystyle\int_0^1 \displaystyle\int_{-\infty}^{\infty} \varrho_t^{\star}(x) \partial_t \varphi (t, x) dx dt + \displaystyle\int_0^1 \displaystyle\int_{-\infty}^{\infty} \varrho_t^{\star}(x) u_t^{\star} (x) \partial_x \varphi (t, x) dx dt = 0.
		\end{flalign*} 
		\item The functions $\varrho^{\star}$ and $u^{\star}$ are smooth in $(t, x)$ on the liquid region $\Omega$ associated with $\varrho^{\star}$.
		\item For each $(t, x) \in \Omega$, the pair $(u^{\star}, \varrho^{\star})$ satisfies 
		\begin{flalign}
			\label{equationrhou}
			\partial_t \varrho_t^{\star}(x) + \partial_x \big( u_t^{\star} (x) \varrho_t^{\star}(x) \big) = 0; \qquad \partial_t \big( u_t^{\star} (x) \varrho_t^{\star}(x) \big) + \partial_x \Big( u_t^{\star} (x)^2 \varrho_t^{\star}(x) - \displaystyle\frac{\pi^2}{3} \varrho_t^{\star}(x)^3 \Big) = 0.
		\end{flalign} 
	\end{enumerate}
\end{lem}

Although we will not make substantial use of it, the next lemma from \cite{LDSI} (we recall its quick proof in \Cref{ProofDerivativeH} below) reformulates \eqref{equationrhou} as an equation for the associated \emph{complex slope} $f = f^{(\mu_0; \mu_1)} : (0, 1) \times \mathbb{R} \rightarrow \overline{\mathbb{H}}$, defined by
\begin{flalign}
	\label{frhou} 
	f(t, x) = u_t^{\star} (x) + \pi \mathrm{i} \varrho_t^{\star} (x).
\end{flalign}

\begin{lem}[{\cite[Equation (2.16)]{LDSI}}]
	
	\label{p:solution}

	Adopting the notation and assumptions of \Cref{t:LDP}, the associated complex slope $f$ defined by \eqref{frhou} satisfies the complex Burgers equation, 
	\begin{equation}
		\label{ftfx}
		\partial_t f(t, x)+f(t, x) \cdot \partial_x f (t ,x)=0,\qquad \text{for all $(t, x) \in \Omega$}.
	\end{equation} 
	
\end{lem}

The following lemma expresses the derivatives of the height function (recall \Cref{hrhot}) associated with non-intersecting Brownian bridges in terms of the $(u^{\star}, \varrho^{\star})$ above. The first statement of \eqref{hurhoequation} follows from the definition \eqref{htxintegral} of the height function. Together with the equation $\partial_t \varrho^{\star} + \partial_x (u^{\star} \varrho^{\star}) = 0$, this formally implies the second statement of \eqref{hurhoequation}; its detailed proof is given in \Cref{ProofDerivativeH} below.

\begin{lem} 
	
	\label{hurho}

	Adopt the notation and assumptions of \Cref{t:LDP}, and let $H^{\star}$ denote the height function associated with $(\mu_t^{\star})$. Then $H^{\star}$ is locally Lipschitz on $(0, 1) \times \mathbb{R}$ and, for almost all $(t, x) \in (0, 1) \times \mathbb{R}$ (with respect to Lebesgue measure) we have
	\begin{flalign} 
		\label{hurhoequation} 
		\partial_x H^{\star} (t, x) = - \varrho_t (x); \qquad \partial_t H^{\star} (t, x) = u_t (x) \varrho_t (x).
	\end{flalign} 
	
\end{lem}

\begin{rem} 
	
	\label{aab1} 
	
	Using \Cref{aab}, let us provide analogs of the above results when the parameters $(a, b)$ and $A$ in \Cref{rhot} are arbitrary. To that end, adopt the notation of \Cref{aab}, and assume (analogously to in \eqref{rho1rho2}) that $\mu_a$ and $\mu_b$ have densities $\varrho_a : \mathbb{R} \rightarrow \mathbb{R}_{\ge 0}$ and $\varrho_b : \mathbb{R} \rightarrow \mathbb{R}_{\ge 0}$ with respect to Lebesgue measure, satisfying
	\begin{flalign}
		\label{rhoarhob}
		\text{$\varrho_a$ and $\varrho_b$ are compactly supported}; \qquad \displaystyle\sup_{x \in \mathbb{R}} \big| \varrho_a (x) \big| < \infty; \qquad \displaystyle\sup_{x \in \mathbb{R}} \big| \varrho_b (x) \big| < \infty.
	\end{flalign}
	
	Letting $(\varrho_t^{\star})_{t \in [a,b]}$ and $(\widetilde{\varrho}_t^{\star})_{t \in [0, 1]}$ denote the densities of $(\mu_t^{\star})_{t \in [a, b]}$ and $(\widetilde{\mu}_t^{\star})_{t \in [0, 1]}$ with respect to Lebesgue measure, respectively (guaranteed to exist by the first part of \Cref{mutrhot}), we have from the first statement of \eqref{muhgaab} that 
	\begin{flalign}
		\label{rhoaab} 
		\widetilde{\varrho}_t^{\star} (x) = A^{-1/2} (b-a)^{1/2} \cdot \varrho_{a + (b-a) t}^{\star} \big( A^{1/2} (b-a)^{1/2} x \big), \qquad \text{for any $(t, x) \in [0, 1] \times \mathbb{R}$}.
	\end{flalign}
	
	\noindent By the second part of \Cref{mutrhot}, there exists $\widetilde{\nu}_t \in \mathscr{P}$ with $\supp \widetilde{\nu}_t \subseteq \supp \widetilde{\mu}_0 + \supp \widetilde{\mu}_1$ and $\widetilde{\mu}_t^{\star} = \widetilde{\nu}_t \boxplus \mu_{\semci}^{(t-t^2)}$, for each $t \in (0, 1)$. Hence, by \Cref{mtscale}, the $\beta = b-a$ case of \Cref{mtscalebeta}, and \eqref{rhoaab}, for any $t \in (a, b)$ there exists some $\nu_t \in \mathscr{P}_{\fin}$ with $\nu_t (\mathbb{R}) = A$, such that $\supp \nu_t \subseteq \supp \mu_a + \supp \mu_b$ and $\mu_t^{\star} = \nu_t \boxplus \mu_{\semci}^{((t-a)(b-t) / (b-a))}$. This verifies that (the rescaled version) of the second part of \Cref{mutrhot} holds for $\mu_t^{\star}$.
	
	Further setting (as in \eqref{hurhoequation}) 
	\begin{flalign}
		\label{} 
		u_t^{\star} (x) = \displaystyle\frac{\partial_t H^{\star} (t, x)}{\varrho_t (x)}, \quad \text{if $\varrho_t (x) > 0$}; \qquad \qquad \widetilde{u}_t^{\star} (x) = \displaystyle\frac{\partial_t \widetilde{H}_t^{\star} (x)}{\widetilde{\varrho}_t^{\star} (x)}, \quad \text{if $\widetilde{\varrho}_t^{\star} (x) > 0$},
	\end{flalign}
	
	\noindent we have from \eqref{rhoaab} and the second statement of \eqref{muhgaab} that 
	\begin{flalign}
		\label{uaab}
		\widetilde{u}_t^{\star} (x) = A^{-1/2} (b-a)^{1/2} \cdot u_{a + (b-a) t}^{\star} \big( A^{1/2} (b-a)^{1/2} x \big).
	\end{flalign}

	\noindent Defining (as in \eqref{frhou}) the functions $f (t, x) = u_t^{\star} (x) + \pi \mathrm{i} \varrho_t^{\star} (x)$ and $\widetilde{f} (t, x) = \widetilde{u}_t^{\star} (x) + \pi \mathrm{i} \widetilde{\varrho}_t^{\star} (x)$, we have from \eqref{rhoaab} and \eqref{uaab} that
	\begin{flalign*}
		\widetilde{f} (t, x) = A^{-1/2} (b-a)^{1/2} \cdot f \big( a + (b-a) t, A^{1/2} (b-a)^{1/2} x \big).
	\end{flalign*}
	
	\noindent Together with the fact that $\widetilde{f}$ satisfies the complex Burgers equation \eqref{ftfx} by \Cref{p:solution}, this implies that $f$ also satisfies \eqref{ftfx} on the liquid region $\Omega$ associated with $\varrho^{\star}$.
\end{rem}

\begin{rem} 
	
	\label{urho0} 
	
	Adopt the notation and assumptions of \Cref{rhot}. Let $H^{\star} : [a, b] \times \mathbb{R} \rightarrow \mathbb{R}$ and $G^{\star} : [a, b] \times [0, A] \rightarrow \mathbb{R}$ denote the height and inverted height functions associated with $\bm{\mu}^{\star}$, respectively, and let $\varrho_t^{\star}$ and $u_t^{\star}$ be as in \Cref{t:LDP} (and \Cref{aab1}); denote the associated inverted liquid region (recall \eqref{omega12}) by $\Omega^{\inv} \subseteq (a,b) \times (0, A)$. By \Cref{hrhot}, we have for $(t, y) \in \Omega^{\inv}$ that $H^{\star} \big( t, G^{\star} (t, y) \big) = y$. Differentiating this equation in $y$ or $t$, we find that \eqref{hurhoequation} can equivalently be formulated in terms $G^{\star}$ as 
	\begin{flalign*}
		\varrho_t^{\star} \big( G^{\star} (t, y) \big) = -\displaystyle\frac{1}{\partial_y G^{\star} (t, y)}, \quad \text{and} \quad u_t^{\star} \big( G^{\star} (t, y) \big) = \partial_t G^{\star} (t, y), \qquad \text{if $(t, y) \in \Omega^{\inv}$},
	\end{flalign*}

	\noindent which defines $\varrho_t^{\star}$ and $u_t^{\star}$ everywhere on $\Omega$ (since $G^{\star}$ bijectively maps $\Omega^{\inv}$ to $\Omega$).

\end{rem}

\subsection{Elliptic Partial Differential Equations for the Height Function}

\label{EquationHG}

In this section we reformulate \eqref{equationrhou} as elliptic partial differential equations satisfied by the height and inverted height functions associated with $\mu_t^{\star}$. The following lemma, which follows quickly from \eqref{equationrhou}, implements this for the height function.

\begin{lem} 
	
	\label{hequation} 
	
	Adopt the notation and assumptions of \Cref{t:LDP}, and let $H^{\star}$ denote the height function associated with $(\mu_t^{\star})$. Recalling the liquid region $\Omega \subset (0, 1) \times \mathbb{R}$ from \eqref{omega12}, we have
	\begin{flalign}
		\label{equationxtb}
		\displaystyle\sum_{i, j \in \{ t, x \}} \mathfrak{b}_{ij} \big( \nabla H^{\star} (t, x) \big) \partial_i \partial_j H^{\star} (t, x) = 0, \qquad \text{for any $(t, x) \in \Omega$},
	\end{flalign}
	
	\noindent where, for any $(u, v) \in \mathbb{R}^2$, we have denoted
	\begin{flalign}
		\label{uvb} 
		\mathfrak{b}_{tt} (u, v) = v^2; \qquad \mathfrak{b}_{tx} (u, v) = -uv = \mathfrak{b}_{xt} (u, v); \qquad \mathfrak{b}_{xx} (u, v) = u^2 + \pi^2 v^4.
	\end{flalign}
\end{lem} 

\begin{proof} 
	
	Applying \Cref{hurho} and the second statement of \eqref{equationrhou} yields
	\begin{flalign}
		\label{tth} 
		\partial_{tt} H + \partial_x \bigg( \displaystyle\frac{\pi^2}{3} (\partial_x H)^3 - \displaystyle\frac{(\partial_t H)^2}{\partial_x H} \bigg) = 0,
	\end{flalign} 
	
	\noindent or equivalently,
	\begin{flalign}
		\label{txh} 
		\partial_{tt} H - \displaystyle\frac{2 \partial_t H}{\partial_x H} \cdot \partial_{tx} H + \displaystyle\frac{(\partial_t H)^2}{(\partial_x H)^2} \cdot \partial_{xx} H + \pi^2 (\partial_x H)^2 \cdot \partial_{xx} H = 0.
	\end{flalign}
	
	\noindent Multiplying this equation by $(\partial_x H)^2$ then yields \eqref{equationxtb}. 
\end{proof}

\begin{rem} 
	
	\label{rhouderivative}
	
	The derivation of \eqref{tth} does not require smoothness of $u^{\star}$ and $\varrho^{\star}$. In particular, if one were to only assume that the second equation in \eqref{equationrhou} holds weakly (in the sense of distributions) on $\Omega$ (as was shown in \cite[Property 2.9(4)]{FOAMI}), the proof of \Cref{hequation} shows that $H$ is a weak solution to the divergence-form ellliptic partial differential equation \eqref{tth}.
	
\end{rem}

The next lemma, which follows from changing variables in \Cref{hequation}, shows that the inverted height function $G^{\star}$ satisfies \eqref{equationxtd}; in view of \Cref{convergepathomega}, this implies that the limiting trajectories of paths in a family of non-intersecting Brownian motions (without north or south boundary) is govered by solutions of \eqref{equationxtd}; in the below, we recall the open set $\Omega^{\inv}$ from \eqref{omega12}. 

\begin{lem} 
	
	\label{gequation}

	Adopt the notation and assumptions of \Cref{t:LDP}, and let $G^{\star}$ denote the inverted height function associated with $(\mu_t^{\star})$. Then $G^{\star}$ satisfies the equation \eqref{equationxtd} for $(t, y) \in \Omega^{\inv}$.

\end{lem} 

\begin{proof}
	
	Let $(t, y) \in \Omega^{\inv}$, and set $x \in \mathbb{R}$ such that $H(t, x) = y$; then, $(t, x) \in \Omega$. Differentiating the equation $G \big( t, H(t, x) \big) = x$, it is quickly verified that  
	\begin{flalign*}
		\partial_y G = \displaystyle\frac{1}{\partial_x H}; \qquad & \partial_t G = -\displaystyle\frac{\partial_t H}{\partial_x H}; \qquad \partial_{yy} G = -\displaystyle\frac{\partial_{xx} H}{(\partial_x H)^3}; \qquad \partial_{ty} G = \displaystyle\frac{\partial_t H}{(\partial_x H)^3} \cdot \partial_{xx} H - \displaystyle\frac{\partial_{tx} H}{(\partial_x H)^2}; \\
		& \quad \partial_{tt} G = \displaystyle\frac{2 \partial_t H}{(\partial_x H)^2} \cdot \partial_{tx} H - \displaystyle\frac{(\partial_t H)^2}{(\partial_x H)^3} \cdot \partial_{xx}H - \displaystyle\frac{\partial_{tt} H}{\partial_x H}.
	\end{flalign*}
	
	\noindent where in the above the arguments of $H$ and $G$ (and their derivatives) are $(t, x)$ and $(t, y)$, respectively. Dividing \eqref{txh} by $\partial_x H$ and applying the above equalities yields 
	\begin{flalign*}
		\partial_{tt} G + \pi^2 (\partial_y G)^{-4} \partial_{yy} G & = \partial_{tt} G - \pi^2 (\partial_x H) \partial_{xx} H \\
		& =  \displaystyle\frac{2 \partial_t H}{(\partial_x H)^2} \cdot \partial_{tx} H - \displaystyle\frac{\partial_{tt} H}{\partial_x H} - \bigg( \displaystyle\frac{(\partial_t H)^2}{(\partial_x H)^3} + \pi^2 \partial_x H \bigg) \cdot \partial_{xx} H = 0,
	\end{flalign*}
	
	\noindent which establishes the lemma.
\end{proof}

\begin{rem} 
	
	\label{eigenvalues2} 
	
	The regularity properties of solutions to \eqref{equationxtb} and \eqref{equationxtd} depend on the $2 \times 2$ matrices 
	\begin{flalign*} 
		& \bm{B} = \bm{B} (u, v) = \left[ \begin{array}{cc} \mathfrak{b}_{tt} (u, v) & \mathfrak{b}_{tx} (u, v) \\ \mathfrak{b}_{xt} (u, v) & \mathfrak{b}_{tt} (u, v) \end{array} \right] = \left[ \begin{array}{cc} v^2 & -uv \\ -uv & u^2 + \pi^2 v^4 \end{array} \right], \quad \text{and} \\
		&  \bm{D} = \bm{D} (u, v) = \left[ \begin{array}{cc}  \mathfrak{d}_{tt} (u, v) & \mathfrak{d}_{ty} (u, v) \\ \mathfrak{d}_{yt} (u, v) & \mathfrak{d}_{yy} (u, v) \end{array} \right] = \left[ \begin{array}{cc} 1 & 0 \\ 0 & \pi^2 v^{-4} \end{array} \right],
	\end{flalign*} 
	
	\noindent respectively; both are nonnegative definite. The two eigenvalues of $\bm{B}$ are bounded above and below whenever $|u|$ is bounded above and $v$ is bounded above and below; the two eigenvalues of $\bm{D}$ are bounded above and below whenever $v$ is. In particular, that the latter only requires bounds on $v = \partial_y G^{\star} = (\partial_x H^{\star})^{-1}$, while the former requires those on both $\partial_x H^{\star}$ and $\partial_t H^{\star}$. 
	
\end{rem} 

Further regularity properties of solutions to \eqref{equationxtd} are provided in \Cref{DerivativesEquation} below. 

\begin{rem}
	
	Formally, one can interpret \eqref{equationxtb} and \eqref{equationxtd} as Euler--Lagrange equations for minimizers of the functionals
	\begin{flalign*}
		& \mathscr{E} (F) = -\displaystyle\int_0^1 \displaystyle\int_{-\infty}^{\infty} \bigg( \displaystyle\frac{\big( \partial_t F(t, x) \big)^2}{\partial_x F(t, x)} + \displaystyle\frac{\pi^2}{3} \big( \partial_x F(t, x) \big)^3 \bigg) dx dt, \qquad \text{and} \\
		& \mathscr{E}^{\inv} (F) = \displaystyle\int_0^1 \displaystyle\int_0^1 \Big( \big( \partial_t F(t, x) \big)^2 + \displaystyle\frac{\pi^2}{3} \big( \partial_x F(t, x) \big)^{-2} \Big) dx dt, \qquad \text{respectively}.
	\end{flalign*}
	
	\noindent However, we will not make this fully precise, as we will not require this interpretation.
	
\end{rem}

It will be useful to make use of invariances of the equations \eqref{equationxtb} and \eqref{equationxtd} under the below (linear and multiplicative) transformations. 

\begin{lem}
	
	\label{invariancesscale} 
	
	Fix a bounded, open subset $\mathfrak{R} \subset \mathbb{R}^2$ and two functions $G, H \in \Adm (\mathfrak{R}) \cap \mathcal{C}^2 (\mathfrak{R})$; assume on $\mathfrak{R}$ that $H$ satisfies \eqref{equationxtb} and that $G$ satisfies \eqref{equationxtd}. Fix real numbers $\alpha, \beta > 0$, and denote the open sets $\widetilde{\mathfrak{R}} = \widetilde{\mathfrak{R}}_{\alpha; \beta} \subset \mathbb{R}^2$ and $\widehat{\mathfrak{R}} = \widehat{\mathfrak{R}}_{\alpha} \subset \mathbb{R}^2$ by
	\begin{flalign*}
		\widetilde{\mathfrak{R}} = \big\{ (t, x) \in \mathbb{R}^2 & : (\alpha t, \beta x) \in \mathfrak{R} \big\}; \quad \widehat{\mathfrak{R}} = \big\{ (t, x) \in \mathbb{R}^2 : (t, x-\alpha t) \in \mathfrak{R} \big\}.
	\end{flalign*}
	
	\begin{enumerate}
		
		\item 
		\begin{enumerate} 
			\item \label{scale11} Define  $\widetilde{H} \in \mathcal{C}^2 (\widetilde{\mathfrak{R}})$ by $\widetilde{H} (t, x) = \alpha \beta^{-2} H(\alpha t, \beta x)$. Then $\widetilde{H}$ satisfies \eqref{equationxtb} on $\widetilde{\mathfrak{R}}$.
			
			\item Define $\widehat{H} \in \mathcal{C}^2 (\widehat{\mathfrak{R}})$ by $\widehat{H}(t, x) = H(t, x-\alpha t)$. Then, $\widehat{H}$ satisfies \eqref{equationxtb} on $\widehat{\mathfrak{R}}$.
		\end{enumerate}
		
		\item 
		\begin{enumerate} 
			\item \label{scale21} Define $\widetilde{G} \in \mathcal{C}^2 (\widetilde{\mathfrak{R}})$ by $\widetilde{G} (t, y) = (\alpha \beta)^{-1/2} G(\alpha t, \beta y)$. Then $\widetilde{G}$ satisfies \eqref{equationxtd} on $\widetilde{\mathfrak{R}}$.
			
			\item \label{galphalinear} Define $\widehat{G} \in \mathcal{C}^2 (\widehat{\mathfrak{R}})$ by $\widehat{G} (t, y) = G(t, y)+\alpha t$. Then, $\widehat{G}$ satisfies \eqref{equationxtd} on $\mathfrak{R}$.
		\end{enumerate}
		
	\end{enumerate} 
	
\end{lem}

\begin{proof}
	
	We only show the first part of the first statement of the lemma, as the proofs of the remaining ones are entirely analogous. To that end, fix $(\widetilde{t}, \widetilde{x}) \in \widetilde{\mathfrak{R}}$ and set $(t, x) = (\alpha \widetilde{t}, \beta \widetilde{x}) \in \mathfrak{R}$. Then,
	\begin{flalign*}
		& \qquad \qquad  \quad \partial_t \widetilde{H} (\widetilde{t}, \widetilde{x}) = \alpha^2 \beta^{-2} \partial_t H(t, x); \qquad
		\partial_x \widetilde{H} (\widetilde{t}, \widetilde{x}) = \alpha \beta^{-1} \partial_x H(t, x); \\
		& \partial_{tt} \widetilde{H} (\widetilde{t}, \widetilde{x}) = \alpha^3 \beta^{-2} \partial_{tt} H(t,x); \quad
		\partial_{tx} \widetilde{H} (\widetilde{t}, \widetilde{x}) = \alpha^2 \beta^{-1} \partial_{tx} H(t,x); \quad
		\partial_{xx} \widetilde{H} (\widetilde{t}, \widetilde{x}) = \alpha \partial_{xx} H(t, x).
	\end{flalign*} 
	
	\noindent It follows from \eqref{equationxtb} and \eqref{uvb} that 
	\begin{flalign*}
		\displaystyle\sum_{i, j \in \{t, x \}} \mathfrak{b}_{ij} \big( \nabla \widetilde{H} (\widetilde{t}, \widetilde{x}) \big) \partial_i \partial_j \widetilde{H} (\widetilde{t}, \widetilde{x}) = \alpha^5 \beta^{-4} \displaystyle\sum_{i,j \in \{ t, x \}} \mathfrak{b}_{ij} \big( \nabla H(t, x) \big) \partial_i \partial_j H(t, x) = 0,
	\end{flalign*}
	
	\noindent verifying that $\widetilde{H}$ satisfies \eqref{equationxtb} on $\widetilde{\mathfrak{R}}$.
\end{proof}

Let us mention that, while the scalings provided by the $\alpha = \beta$ cases of parts (\ref{scale11}) and (\ref{scale21}) of \Cref{invariancesscale} are satisfied by any equation of the general form \eqref{equationxtb} (only assuming that the $\mathfrak{b}_{ij}$ and $\mathfrak{d}_{ij}$ are measurable), those provided by general $(\alpha, \beta)$ are special to the explicit forms of the coefficients \eqref{uvb} and \eqref{uvd}.

\begin{rem} 
	
	\label{aabequation}
	
	Let us extend the above statements to the case of general parameters $a < b$ and $A > 0$. Adopt the notation and assumptions of \Cref{aab1}. Since $\widetilde{H}^{\star}$ satisfies \eqref{equationxtb}, the second statement of \eqref{muhgaab} and the $(\alpha, \beta) = \big( (b-a)^{-1}, A^{-1/2} (b-a)^{-1/2} \big)$ case of part (\ref{scale11}) of \Cref{invariancesscale} implies that $H^{\star}$ also satisfies \eqref{equationxtb} on the liquid region $\Omega$ associated with $(\varrho_t^{\star})$. Similarly, since $\widetilde{G}^{\star}$ satisfies \eqref{equationxtd}, the third statement of \eqref{muhgaab} and the $(\alpha, \beta) = \big( (b-a)^{-1}, A^{-1} )$ case of part (\ref{scale21}) of \Cref{invariancesscale} implies that $G^{\star}$ satisfies \eqref{equationxtd} on the inverted liquid region $\Omega^{\inv}$ associated with $(\varrho_t^{\star})$. 
	
\end{rem} 

We conclude this section by mentioning that, as stated, the results of \Cref{EquationLimitrhou} and \Cref{EquationHG} largely imposed the assumption \eqref{rhoarhob} (due to \eqref{rho1rho2}). However, the following lemma indicates that many of these results continue to hold if we only assume that the boundary measures $\mu_a$ and $\mu_b$ are compactly supported (without imposing that they admit a bounded density, as in \eqref{rhoarhob}). The proof is given in \Cref{ProofDerivativeH} below, by combining the results of \Cref{EquationLimitrhou} and \Cref{EquationHG} with \Cref{mutrhot} and \Cref{gtabab} (to verify \eqref{rhoarhob} in the interior of the domain). 

\begin{lem} 
	
	\label{muamubrho}

	Adopt the notation and assumptions of \Cref{rhot}, and assume that $\mu_a$ and $\mu_b$ are compactly supported. Denote the density process associated with $\bm{\mu}$ by $(\varrho_t^{\star})_{t \in (a, b)}$; denote the associated height and inverted height functions by $H^{\star} : [a, b] \times \mathbb{R} \rightarrow \mathbb{R}$ and $G^{\star} : [a, b] \times [0, A] \rightarrow \mathbb{R}$, respectively; and denote the associated liquid and inverted liquid regions by $\Omega \subset (a, b) \times \mathbb{R}$ and $\Omega^{\inv} \subseteq (a, b) \times (0, A)$, respectively. Then, the following statements hold.
	
	\begin{enumerate}
		\item The function $H^{\star} (t, x)$ is smooth for $(t, x) \in \Omega$; moreover, $G^{\star} (t, y)$ is smooth for $(t, y) \in \Omega^{\inv}$. 
		\item Defining $u^{\star} : \Omega \rightarrow \mathbb{R}$ as in \Cref{urho0} and the complex slope $f : \Omega \rightarrow \overline{\mathbb{H}}$ as in \eqref{frhou}, $f$ satisfies \eqref{ftfx}  on $\Omega$. 
		\item The function $H^{\star}$ satisfies \eqref{equationxtb} on $\Omega$, and $G^{\star}$ satisfies \eqref{equationxtd} on $\Omega^{\inv}$.
	\end{enumerate}

\end{lem}

\section{Proof of the Concentration Estimate} 

\label{BoundaryGH}

In this section we reduce \Cref{gh} to \Cref{gh1} below, which is an inductive statement that bounds the height function associated with the family $\bm{x}$ of non-intersecting Brownian motion by a sequence of barrier functions.

\subsection{Proof of \Cref{gh}} 

\label{Prooffgb2}

The proof of \Cref{gh} will proceed by first establishing the following variant of it for the associated height function instead of for the path trajectories $x_j (t)$. We begin with the following definition for a certain class of open sets. 

\begin{definition} 
	
\label{pfg} 

For any real numbers $a < b$ and functions $f, g:  [a, b] \rightarrow \mathbb{R}$ with $f < g$, we define the open set $\mathfrak{P} = \mathfrak{P}_{f; g} \subset \mathbb{R}^2$ by
\begin{flalign*}
	\mathfrak{P} = \big\{ (t, x) \in (a, b) \times \mathbb{R} : f(t) < x < g(t) \big\}.
\end{flalign*}

\noindent We call any open set of the form $\mathfrak{P}_{f; g}$ for some $f, g: [a, b] \rightarrow \mathbb{R}$ with $f < g$ a \emph{strip domain}. The \emph{width} of $\mathfrak{P}_{f; g}$ is defined to be $b-a$.

\end{definition}

Observe in particular that the non-intersecting Brownian motions from \Cref{qxyfg} all lie within the strip domain $\mathfrak{P}_{f; g}$.

 The following proposition, which quickly implies \Cref{gh}, is established in \Cref{ProofHeighteta} below. In what follows we recall that any function $G : [a, b] \times [0, 1] \rightarrow \mathbb{R}$, that is nonincreasing in its second coordinate, is associated with a unique height function $H : [a, b] \times \mathbb{R} \rightarrow [0, 1]$ through \eqref{gty}, namely, by setting $H(t, x) = \displaystyle\inf \big\{ y \in [0, 1] : G(t, y) \ge x \big\}$. We also recall the height function $\mathsf{H}^{\bm{x}}$ associated with a family $\bm{x}$ of non-intersecting curves from \eqref{htx}.

\begin{prop}
	\label{gh0} 
	
	Adopt the notation and assumptions of \Cref{gh}; assume that $\varkappa = 0$ and that $G(0, 0) = 0$. There exist constants $c = c(\varepsilon, \delta, B_0, m) > 0$ and $C = C(\varepsilon, \delta, B_0, m) > 1$ such that the following holds. Letting $H : [0, L^{-1}] \times \mathbb{R} \rightarrow \mathbb{R}$ denote the height function associated with $G$, we have  
	\begin{flalign*}
		\mathbb{P} \bigg[ \displaystyle\sup_{z \in \mathfrak{P}_{f; g}} \big| n^{-1} \mathsf{H}^{\bm{x}} (z) - H (z) \big| > C n^{2/m+\delta-1} \bigg] < C e^{-c (\log n)^2}.
	\end{flalign*} 
\end{prop}

In the below, we will use the fact that under the notation of \Cref{gh0} we have 
\begin{flalign}
	\label{derivativexhtx}
	-\varepsilon^{-1} < \partial_x H(t, x) < -\varepsilon, \qquad \text{for each $(t, x) \in \mathfrak{P}_{f; g}$},
\end{flalign}

\noindent which follows from \eqref{gty} and the fact that $G \in \Adm_{\varepsilon} (\mathfrak{R})$ (recall \Cref{fgr}).  Also observe that there exists a constant $B = B (\varepsilon, B_0, m) > 1$ such that 
\begin{flalign}
	\label{hderivativem} 
	\| H \|_{\mathcal{C}^{m+1} (\mathfrak{P}_{f;g})} < B,
\end{flalign}

\noindent which follows from \eqref{gty} and the facts that $G \in \Adm_{\varepsilon} (\mathfrak{R})$ (so $H \big( t, G(t, y) \big) = y$ for each $(t, y)$ such that $G(t, y) \in \mathfrak{P}_{f; g}$) and that $\|G \|_{\mathcal{C}^{m+1} (\mathfrak{R})} \le B_0$.

\begin{proof}[Proof of \Cref{gh}]
	
	We may assume by shifting in what follows that $G(0, 0) = 0$. Let us first address the case when $\varkappa = 0$. By \Cref{gh0}, it suffices to show for any real number $C > 1$ the inclusions of events 
\begin{flalign}
	\label{xjthh} 
	\begin{aligned}
 \Bigg\{ \displaystyle\max_{j \in \llbracket 1, n \rrbracket} \displaystyle\sup_{t \in [0, 1/L]} \big| x_j (t) - G( & t, jn^{-1}) \big| \le  C \varepsilon^{-1} n^{2/m + \delta - 1} \Bigg\} \\
 & \subseteq \Bigg\{ \displaystyle\sup_{z \in \mathfrak{P}_{f; g}} \big| n^{-1} \mathsf{H}^{\bm{x}} (z) - H(z) \big| \le C n^{2/m + \delta - 1} \Bigg\}.
 	\end{aligned} 
\end{flalign}

To that end, we begin by restricting to the event on the right side of \eqref{xjthh}. Fix $t_0 \in (0, L^{-1})$, and apply that event twice, first with $z = (t_0, y_1)$ and then with $z = (t_0, y_2)$, where $y_1 = G \big(t_0, jn^{-1}  + C n^{2/m+  \delta-1} \big)$ and $y_2 = G \big( t_0, jn^{-1}  - C n^{2/m + \delta - 1} \big)$. Since \eqref{gty} and the fact that $G \in \Adm (\mathfrak{R})$ imply $H \big( t_0, G(t_0, y) \big) = y$ for each $y \in (0, 1)$, this yields 
\begin{flalign*}
	 \mathsf{H}^{\bm{x}} \big(t_0, G(t_0, jn^{-1} - C n^{2/m + \delta-1}) \big) \le j \le \mathsf{H}^{\bm{x}} \big( t_0, G(t_0, jn^{-1} + C n^{2/m+\delta-1}) \big),
\end{flalign*} 

\noindent and hence
\begin{flalign*} 
		G \big(t_0, jn^{-1} + Cn^{2/m+\delta-1} \big) \le x_j (t_0) \le G \big(t_0, jn^{-1} - C n^{2/m+\delta-1} \big).
\end{flalign*}

\noindent Together with the fact that $G \in \Adm_{\varepsilon} (\mathfrak{R})$, this yields 
\begin{flalign*}
	G(t_0, jn^{-1}) - C \varepsilon^{-1} n^{2/m + \delta - 1} & \le x_j (t_0)  \le G (t_0, jn^{-1}) + C \varepsilon^{-1} n^{2/m + \delta - 1},
\end{flalign*} 

\noindent confirming \eqref{xjthh}. This establishes the theorem if $\varkappa = 0$.

If $\varkappa \ne 0$, then define the $n$-tuples $\bm{u}^-, \bm{u}^+, \bm{v}^-, \bm{v}^+ \in \overline{\mathbb{W}}_n$ and the functions $f^-, f^+, g^-, g^+ : [0, L^{-1}] \rightarrow \mathbb{R}$ by setting 
\begin{flalign*}
	& u_j^{\pm} = G(0, jn^{-1}) \pm \varkappa; \quad v_j^{\pm} = G(L^{-1}; jn^{-1}) \pm \varkappa; \quad f^{\pm} (s) = f(s) \pm \varkappa; \qquad g^{\pm} (s) = g(s) \pm \varkappa,
\end{flalign*}

\noindent for each index $\pm \in \{ +, - \}$ and pair $(j, s) \in \llbracket 1, n \rrbracket \times [0, L^{-1}]$. Sample non-intersecting Brownian bridges $\bm{x}^- = (x_1^-, x_2^-, \ldots , x_n^-) \in \llbracket 1, n \rrbracket \times \mathcal{C} \big( [0, L^{-1}] \big)$ and $\bm{x}^+ = (x_1^+, x_2^+, \ldots , x_n^+) \in \llbracket 1, n \rrbracket \times \mathcal{C} \big( [0, L^{-1}] \big)$ under the measures $\mathfrak{Q}_{f^-; g^-}^{\bm{u}^-; \bm{v}^-}$ and $\mathfrak{Q}_{f^+; g^+}^{\bm{u}^+; \bm{v}^+}$, respectively. Observe that $\bm{u}^- \le \bm{u} \le \bm{u}^+$ and $\bm{v}^- \le \bm{v} \le \bm{v}^+$ by \eqref{uvg}, and that $f^- \le f \le f^+$ and $g^- \le g \le g^+$. Thus, by \Cref{monotoneheight}, we may couple $(\bm{x}^-, \bm{x}, \bm{x}^+)$ such that 
\begin{flalign}
	\label{xjs3} 
	x_j^- (s) \le x_j (s) \le x_j^+ (s), \qquad \text{for each $(j, s) \in \llbracket 1, n \rrbracket \times [0, L^{-1}]$.}
\end{flalign} 

Then, \Cref{gh} applies to both $\bm{x}^-$ and $\bm{x}^+$, with the $G$ there given by $G^- = G - \varkappa$ and $G^+ = G + \varkappa$ here, respectively. It yields a constant $c = c(\varepsilon, \delta, B_0, m) > 0$ such that, on an event of probability at least $1 - c^{-1} e^{-c(\log n)^2}$, we have $x_j^- (s) \ge G^- (s, jn^{-1}) - n^{2/m+\delta-1} = G(s, jn^{-1}) - \varkappa - n^{2/m+\delta-1}$ and similarly $x_j^+ \le G (s, jn^{-1}) + \varkappa + n^{2/m+\delta-1}$. This, with \eqref{xjs3} and a union bound, yields the theorem.
\end{proof}

To establish \Cref{gh0} , we make use of the following initial estimate that provides a high-probability $\frac{1}{2}$-H\"{o}lder bound, with associated constant $n^{\delta/2}$, for the path trajectories $x_j (t)$ from \Cref{gh0} .

\begin{lem}
	
	\label{n10estimate0} 
	
	Adopt the notation and assumptions of \Cref{gh}; assume that $\varkappa = 0$. There exists a constant $C = C(\delta, B_0) > 1$ such that 
	\begin{flalign*}
		\mathbb{P} \bigg[ \displaystyle\max_{j \in \llbracket 1, n \rrbracket} \displaystyle\sup_{0 \le t < t+s \le 1/L} \big| x_j (t+s) - x_j (t) \big| \le n^{\delta/2}  s^{1/2} + n^{-5} \bigg] \le C e^{-n}.
	\end{flalign*} 
\end{lem}

\begin{proof}
	
	First observe by \Cref{fgr} that $\big| f' (s) \big| + \big| g' (s) \big| \le 2 \big\| G - G(0,0) \big\|_{\mathcal{C}^{m+1} (\mathfrak{R})} \le 2B_0$, for each $s \in (0, L^{-1})$. Thus we may apply the second part of \Cref{estimatexj3} with the $(A, B, T)$ there equal to $(2B_0, C_1, L^{-1})$ here, for some sufficiently large constant $C_1 > 1$. Since $L \le n^{\delta}$, this yields some constant $C_2 > 1$ such that 
	\begin{flalign}
		\label{xjts12} 
		\begin{aligned}
			\mathbb{P} \Bigg[ \bigcup_{j=1}^n \bigcup_{a \le t < t+s \le b} \bigg\{ \big|  x_j (t+s &) - x_j (t) - sL (v_j - u_j) \big|  \ge C_2 s^{1/2} \big( |\log 9s^{-1} n^{\delta}  | + B_0 \big)^2 \bigg\} \Bigg] \le C_2 e^{-n}.
		\end{aligned} 
	\end{flalign}
	
	\noindent Since $\delta < \frac{1}{80}$, by \Cref{fgr}, there is a constant $C_3 = C_3 (B_0) > 1$ such that for $n > C_3$ we have
	\begin{flalign*}
		C_2 s^{1/2} \big( \log |9s^{-1} n^{\delta}| + B_0 \big)^2 & \le C_2 B_0^2 s^{1/2} \Big( ( \log n^{21} + 1)^2 \cdot \textbf{1}_{s \ge n^{-20}} + \big( \log |s n| + 1 \big)^2 \cdot \textbf{1}_{s \le n^{-20}} \Big) \\ 
		& \le n^{\delta/4} s^{1/2} + n^{-5}.
	\end{flalign*}
	
	\noindent Inserting this into \eqref{xjts12}, and using the facts that $s \le L^{-1} \le B_0$, that $|v_j - u_j| L \le 2 \big\| G - G(0, 0) \big\|_{\mathcal{C}^{m+1} (\mathfrak{R})} \le 2B_0$ (by \Cref{fgr}), and that $n$ is sufficiently large, yields the lemma.
\end{proof}

\subsection{Proof of \Cref{gh0}} 

\label{ProofHeighteta}

In this section we establish \Cref{gh0} assuming \Cref{gh1} below. Define the real numbers $\rho, \omega > 0$; fix a real number $\zeta$; and define the function $\psi : \mathbb{R}^2 \rightarrow \mathbb{R}$ so that   
\begin{flalign}
	\label{rhoomega} 
	 \rho = n^{\delta - 1/m}; \qquad \omega = n^{2/m + \delta - 1}; \qquad \zeta > 20000 (B^5 + \varepsilon^{-5}); \qquad \psi (t, x) = \Upsilon - e^{\zeta t} - e^{\zeta x},
\end{flalign}

\noindent where $\Upsilon > 0$ is chosen so that $\inf_{z \in \mathfrak{P}_{f; g}} \psi (z) = 2$. Set $J = \lfloor 3 \omega^{-1} \rfloor - 1$, and let $K$ be the smallest integer so that $(1 - K \rho^2 n^{-\delta}) \sup_{z \in \mathfrak{P}_{f; g}} \psi (z) \le \frac{1}{4}$. Analogously to in \cite[Equation (52)]{LTGDMLS}, for any integers $j \in \llbracket 0, J \rrbracket$ and real number (typically an integer) $k \in [0, K]$, define the functions $\varphi_{j;k}^+, \varphi_{j;k}^- : \overline{\mathfrak{P}_{f; g}} \rightarrow \mathbb{R}$ by for $z \in \mathfrak{P}_{f; g}$ setting
\begin{flalign}
	\label{functionsjkz}
	\begin{aligned}
	  & \varphi_{j;k}^+ (z) = H(z) - (j+1) \omega + ( 1 - k \rho^2 n^{-\delta} ) \omega \psi (z) + 3; \\
	  & \varphi_{j;k}^- (z) = H(z) + (j+1) \omega - ( 1 - k \rho^2 n^{-\delta} ) \omega \psi (z) - 3,
	  \end{aligned} 
\end{flalign} 

\noindent  and by for $z \in \partial \mathfrak{P}_{f; g}$ setting $\varphi_{j;k}^- (z) = H(z)$ and $\varphi_{j;k}^+ (z) = H(z)$. Further define the events
\begin{flalign*}
	& \Omega_{j;k}^+ (z) = \big\{ n^{-1} \mathsf{H}^{\bm{x}} (z) > \varphi_{j;k}^+ (z) \big\}; \qquad \Omega_{j;k}^- (z) = \big\{ n^{-1} \mathsf{H}^{\bm{x}} (z) < \varphi_{j;k}^- (z) \big\}. 
\end{flalign*}

 The following proposition indicates that if $H$ lies below (or above) $\varphi_{j;k-1}^+$ (or $\varphi_{;k-1}^-$, respectively), then with high probability it also lies below (or above) $\varphi_{j;k}^+$ (or $\varphi_{j;k}^-$, respectively); it will be established in \Cref{ProofConcentration0} below. Here, we denote the complement of any event $\mathscr{E}$ by $\mathscr{E}^{\complement}$. 

\begin{prop}
	
	\label{gh1} 
	
	Adopt the notation and assumptions of \Cref{gh0}. For any index  $\pm \in \{ +, - \}$, integer pair $(j, k) \in \llbracket 0, J \rrbracket \times [0, K]$, and point $z \in \mathfrak{P}_{f; g}$, we have 
	\begin{flalign*}
		\mathbb{P} \Bigg[ \Omega_{j; k}^{\pm} (z) \cap \bigcap_{w \in \mathfrak{P}_{f;g}} \Omega_{j;k-1}^{\pm} (w)^{\complement}  \Bigg] \le C e^{-c(\log n)^2}. 
	\end{flalign*}
\end{prop}

Given \Cref{gh1}, we can quickly establish \Cref{gh0}. 

\begin{proof}[Proof of \Cref{gh0}]
	
	It suffices to show that 
	\begin{flalign}
		\label{ggamma01} 
		\begin{aligned}
		& \mathbb{P} \Bigg[  \bigcup_{z \in \mathfrak{P}_{f; g}} \big\{ n^{-1} \mathsf{H}^{\bm{x}} (z) \ge H(z)  + 2 n^{2/m + \delta - 1}  \big\} \Bigg] < Ce^{-c(\log n)^2}; \\ 
		& \mathbb{P} \Bigg[ \bigcup_{z \in \mathfrak{P}_{f; g}} \big\{ n^{-1} \mathsf{H}^{\bm{x}} (z) \le H(z)  - 2 n^{2/m + \delta - 1}  \big\} \Bigg] < Ce^{-c(\log n)^2}.
		\end{aligned} 
	\end{flalign}

	\noindent We only establish the first bound in \eqref{ggamma01}, for the proof of the latter is entirely analogous. To that end, for any integer $j \in \llbracket 0, J \rrbracket$, define the function $\Phi_j : \overline{\mathfrak{P}_{f; g}} \rightarrow \mathbb{R}$ by 
	\begin{flalign*} 
		\Phi_j (z) = H(z) - j \omega  + 3.
	\end{flalign*} 

	\noindent and for any $z \in \mathfrak{P}_{f; g}$ define the event
	\begin{flalign*} 
	\Omega_j (z) = \big\{ n^{-1} \mathsf{H}^{\bm{x}} (z) > \Phi_j (z) \big\}.
	\end{flalign*} 

	\noindent Then, observe from \eqref{functionsjkz} and \eqref{rhoomega} that 
	\begin{flalign}
		\label{functions2} 
		\varphi_{j; K}^+ (z) \le \Phi_j (z) \le \Phi_{j-1} (z) \le \varphi_{j; 0}^+ (z),
	\end{flalign}

	\noindent where the first bound follows from the fact that $(1 - K \rho^2 n^{-\delta}) \sup_{z \in \mathfrak{P}_{f; g}} \psi (z) \le \frac{1}{4}$; the second from the definition of $\Phi_j$; and the third from the fact that $\inf_{z \in \mathfrak{P}_{f; g}} \psi (z) = 2$.  
 
	Define the set $\mathcal{S} \subset \mathfrak{P}_{f; g}$ consisting of all points $(t, x) \in \mathfrak{P}_{f; g}$ such that $t \in n^{-15} \cdot \mathbb{Z}$ and $x = G(t, y)$, for some $y \in n^{-1} \cdot \mathbb{Z}$; then, $|\mathcal{S}| \le n^{20}$. Taking a union over $z \in \mathcal{S}$ of \Cref{gh1} yields constants $c_1 = c_1 (\varepsilon, \delta, B_0, m) > 0$ and $C_1 = C_1 (\varepsilon, \delta, B_0, m) > 1$ satisfying   
	\begin{flalign}
		\label{ggamma03} 
		\mathbb{P} \left[\bigcap_{w \in \mathfrak{P}_{f; g}} \Omega_{j;k-1}^+ (w)^{\complement}\setminus \mathscr{E}_{j; k}\right] \leq C_1 e^{-c_1 (\log n)^2}, 
	\end{flalign}
where 
\begin{align*}
\mathscr{E}_{j; k} = \bigcap_{z' \in \mathcal{S}} \Omega_{j;k}^+ (z')^{\complement} \cap \bigcap_{w \in \mathfrak{P}_{f; g}} \Omega_{j;k-1}^+ (w)^{\complement}.
\end{align*}
	\noindent By \Cref{n10estimate0}, there also exists a constant $C_2 = C_2 (\delta, B_0) > 1$ such that
	\begin{flalign}
		\label{eventf1}
		\mathbb{P}[\mathscr{F}] \ge 1 - C_2 e^{-n}, \quad \text{where} \quad \mathscr{F} = \Bigg\{ \displaystyle\max_{j \in \llbracket 1, n \rrbracket} \displaystyle\sup_{0 < t < t+s \le 1/L} \big| x_j (t+s) - x_j (t) \big| \le n^{\delta/2} s^{1/2} + n^{-5} \Bigg\}.
	\end{flalign}
	
	Next, for any $j \in \llbracket 0, J \rrbracket$, $k \in [0, K]$, $s \in [0, L^{-1}]$, and $r \ge 0$, let 
	\begin{flalign*} 
		\gamma_{j;k}^+ (r; s) = \sup \big\{ \gamma \in \mathbb{R} : \varphi_{j;k}^+ (s, \gamma) \ge r \big\}.
	\end{flalign*} 

	\noindent Fix $(i, s) \in \llbracket 1, n \rrbracket \times [0, L^{-1}]$, and let $s' \in [0, L^{-1}]$ be such that $\big( s', G(s', in^{-1}) \big) \in \mathcal{S}$ and $|s-s'|$ is minimal; then, $|s-s'| \le n^{-15}$. On the event $\mathscr{E}_{j; k} \cap \mathscr{F}$, we have for sufficiently large $n$ that
	\begin{flalign*}
		x_i (s) \le x_i (s') + 2n^{-5} \le \gamma_{j; k}^+ (in^{-1}; s') + 2n^{-5} < \gamma_{j;k}^+ (in^{-1}; s) + 4n^{-5} < \gamma_{j;k}^+ (in^{-1} + n^{-4}; s).
	\end{flalign*}

	\noindent Here, the first inequality follows from restricting to $\mathscr{F}$; the second follows from restricting to $\mathscr{E}_{j;k}$; and the third and fourth follow from the facts that $\varphi_{j;k}^+ \in \Adm_{\varepsilon/2} (\mathfrak{P}_{f; g})$, that $\| \varphi_{j;k}^+ \|_{\mathcal{C}^1 (\mathfrak{P}_{f; g})} \le C_3$ for some constant $C_3 = C_3 (\varepsilon, \delta, B, m) > 1$, and that $n$ sufficiently large (where these bounds follow from \eqref{functionsjkz}, since $H \in \Adm_{\varepsilon} (\mathfrak{P}_{f; g})$ by \eqref{derivativexhtx}, $\| H \|_{\mathcal{C}^1 (\mathfrak{P}_{f; g})} \le B$ by \eqref{hderivativem}, $\| \psi \|_{\mathcal{C}^1 (\mathfrak{P}_{f; g})} \le 10 \zeta e^{\zeta}$, and $\omega = n^{2/m+\delta-1}$). Since this holds for each $(i, s) \in \llbracket 1, n \rrbracket \times [0, L^{-1}]$, it follows on $\mathscr{E}_{j;k} \cap \mathscr{F}$ that   
	\begin{flalign*}
		 n^{-1} \mathsf{H}^{\bm{x}} (z) \le \varphi_{j;k}^+ (z) + n^{-4} < \varphi_{j;k-1/2}^+ (z), \qquad \text{for each $z \in \mathfrak{P}_{f; g}$, on the event $\mathscr{E}_{j;k} \cap \mathscr{F}$}.
	\end{flalign*}

	\noindent Applying this bound (decreasing $k$ by $\frac{1}{2}$ at each point) $2K+2 \le 2n$ times, and using \eqref{ggamma03} and \eqref{eventf1}, yields constants $c_2 = c_2 (\varepsilon, \delta, B_0, m) > 0$ and $C_4 = C_4 (\varepsilon, \delta, B_0, m) > 1$ such that
	\begin{flalign*}
		\mathbb{P} \Bigg[ \bigcup_{z \in \mathfrak{P}_{f; g}} \big\{ n^{-1} \mathsf{H}^{\bm{x}} (z) > \varphi_{j;K}^+ (z) \big\} \cap \bigcap_{w \in \mathfrak{P}_{f; g}} \Omega_{j; 0}^+ (w)^{\complement}  \Bigg] \le C_4 e^{-c_2 (\log n)^2}.  
	\end{flalign*} 

	\noindent By \eqref{functions2}, we deduce
	\begin{flalign*}
		\mathbb{P} \Bigg[ \bigcup_{z \in \mathfrak{P}_{f; g}} \Omega_j (z) \cap \bigcap_{w \in \mathfrak{P}_{f; g}} \Omega_{j-1} (z)^{\complement}  \Bigg] \le C_4 e^{-c_2 (\log n)^3},
	\end{flalign*} 
		
	\noindent which upon applying $J \le n$ times implies  
	\begin{flalign*}
		\mathbb{P} \Bigg[ \bigcup_{z \in \mathfrak{P}_{f; g}} \Omega_J (z) \cap \bigcap_{w \in \mathfrak{P}_{f; g}} \Omega_0 (z)^{\complement}  \Bigg] \le C_4 n e^{-c_2 (\log n)^2}. 
	\end{flalign*} 

	\noindent With a union bound and the facts that $\Phi_J (z) \le H(z)  + 2 \omega$ (since $J = \lfloor 3 \omega^{-1} \rfloor - 1$) and $\Phi_0 (z) \le H(z)  + 3$, this yields 
	\begin{flalign*}
		\mathbb{P} \Bigg[ \bigcup_{z \in \mathfrak{P}_{f; g}} \big\{ n^{-1} \mathsf{H}^{\bm{x}} (z) \ge H(z)  & + 2 n^{2/m + \delta-1} \big\}  \Bigg] \\
		& < \mathbb{P} \Bigg[ \bigcup_{z \in \mathfrak{P}_{f; g}} \big\{ n^{-1} \mathsf{H}^{\bm{x}} (z) \ge H(z)  + 3 \}  \Bigg] + C_4 n e^{-c_2 (\log n)^2}.
	\end{flalign*}

	\noindent Since $n^{-1} \mathsf{H}^{\bm{x}} (z) \le 1 \le H(z) + 3$ holds deterministically, this yields the first bound in \eqref{ggamma01} and thus the proposition. 
\end{proof}

\section{The Local Behaviors of Height Profiles}

\label{HeightLocal} 

In this section we explain how to approximate the first several derivatives (which we refer to as a ``local profile'') of a limit shape by another one, associated with a family of non-intersecting Brownian bridges with known concentration bounds. 

\subsection{Local Profiles}

\label{Derivatives0} 

In this section we introduce the notion of a local profile, which constitutes information about the first several derivatives of a function. To that end, for any integer $k \ge 1$; open subset $\mathfrak{R} \subset \mathbb{R}^2$; and function $F : \mathcal{C}^k (\overline{\mathfrak{R}})$, and point $z \in \mathfrak{R}$, define the \emph{$k$-th order local profile} $\bm{q}_F^{(k)} (z) \in \mathbb{R}^{k+1}$ by 
\begin{flalign}
	\label{kqf}
	\bm{q}_F^{(k)} (z) = \big( \partial_x^k  F(z), \partial_t \partial_x^{k-1} F(z), \ldots , \partial_t^k F(z) \big).
\end{flalign}

\noindent For any integer $m \ge 1$, the \emph{full-$m$ local profile} is the sequence of vectors $\bm{Q}_F^{(m)} (z) = \big( \bm{q}_F^{(k)} (z) \big)_{k \in \llbracket 1, m \rrbracket}$.

We next describe two properties that a full-$m$ local profile $\bm{Q}_F^{(m)}$ must satisfy, if some admissible $F$ solves the equation \eqref{equationxtb} on some open subset $\mathfrak{R} \subseteq \mathbb{R}^2$. The first is that the second coordinate of $\bm{q}_F^{(1)}$ must be negative, due to the admissibility of $F$. The second is that $F$ must satisfy \eqref{equationxtb}, which imposes a consistency relation between $\bm{q}_F^{(1)}$ and $\bm{q}_F^{(2)}$. More generally, applying the differential operator $\partial_t^i \partial_x^j$ to \eqref{equationxtb} imposes a relation between $\big( \bm{q}_F^{(1)}, \bm{q}_F^{(2)}, \ldots , \bm{q}_F^{(i+j+2)} \big)$, for each $(i, j) \in \mathbb{Z}_{\ge 0}^2$ such that $i + j \le m - 2$. These admissibility and consistency properties are made precise through the following definition. Throughout, unless mentioned otherwise, we will denote sequences of $m$ vectors $\bm{Q} = \big( \bm{q}^{(k)} \big)_{\llbracket 1, m \rrbracket} \in \mathbb{R}^2 \times \mathbb{R}^3 \times \cdots \times \mathbb{R}^{m+1}$ by $\bm{q}^{(k)} = \big( q_0^{(k)}, q_1^{(k)}, \ldots , q_k^{(k)} \big)$, for each $k \in \llbracket 1, m \rrbracket$.

\begin{definition} 
	
\label{admissibleconsistent}

A sequence of $m$ vectors $\bm{Q} = \big( \bm{q}^{(k)} \big) \in \mathbb{R}^2 \times \mathbb{R}^3 \times \cdots \times \mathbb{R}^{m+1}$ is \emph{admissible} if $q_0^{(1)} < 0$. For any integers $i, j \ge 0$, define the function $\mathfrak{v}_{i,j}: \mathbb{R}^2 \times \mathbb{R}^3 \times \cdots \times \mathbb{R}^{i+j+3} \rightarrow \mathbb{R}$ so that
\begin{flalign*}
	\mathfrak{v}_{i,j} \big( \bm{Q}_F^{(i+j+2)} (z) \big) = \partial_t^i \partial_x^j \Bigg( \displaystyle\sum_{i', j' \in \{ t, x \}} \mathfrak{b}_{i'j'} \big( \nabla F(z) \big) \partial_{i'} \partial_{j'} F(z) \Bigg).
\end{flalign*}

\noindent We say that the sequence $\bm{Q}$ is \emph{consistent} if it is admissible and
\begin{flalign}
	\label{vijequation}
\mathfrak{v}_{i,j} \big( \bm{q}^{(1)}, \bm{q}^{(2)}, \ldots , \bm{q}^{(i+j+2)} \big) = 0, \qquad \text{for each $(i, j) \in \mathbb{Z}_{\ge 0}^2$ with $i + j \le m - 2$}. 
\end{flalign}

\end{definition} 

\begin{example} 

Since
\begin{flalign*}
	\mathfrak{v}_{1,0} \big( \bm{Q}_F^{(3)} (z) \big) & = \mathfrak{b}_{tt} \big( \nabla F(z) \big) \partial_{ttt} F(z) + 2 \mathfrak{b}_{tx} \big( \nabla F(z) \big) \partial_{ttx} F(z) + \mathfrak{b}_{xx} \big( \nabla F(z) \big) \partial_{txx} F(z) \\
	& \qquad + \Big( \partial_{tt} F(z) \partial_u \mathfrak{b}_{tt} \big( \nabla F(z) \big) + \partial_{tx} F(z) \partial_v \mathfrak{b}_{tt} \big( \nabla F(z) \big) \Big) \partial_{tt} F(z) \\
	& \qquad + 2 \Big( \partial_{tt} F(z) \partial_u \mathfrak{b}_{tx} \big( \nabla F(z) \big) + \partial_{tx} F(z) \partial_v \mathfrak{b}_{tx} \big( \nabla F(z) \big) \Big) \partial_{tx} F(z) \\
	& \qquad + \Big( \partial_{tt} F(z) \partial_u \mathfrak{b}_{xx} \big( \nabla F(z) \big) + \partial_{tx} F(z) \partial_v \mathfrak{b}_{xx} \big( \nabla F(z) \big) \Big) \partial_{xx} F(z),
\end{flalign*}

\noindent we have, for any sequence $\bm{Q}^{(3)} = \big( \bm{q}^{(k)} \big)_{k \in \llbracket 1, 3 \rrbracket}$ with $\bm{q}^{(k)} = \big( q_0^{(k)}, q_1^{(k)}, \ldots,  q_k^{(k)} \big)$ for each $k \in \llbracket 1, 3 \rrbracket$,
\begin{flalign*}
	\mathfrak{v}_{1,0} \big( \bm{Q}^{(3)} \big) & = \mathfrak{b}_{tt} \big( \bm{q}^{(1)} \big) q_3^{(3)} + 2 \mathfrak{b}_{tx} \big( \bm{q}^{(1)} \big) q_2^{(3)} + \mathfrak{b}_{xx} \big( \bm{q}^{(1)} \big) q_1^{(3)}  + \Big( q_2^{(2)} \partial_u \mathfrak{b}_{tt} \big( \bm{q}^{(1)} \big) + q_1^{(2)} \partial_v \mathfrak{b}_{tt} \big( \bm{q}^{(1)} \big) \Big) q_2^{(2)} \\
	& \quad + 2 \Big( q_2^{(2)} \partial_u \mathfrak{b}_{tx} \big( \bm{q}^{(1)} \big) + q_1^{(2)} \partial_v \mathfrak{b}_{tx} \big( \bm{q}^{(1)} \big) \Big) q_1^{(2)} + \Big( q_2^{(2)} \partial_u \mathfrak{b}_{xx} \big( \bm{q}^{(1)} \big) + q_1^{(2)} \partial_v \mathfrak{b}_{xx} \big( \bm{q}^{(1)} \big) \Big) q_0^{(2)}.
\end{flalign*}

\end{example}

\begin{rem}
	
	\label{fconsistent}
	
	For any function $F \in \mathcal{C}^2 (\overline{\mathfrak{R}}) \cap \Adm (\mathfrak{R})$ satisfying \eqref{equationxtb}, the sequence $\bm{Q}_F^{(m)} (z)$ is consistent, for each integer $m \ge 1$ and point $z \in \mathfrak{R}$; this follows from applying $\partial_t^i \partial_x^j$ to \eqref{equationxtb} (for each $(i, j) \in \mathbb{Z}_{\ge 0}^2$ with $i + j \le m + 2)$.
	 
\end{rem}

Observe that consistency imposes $\binom{m}{2}$ relations (one for each $(i, j) \in \mathbb{Z}_{\ge 0}$ with $i + j \le m - 2$) between the $\binom{m+2}{2} - 1$ coordinates of $\bm{Q}$; this suggests that there are $\binom{m+2}{2} - 1 - \binom{m}{2} = 2m$ ``free parameters'' describing a consistent $\bm{Q}$. The following lemma indicates this to indeed be the case, taking these $2m$ parameters to be the first two coordinates $\big( q_0^{(k)}, q_1^{(k)} \big)_{k \in \llbracket 1, m \rrbracket}$ of each vector in $\bm{Q}$. 

\begin{lem} 
	
\label{q01j}

Fix an integer $m \ge 1$ and $m$ pairs of real numbers $\big( q_0^{(k)}, q_1^{(k)} \big) \in (\mathbb{R}^2)^m$ with $q_0^{(1)} < 0$. 

\begin{enumerate} 
	
	\item There is a unique consistent sequence of $m$ vectors $\bm{Q} = \big( \bm{q}^{(k)} \big) \in \mathbb{R}^2 \times \mathbb{R}^3 \times \cdots \times \mathbb{R}^{m+1}$ so that, for each $k \in \llbracket 1, m \rrbracket$, the first two coordinates of $\bm{q}^{(k)}$ are $\big( q_0^{(k)}, q_1^{(k)} \big)$.
	\item For any real numbers $\varepsilon \in (0, 1)$ and $B > 1$, there exists a constant $C = C (\varepsilon, B, m) > 1$ so that, if 
	\begin{flalign}
		\label{q0q1estimate}
		q_0^{(1)} \le -\varepsilon, \quad \text{and} \quad \max_{1 \le k \le m} \Big( \big| q_0^{(k)} \big| + \big| q_1^{(k)} \big| \Big) \le B,
	\end{flalign}

	\noindent then the consistent sequence $\bm{Q}$ sastisfies $\big| q_j^{(k)} \big| \le C$ for each $k \in \llbracket 1, m \rrbracket$ and $j \in \llbracket 0, k \rrbracket$. 
\end{enumerate} 

\end{lem} 

\begin{proof}
	
	Observe for any $(i, j) \in \mathbb{Z}_{> 0}^2$ that 
	\begin{flalign*}
		\partial_t^i \partial_x^j \Bigg( \displaystyle\sum_{i', j' \in \{ t, x \}} \mathfrak{b}_{i' j'} \big( \nabla F(z) \big) \partial_{i'} \partial_{j'} F(z) \Bigg) = \mathfrak{b}_{tt} \big( \nabla F(z) \big) \partial_t^{i+2} \partial_x^j F(z) + \mathfrak{u}_{i, j} \Big( \big( \partial_t^{i'} \partial_x^{j'} F (z) \big)_{i', j'} \Big),
	\end{flalign*}
	
	\noindent for some function $\mathfrak{u}_{i, j}$ that depends polynomially (as the $\mathfrak{b}_{ij} (u, v)$ from \eqref{uvb} are polynomials in $(u, v)$) only on those derivatives $\partial_t^{i'} \partial_x^{j'} F(z)$ such that either $i' \le i + 1$ or $i' + j' < i + j + 2$. As such, for any sequence of $m$ vectors $\bm{Q} = \big( \bm{q}^{(1)}, \bm{q}^{(2)}, \ldots , \bm{q}^{(m)} \big) \in \mathbb{R}^2 \times \mathbb{R}^3 \times \cdots \mathbb{R}^{m+1}$, we have 
	\begin{flalign}
		\label{vijuij2}
		\mathfrak{v}_{i, j} \big( \bm{q}^{(1)}, \bm{q}^{(2)}, \ldots , \bm{q}^{(i+j+2)} \big) = q_{i+2}^{(i+j+2)} \mathfrak{b}_{tt} \big( \bm{q}^{(1)} \big) + \mathfrak{u}_{i, j} \Big( \big( q_{i'}^{(k')} \big)_{i', k} \Big),
		\end{flalign}
	
	\noindent where again $\mathfrak{u}_{i,j}$ is a polynomial in those $q_{i'}^{(k')}$ with either $i' \le i + 1$ or $k' \le i + j + 1$. Thus, 
	\begin{flalign}
		\label{vijuij}
		\mathfrak{v}_{i, j} \big( \bm{q}^{(1)}, \bm{q}^{(2)}, \ldots , \bm{q}^{(i+j+2)} \big) = 0, \quad \text{if and only if} \quad q_{i+2}^{(i+j+2)} = -\displaystyle\frac{\mathfrak{u}_{i,j} \Big( \big (q_{i'}^{(k')} \big)_{i', k'} \Big)}{\mathfrak{b}_{tt} \big( \bm{q}^{(1)} \big)}.
	\end{flalign}

	Hence since $q_0^{(1)} < 0$, meaning $\mathfrak{b}_{tt} \big( \bm{q}^{(1)} \big) > 0$, the consistency of $\bm{Q}$ provides a unique solution for $q_{i+2}^{(i+j+2)}$, in terms of all $\big( q_{i'}^{(k')} \big)$ such that either $i' \le i + 1$ or $k' \le i+j+1$. In particular, given the pairs $\big( q_0^{(k)}, q_1^{(k)} \big)$, this by induction on the lexicographic pair $(i+j+2, i+2)$ determines a unique consistent solution for the remaining entries of $\bm{Q}$, establishing the first statement of the lemma. 
	
	To establish the second, observe from \eqref{uvb} that $\mathfrak{b}_{tt} \big( \bm{q}^{(1)} \big) \ge \varepsilon^2$, as $q_0^{(1)} \le -\varepsilon$. Thus, since the $\mathfrak{u}_{i,j}$ are polynomials with coefficients and degrees only dependent on $i, j \in \llbracket 1, m \rrbracket$, it follows by \eqref{vijuij}\, \eqref{q0q1estimate}, and induction on the lexicographic pair $(i+j+2, i+2) \in \llbracket 2, m \rrbracket^2$ that for each $(i, j) \in \llbracket 0, m \rrbracket^2$ there exists a constant $C_{i,j} = C_{i,j} (\varepsilon, B) > 1$ such that $\big| q_{i+2}^{(i+j+2)} \big| \le C_{i, j}$. This (together with the fact from \eqref{q0q1estimate} that $\big| q_0^{(k)} \big| + \big| q_1^{(k)} \big| \le B$ for each $k \in \llbracket 1, m \rrbracket$) yields the second statement of the lemma, with $C = \max_{i+j \le m-2} C_{i,j} + B$. 
\end{proof} 

\begin{definition}
	
	\label{qextend} 
	
	Adopting the notation of \Cref{q01j}, we call $\bm{Q} = \big( \bm{q}^{(k)} \big)$ the \emph{consistent extension} of the sequence of $m$ pairs $\big( q_0^{(k)}, q_1^{(k)} \big) \in (\mathbb{R}^2)^m$. 
\end{definition}

We conclude this section with the following two lemmas. The first is a continuity bound, saying that if two sequences of real pairs are close, then their consistent extensions are as well. The second implies that, if $\widetilde{\bm{Q}}$ is a sequence that ``almost satisfies''  the consistency relations \eqref{vijequation}, then there is a consistent sequence $\bm{Q}$ close to it.

\begin{lem} 
	
	\label{continuityextend}
	
	For any integer $m \ge 1$ and real numbers $\varepsilon \in (0, 1)$ and $B > 1$, there exists a constant $C = C(\varepsilon, B, m) > 1$ such that the following holds. Let $\big( q_0^{(k)}, q_1^{(k)} \big) \in (\mathbb{R}^2)^m$ and $\big( \widetilde{q}_0^{(k)}, \widetilde{q}_1^{(k)} \big) \in (\mathbb{R}^2)^m$ be two sequences of $m$ real number pairs such that 
	\begin{flalign*}
		\displaystyle\max \big\{ q_0^{(1)}, \widetilde{q}_0^{(1)} \} \le -\varepsilon; \qquad \displaystyle\max_{1 \le k \le m} \Big( \big| q_0^{(k)} \big| + \big| q_1^{(k)} \big| + \big| \widetilde{q}_0^{(k)} \big| + \big| \widetilde{q}_1^{(k)} \big| \Big) \le B.
	\end{flalign*}
	
	\noindent Letting $\bm{Q} = \big( \bm{q}^{(k)} \big)$ and $\widetilde{\bm{Q}} = \big( \widetilde{\bm{q}}^{(k)} \big)$ denote the consistent extensions of $\big( q_0^{(k)}, q_1^{(k)} \big)$ and $\big( \widetilde{q}_0^{(k)}, \widetilde{q}_1^{(k)} \big)$, respectively, we have 
	\begin{flalign*}
		\big\| \bm{q}^{(k)} - \widetilde{\bm{q}}^{(k)} \big\| \le C \displaystyle\sum_{k' = 1}^k \Big( \big| q_0^{(k')} - \widetilde{q}_0^{(k')} \big| + \big| q_1^{(k')} - \widetilde{q}_1^{(k')} \big| \Big), \qquad \text{for each $k \in \llbracket 1, m \rrbracket$}. 
	\end{flalign*}
	
\end{lem} 

\begin{proof}
	
	Throughout, we recall the notation from the proof of \Cref{q01j}. First observe from the second part of \Cref{q01j} that there exists a constant $C_1 = C_1 (\varepsilon, B, m) > 1$ such that 
	\begin{flalign}
		\label{qsumq} 
		\displaystyle\sum_{k = 1}^m \Big( \big| \bm{q}^{(k)} \big| + \big| \widetilde{\bm{q}}^{(k)} \big| \Big) \le C_1.
	\end{flalign}
	
	\noindent Next, due to \eqref{qsumq}; the definition \eqref{uvb} of $\mathfrak{b}_{tt}$; and the fact that $\max \big\{ q_0^{(1)}, \widetilde{q}_0^{(1)} \big\} \le -\varepsilon$, there exists a constant $C_2 = C_2 (\varepsilon, m) > 1$ such that
	\begin{flalign}
		\label{aestimatea}
		& \bigg| \displaystyle\frac{1}{\mathfrak{b}_{tt} \big( \bm{q}^{(1)} \big)} - \displaystyle\frac{1}{\mathfrak{b}_{tt} \big( \widetilde{\bm{q}}^{(1)} \big)} \bigg| \le C_2 \big| \bm{q}^{(1)} - \widetilde{\bm{q}}^{(1)} \big|; \qquad \bigg| \displaystyle\frac{1}{\mathfrak{b}_{tt} \big( \bm{q}^{(1)} \big)} \bigg| + \bigg| \displaystyle\frac{1}{\mathfrak{b}_{tt} \big( \widetilde{\bm{q}}^{(1)} \big)} \bigg| \le C_2.
	\end{flalign}
	
	\noindent Moreover, by \eqref{qsumq}; the fact that the $\mathfrak{u}_{i,j}$ (satisfying \eqref{vijuij}) are polynomials that only depend on the $q_{i'}^{(k')}$ with either $i' \le i+1$ or $k' \le i+j+1$; and the fact that the coefficients and degrees of $\mathfrak{u}_{i,j}$ only depend on $i, j \in \llbracket 1, m \rrbracket$, there exists a constant $C_3 = C_3 (\varepsilon, B, m) > 1$ such that
	\begin{flalign}
		\label{uijuij}
		\begin{aligned}
			& \bigg| \mathfrak{u}_{i,j} \Big( \big( q_{i'}^{(k')} \big)_{i',k'} \Big) - \mathfrak{u}_{i,j} \Big( \big( \widetilde{q}_{i'}^{(k')} \big)_{i', k'} \Big) \bigg| \le C_3 \Bigg( \displaystyle\sum_{i' = 0}^{i+1} \displaystyle\sum_{k'=1}^{i+j+2} \big| q_{i'}^{(k')} - \widetilde{q}_{i'}^{(k')} \big| + \displaystyle\sum_{i' = 0}^{i+2} \displaystyle\sum_{k' = 1}^{i+j+1} \big| q_{i'}^{(k')} - \widetilde{q}_{i'}^{(k')} \big| \Bigg); \\
			& \bigg| \mathfrak{u}_{i,j} \Big( \big( q_{i'}^{(k')} \big)_{i', k'} \Big) \bigg| + \bigg| \mathfrak{u}_{i,j} \Big( \big( \widetilde{q}_{i'}^{(k')} \big)_{i', k'} \Big) \bigg| \le C_3.
		\end{aligned} 
	\end{flalign}
	
	Thus, \eqref{vijuij}, \eqref{aestimatea}, and \eqref{uijuij} together yield a constant $C_4 = C_4 (\varepsilon, B, m) > 1$ so that
	\begin{flalign*}
		\big| q_{i+2}^{(i+j+2)} - \widetilde{q}_{i+2}^{(i+j+2)} \big| \le C_4 \Bigg( \displaystyle\sum_{i' = 0}^{i+1} \displaystyle\sum_{k'=1}^{i+j+2} \big| q_{i'}^{(k')} - \widetilde{q}_{i'}^{(k')} \big| + \displaystyle\sum_{i' = 0}^{i+2} \displaystyle\sum_{k' = 1}^{i+j+1} \big| q_{i'}^{(k')} - \widetilde{q}_{i'}^{(k')} \big| \Bigg).
	\end{flalign*}
	
	\noindent Induction on the lexicographic pair $(i+j+2, i+2) \in \llbracket 2, m \rrbracket^2$ then gives a constant $C_{i, j} = C_{i, j} (\varepsilon, B, m) > 1$ for any $(i, j) \in \llbracket 0, m \rrbracket^2$ with $i + j + 2 \le m$ such that
	\begin{flalign*}
		\big| q_{i+2}^{(i+j+2)} - \widetilde{q}_{i+2}^{(i+j+2)} \big| \le C_{i, j} \displaystyle\sum_{k' = 1}^{i+j+2} \Big( \big| q_0^{(k')} - \widetilde{q}_0^{(k')} \big| + \big| q_1^{(k')} - \widetilde{q}_1^{(k')} \big| \Big).
	\end{flalign*}
	
	\noindent Summing over all such $(i, j)$ with $i+j = k-2$, for any $k \in \llbracket 1, m \rrbracket$, we deduce the lemma.
\end{proof}

\begin{lem} 

\label{qestimateq}

For any integer $m \ge 1$ and real numbers $\varepsilon \in (0, 1)$ and $B > 1$, there exists a constant $C = C(\varepsilon, B, m) > 1$ such that the following holds. Let $\big( q_0^{(k)}, q_1^{(k)} \big) \in (\mathbb{R}^2)^m$ denote a sequence of $m$ pairs of real numbers satisfying \eqref{q0q1estimate}, and let $\bm{Q} = \big( \bm{q}^{(k)} \big)$ denote its consistent extension. Let $\widetilde{\bm{Q}} = \big( \widetilde{\bm{q}}^{(k)} \big) \in \mathbb{R}^2 \times \mathbb{R}^3 \times \cdots \times \mathbb{R}^{m+1}$ be such that $\widetilde{q}_0^{(k)} = q_0^{(k)}$ and $\widetilde{q}_1^{(k)} = q_1^{(k)}$ for each $k \in \llbracket 1, m \rrbracket$. If $\big| \widetilde{q}_j^{(k)} \big| \le B$ for each $k \in \llbracket 1, m \rrbracket$ and $j \in \llbracket 0, k \rrbracket$, then
\begin{flalign*}
	\big| \bm{q}^{(k)} - \widetilde{\bm{q}}^{(k)} \big| \le C \displaystyle\sum_{h=2}^k \displaystyle\sum_{i=0}^{h-2} \Big| \mathfrak{v}_{i,h-i-2} \big( \widetilde{\bm{q}}^{(1)}, \widetilde{\bm{q}}^{(2)}, \ldots , \widetilde{\bm{q}}^{(h)} \big) \Big|, \qquad \text{for each $k \in \llbracket 1, m \rrbracket$}. 
\end{flalign*}

\end{lem} 

\begin{proof}

Throughout, we recall the notation from the proof of \Cref{q01j}. Following the beginning of the proof of \Cref{continuityextend}, we deduce the existence of a constant $C_1 = C_1 (\varepsilon, B, m) > 1$ such that \eqref{qsumq} holds (by the assumption of this lemma, as well as the second part of \Cref{q01j}); a constant $C_2 = C_2 (\varepsilon, m) > 1$ such that \eqref{aestimatea} holds (due to \eqref{qsumq}; the definition \eqref{uvb} of $\mathfrak{b}_{tt}$; and the fact that $q_0^{(1)} = \widetilde{q}_0^{(1)} \le -\varepsilon$); and a constant $C_3 = C_3 (\varepsilon, B, m) > 1$ such that \eqref{uijuij} holds (by \eqref{qsumq}; the fact that the $\mathfrak{u}_{i,j}$ are polynomials that only depend on the $q_{i'}^{(k')}$ with either $i' \le i+1$ or $k' \le i+j+1$; and the fact the coefficients and degrees of $\mathfrak{u}_{i,j}$ only depend on $i, j \in \llbracket 1, m \rrbracket$). 

Next, by \eqref{vijuij2}, we have 
\begin{flalign*}
	q_{i+2}^{(i+j+2)} = - \displaystyle\frac{\mathfrak{u}_{i,j} \Big( \big( q_{i'}^{(k')} \big)_{i',k'} \Big)}{\mathfrak{b}_{tt} \big( \bm{q}^{(1)} \big)}; \qquad \widetilde{q}_{i+2}^{(i+j+2)} = \displaystyle\frac{\mathfrak{v}_{i,j} \big( \widetilde{\bm{q}}^{(1)}, \widetilde{\bm{q}}^{(2)}, \ldots , \widetilde{\bm{q}}^{(i+j+2)} \big)}{\mathfrak{b}_{tt} \big( \widetilde{\bm{q}}^{(1)} \big)}- \displaystyle\frac{\mathfrak{u}_{i,j} \Big(\big( \widetilde{q}_{i'}^{(k')} \big)_{i',k'} \Big)}{\mathfrak{b}_{tt} \big( \widetilde{\bm{q}}^{(1)} \big)}.
\end{flalign*}

\noindent This, together with \eqref{aestimatea} and \eqref{uijuij}, yields a constant $C_4 = C_4 (\varepsilon, m) > 1$ so that
\begin{flalign*}
	\big| q_{i+2}^{(i+j+2)} - \widetilde{q}_{i+2}^{(i+j+2)} \big| \le C_4 \Bigg( \displaystyle\sum_{i' = 0}^{i+1} \displaystyle\sum_{k'=1}^{i+j+2}  \big| q_{i'}^{(k')} - \widetilde{q}_{i'}^{(k')} \big|  & + \displaystyle\sum_{i' = 0}^{i+2} \displaystyle\sum_{k' = 1}^{i+j+1} \big| q_{i'}^{(k')} - \widetilde{q}_{i'}^{(k')} \big| \\
	& + \Big| \mathfrak{v}_{i,j} \big( \widetilde{\bm{q}}^{(1)}, \widetilde{\bm{q}}^{(2)}, \ldots , \widetilde{\bm{q}}^{(i+j+2)} \big) \Big| \Bigg).
\end{flalign*}

\noindent Induction on the lexicographic pair $(i+j+2, i+2) \in \llbracket 2, m \rrbracket^2$ (together with the fact that $\big( q_0^{(k)}, q_1^{(k)} \big) = \big( \widetilde{q}_0^{(k)}, \widetilde{q}_1^{(k)} \big)$, for each $k \in \llbracket 1, m \rrbracket$) then gives a constant $C_{i, j} = C_{i, j} (\varepsilon, B, m) > 1$ for any $(i, j) \in \llbracket 0, m \rrbracket^2$ with $i + j + 2 \le m$ such that
\begin{flalign*}
	\big| q_{i+2}^{(i+j+2)} - \widetilde{q}_{i+2}^{(i+j+2)} \big| \le C_{i, j} \displaystyle\sum_{h = 2}^{i+j+2} \displaystyle\sum_{i' = 0}^{i+2} \Big| \mathfrak{v}_{i',h-i'-2} \big( \widetilde{\bm{q}}^{(1)}, \widetilde{\bm{q}}^{(2)}, \ldots , \widetilde{\bm{q}}^{(h)} \big) \Big|.
\end{flalign*}

\noindent Summing over all such $(i, j)$ with $i + j  = k-2$, for any $k \in \llbracket 2, m \rrbracket$ (recall $\bm{q}^{(1)} = \widetilde{\bm{q}}^{(1)}$, addressing the case $k = 1$), we deduce the lemma.
\end{proof}

\subsection{Concentrating Profiles} 

\label{ProfileEstimate} 

In this section we discuss families of (deterministic) height profiles that are limit shapes for systems of non-intersecting Brownian bridges that admit a (nearly) optimal concentration bound. They are given by the below definition.

\begin{definition} 

\label{hhestimate0}
 
Fix real numbers $a < b$ and two functions $f , g : [a, b] \rightarrow \mathbb{R}$ with $f < g$. Denote the strip domain $\mathfrak{S} = \mathfrak{P}_{f; g} \subset \mathbb{R}^2$, and let $\mathfrak{H} : \overline{\mathfrak{S}} \rightarrow \mathbb{R}$ be a function. Given real numbers $C > 1$ and $c, \delta \in \big(0, \frac{1}{2} \big)$, we say that $\mathfrak{H}$ is \emph{$(C; c; \delta)$-concentrating} if it is satisfies the equation \eqref{equationxtb} for each $z \in \mathfrak{S}$, and the following holds for any integer $n \ge 1$. There exists an integer $n' \in [cn, n]$; starting data $\bm{u} \in \overline{\mathbb{W}}_{n'}$; and ending data $\bm{v} \in \overline{\mathbb{W}}_{n'}$ on $\mathfrak{S}$ so that, sampling $n'$ non-intersecting Brownian bridges $\bm{y} =( y_1, y_2, \ldots , y_{n'}) \in \llbracket 1, n' \rrbracket \times \mathcal{C} \big( [a, b] \big)$ from the measure $\mathfrak{Q}_{f; g}^{\bm{u}; \bm{v}} (n^{-1})$ and denoting the associated height function by $\mathsf{H}^{\bm{y}}$ (recall \eqref{htx}), we have 
\begin{flalign*}
	\mathbb{P} \Bigg[ \displaystyle\sup_{z \in \mathfrak{S}} \big| n^{-1} \mathsf{H}^{\bm{y}} (z) - \mathfrak{H} (z) \big| \ge n^{\delta-1} \Bigg] \le C e^{-c(\log n)^2}. 
\end{flalign*} 
	
\end{definition}

We will be interested in height functions that are $(C; c; \delta)$-concentrating with small $\delta > 0$. While it is plausible that essentially any solution of \eqref{equationxtb} on $\mathfrak{R}$ has this property, it is not transparent to us how to show this directly. We instead show the following result stating that, if $\bm{Q} \in \mathbb{R}^2 \times \mathbb{R}^3 \times \cdots \times \mathbb{R}^{m+1}$ is an admissible sequence of $m$ vectors in the sense of \Cref{admissibleconsistent}, then (for small $\delta$) there exists a $(C; c; \delta)$-concentrating height function $\mathfrak{H}$ whose full-$m$ local profile at $z = (0, 0)$ is given by $\bm{Q}$; see \Cref{f:H_illustration} for a depiction. It will be established in \Cref{ConcentrateProfile} below. 

\begin{figure}
\center
\includegraphics[width=0.4\textwidth]{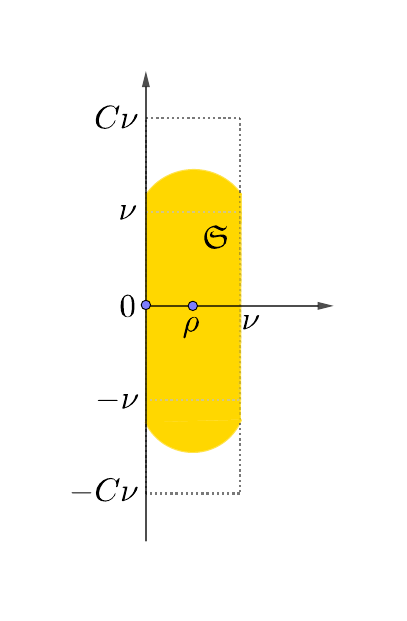}
\caption{Shown above is a depiction of \Cref{hconcentrateq} and \Cref{hconcentrateq2}. }
\label{f:H_illustration}
\end{figure}

\begin{prop} 
	
\label{hconcentrateq} 

For any integer $m \ge 1$ and real numbers $\varepsilon \in (0, 1)$ and $B > 1$, there exist constants $\nu = \nu (\varepsilon, B, m) \in (0, 1)$, $c = c(\varepsilon, B, m) > 0$, and $C = C(\varepsilon, B, m) > 1$ such that the following holds. Fix a sequence of $m$ pairs $\big( q_0^{(k)}, q_1^{(k)} \big) \in (\mathbb{R}^2)^m$ satisfying \eqref{q0q1estimate}; denote its consistent extension by $\bm{Q}$. There is a strip domain $\mathfrak{S} \subset \mathbb{R}^2$ of width $\nu$, such that $[0, \nu] \times [-\nu, \nu] \subseteq \mathfrak{S} \subseteq [0, \nu] \times [-C\nu, C\nu]$, and a function $\mathfrak{H}: \overline{\mathfrak{S}} \rightarrow \mathbb{R}$, such that $\| \mathfrak{H} \|_{\mathcal{C}^{m+1} (\overline{\mathfrak{S}})} \le C$ and $\sup_{z \in \mathfrak{S}} \partial_x \mathfrak{H} (z) \le -C^{-1}$, which satisfies the following two properties.

\begin{enumerate}
	\item The full-$m$ local profile of $\mathfrak{H}$ satisfies  $\lim_{z \rightarrow (0, 0)} \bm{Q}_{\mathfrak{H}}^{(m)} (z) = \bm{Q}$.
	\item For any real number $\delta \in (0, 1)$, there exists a constant $C_1 = C_1 (\varepsilon, \delta, B, m) > 1$ such that $\mathfrak{H}$ is $(C_1; c; \delta)$-concentrating. 
\end{enumerate} 

\end{prop} 

Observe that \Cref{hconcentrateq} fixes the local profile of $\mathfrak{H}$ at the point $(0, 0)$ on the boundary of $\partial \mathfrak{R}$ . However, it will be useful to (at least approximately) fix the local profile of $\mathfrak{H}$ at some point in the interior of $\mathfrak{R}$, to which end we have the following lemma and corollary; see \Cref{f:H_illustration} for a depiction of the latter.

\begin{lem} 
	
	\label{qrestimate} 
	
For any integer $m \ge 1$ and real numbers $\varepsilon \in (0, 1)$, $\nu \in (0, 1)$, and $B > 1$, there exist constants $c = c(\varepsilon, B, m) \in \big( 0, \frac{\nu}{2} \big)$ and $C = C(\varepsilon, B, m) > 1$ such that the following holds. Fix a consistent sequence $\bm{Q} = \big( \bm{q}^{(k)} \big) \in \mathbb{R}^2 \times \mathbb{R}^3 \times \cdots \times \mathbb{R}^{m+1}$; denoting $\bm{q}^{(k)} = \big( q_0^{(k)}, q_1^{(k)}, \ldots , q_k^{(k)} \big) \in \mathbb{R}^{k+1}$ for each $k \in \llbracket 1, m \rrbracket$, assume that \eqref{q0q1estimate} holds. Fix a real number $\rho \in (0, c)$ and set 
\begin{flalign}
	\label{r0r1}                                                                                                                           
	r_0^{(k)} = \displaystyle\sum_{i = 0}^{m-k} \displaystyle\frac{(-\rho)^i}{i!} q_i^{(k+i)}, \quad \text{and} \quad r_1^{(k)} = \displaystyle\sum_{i = 0}^{m-k} \displaystyle\frac{(-\rho)^i}{i!} q_{i+1}^{(k+i)}, \qquad \text{for each $k \in \llbracket 1, m \rrbracket$}.
\end{flalign}                                                                                                                           
                                                                                                                                          
\noindent Let $\bm{R} \in \mathbb{R}^2 \times \mathbb{R}^3 \times \cdots \times \mathbb{R}^{m+1}$ denote the consistent extension of the sequence of $m$ pairs $\big( r_0^{(k)}, r_1^{(k)} \big) \in (\mathbb{R}^2)^m$. If there exists a function $\mathfrak{H} : [0, \nu] \times [-\nu, \nu] \rightarrow \mathbb{R}$ satisfying \eqref{equationxtb} on $[0, \nu] \times [-\nu, \nu]$, with
\begin{flalign}                                                                                                                           
	\label{hqhestimate0}                                                                                                                     
	\| \mathfrak{H} \|_{\mathcal{C}^{m+1}} \le B, \quad \text{and} \quad \displaystyle\lim_{z \rightarrow (0, 0)} \bm{Q}_{\mathfrak{H}}^{(m)} (z) = \bm{R},
\end{flalign}                                                                                                                             
                                                                                                                                          
\noindent then $\big| \bm{q}_{\mathfrak{H}}^{(k)} (\rho, 0) - \bm{q}^{(k)} \big| \le C \rho^{m-k+1}$ for each $k \in \llbracket 1, m \rrbracket$, where $\bm{q}_{\mathfrak{H}}^{(k)}$ is as defined in \eqref{kqf}.
                                                                                                                                          
\end{lem}                                                                                                                                 
                                                                                                                                          
\begin{proof}                                                                                                                             
	                                                                                                                                         
	Throughout this proof, we set $q_0^{(0)} = 0$. First observe that, since \eqref{q0q1estimate} holds, \Cref{q01j} yields a constant $C_1 = C_1 (\varepsilon, B, m) > 1$ such that
	\begin{flalign}                                                                                                                          
		\label{qjkc}                                                                                                                            
		\displaystyle\max_{k \in \llbracket 0, m \rrbracket} \displaystyle\max_{j \in \llbracket 0, k \rrbracket} \big| q_j^{(k)} \big| \le C_1.                                     
	\end{flalign}                                                                                                                            
                                                                                                                                          
	\noindent This, together with the fact that $q_0^{(1)} \le -\varepsilon$ (by \eqref{q0q1estimate}) implies (by the definition \eqref{r0r1} of $r_0^{(k)}$ and $r_1^{(k)}$) that there exist constants $c_1 = c_1(\varepsilon, B, m) > 0$ and $C_2 = C_2 (\varepsilon, B, m) > 1$ such that for $\rho \in (0, c_1)$ we have $r_0^{(1)} \le -\frac{\varepsilon}{2}$ and $\big| r_i^{(k)} \big| \le C_2$, for each $i \in \{ 0, 1 \}$ and $k \in \llbracket 1, m \rrbracket$. Hence, the second part of \Cref{q01j} yields a constant $C_3 = C_3 (\varepsilon, B, m) > 1$ such that 
	\begin{flalign}
		\label{rjkc} 
		r_0^{(1)} \le -\frac{\varepsilon}{2}; \qquad \displaystyle\max_{k \in \llbracket 0, m \rrbracket} \displaystyle\max_{j \in \llbracket 0, k \rrbracket} \big| r_j^{(k)} \big| \le C_3.
	\end{flalign} 

	Next, define the (not necessarily consistent) sequence of $m+1$ vectors $\widetilde{\bm{R}} = \big( \widetilde{\bm{r}}^{(k)} \big)_{k \in \llbracket 0, m \rrbracket} \in \mathbb{R} \times \mathbb{R}^2 \times \cdots \times \mathbb{R}^{m+1}$, where $\widetilde{\bm{r}}^{(k)} = \big( \widetilde{r}_0^{(k)}, \widetilde{r}_1^{(k)}, \ldots , \widetilde{r}_k^{(k)} \big) \in \mathbb{R}^{k+1}$ for each $k \in \llbracket 0, m \rrbracket$, by 
	\begin{flalign}
		\label{rki} 
		\widetilde{r}_i^{(k)} = \displaystyle\sum_{i' = 0}^{m-k} \displaystyle\frac{(-\rho)^{i'}}{i'!} q_{i+i'}^{(k+i')},
	\end{flalign}
	
	\noindent so in particular 
	\begin{flalign} 
		\label{r0r10} 
		\widetilde{r}_0^{(k)} = r_0^{(k)};  \qquad \widetilde{r}_1^{(k)} = r_1^{(k)}.
	\end{flalign} 

	\noindent Let us express the $q_i^{(k)}$ through the $\widetilde{r}_i^{(k)}$. To that end, define the polynomials $A_0, A_1 : \mathbb{R}^2 \rightarrow \mathbb{R}$ by
	\begin{flalign}
		\label{a0a1} 
		A_0 (t, x) = \displaystyle\sum_{k = 1}^m \displaystyle\sum_{j=0}^k \displaystyle\frac{q_j^{(k)}}{j! (k-j)!} t^j x^{k-j}; \qquad A_1 (t, x) = \displaystyle\sum_{k=0}^m \displaystyle\sum_{j=0}^k \displaystyle\frac{\widetilde{r}_j^{(k)}}{j! (k-j)!} t^j x^{k-j}.
	\end{flalign}

	\noindent so that 
	\begin{flalign}
		\label{a0a13} 
		\begin{aligned}
		A_0 (t - \rho, x) & = \displaystyle\sum_{k = 1}^m \displaystyle\sum_{j=0}^k \displaystyle\sum_{i = 0}^{j} \displaystyle\frac{(-\rho)^{j-i}}{(k-j)! i! (j-i)!} q_j^{(k)} t^i x^{k-j} \\
		& = \displaystyle\sum_{i=0}^m \displaystyle\sum_{j'=0}^{m-i} \displaystyle\frac{t^i x^{j'}}{i! j'!} \displaystyle\sum_{k'=0}^{m-i-j'} \displaystyle\frac{(-\rho)^{k'}}{k'!} q_{i+k'}^{(i+j'+k')} = \displaystyle\sum_{i=0}^m \displaystyle\sum_{j'=0}^{m-i} \displaystyle\frac{\widetilde{r}_i^{(i+j')}}{i! j'!} t^i x^{j'} = A_1 (t, x),
		\end{aligned} 
	\end{flalign}

	\noindent where in the second equality we changed variables $(j', k') = (k - j, j-i)$ (and used the fact that $q_0^{(0)} = 0$), in the third we used the definition \eqref{rki} of $\widetilde{r}_i^{(i+j)}$, and in the fourth we again changed variables $(i, j') = (j, k-j)$. It follows (again since $q_0^{(0)} = 0$) that 
	\begin{flalign*}
		\displaystyle\sum_{k=0}^m \displaystyle\sum_{j = 0}^k \displaystyle\frac{t^j x^{k-j}}{j! (k-j)!} q_j^{(k)} = A_0 (t, x) = A_1 (t + \rho, x) & = \displaystyle\sum_{k=0}^m \displaystyle\sum_{j=0}^k \displaystyle\sum_{i=0}^j \displaystyle\frac{\rho^{j-i}}{(k-j)! i! (j-i)!} \widetilde{r}_j^{(k)} t^i x^{k-j} \\
		& = \displaystyle\sum_{i=0}^m \displaystyle\sum_{j'=0}^{m-i} \displaystyle\frac{t^i x^{j'}}{i! j'!} \displaystyle\sum_{k' = 0}^{m-i-j'} \displaystyle\frac{\rho^{k'}}{k'!} \widetilde{r}_{i+k'}^{(i+j'+k')} \\
		& = \displaystyle\sum_{k=0}^m \displaystyle\sum_{j'=0}^k \displaystyle\frac{t^j x^{k-j}}{j! (k-j)!} \displaystyle\sum_{i'=0}^{m-k} \displaystyle\frac{\rho^{i'}}{i'!} \widetilde{r}_{j+i'}^{(k+i')},
	\end{flalign*}

	\noindent where in the fourth equality we again changed variables $(j', k') = (k-j, j-i)$, and in the fifth we changed variables $(i,j',k') = (j, k-j, i')$. Comparing coefficients of $t^j x^{k-j}$, we deduce
	\begin{flalign}
		\label{qkir} 
		q_i^{(k)} = \displaystyle\sum_{i' = 0}^{m-k} \displaystyle\frac{\rho^{i'}}{i'!} \widetilde{r}_{i+i'}^{(k+i')}, \qquad \text{for each $k \in \llbracket 0, m \rrbracket$ and $i \in \llbracket 0, k \rrbracket$}.
	\end{flalign}

	Next, set $\bm{R} = \big( \bm{r}^{(k)} \big)$, where $\bm{r}^{(k)} = \big( r_0^{(k)}, r_1^{(k)}, \ldots , r_k^{(k)} \big) \in \mathbb{R}^{k+1}$ for each $k \in \llbracket 1, m \rrbracket$. Due to \eqref{hqhestimate0} and a Taylor expansion, we have
	\begin{flalign*}
		\Bigg| \partial_t^i \partial_x^j \mathfrak{H} (\rho, 0) - \displaystyle\sum_{i'=0}^{m-i-j} \displaystyle\frac{\rho^{i'}}{i'!} r_{i+i'}^{(i+ i' + j)} \Bigg| & = \Bigg| \partial_t^i \partial_x^j \mathfrak{H} (\rho, 0) - \displaystyle\sum_{i'=0}^{m-i-j} \displaystyle\frac{\rho^{i'}}{i'!} \partial_t^{i+i'} \partial_x^j \mathfrak{H} (0, 0) \Bigg| \le B \rho^{m-i-j+1}.
	\end{flalign*} 

	\noindent Hence, by \eqref{qkir}, it suffices to show  
	\begin{flalign}
		\label{rrestimate} 
		\big| r_i^{(k)} - \widetilde{r}_i^{(k)} \big| \le C \rho^{m-k+1}, \qquad \text{for any $k \in \llbracket 1, m \rrbracket$ and $i \in \llbracket 0, k \rrbracket$}.
	\end{flalign}

	To that end, define $D: \mathbb{R}^2 \rightarrow \mathbb{R}$ by, for any $(t, x) \in \mathbb{R}^2$, setting
	\begin{flalign*}
		D(t, x) =  \displaystyle\sum_{j,k \in \{ t, x \}} \mathfrak{b}_{jk} \big( \nabla A_1 (t,x) \big) \partial_j \partial_j A_1 (t, x) = \displaystyle\sum_{j,k \in \{ t, x\}} \mathfrak{b}_{j,k} \big( \nabla A_0 (t-\rho, x) \big) \partial_j \partial_k A_0 (t-\rho, x),
	\end{flalign*}

	\noindent were in the last equality we used \eqref{a0a13}. We will use a Taylor expansion to view $D$ as an approximation of the left side of \eqref{equationxtb}, where the $H$ there is $A$ here. For each integer $k \in \llbracket 0, m \rrbracket$, set $\bm{Q}^{(k)} = \big( \bm{q}^{(1)}, \bm{q}^{(1)}, \ldots , \bm{q}^{(k)} \big) \in \mathbb{R}^2 \times \mathbb{R}^3 \times \cdots \times \mathbb{R}^{k+1}$ and $\widetilde{\bm{R}}^{(k)} = \big( \widetilde{\bm{r}}^{(1)}, \widetilde{\bm{r}}^{(2)}, \ldots , \widetilde{\bm{r}}^{(k)} \big) \in \mathbb{R}^2 \times \mathbb{R}^3 \times \cdots \times \mathbb{R}^{k+1}$.  Under this notation, we have from \Cref{admissibleconsistent} (and the fact from \eqref{a0a1} that $\bm{Q}_{A_0}^{(k)} (0, 0) = \bm{Q}^{(k)}$ and $\bm{Q}_{A_1}^{(k)} (0, 0) = \widetilde{\bm{R}}^{(k)}$ for each $k \in \llbracket 1, m \rrbracket$), that
	\begin{flalign*}
		\partial_t^i \partial_x^j D(0,0) = \mathfrak{v}_{i,j} \big( \widetilde{\bm{R}}^{(i+j+2)} \big); \qquad \partial_t^i \partial_x^j D(\rho, 0) = \mathfrak{v}_{i,j} \big( \bm{Q}^{(i+j+2)} \big), 
	\end{flalign*} 

 	\noindent for each $(i, j) \in \mathbb{R}_{\ge 0}^2$ with $i + j \le m -2$. Since $\bm{Q}$ is consistent, we have from \eqref{vijequation} and \eqref{qjkc} (and the fact that all derivatives of $D$ of order at most $2m$ are uniformly bounded on compact sets) that there exists a constant $C_4 = C_4 (\varepsilon, B, m) > 1$ with 
	\begin{flalign*}
		\partial_t^i \partial_x^j D(\rho,0) = 0, \quad \text{for any $(i, j) \in \mathbb{Z}_{\ge 0}^2$ with $i + j \le m-2$}; \qquad \displaystyle\max_{i, j \in \llbracket 0, m \rrbracket} \displaystyle\sup_{|z| \le 1} \big| \partial_t^i \partial_x^j D(z) \big| \le C_4.
	\end{flalign*}
	
	\noindent Hence, it follows from a Taylor expansion that 
	\begin{flalign*}
		\Big| \mathfrak{v}_{i,j} \big( \widetilde{\bm{R}}^{(i+j+2)} \big) \Big|  = \big| \partial_t^i \partial_x^j D(0, 0) \big| \le C_4 \rho^{m-i-j-1}.
	\end{flalign*}
	
	\noindent This, with \Cref{qestimateq}, \eqref{rjkc}, \eqref{qjkc}, \eqref{rki}, and \eqref{r0r10}, implies \eqref{rrestimate} and thus the lemma.
\end{proof}

\begin{cor}
	
	\label{hconcentrateq2}
	
	Adopting the notation and assumptions of \Cref{hconcentrateq}, there exist constants $\nu = \nu (\varepsilon, B, m) \in (0, 1)$, $c = c(\varepsilon, B, m) > 0$, and $C = C(\varepsilon, B, m) > 1$; a strip domain $\mathfrak{S} \subset \mathbb{R}^2$; and a function $\mathfrak{H} : \overline{\mathfrak{S}} \rightarrow \mathbb{R}$, such that all statements but the first enumerated one in \Cref{hconcentrateq} hold. Moreover, the following variant of that first statement holds, for any real number $\rho \in (0, c)$.
	\begin{enumerate}
		\item[1.] Denoting $\bm{Q}_{\mathfrak{H}}^{(m)} (\rho, 0) = \big( \bm{q}_{\mathfrak{H}}^{(k)} (\rho, 0) \big)_{k \in \llbracket 1, m \rrbracket}$, we have $\big| \bm{q}_{\mathfrak{H}}^{(k)} (\rho, 0) - \bm{q}^{(k)} \big| \le C \rho^{m-k+1}$ for all $k \in \llbracket 1, m \rrbracket$.
\end{enumerate}
\end{cor}

\begin{proof}
	
	Define the $m$ real pairs $\big( r_0^{(k)}, r_1^{(k)} \big) \in (\mathbb{R}^2)^m$ from $\bm{Q}$ through \eqref{r0r1}, and let $\bm{R}$ denote its consistent extension. By \Cref{hconcentrateq} (applied with the $\bm{Q}$ there equal to $\bm{R}$ here) and \eqref{rjkc}, there exist constants $\nu = \nu (\varepsilon, B, m) \in (0, 1)$, $c = c(\varepsilon, B, m) > 1$, and $C_1 = C_1 (\varepsilon, B, m) > 1$; a strip domain $\mathfrak{S} \subset \mathbb{R}^2$ of width $\nu$, such that $[0, \nu] \times [-\nu, \nu] \subseteq \mathfrak{S} \subseteq [0, \nu] \times [-C_1 \nu, C_1 \nu]$; and a function $\mathfrak{H}: \overline{\mathfrak{S}} \rightarrow \mathbb{R}$ with $\| \mathfrak{H} \|_{\mathcal{C}^{m+1} (\overline{\mathfrak{S}})} \le C_1$ and $\inf_{z \in \mathfrak{S}} \partial_x \mathfrak{H} (z) \le -C_1^{-1}$, satisfying $\lim_{z \rightarrow (0, 0)} \bm{Q}_{\mathfrak{H}}^{(m)} (z) = \bm{R}$ and the following property. For any real numbers $\delta \in (0, 1)$, there exists a constant $C_2 = C_2 (\varepsilon, \delta, B, m) > 1$ such that $\mathfrak{H}$ is $(C_2; c; \delta)$-concentrating. This $\mathfrak{H}$ satisfies the last condition of the corollary, as \Cref{qrestimate} yields a constant $C_3 = C_3 (\varepsilon, B, m) > 1$ such that $\big| \bm{q}_{\mathfrak{H}}^{(k)} (\rho, 0) - \bm{q}^{(k)} \big| \le C_3 \rho^{m-k+1}$, for each $k \in \llbracket 1, m \rrbracket$ and sufficiently small $\rho$. 
\end{proof}

\section{Proof of Concentration Bounds} 

\label{Proof1} 

In this section we establish \Cref{gh1}.

\subsection{Local Approximations of $\mathfrak{H}$} 

\label{HLocal} 

In this section we will establish various statements needed to show \Cref{gh1} in \Cref{ProofConcentration0} below; we only establish the bound in that statement on $\Omega_{j; k}^+ (z)$, as the proof of the bound on $\Omega_{j;k}^- (z)$ is entirely analogous. Throughout, we recall the notation from \Cref{ProofHeighteta}; fix $(j, k) \in \llbracket 0, J \rrbracket \times [0, K]$ and a point $z \in \mathfrak{P}_{f; g}$; and set (recall \eqref{functionsjkz})
\begin{flalign*}
	\varphi = \varphi_{j; k-1}^+; \qquad \Omega = \bigcup_{w \in \mathfrak{P}} \Omega_{j; k-1}^+; \qquad \kappa = 1 - (k-1) \rho^2 n^{-\delta}; \qquad \mathfrak{P} = \mathfrak{P}_{f; g}.
\end{flalign*} 

\noindent Throughout this section, we further restrict to the event $\Omega^{\complement}$.  

Let us now briefly outline how we will establish \Cref{gh1}. We will proceed by locally comparing the Brownian bridges $\bm{x}$ in a neighborhood of $z$ with a family $\bm{y}$ of non-intersecting Brownian bridges, whose associated normalized height function $n^{-1} \mathsf{H}^{\bm{y}}$ is known to concentrate around its limit shape $\mathfrak{H}$. By \Cref{hconcentrateq2}, there are many choices for this limit shape $\mathfrak{H}$; we will make its second derivatives slightly smaller than those of $\varphi$ (while making its gradient match that of $\varphi$), so as to ensure that $\varphi_{j;k}^+ (z) > \mathfrak{H} (z)$ while $\varphi  |_{\partial \mathfrak{B}_{\rho} (z)} < \mathfrak{H} |_{\partial \mathfrak{B}_{\rho} (z)}$. Since the normalized height function $n^{-1} H$ of $\bm{x}$ is bounded above by $\varphi$ on the event $\Omega^{\complement}$, comparing $\bm{x}$ to $\bm{y}$ using monotonicity (\Cref{monotoneheight}) and applying the concentration bound for $\bm{y}$ (\Cref{hconcentrateq2}) will yield $n^{-1} \mathsf{H}^{\bm{x}} (z) \le \mathfrak{H} (z) < \varphi_{j;k}^+ (z)$ with high probability, establishing \Cref{gh1}. 

To implement this, we must fix the derivatives of the function $\mathfrak{H}$, to which end we require some additional notation. Recalling the $\mathfrak{b}_{ij}$ from \eqref{uvb}, define for any $(u, v) \in \mathbb{R} \times \mathbb{R}_{< 0}$ the vector 
\begin{flalign*} 
	\bm{p} (u, v) = \big( \mathfrak{b}_{xx} (u, v), 2 \mathfrak{b}_{tx} (u, v), \mathfrak{b}_{tt} (u, v) \big) \in \mathbb{R}^3.
\end{flalign*} 

\noindent Observe under this notation, and recalling $\bm{q}_F^{(2)}$ for any function $F \in \mathcal{C}^2 (\mathfrak{P})$ from \eqref{kqf}, the constraint \eqref{equationxtb} is equivalent to $\bm{p} \big(\nabla F(z) \big) \cdot \bm{q}_F^{(2)} (z) = 0$.     

\begin{lem}

\label{pqg}

The following two statements hold for $n$ sufficiently large.

\begin{enumerate}
	\item For any $z_0 \in \mathfrak{P}$, we have 
	\begin{flalign*} 
		\partial_x \varphi (z_0) \le -\frac{\varepsilon}{2}; \qquad \| \varphi \|_{\mathcal{C}^{m+1} (\overline{\mathfrak{P}})} \le 2B.
	\end{flalign*} 

	\item For any $z_0 \in \mathfrak{P}$, we have 
	\begin{flalign*}
		\bigg| \displaystyle\frac{\bm{p} \big( \nabla \varphi (z_0) \big) \cdot \bm{q}_H^{(2)} (z_0)}{\bm{p} \big(\nabla \varphi (z_0) \big) \cdot \bm{q}_{\psi}^{(2)} (z_0)} \bigg| \le \displaystyle\frac{\kappa \omega}{2}.
	\end{flalign*}
\end{enumerate} 
\end{lem} 

\begin{proof}
	
	From the definition \eqref{functionsjkz} of $\varphi = \varphi_{j;k-1}^+$, we have for any $z_0 = (t, x) \in \mathfrak{P}$ that
	\begin{flalign}
		\label{derivativex1} 
		\big| \nabla \varphi (z_0) - \nabla H(z_0) \big| \le \kappa \omega \big| \nabla \psi (z_0) \big| \le 2 \kappa \omega \zeta \max \{ e^{\zeta t}, e^{\zeta x} \}.
	\end{flalign} 

	\noindent where in the second inequality we used the definition \eqref{rhoomega} of $\psi$. Together with the fact that any $z_0 = (t, x) \in \mathfrak{P}$ satisfies $|x| + |t| \le 3 B_0$ (as $|x| + |t| \le \| f \|_{\mathcal{C}^{m+1}} + \| g \|_{\mathcal{C}^{m+1}} + L^{-1} \le 2 \|G \|_{\mathcal{C}^{m+1}} + B_0 \le 3B_0$) and the definition \eqref{rhoomega} of $\omega$, this implies for sufficiently large $n$ that 
	\begin{flalign} 
		\label{derivativex2} 
		\big| \partial_x \varphi (z_0) - \partial_x H(z_0) \big| \le 2 \kappa \omega \zeta e^{3 B_0 \zeta} \le \frac{\varepsilon}{2}.
	\end{flalign} 

	\noindent Since \eqref{derivativexhtx} gives $\partial_x H(z_0) \le -\varepsilon$, this yields the first part of the first statement of the lemma; the proof of the second part of the first statement is entirely analogous.
	
	To establish the second, first observe from \eqref{derivativex2} and \eqref{hderivativem} that $\big| \nabla \varphi (z_0) \big| \le \big| \nabla H(z_0) \big| + \varepsilon \le 2B$ for each $z_0 \in \mathfrak{P}$. Moreover, by \eqref{uvb}, each $\mathfrak{b}_{ij} (u, v)$ is $100 B^3$-Lipschitz on the region where $\max \big\{ |u|, |v| \big\} \le 2B$. Thus, since $\bm{p} \big( \nabla H(z_0) \big) \cdot \bm{q}_H^{(2)} (z_0) = 0$ (by the fact that $H$ solves \eqref{equationxtb}), we have for any $z_0 \in \mathfrak{P}$ that
	\begin{flalign}
		\label{pqestimate} 
		\begin{aligned}
		\Big| \bm{p} \big( \nabla \varphi (z_0) \big) \cdot \bm{q}_H^{(2)} (z_0) \Big| & = \bigg| \Big( \bm{p} \big( \nabla \varphi(z_0) \big) - \bm{p} \big( \nabla H(z_0) \big) \Big) \cdot \bm{q}_H^{(2)} (z_0) \bigg| \\
		& \le 200 B^3 \big| \nabla \varphi (z_0) - \nabla H(z_0) \big| \big| \bm{q}_H^{(2)} (z_0) \big|  \le 200 B^4 \big| \nabla \varphi (z_0) - \nabla H(z_0) \big|,
		\end{aligned} 
	\end{flalign}

	\noindent where in the last bound we used \eqref{hderivativem}. Again by \eqref{uvb} and the first part of the lemma, we have 
	\begin{flalign*} 
		\inf_{z_0 \in \mathfrak{P}} \min \big\{ \mathfrak{b}_{tt} \big( \nabla \varphi(z_0) \big), \mathfrak{b}_{xx} \big( \nabla \varphi (z_0) \big) \big\} \ge \frac{\varepsilon^4}{16}.
	\end{flalign*} 

	\noindent Since \eqref{rhoomega} gives for $z_0 = (t, x) \in \mathfrak{P}$ that $\bm{q}_{\psi}^{(2)} (z_0) = (-\zeta^2 e^{\zeta x}, 0, -\zeta^2 e^{\zeta t})$, this implies  
	\begin{flalign*}
		\Big| \bm{p} \big( \nabla \varphi (z_0) \big) \cdot \bm{q}_{\psi}^{(2)} (z_0) \Big| \ge \displaystyle\frac{\zeta^2 \varepsilon^4}{16} \displaystyle\max \{ e^{\zeta t}, e^{\zeta x} \}.
	\end{flalign*}

	\noindent Together with \eqref{pqestimate} and \eqref{derivativex1}, this yields
	\begin{flalign*}
		\displaystyle\sup_{z_0 \in \mathfrak{P}} \displaystyle\frac{\Big| \bm{p} \big( \nabla \varphi (z_0) \big) \cdot \bm{q}_H^{(2)} (z_0) \Big|}{\Big| \bm{p} \big( \nabla \varphi (z_0) \big) \cdot \bm{q}_{\psi}^{(2)} (z_0) \Big|} \le \displaystyle\frac{6400B^4}{\varepsilon^4 \zeta} \kappa \omega,
	\end{flalign*}

	\noindent which by \eqref{rhoomega} yields the lemma.
	\end{proof}

Now let us proceed to prescribe the derivatives of $\mathfrak{H}$ at $z$, which will be a higher order variant of what was done in \cite[Equation (72)]{LTGDMLS}. They will be set to be near an admissible, consistent profile $\bm{Q} = \big( \bm{q}^{(k)} \big) \in \mathbb{R}^2 \times \mathbb{R}^3 \times \cdots \times \mathbb{R}^{m+1}$, defined as follows. Denoting $\bm{q}^{(k)} = \big( q_0^{(k)}, q_1^{(k)}, \ldots , q_k^{(k)} \big) \in \mathbb{R}^{k+1}$ for each $k \in \llbracket 3, m \rrbracket$, set 
\begin{flalign}
	\label{q0q1k} 
	& \bm{q}^{(1)} = \nabla \varphi (z); \quad  \bm{q}^{(2)} = \bm{q}_H^{(2)} (z) - \displaystyle\frac{\bm{p} \big( \nabla \varphi (z) \big) \cdot \bm{q}_H^{(2)} (z)}{\bm{p} \big( \nabla \varphi (z) \big) \cdot \bm{q}_{\psi}^{(2)} (z)} \bm{q}_{\psi}^{(2)} (z); \quad q_0^{(k)} = \partial_x^k \varphi (z); \quad q_1^{(k)} = \partial_t \partial_x^{k-1} \varphi (z).
\end{flalign}

\noindent Observe from the first part of \Cref{pqg} that $q_0^{(1)} < 0$. Thus, the sequence of $m$ pairs $\big( q_0^{(k)}, q_1^{(k)} \big) \in (\mathbb{R}^2)^m$ admits a consistent extension,\footnote{Recall in \Cref{qextend} that consistent extensions were defined on sequences of pairs, while $\bm{q}^{(2)}$ above is a triple. Our choice of $\bm{q}^{(2)}$ ensures that this introduces no conflict in notation. Indeed, \eqref{q0q1k} implies that $\bm{p} \big( \bm{q}^{(1)} \big) \cdot \bm{q}^{(2)} = 0$, which by \eqref{equationxtb} and \eqref{vijequation} gives $\mathfrak{v}_{0, 0} \big( \bm{q}^{(1)}, \bm{q}^{(2)} \big) = 0$. In particular, the second vector of the admissible extension $\bm{Q}$ of $\big( q_0^{(k)}, q_1^{(k)} \big)$ must indeed be $\bm{q}^{(2)}$.} which we denote by $\bm{Q} = \big( \bm{q}^{(k)} \big)$.

By \Cref{hconcentrateq2} (and the first statement of \Cref{pqg}), there exist constants $\nu = \nu (\varepsilon, B, m) > 0$ and $C_1 = C_1 (\varepsilon, B, m) > 1$; a strip domain $\mathfrak{S} = \mathfrak{P}_{\mathfrak{f}; \mathfrak{g}} \subset \mathbb{R}^2$ with $[0, \nu] \times [-\nu, \nu] \subseteq \mathfrak{S} \subseteq [0, \nu] \times [-C_1 \nu, C_1 \nu]$; and a function $\mathfrak{H}: \overline{\mathfrak{S}} \rightarrow \mathbb{R}$ satisfying (recall \eqref{kqf}) 
\begin{flalign}
	\label{qh2} 
	  \| \mathfrak{H} \|_{\mathcal{C}^{m+1} (\overline{\mathfrak{S}})} \le C_1; \quad \displaystyle\inf_{z \in \mathfrak{S}} \partial_x \mathfrak{H} (z) \le - \displaystyle\frac{1}{C_1}; \quad \text{$\big| \bm{q}_{\mathfrak{H}}^{(k)} (\rho, 0) - \bm{q}^{(k)} \big| \le C_1 \rho^{m-k+1}$, for each $k \in \llbracket 1, m \rrbracket$},
\end{flalign}   

\noindent and the below property. There exist constants $c_0 = c_0 (\varepsilon, \delta, B, m) > 0$ and $C_2 = C_2 (\varepsilon, \delta, B, m) > 1$ so that $\mathfrak{H}$ is $\big( C_2; c_0; \delta \big)$-concentrating, meaning (recall \Cref{hhestimate0}) that $\mathfrak{H}$ satisfies \eqref{equationxtb} and the following. There exist an integer $n' \in [c_0 n, n]$, entrance data $\overline{\bm{u}} \in \overline{\mathbb{W}}_{n'}$ and ending data $\overline{\bm{v}} \in \mathbb{W}_{n'}$ on $\mathfrak{S}$ such that, sampling non-intersecting Brownian bridges $\bm{y} = (y_1, y_2, \ldots , y_{n'})$ from $\mathfrak{Q}_{\mathfrak{f}; \mathfrak{g}}^{\overline{\bm{u}}; \overline{\bm{v}}} (n^{-1})$ and denoting the associated height function by $\mathsf{H}^{\bm{y}} : \overline{\mathfrak{S}} \rightarrow \mathbb{Z}$, we have
\begin{flalign*}
	\mathbb{P} \Bigg[ \displaystyle\sup_{z \in \mathfrak{S}} \big| n^{-1} \mathsf{H}^{\bm{y}} (z) - \mathfrak{H} (z) \big| \ge n^{\delta-1} \Bigg] \le C_2 e^{-c_0 (\log n)^2}.
\end{flalign*} 

The following lemma bounds the distance between $\bm{Q}_{\varphi} (z)$ and $\bm{Q}_{\mathfrak{H}} (z)$ under this notation.

\begin{lem}

\label{hestimateq}

There exists a constant $C = C (\varepsilon, B, m) > 1$ such that 
\begin{flalign}
	\label{hestimateq0}
	\big| \bm{q}_{\varphi}^{(k)} (z) - \bm{q}_{\mathfrak{H}}^{(k)} (\rho, 0) \big| \le C (\rho^{m-k+1} + \omega), \qquad \text{for any $k \in \llbracket 1, m \rrbracket$}.
\end{flalign}

\end{lem} 

\begin{proof}
	
	The bound \eqref{hestimateq0} holds at $k = 1$ by the first statement of \eqref{q0q1k} and the second statement of \eqref{qh2}. So, we may assume in what follows that $k \in \llbracket 2, m \rrbracket$. To establish \eqref{hestimateq0} in that case, it suffices to show that 
	\begin{flalign}
		\label{qhqg}
		\big| \bm{q}_{\varphi}^{(k)} (z) - \bm{q}_H^{(k)} (z) \big| \le C \omega; \qquad \big| \bm{q}^{(k)} - \bm{q}_{\mathfrak{H}}^{(k)} (\rho, 0) \big| \le C \rho^{m-k+1}; \qquad \big| \bm{q}_H^{(k)} (z) - \bm{q}^{(k)} \big| \le C \omega.
	\end{flalign}
	
	\noindent To verify the first, observe from \eqref{functionsjkz} and \eqref{rhoomega} that 
	\begin{flalign}
		\label{qkqpsi} 
		\big| \bm{q}_{\varphi}^{(k)} (z) - \bm{q}_H^{(k)} (z) \big| \le \omega \big| \bm{q}_{\psi}^{(k)} (z) \big| \le (m+1) \zeta^m e^{\zeta |z|} \omega \le C_3 \omega, 
	\end{flalign}
	
	\noindent for some constant $C_3 = C_3 (\varepsilon, B, m) > 1$ (where in the last inequality we implicitly used the fact that $|z| \le 3B_0$, which holds since $|z| \le \| f \|_{\mathcal{C}^{m+1}} + \| g \|_{\mathcal{C}^{m+1}} + L^{-1} \le 3B_0$). The second follows from the last statement of \eqref{qh2}.
	
	To establish the third statement of \eqref{qhqg}, observe since $\bm{Q}_H^{(m)} (z)$ and $\bm{Q}$ are consistent that it suffices by \Cref{continuityextend} (and the first part of \Cref{pqg}) to show the existence of a constant $C_4 = C_4 (\varepsilon, B, m) > 1$ such that
	\begin{flalign}
		\label{qgtk}
		\big| \partial_x^k H(z) - q_0^{(k)} \big| \le C_4 \omega; \qquad \big| \partial_t \partial_x^{k-1} H(z) - q_1^{(k)} \big| \le C_4 \omega, \qquad \text{for each $k \in \llbracket 1, m \rrbracket$}.
	\end{flalign}

	\noindent For $k \ne 2$, this follows from \eqref{q0q1k} and \eqref{qkqpsi}, and for $k = 2$ it follows from the fact that 
	\begin{flalign*}
		\big| \bm{q}_H^{(2)} (z) - \bm{q}^{(2)} \big| = \Bigg| \displaystyle\frac{\bm{p} \big( \nabla \varphi (z) \big) \cdot \bm{q}_H^{(2)} (z)}{\bm{p} \big( \nabla \varphi (z) \big) \cdot \bm{q}_{\psi}^{(2)} (z)} \Bigg| \big| \bm{q}_{\psi}^{(2)} (z) \big| \le \big| \bm{q}_{\psi}^{(2)} (z) \big| \frac{\omega\kappa}{2} \le 2\zeta^2 e^{3 B_0 \zeta} \omega\kappa,
	\end{flalign*}

	\noindent where the first statement holds by \eqref{q0q1k}, the second by the second part of \Cref{pqg}; and the third from the explicit form \eqref{rhoomega} for $\psi$, with the fact that $|z| \le 3B_0$ for $z \in \mathfrak{P}$. This confirms \eqref{qgtk} and thus the lemma.
\end{proof}

In the below, for any point $w \in \mathbb{R}^d$ and real number $r > 0$, we denote the open disk $\mathcal{B}_r (w) = \big\{ w' \in \mathbb{R}^d : |w'-w| \le r \big\}$. We next have the following lemma, which is similar to \cite[Equation (88)]{LTGDMLS}, indicating the relative height of $\varphi (z)$ with respect to $\varphi |_{\partial \mathcal{B}_{\rho} (z)}$ is larger than that of $\mathfrak{H} (\rho, 0)$ with respect to that of $\mathfrak{H} |_{\partial \mathcal{B}_{\rho} (\rho, 0)}$.  

\begin{lem} 
	
	\label{hwhz} 

There exists a constant $c =c(\varepsilon, B, m) > 0$ such that the following holds. Fix $w \in \mathfrak{P}$ with $|z-w| \le n^{1/2m^2} \rho$; set $z_0 = (\rho, 0) \in \mathfrak{S}$; and set $w_0 = z_0 - z + w$. If $w_0 \in \mathfrak{S}$, then   
\begin{flalign*}
	\varphi (z) - \varphi (w) \ge \mathfrak{H} (z_0) - \mathfrak{H} (w_0) + c \omega |z - w|^2 - c^{-1} \rho^m |z-w|.
\end{flalign*}

\end{lem} 

\begin{proof}

	For any continuously twice-differentiable function $F$ and point $z'$ in its domain, let the $2 \times 2$ matrix $\bm{H}_F (z')$ denote its Hessian at $z'$, given by 
\begin{flalign*} 
	\bm{H}_F (z') = \left[ \begin{array}{cc} \partial_t^2 F (z') & \partial_t \partial_x F(z') \\ \partial_t \partial_x F(z') & \partial_x^2 F(z') \end{array} \right].  
\end{flalign*}

\noindent Then, by \eqref{q0q1k}, \eqref{qh2}, a Taylor expansion, the fact that $w_0 - z_0 = w - z$, the first part of \Cref{pqg}, and \Cref{hestimateq}, there is a constant $C_3 = C_3 (\varepsilon, B, m) > 1$ such that
\begin{flalign}
	\label{wzwz} 
	\begin{aligned}
	 \Big|  \varphi (  w) - \varphi(z) - \big( & \mathfrak{H} (w_0) - \mathfrak{H} (z_0) \big) - (w - z) \cdot \big( \bm{H}_{\varphi} (z) - \bm{H}_{\mathfrak{H}} (z_0) \big) (w - z)\Big| \\
	&  \le |z-w| \cdot \big| \bm{q}_{\varphi}^{(1)} (z) - \bm{q}_{\mathfrak{H}}^{(1)} (z_0) \big| + \displaystyle\sum_{k=3}^m |z - w|^k \big| \bm{q}_{\varphi}^{(k)} (z) - \bm{q}_{\mathfrak{H}}^{(k)} (z_0) \big| \\
	& \qquad \qquad \qquad  \qquad \qquad + |z - w|^{m+1} \big( \| \varphi \|_{\mathcal{C}^{m+1} (\overline{\mathfrak{S}})} + \| \mathfrak{H} \|_{\mathcal{C}^{m+1} (\overline{\mathfrak{S}})} \big) \\
	& \le C_3 \rho^m |z-w| + C_3  |z-w|^{m+1} + C_3 \displaystyle\sum_{k=3}^m (\rho^{m-k+1}  + \omega) |z-w|^k. 
	\end{aligned} 
\end{flalign}

Next, observe that 
\begin{flalign}
	\label{h1h2}
	\begin{aligned}
	(w-z) & \cdot \big( \bm{H}_{\varphi} (z) - \bm{H}_{\mathfrak{H}} (z_0) \big) (w-z) \\
	&  = (w-z) \cdot \big( \bm{H}_H (z) - \bm{H}_{\mathfrak{H}} (z_0) + \kappa \omega \bm{H}_{\psi} (z) \big) (w-z) \\
	& \le C_1 \rho^{m-1} |z-w|^2 + \Bigg( \kappa \omega - \displaystyle\frac{\bm{p} \big( \nabla \varphi (z) \big) \cdot \bm{q}_H^{(2)} (z)}{\bm{p} \big( \nabla \varphi (z) \big) \cdot \bm{q}_{\psi}^{(2)} (z)} \Bigg) (w-z) \cdot \big( \bm{H}_{\psi} (z) \big) (w-z).
	\end{aligned} 
\end{flalign}

\noindent where for the first statement we applied the definition \eqref{functionsjkz} of $\varphi$ and for the second we applied \eqref{q0q1k} and \eqref{qh2}. By the second part of \Cref{pqg}, and the fact that $\bm{q}_{\psi}^{(2)}$ has all negative entries (by the explicit form \eqref{rhoomega} for $\psi$), we have 
\begin{flalign*} 
	\Bigg( \kappa \omega - \displaystyle\frac{\bm{p} \big( \nabla \varphi(z) \big) \cdot \bm{q}_H^{(2)} (z)}{\bm{p} \big( \nabla \varphi (z) \big) \cdot \bm{q}_{\psi}^{(2)} (z)} \Bigg) \bm{q}_{\psi}^{(2)} (z) \le \frac{\kappa \omega}{2} \bm{q}_{\psi}^{(2)} (z), \qquad \text{entrywise}.
\end{flalign*} 

\noindent Since $\bm{q}_{\psi}^{(2)} (t, x) = (-\zeta^2 e^{\zeta t},  0, -\zeta^2 e^{\zeta x})$, since $\kappa \ge \frac{1}{4}$, and since $|z| \le 3B_0$ for any $z \in \mathfrak{P}$ (as $|z| \le \| f \|_{\mathcal{C}^{m+1}} + \| g \|_{\mathcal{C}^{m+1}} + L^{-1} \le 3B_0$), it follows that there is a constant $c_1 = c_1 (\varepsilon, m) > 0$ such that 
\begin{flalign*}
	 \Bigg( \kappa \omega - \displaystyle\frac{\bm{p} \big( \nabla \varphi (z) \big) \cdot \bm{q}_H^{(2)} (z)}{\bm{p} \big( \nabla \varphi (z) \big) \cdot \bm{q}_{\psi}^{(2)} (z)} \Bigg) (w-z) \cdot \big( \bm{H}_{\psi} (z) \big) (w-z) \le -c_1 \omega |z-w|^2.
\end{flalign*}

\noindent Together with \eqref{wzwz} and \eqref{h1h2}, this gives
\begin{flalign*}
	 \varphi(w) - \varphi(z) - \big( \mathfrak{H} (w_0) - \mathfrak{H}(z_0) \big) & \le C_3 \displaystyle\sum_{k=1}^{m+1} \rho^{m-k+1} |z-w|^k + C_3 m \omega |z-w|^3 - c_1 \omega |z-w|^2.
\end{flalign*}

\noindent This, since $ \rho^{m-k+1} |z-w|^k + \omega |z-w|^3 \le 2 n^{-1/m^2} \omega |z-w|^2$ for each $k \in \llbracket 2, m+1 \rrbracket$ (which holds by \eqref{rhoomega}, the fact that $|z-w| \le n^{1/2m^2} \rho$, and the fact that $\delta \le \frac{1}{5m^2}$), implies the lemma.
\end{proof}

\subsection{Proof of \Cref{gh1}}

\label{ProofConcentration0} 

In this section we establish \Cref{gh1}. As indicated in \Cref{HLocal}, we only show the bound in that proposition on $\Omega_{j;k}^+ (z)$, as the proof of the bound on $\Omega_{j;k}^- (z)$ is entirely analogous; recall we assume that $\Omega^{\complement}$ holds. We proceed by comparing the bridges $\bm{x}(t)$ with those $\bm{y}(t)$, concentrating around $\mathfrak{H}$, in a local neighborhood (of diameter about $\rho$) around $z$. To implement this, we must reindex the bridges in $\bm{x}$ and $\bm{y}$ to form $\bm{x}'$ and $\bm{y}'$, defined as follows.
	
\begin{figure}
\center
\includegraphics[width=1\textwidth]{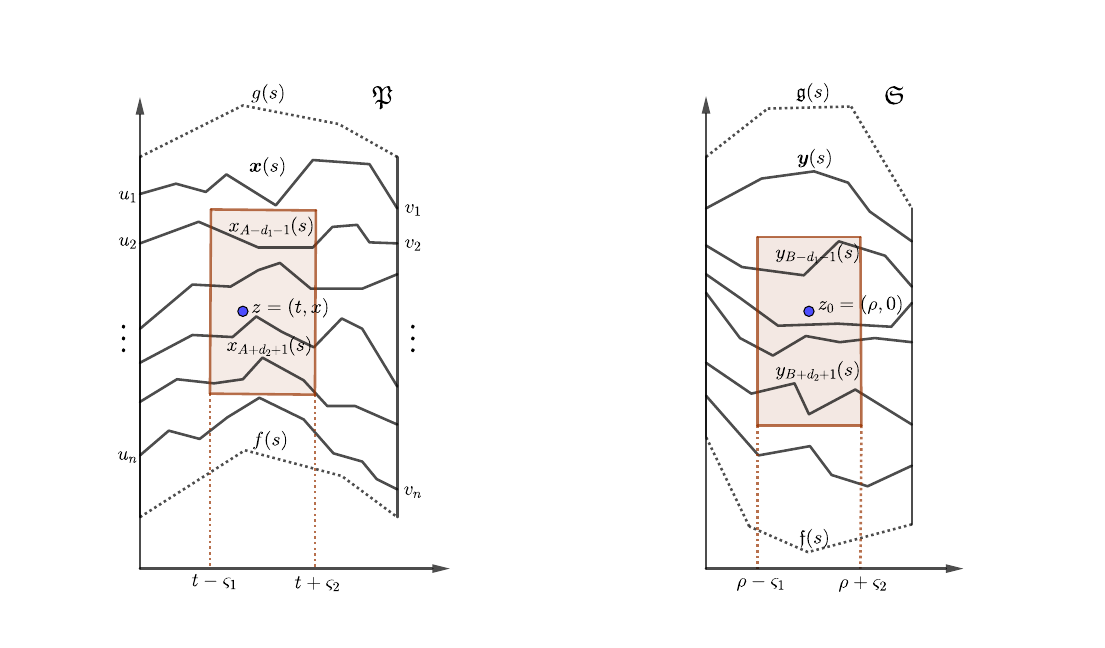}
\caption{For the proof of \Cref{gh1}, we compare the non-intersecting Brownian bridges $\bm x(t)$ in a neighborhood of $z$ with bridges $\bm y(t)$ in one of $z_0$.}
\label{f:coupling}
\end{figure}
	
	Let $z = (t, x) \in \mathfrak{P}$ and $z_0 = (\rho, 0) \in \mathfrak{S}$; see \Cref{f:coupling} for a depiction. Letting $c_1 = c_1 (\varepsilon, B, m) > 0$ denote the constant $c$ from \Cref{hwhz} (and recalling $C_1$ from \eqref{qh2}), set 
	\begin{flalign}
		\label{absigma} 
		\begin{aligned}
		 & A = \bigg\lfloor  n \varphi (z) - \displaystyle\frac{c_1 \rho^2 \omega n}{2} \bigg\rfloor; \quad B = \big\lfloor n \mathfrak{H} (z_0) + n^{\delta} \big\rfloor; \quad \varsigma_1 = \min \{ \rho, t \}; \quad \varsigma_2 = \min \{ \rho, L^{-1} - t \}; \\
		& d_1 = \min \big\{ \lceil 5 C_1^3 \rho n \rceil, A-1 \big\}; \quad d_2 = \min \big\{ \lceil 5 C_1^3 \rho n \rceil, n-A \big\}; \quad \varsigma = \varsigma_1 + \varsigma_2; \quad d = d_1 + d_2.
		\end{aligned} 
	\end{flalign}
	
	\noindent In what follows, we omit the floors and ceilings, as they do not affect the proofs; we may also assume that $A \le n - 5 n^{\delta}$, for otherwise $\varphi (z) \ge 1 + n^{-\delta} \rho^2 \omega$ and so $n^{-1} \mathsf{H}^x (z) \le 1 \le \varphi (z) - n^{-\delta} \rho^2 \omega = \varphi_{j;k} (z)$ holds deterministically. For each integer $j \in \llbracket 0, d+2 \rrbracket$ and real number $s \in [0, \varsigma]$, set
	\begin{flalign}
		\label{xjsyjs} 
		x_j' (s) = x_{A+j-d_1-1} (s+t-\varsigma_1) - x; \qquad y_j' (s) = y_{B+j-d_1 - 1} (s + \rho -  \varsigma_1),
	\end{flalign}

	\noindent where $x_0 = g$ and $x_{n+1} = f$. Further denote $\bm{x}' = ( x_1', x_2', \ldots , x_{d+1}') \in \llbracket 1, d+1 \rrbracket \times \mathcal{C} \big( [0, \varsigma] \big)$ and $\bm{y}' = ( y_1', y_2', \ldots , y_{d+1}') \in \llbracket 1, d + 1 \rrbracket \times \mathcal{C} \big( [0, \varsigma] \big)$. Also define the functions $\mathfrak{x}^-, \mathfrak{x}^+, \mathfrak{y}^-, \mathfrak{y}^+ : [0, \varsigma] \rightarrow \mathbb{R}$ by for each $s \in [0, \varsigma]$ setting 
	\begin{flalign*}
		\mathfrak{x}^- (s) = x_{d+2}' (s); \qquad \mathfrak{x}^+ (s) = x_0' (s); \qquad \mathfrak{y}^- (s) = y_{d+2}' (s); \qquad \mathfrak{y}^+ (s) = y_0' (s).
	\end{flalign*}

	\noindent Further define the entrance data $\bm{u}', \bm{u}'' \in \overline{\mathbb{W}}_{d+1}$ and ending data $\bm{v}', \bm{v}'' \in \overline{\mathbb{W}}_{d+1}$ by setting $\bm{u}' = \bm{x}' (0)$, $\bm{v}' = \bm{x}' (\varsigma)$, $\bm{u}'' = \bm{y}' (0)$, and $\bm{v}'' = \bm{y}' (\varsigma)$. 
	
	Then, the laws of $\bm{x}'$ and $\bm{y}'$ are given by non-intersecting Brownian bridges sampled from the measures $\mathfrak{Q}_{\mathfrak{x}^-; \mathfrak{x}^+}^{\bm{u}'; \bm{v}'} (n^{-1})$ and $\mathfrak{Q}_{\mathfrak{y}^-; \mathfrak{y}^+}^{\bm{u}''; \bm{v}''} (n^{-1})$, respectively. Denote their height functions (recall \eqref{htx}) by $\mathsf{H}^{\bm{x}'} : \mathfrak{P}_{\mathfrak{x}^-; \mathfrak{x}^+} \rightarrow \mathbb{R}$ and $\mathsf{H}^{\bm{y}'} : \mathfrak{P}_{\mathfrak{y}^-; \mathfrak{y}^+} \rightarrow \mathbb{R}$, respectively. Observe by \eqref{xjsyjs} that for any $(t_0, x_0) \in \mathfrak{P}_{\mathfrak{x}^-, \mathfrak{x}^+} \cap \mathfrak{P}_{\mathfrak{y}^-, \mathfrak{y}^+}$ we have for any $(t_0, x_0) \in [0, \varsigma] \times \mathbb{R}$ that 
	\begin{flalign}
		\label{hxihxi} 
		\begin{aligned}
		& \mathsf{H}^{\bm{x}'} (t_0, x_0) = \mathsf{H}^{\bm{x}} (t_0 + t - \varsigma_1, x_0 + x) - A + d_1 + 1; \\ 
		& \mathsf{H}^{\bm{y}'} (t_0, x_0) = \mathsf{H}^{\bm{y}} (t_0 + \rho - \varsigma_1, x_0) - B + d_1 + 1.
		\end{aligned}
	\end{flalign}
	
	\begin{lem} 
		
	\label{eventf} 
	
	There exists a constant $c = c (\varepsilon, B, m) > 1$ and a coupling between $\bm{x}'$ and $\bm{y}'$ such that the following holds. On an event $\mathscr{F}$ with $\mathbb{P}[\mathscr{F}] \ge 1 - c^{-1} e^{-c (\log n)^2}$, we have $x_j' (s) \leq y_j' (s)$ for each $j \in \llbracket 1, d+1 \rrbracket $ and $s \in [0, \varsigma]$.
	
	\end{lem} 

	\begin{proof} 
		
	 By monotonicity (\Cref{monotoneheight}), we must show on an event of probability at least $1 - c^{-1} e^{-c (\log n)^2} $ that
	\begin{flalign}
		\label{xyxy} 
		\begin{aligned} 
			& x_j' (0) \le y_j' (0), \quad \text{and} \quad x_j' (\varsigma) \le y_j' (\varsigma), \qquad \quad \text{for each $j \in \llbracket 1, d+1 \rrbracket$};\\
			& x_0' (s) \le y_0' (s), \quad \text{and} \quad x_{d+2}' (s) \le y_{d+2}' (s), \quad \text{for each $s \in [0, \varsigma]$}.
		\end{aligned} 
	\end{flalign}	

	\noindent Defining the event 
	\begin{flalign}
		\label{evente0} 
		\mathscr{E} = \Bigg\{ \displaystyle\sup_{z \in \mathfrak{S}} \big| n^{-1} \mathsf{H}^{\bm{y}} (z) - \mathfrak{H} (z)  \big| \le n^{\delta-1} \Bigg\}, \qquad \text{we have} \qquad \mathbb{P} [\mathscr{E}] \ge 1 - C_2 e^{- c_0 (\log n)^2},
	\end{flalign}

	\noindent since $\mathfrak{H}$ is $(C_2; c_0; \delta)$-concentrating. On $\mathscr{E} \cap \Omega^{\complement}$, we have  
	\begin{flalign*}
		\varphi \big( s, x_j (s) \big) \ge \displaystyle\frac{j}{n}; \qquad \mathfrak{H} \big( s, y_{j-n^{\delta}} (s) \big) \le \displaystyle\frac{j}{n}, \qquad \text{for any $j \in \llbracket 1, n \rrbracket$},
	\end{flalign*}
	\noindent where the first statement holds from $\Omega^{\complement}$ and the second holds from $\mathscr{E}$. Replacing $(j,s)$ with $(A+j-d_1-1,s+t-\varsigma_1)$ in the former; replacing $(j, s)$ with $(B+j+n^\delta-d_1-1, s)$ in latter; and using \eqref{xjsyjs} and \eqref{absigma}, we find for $j \in \llbracket 1, n \rrbracket$ and $s \in [0, \varsigma]$ that 
	\begin{flalign*}
		& \varphi \big( s + t - \varsigma_1, x_j' (s) + x \big) \ge \varphi (z) - \displaystyle\frac{c_1 \rho^2 \omega}{2} + \displaystyle\frac{j-d_1-1}{n}; \\
		& \mathfrak{H} \big( s + \rho - \varsigma_1, y_j' (s) \big) \le \mathfrak{H} (z_0) + 2n^{\delta-1} + \displaystyle\frac{j-d_1-1}{n}.
	\end{flalign*}

	\noindent It follows that 
	\begin{flalign}
		\label{hxj}
		\begin{aligned}
		\varphi (z) - \varphi \big( s + t - \varsigma_1, x_j' (s) + x \big) & \le \displaystyle\frac{c_1 \rho^2 \omega}{2} + \displaystyle\frac{d_1-j+1}{n} \\ 
		& \le \mathfrak{H} (z_0) - \mathfrak{H} \big( s + \rho - \varsigma_1, y_j' (s) \big) + \displaystyle\frac{c_1 \rho^2 \omega}{2} + 2n^{\delta-1}.
		\end{aligned} 
	\end{flalign} 

	We next apply \Cref{hwhz} with that $w_0 = \big( s + \rho - \varsigma_1, y_j' (s) \big)$ for either $(j, s) \in \{ 0, d+2 \} \times [0, \varsigma]$ or $(j, s) \in \llbracket 1, d+2 \rrbracket \times \{ 0, \varsigma \}$. To that end, we must first verify that $|z_0 - w_0| \le n^{1/2m^2} \rho$. In both cases, we have from the event $\mathscr{E}$ from \eqref{evente0}, the first two bounds in \eqref{qh2}, and \eqref{absigma} that
	\begin{flalign*}
		|z_0 - w_0| \le \big| y_j' (s) \big| + \rho & = \big| y_{B + j - d_1 - 1} (s + \rho - \varsigma_1) \big| + \rho \\ 
		& \le C_1 \bigg| \displaystyle\frac{B+j-d_1-1}{n} - \mathfrak{H} (s + \rho - \varsigma_1, 0) \bigg| + C_1 n^{\delta-1} + \rho \\
		& \le C_1 \bigg| \displaystyle\frac{B+j-d_1-1}{n} - \mathfrak{H} (z_0) \bigg| + C_1 n^{\delta-1} + 2C_1^2 \rho \le 20C_1^4 \rho \le n^{1/2m^2} \rho,
	\end{flalign*}
	
	\noindent which establishes the bound $|w_0 - z_0| \le n^{1/2m^2} \rho$. We may thus apply \Cref{hwhz} with $w_0 = \big( s + \rho - \varsigma_1, y_j' (s) \big)$. By \eqref{hxj}, this yields
	\begin{flalign}
		\label{hz0hw0}
		\begin{aligned}
		\varphi (z) - \varphi \big( s + t - \varsigma_1, x_j' (s) + x \big) & \le \mathfrak{H} (z_0) - \mathfrak{H} (w_0) + \displaystyle\frac{c_1 \rho^2 \omega}{2} + 2n^{\delta-1} \\
		& \le \varphi (z) - \varphi \big( s + t - \varsigma_1, y_j' (s) + x \big) + \displaystyle\frac{c_1 \rho^2 \omega}{2} + 2n^{\delta-1} \\ 
		& \qquad - c_1 \omega |w_0 - z_0|^2 + c_1^{-1} \rho^m |w_0 - z_0|.
		\end{aligned}
	\end{flalign}

	\noindent If $|w_0 - z_0| \ge \rho$, then since $c_1 \rho^2 \omega = c_1 n^{3 \delta - 1} \ge 4 n^{\delta-1}$ and $c_1^{-1} \rho^m |w_0 -z_0| \le c_1^{-1} n^{1/2m^2} \rho^{m+1} \le c_1^{-1} n^{(m+1)\delta-1-1/m + 1/2m^2} \le n^{\delta-1} $ (by \eqref{rhoomega} and the fact that $\delta \le \frac{1}{5m^2}$) it follows that 
	\begin{flalign*} 
		\varphi \big(s + t - \varsigma_1, y_j' (s) + x \big) \le \varphi \big( s + t - \varsigma_1, x_j' (s) + x \big).
	\end{flalign*} 

	\noindent Since $\partial_x \varphi < 0$ (by the first part of \Cref{pqg}), this yields $y_j' (s) \ge x_j' (s)$, confirming \eqref{xyxy}. 
	
	If instead $|w_0 - z_0| < \rho$, then we can improve \eqref{hz0hw0} using the fact that $\varphi |_{\partial \mathfrak{P}} = H |_{\partial \mathfrak{P}}$. In particular, assume first in \eqref{xyxy} that $s \in \{ 0, \varsigma \}$; since the two cases are entirely analogous, we let $s = 0$. Since $\varsigma_1 \le |w_0 - z_0| < \rho$, it follows from \eqref{absigma} that $t < \rho$ and so $\varsigma_1 = t$. Then, $\varphi \big( s + t - \varsigma_1, x_j' (s) + x \big) = \varphi \big( 0, x_j' (s) + x \big)$. We next claim that 
	\begin{align}\label{e:boundarygap}
	\varphi \big( 0, y_j' (s) + x \big) = H \big( 0, y_j' (s) + x \big) \le \varphi \big( 0^+, y_j' (s) + x \big) - c_2\omega, \qquad \text{where $c_2 = \frac{1}{8 e^{3 B_0 \zeta}}$}. 
	\end{align}

	\noindent Indeed, the first statement follows from the definition of $\varphi$ (see below \eqref{functionsjkz}) with the fact that $\big( 0, y_j' (s) + x \big) \in \partial \mathfrak{P}$. To see the second, first observe that 
\begin{align}
	\label{psi0} 
\sup_{z \in \mathfrak{P}} \psi (z) = \sup_{z \in \mathfrak{P}} (\Upsilon - e^{\zeta t} - e^{\zeta x}) 
\le 2 + \sup_{z \in \mathfrak{P}} (e^{\zeta t} + e^{\zeta x}) 
\le 2 e^{3 B_0 \zeta} + 2 \le 4 e^{3 B_0 \zeta},
\end{align}
where in the first statement we used the definition \eqref{rhoomega} of $\psi$; in the second we used the fact that $\inf_{z \in \mathfrak{P}} \psi (z) = 2$; in the third we used the fact that $|t| + |x| \le \sup_{t \in [0, 1/L]} \big| f(t) \big| + \sup_{t \in [0, 1/L]} \big| g(t) \big| + L^{-1} \le 3 B_0$; and in the fourth we used the fact that $B_0, \zeta \ge 0$. Next, observe for any $z \in \mathfrak{P}$ that 
\begin{align*}
\varphi(z) - H(z) 
= 3  - (j+1) \omega +  \kappa \omega \psi(z)
& \ge \kappa \omega \psi (z) \\
&\ge \displaystyle\frac{\omega}{4}   \sup_{z \in \mathfrak{P}} \psi (z) \cdot \frac{ \inf_{z \in \mathfrak{P}} \psi (z) }{ \sup_{z \in \mathfrak{P}} \psi (z) }
\ge \frac{\omega}{4}  \cdot \frac{2} {4 e^{3 B_0 \zeta}}
\ge \frac{\omega} {8 e^{3 B_0 \zeta}}
= c_2 \omega,
\end{align*}
where the first statement is from \eqref{functionsjkz}; the second is from the fact that $j \le J = 3 \omega^{-1} - 1$; the third is from the fact $\kappa \ge \frac{1}{4}$; the fourth is from the fact that $\inf_{z \in \mathfrak{P}} \psi (z) \ge 2$ (from below equation  \eqref{rhoomega}) and \eqref{psi0}; and the fifth is from the definition of $c_2$ as in \eqref{e:boundarygap}. This verifies \eqref{e:boundarygap}.

	The estimate \eqref{e:boundarygap}, together with the first bound in \eqref{hz0hw0}, (the $w_0 = \big( s^+ - \varsigma_1 + \rho, y_j' (s) \big)$ case of) \Cref{hwhz}, and \eqref{rhoomega}, yields (recalling that $s = 0$ and $t = \varsigma_1$)
	\begin{flalign*}
		\varphi (z) - \varphi \big( & 0, x_j' (s) + x \big) \\
		 & \le \mathfrak{H} (z_0) - \mathfrak{H} \big( s^+ + \rho - \varsigma_1, y_j' (s) \big) + \displaystyle\frac{c_1 \rho^2 \omega}{2} + 2n^{\delta-1} \\
		& \le \varphi (z) - \varphi \big( 0^+, y_j' (s) + x \big) + \displaystyle\frac{c_1 \rho^2 \omega}{2} + 2n^{\delta-1} - c_1 \omega |w_0 - z_0|^2 + c_1^{-1} \rho^m |w_0 - z_0| \\
		& \le \varphi (z) - \varphi \big( 0, y_j' (s) + x \big) + \displaystyle\frac{c_1 \rho^2 \omega - 2c_2\omega}{2} + 2n^{\delta-1} + c_1^{-1} \rho^{m+1} \le \varphi (z) - \varphi \big( 0, y_j' (s) + x \big),
	\end{flalign*}

	\noindent from which it again follows that $y_j' (s) \ge x_j' (s)$ since $\partial_x \varphi < 0$.

	Now assume that $j \in \{ 0, d + 2 \}$; the two cases are again entirely analogous, so let $j = 0$. Then, 
	\begin{flalign*} 
		\rho > y_0' (s) = y_{B - d_1 - 1} (s + \rho - \varsigma_1) & \ge \displaystyle\frac{d_1}{C_1 n} - C_1 \big| n^{-1} B - \mathfrak{H} (s + \rho - \varsigma_1, 0) \big| - C_1 n^{\delta-1} \\
		&  \ge \displaystyle\frac{d_1}{C_1 n} - C_1 \big| \mathfrak{H} (\rho,0) - \mathfrak{H} (s + \rho - \varsigma_1, 0) \big| - 2 C_1 n^{\delta-1} \\
		&  \ge \displaystyle\frac{d_1}{C_1 n} - C_1^2 \rho - 2 C_1 n^{\delta-1}.
	\end{flalign*}

	\noindent Here, the first statement follows from the fact that $y_0' (s) \le |w_0 - z_0| < \rho$; the second follows from \eqref{xjsyjs}; the third follows from the facts that we are restricting to the event $\mathscr{E}$ from \eqref{evente0}, and also from the first two bounds in \eqref{qh2}; the fourth follows from the definition \eqref{absigma} of $B$; and the fifth follows from the first bound in \eqref{qh2}. Hence, $d_1 < 5C_1^3 \rho n$, implying by \eqref{absigma} that $d_1 = A-1$.
	
	We must show that $y_0' (s) \ge x_0' (s)$. Since $x_0' (s) = x_{A-d-1} (s + t - \varsigma_1) - x = x_0 (s + t - \varsigma) - x = g(s + t - \varsigma_1) - x$ for each $s \in [0, \varsigma]$, this is equivalent to $y_{B-d_1 - 1} (s + \rho) \ge g(s + t) - x$ for each $s \in [-\varsigma_1, \varsigma_2]$. To that end, observe that 
	\begin{flalign}
		\label{hrho02}
		\begin{aligned} 
		\mathfrak{H} & (\rho, 0) - \varphi (t, x) \\ 
		& \le \mathfrak{H} \big(s + \rho - \varsigma_1, g (t + s - \varsigma_1) - x^+ \big) - \varphi (t + s - \varsigma_1, g (t + s - \varsigma_1) - 0^+ \big) + c_1^{-1} n^{1/2m^2} \rho^{m+1} \\
		& \le \mathfrak{H} \big( s + \rho - \varsigma_1, g(s + t - \varsigma_1) - x \big)  - \displaystyle\frac{c_2\omega}{2}.
		\end{aligned} 
	\end{flalign}

	\noindent Here, in the first inequality, we applied \Cref{hwhz}; in the second, we used the fact that $\varphi \big( t - \varsigma_1, g (t+s-\varsigma_1) - 0^+ \big) \ge c_2\omega\ge 2c_1^{-1} n^{1/2m^2} \rho^{m+1}$ by \eqref{rhoomega} and \eqref{e:boundarygap} (since $\mathfrak{H} \big( t + s - \varsigma_1, g(t + s - \varsigma_1) \big) = 0$). Moreover, 
	\begin{flalign*}
		\mathfrak{H} \big( s + \rho - \varsigma_1, y_0' (s) \big) & = \mathfrak{H} \big( s + \rho - \varsigma_1, y_{B-d_1-1} (s + \rho - \varsigma_1) \big) \\ 
		& \le n^{-1} (B - d_1 - 1) + n^{\delta-1} = n^{-1} (B-A) + n^{\delta-1},
	\end{flalign*}

	\noindent where in the first statement we used the definition of $y_j'$ from \eqref{xjsyjs}; in the second we used the fact that we are restricting to $\mathscr{E}$ (from \eqref{evente0}); and in the third we used the equality $d_1 = A-1$. Thus,
	\begin{flalign*} 
		\mathfrak{H} \big( s + \rho - \varsigma_1, y_0' (s) \big) & \le n^{-1} (B-A) + n^{\delta-1} \\
		& \le \mathfrak{H} (\rho, 0) - \varphi (t, x)  + \displaystyle\frac{c_1 \rho^2 \omega}{2} + 2n^{\delta-1} \\
		& \le \mathfrak{H} \big( s + \rho - \varsigma_1, g(s+t-\varsigma_1) - x \big) + \displaystyle\frac{c_1 \rho^2 \omega - c_2\omega}{2} + 2 n^{\delta-1} \\
		& \le \mathfrak{H} \big(s - \rho - \varsigma_1, g(s+t-\varsigma_1) - x \big).
	\end{flalign*}

	\noindent where in the second bound we used the definitions \eqref{absigma} of $A$ and $B$; in the third we applied \eqref{hrho02}; and in the fourth we used the fact that $c_2\omega - c_1 \rho^2 \omega  \ge \frac{c_2\omega}{2} > 4 n^{\delta-1}$ for sufficiently large $n$. Since $\partial_x \mathfrak{H} < 0$, it follows that $y_0' (s) \ge g(s+t-\varsigma_1) - x$. In particular, $x_0' (s) \le y_0' (s)$, which again verifies \eqref{xyxy} and thus the lemma.     
	\end{proof}

	We can now establish \Cref{gh1}.

	\begin{proof}[Proof of \Cref{gh1}]
		
	By \Cref{eventf}, there exists a constant $c_2 = c_2 (\varepsilon, B, m)  > 0$ and an event $\mathscr{F}$ with $\mathbb{P} [\mathscr{F}] \ge 1 - c_2^{-1} e^{-c_2 (\log n)^2}$ such that $x_j' (s) \ge y_j' (s)$ for each $(j, s) \in \llbracket 1, d+1 \rrbracket \times [0, \varsigma]$ on $\mathscr{F}$. By \eqref{evente0} (and recalling the event $\mathscr{E}$ there), we may assume that $\mathscr{F} \subseteq \mathscr{E}$, for otherwise we can replace $\mathscr{F}$ with $\mathscr{E} \cap \mathscr{F}$; we restrict to $\mathscr{F}$ in what follows. Observe on $\mathscr{E}$ (and thus on $\mathscr{F}$) that $(\varsigma_1, 0) \in \mathfrak{P}_{\mathfrak{y}^-; \mathfrak{y}^+}$, since by \eqref{hxihxi} and \eqref{absigma} we have (as $A \le n - 5 n^{\delta}$) 
	\begin{flalign*} 
		\mathsf{H}^{\bm{y}'} (\varsigma_1, 0) = \mathsf{H}^{\bm{y}} (z_0) - n \big( \mathfrak{H} (z_0) + n^{\delta-1}) + d_1 + 1 \in [d_1 - 2n^{\delta} + 1, d_1 + 2n^{\delta} + 1] \subseteq [1, d+1].
	\end{flalign*} 

	\noindent The above coupling thus implies that $\mathsf{H}^{\bm{x}'} (\varsigma_1, 0) \le \mathsf{H}^{\bm{y}'} (\varsigma_1, 0)$, so by \eqref{hxihxi} we deduce
	\begin{flalign*}
		\mathsf{H}^{\bm{x}} (z) = \mathsf{H}^{\bm{x}'} (\varsigma_1, 0) + A - d_1 - 1 & \le \mathsf{H}^{\bm{y}'} (\varsigma_1, 0) + A - d_1 - 1 \\
		& = \mathsf{H}^{\bm{y}} (\rho, 0) + A - B  \le n \bigg( \varphi (z) - \displaystyle\frac{c_1 \rho^2 \omega}{2} \bigg) + \mathsf{H}^{\bm{y}} (\rho, 0) - B.
	\end{flalign*}

	\noindent Since on the event $\mathscr{F} \subseteq \mathscr{E}$, we have $\mathsf{H}^{\bm{y}} (\rho, 0) \le B$, so
	\begin{flalign*}
		n^{-1} \mathsf{H}^{\bm{x}} (z) \le  \varphi (z) - \displaystyle\frac{c_1 \rho^2 \omega}{2} \le \varphi - n^{-\delta} \rho^2 \omega \psi (z) \le \varphi_{j;k}^+ (z),
	\end{flalign*}

	\noindent where in the second bound we used the definition \eqref{functionsjkz} of $\varphi_{j;k}^+$ (and the uniform boundedness of $\psi$, from \eqref{rhoomega}); this establishes the proposition.
\end{proof}

\section{Concentrating Functions With Given Local Profiles}

\label{ProofConcentration}

In this section we establish \Cref{hconcentrateq}.

\subsection{Taylor Approximations Through Dyson Brownian Motion}

\label{MeasureApproximate}

The following proposition produces measures whose Stieltjes transforms have specified derivatives. In what follows, we recall the Stieltjes transform $m_{\mu}$ of any measure $\mu \in \mathscr{P}$ from \Cref{TransformConvolution}.

\begin{prop}
	\label{p:constructrho}
	
	For any integer $m \ge 1$ and real numbers $\varpi > 0$, $\varepsilon \in (0, 1)$, and $A, B > 1$, there exists a constant $\delta = \delta (\varepsilon, \varpi, A, B) > 0$ such that the following holds. Let $\bm{r} = (r_0, r_1, r_2, \ldots , r_m) \in \mathbb{R}^{m+1}$ be an $(m+1)$-tuple of real numbers with $|r_j| \le A$ for each $j \in \llbracket 0, m \rrbracket$. Also let $\mu \in \mathscr{P}$ be a compactly supported probability measure with density $\varrho : \mathbb{R} \rightarrow \mathbb{R}_{\ge 0}$ that is $\mathcal{C}^m$ in a neighborhood of $0$, satisfying $\varepsilon \le  \varrho(x) \le B$ for each $x \in [-\varpi, \varpi]$; $\varrho (x) = 0$ for $|x| > \varpi$; and $\big| \Real m_{\mu} (0) \big| \le A$. There exists a measure $\widetilde{\mu} \in \mathscr{P}_0$ with density $\widetilde{\varrho} : \mathbb{R} \rightarrow \mathbb{R}_{\ge 0}$ compactly supported on $[-\varpi, \varpi]$, satisfying the following two properties. 
	
	\begin{enumerate} 
		
		\item We have $\frac{\varepsilon}{2} \le  \widetilde{\varrho} (x) \le 2B$ for each $x \in [-\varpi, \varpi]$, and $\widetilde{\varrho} (x) = \varrho (x)$ for each $x \in [-\delta, \delta]$. 
		
		\item For each integer $k \in \llbracket 0, m \rrbracket$ we have
		\begin{align}
			\label{mmurmmu}
			\partial_z^k m_{\tilde{\mu}} (0)= r_k+ \mathrm{i} \Imaginary \partial_z^k m_{\mu} (0).    
		\end{align}
	\end{enumerate}
\end{prop}

\begin{proof}[Proof of \Cref{p:constructrho}]
	
	If $\Real m_{\mu} (0) = r_0$, then we will establish the lemma with the stronger bound on $\widetilde{\varrho}$ given by $\frac{3\varepsilon}{4} \le \widetilde{\varrho} (x) \le B + \frac{\varepsilon}{4}$. Let us first reduce to this statement. If $\Real m_{\mu} (0) \ne r_0$, then denote $\vartheta = r_0 - \Real m_{\mu} (0)$; assume that $\vartheta > 0$, as the case when $\vartheta < 0$ is entirely analogous. By our assumption, $\vartheta\leq |r_0|+\big| \Real m_{\mu}(0) \big| \leq A+B$.  Then replace $\varrho (x)$ with $\varrho_0 : \mathbb{R} \rightarrow \mathbb{R}$, defined by for any $x \in \mathbb{R}$ setting
	\begin{flalign}
		\label{rho0rho}
		 \varrho_0 (x) = \varrho(x) + \displaystyle\frac{\varepsilon}{4} \cdot \big( \bm{1} (d \le x \le e^{8 \vartheta / \varepsilon} d) - \bm{1} (e^{-4 \vartheta/\varepsilon} d \le x \le d) \big),
	\end{flalign}

	\noindent for any real number $d < e^{-10(A+B)/\varepsilon} \varpi$. Then, $\varrho_0$ continues to be a probability measure supported in $[-\varpi, \varpi]$ and satisfies $\frac{3 \varepsilon}{4} \le \varrho_0 (x) \le B + \frac{\varepsilon}{4}$ for $x \in [-\varpi, \varpi]$. Moreover, $\varrho_0 = \varrho$ in a neighborhood of $0$, and thus $\Imaginary m(0) = \Imaginary m_{\mu_0} (0)$ by the first statement of \eqref{mrho}, where $m_{\mu_0}$ denotes the Stieltjes transform of the measure $\mu_0 = \varrho_0 (x) dx \in \mathscr{P}_0$. Additionally, 
	\begin{flalign*}
		\Real m_{\mu_0} (0) = \Real m(0) + \displaystyle\frac{\varepsilon}{4} \cdot (  8 \varepsilon^{-1} \vartheta - 4 \varepsilon^{-1} \vartheta )= \Real m(0) + \vartheta = r_0.
	\end{flalign*}

	\noindent Thus, $\mu_0$ and $\varrho_0$ continue to satisfy the hypotheses of the lemma (after replacing $(\varepsilon, B)$ with $\big( \frac{3 \varepsilon}{4}, B + \frac{\varepsilon}{4} \big)$ and altering the $r_j$ for $j \in \llbracket 1, m \rrbracket$ to account for the difference in \eqref{rho0rho}) but now with $\Real m_{\mu_0} (0) = r_0$. It follows that there exists a measure $\widetilde{\mu}$ with density $\widetilde{\varrho}$ satisfying the statement of the lemma (and in fact the bounds $\frac{\varepsilon}{2} \le \frac{9\varepsilon}{16} \le \widetilde{\varrho} (x) \le B + \frac{7 \varepsilon}{16} \le 2B$).

	Thus, we will assume in what follows that $\Real m_{\mu} (0) = r_0$. We will take $\widetilde \varrho(x)$ to be of the form
	\begin{align}
		\label{rho2} 
		\widetilde \varrho(x)=\varrho(x)+\sum_{\ell=1}^{m+2} c_\ell \cdot \bm1 \big( a(y_\ell+y)\le x\le a(y_\ell+y+b) \big),
	\end{align}
	
	\noindent where the coefficients $c_{\ell} \in \mathbb{R}$ and positive real numbers $a, b, y, y_\ell>0$ will be chosen later, in a way continuous in the parameters $(\varepsilon, A, B)$, so that their uniform boundedness (fixing $(\varepsilon, A, B)$) follows from compactness. In this way, for any complex number $z \in \mathbb{C}$ in a neighborhood of $0$, we have
	\begin{align}
		\label{msumc}
		m_{\tilde{\mu}}(z)=m_{\mu}(z)+\sum_{\ell=1}^{m+2} c_\ell \log \left(\frac{a(y_\ell+y+b)-z}{a(y_\ell+y)-z} \right).
	\end{align}

	\noindent Thus, for any integer $k \ge 1$, we have
	\begin{align}
		\label{mkz}
		\partial_z^k m_{\tilde{\mu}}(z)= \partial_z^k m_{\mu} (z)+\sum_{\ell=1}^{m+2} c_\ell (-1)^{k-1} (k-1)! \left(\frac{1}{\big(z-a(y_\ell+y+b) \big)^k}-\frac{1}{\big( z-a(y_\ell+y) \big)^k}\right).
	\end{align}
	
	Since the parameters $a$, $b$, $y$, and $y_{\ell}$ are all positive, we have $\widetilde{\varrho} (x) = \varrho (x)$ for $x \in \mathbb{R}$ in a sufficiently small neighborhood of $0$; this verifies the second part of the first statement of the proposition. Moreover, since by \eqref{mrho} (as $m_{\mu}$ and $m_{\tilde{\mu}}$ are well-defined near $0$, by the above) we have $\Imaginary m_{\tilde{\mu}} (x) = \pi \widetilde{\varrho} (x) = \pi \varrho (x) = \Imaginary m_{\mu} (x)$ for $x$ in a neighborhood of $0$. Differentiating this equation in $x$, it follows that $ \partial_z^k \Imaginary m_{\tilde{\mu}} (0) = \partial_z^k \Imaginary m_{\mu} (0)$ for any integer $k \ge 1$, from which the imaginary part of \eqref{mmurmmu} follows.
	
	By \eqref{mkz}, the real part of the second part of \eqref{mmurmmu} holds for any $k \in \llbracket 1, m \rrbracket$ if and only if  
	\begin{align}
		\label{e:tosolve}
			a^k \big(  \partial_z^k \Real m_{\mu} (0) -r_k \big)=\sum_{\ell=1}^{m+2} c_\ell  (k-1)! \left(\frac{1}{(y_\ell+y+b)^k}-\frac{1}{(y_\ell+y)^k}\right)
	\end{align}
	
	\noindent Moreover, $\widetilde{\varrho}$ has total mass $1$ (that is, $\int_{-\infty}^{\infty} \widetilde{\varrho} (x) dx = 1$) and the second part of \eqref{mmurmmu} holds at $k = 0$ if and only if 
	\begin{flalign} 
		\label{sumc2} 
	\displaystyle\sum_{\ell=1}^{m+2} c_{\ell} = 0, \quad \text{and} \quad \displaystyle\sum_{j=1}^{m+2} c_{\ell} \log \bigg( \displaystyle\frac{y_{\ell} + y + b}{y_{\ell} + y} \bigg) = 0,
	\end{flalign} 

	\noindent respectively, by \eqref{rho2} (since $\varrho$ has total mass $1$) and \eqref{msumc} (since $\Real m_{\mu} = r_0$). We denote the $(m+2)$-dimensional vectors $\bm{c}=(c_1, c_2,\cdots, c_{m+2})$ and $\bm{p}=(p_1, p_2,\cdots, p_{m+2}) $, where $p_1 = 0 = p_2$ and 
	\begin{flalign*} 
		p_k=  \partial_z^{k-2} \Real m_{\mu} (0)-r_{k-2}, \qquad \text{for each $k \in \llbracket 3, m+2 \rrbracket$}.
	\end{flalign*} 
	
	\noindent We also define the $(m+2) \times (m+2)$ matrices $\bm{A} = \bm{A}(a) = \diag (a, a, a, a^2, \ldots , a^m)$ and $\bm{W} (\bm{y}, b) = \big\{ W_{k, \ell} (\bm{y}, b) \big\}_{\ell, k \in \llbracket 1, m+2 \rrbracket}$, with entries given by 
	\begin{flalign*}
		& W_{k, \ell}(\bm{y}, b)= (k-3)! \left(\frac{1}{(y_\ell+y+b)^{k-2}}-\frac{1}{(y_\ell+y)^{k-2}}\right), \quad \text{for $k \in \llbracket 3, m+2 \rrbracket$}; \\
		& \qquad \qquad \qquad W_{2, \ell} (\bm{y}, b) = 1, \qquad  W_{1, \ell} (\bm{y}, b) = \log \bigg( \displaystyle\frac{y_{\ell} + y + b}{y_{\ell} + y} \bigg), 
	\end{flalign*} 

	\noindent for each $\ell \in \llbracket 1, m+2 \rrbracket$. Then we can rewrite the relations \eqref{e:tosolve} and \eqref{sumc2} as $\bm{A} \cdot \bm{p} = \bm{W}(\bm{y}, b) \cdot \bm{c}$ or equivalently as $\bm{c}=\bm{W} (\bm{y}, b)^{-1} \bm{A} \cdot \bm{p}$. 
	
	Next we show there exists some choice of $y, y_{\ell}, b>0$ such that $\bm{W}(\bm{y}, b)$ is invertible. Observe that, for generic $b$, the first row of $\bm{W} (\bm{y}, b)$ is linearly independent from the bottom $m+1$ rows (since the logarithmic function is not rational). It therefore suffices to show that the left-bottom $(m+1)\times (m+1)$ block $\bm{W}_0 (\bm{y}, b)$ of $\bm{W} (\bm{y}, b)$ (constituting its $(k, \ell) \in \llbracket 2, m+2 \rrbracket \times \llbracket 1, m+1 \rrbracket$ entries) is invertible. Let us write this block as $\bm{W}_0 (\bm{y}, b) = \bm{V} (\bm{y}) - \bm{E} (\bm{y}, b)$, where the $(m+1) \times (m+1)$ matrices $\bm{V} = \bm{V} (\bm{y}) = \big\{ V_{k, \ell} (\bm{y}) \big\}$ and $\bm{E} = \bm{E} (\bm{y}, b) = \big\{ E_{k, \ell} (\bm{y}, b) \big\}$ are defined by
	\begin{flalign*}
		& V_{1, \ell} (\bm{y}) = 1, \qquad \text{and} \qquad V_{k, \ell} (\bm{y}) = -(k-2)! (y_{\ell} + y)^{1-k},  \quad    \text{for $k  \in  \llbracket 2, m \rrbracket$}; \\
		& E_{1, \ell} (\bm{y}, b) = 0, \qquad \text{and} \qquad E_{k, \ell} (\bm{y}, b) = -(k-2)! (y_{\ell} + y + b)^{1-k}, \quad \text{for $k \in  \llbracket 2, m \rrbracket$},  
	\end{flalign*}
	
	\noindent for each $\ell \in \llbracket 1, m+1 \rrbracket$. Then, $\bm{V}(\bm y)$ is a Vandermonde matrix (up to multiplication by a diagonal matrix), and is therefore invertible if we take $y_1, y_2, \ldots , y_{m+1} > 0$ mutually distinct. Moreover, by taking $b$ sufficiently large (and generic) the norm of $\bm{E} (\bm y, b)$ can be taken to be sufficiently small so as to guarantee that $\bm{W} (\bm y) = \bm{V} (\bm y) - \bm{E} (\bm y, b)$ remains invertible. 
	
	There thus exists a choice of the parameters $y, b, y_1, y_2, \ldots , y_{m+1}$ such that $\bm{W}_0 (\bm{y}, b)$ and thus $\bm{W} (\bm{y}, b)$ is invertible; then, $\bm{c}=\bm{W} (y)^{-1} \bm{A} \cdot \bm{p}$. By taking $a$ sufficiently small, we can ensure that each entry of $\bm{c}$ is smaller than $(4m+8)^{-1} \varepsilon$ and that $a(y_\ell+y+b)\le \varpi$ for each $\ell \in \llbracket 1, m+2 \rrbracket$.  Since $\varepsilon \le \varrho(x) \le B$ for $x\in [-\varpi, \varpi]$ and $\varrho (x) = 0$ for $|x| > -\varpi$, it follows from \eqref{rho2} that $\frac{3 \varepsilon}{4} \le \big| \widetilde \varrho(x) \big| \le B + \frac{\varepsilon}{4}$ for $x \in [-\varpi, \varpi]$ and that $\widetilde{\varrho} (x) = 0$ for $|x| > \varpi$. In particular, $\widetilde{\varrho}$ is a nonnegative measure satisfying the (previously mentioned improvement of the) first statement of the proposition. That it is a probability measure satisfying \eqref{mmurmmu} follows from our choices of $y$, $b$, and the $(y_{\ell})$.  
\end{proof}

The next lemma, which is a quick consequence of \Cref{concentrationequation}, indicates that bridge-limiting measures (recall \Cref{mutmu0mu1}) can be obtained through free convolution (recall \Cref{TransformConvolution}). In what follows, we recall the rescaled semicircle law $\mu_{\semci}^{(t)}$ from \eqref{rhosct}.

\begin{lem}

	\label{nu0nutlimit}
	
	Fix a compactly supported measure $\nu_0 \in \mathscr{P}_{\fin}$ and a real number $t > 0$. For any real number $s > 0$, denote the measure $\nu_s = \nu \boxplus \mu_{\semci}^{(s)} \in \mathscr{P}_{\fin}$. Then, $\bm{\nu} = (\nu_s)_{s \in [0, t]}$ is the bridge-limiting measure process on $[0, t]$ with boundary data $(\nu_0; \nu_t)$. 
	
\end{lem}

\begin{proof}

By \Cref{mtscale} and \Cref{aab}, we may assume that $\nu_0$ (and thus $\nu_s$ for each $s \ge 0$) is a probability measure. Recalling (from \Cref{gammarho}) the classical locations $\gamma_{j;n}^{\mu} \in \mathbb{R}$ with respect to any probability measure $\mu$, abbreviate $\gamma_{j;n}(s) = \gamma_{j;n}^{(\nu_s)}$, for each real number $s \ge 0$ and integers $n \ge 1$ and $j \in \llbracket 1, n \rrbracket$. Further define the $n$-tuple $\bm{u}^n = (u_1^n, u_2^n, \ldots , u_n^n) \in \overline{\mathbb{W}}_n$ by setting $u_j^n = \gamma_{j;n} (0)$, and let $\bm{u}^n (s)$ denote the process obtained by running Dyson Brownian motion for time $s$ on $\bm{u}$; set $\bm{v}^n = \bm{u}^n (t)$. Denote $\mu_s^n = \emp \big( \bm{u}^n (s) \big)$ for each $s \in [0, t]$ (recall \eqref{aemp}). 

By (the $A=2$ case of) \Cref{concentrationequation}, we have for sufficiently large $n$ that, with probability at least $1 - n^{-1}$, 
\begin{flalign}
	\label{munsnus}
	\displaystyle\int_{-\infty}^{y-2(\log n)^5/n} \nu_s (dx) \le \displaystyle\int_{-\infty}^y \mu_s^n (dx) \le \displaystyle\int_{-\infty}^{y+2(\log n)^5/n} \nu_s (dx), \qquad \text{for all $(s, y) \in [0, t] \times \mathbb{R}$}.
\end{flalign}

\noindent Moreover, the second part of \Cref{lambdat} indicates that, conditional on $\bm{v}^n = \bm{u}^n (t)$, the law of $\bm{u}^n (s)$ is given by $\mathfrak{Q}^{\bm{u}^n; \bm{v}^n}$. This, together with \eqref{munsnus}, \Cref{rhot}, and \Cref{mutmu0mu1}, indicates that $\bm{\nu} = (\nu_s)_{s \in [0, t]}$ is the bridge-limiting measure process on $[0, t]$ with boundary data $(\nu_0; \nu_t)$. 
\end{proof}

\subsection{Proof of \Cref{hconcentrateq}} 

\label{ConcentrateProfile} 

In this section we establish \Cref{hconcentrateq}. In what follows, we recall the notation from that proposition; on free convolutions and Stieltjes and Hilbert transforms from \Cref{TransformConvolution}; on bridge-limiting measure processes from \Cref{mutmu0mu1}; and on the associated (inverted) height functions from \Cref{hrhot}. 
 	
 	The function $\mathfrak{H}$ of the proposition will eventually be the height function associated with a bridge-limiting measure process arising from a free convolution (as in \Cref{nu0nutlimit}), to which end we must define the initial data for the free convolution. This will proceed using \Cref{p:constructrho}, and so we must first define the real numbers $\bm{r}$ and measure $\varrho$ appearing there; we begin with the former. Define real numbers $r_0, r_1, \ldots ,r_{m-1} \in \mathbb{R}$ through the equations
	\begin{flalign}
		\label{rj}
		\displaystyle\sum_{j=0}^{k-1} \binom{k-1}{j} q_0^{(k-j)} r_j = q_1^{(k)}, \qquad \text{for each $k \in \llbracket 1, m \rrbracket$}.
	\end{flalign}

	\noindent Since $q_0^{(1)} \le -\varepsilon$ by \eqref{q0q1estimate}, these equations uniquely define the $r_j$ by induction on $j$. Furthermore, by \eqref{q0q1estimate} and induction on $j$, there exists a constant $C_2 = C_2 (\varepsilon, B, m) > 1$ such that $|r_j| < C_2$ for each $j \in \llbracket 0, m-1 \rrbracket$. 
	
	Next, observe that there exists a constant $C_3 = C_3 (\varepsilon, B, m) > 1$ and a probability measure $\mu \in \mathscr{P}_0$ with density $\varrho \in L^2 (\mathbb{R})$ such that $\supp \varrho = [-1, 1]$, such that $\| \varrho \|_{\mathcal{C}^{m+3} ([-1, 1])} \le C_2$ and such that
	\begin{flalign}
		\label{derivativerho}
		&  \partial_x^k \varrho (0) = -q_0^{(k+1)}, \quad \text{for each $k \in \llbracket 0, m-1 \rrbracket$}, \qquad \text{and} \qquad C_3^{-1} < \varrho (x) \le C_3, \quad \text{for each $x \in [-1, 1]$}.
	\end{flalign}
Moreover, after increasing $C_3$ if necessary, it follows from \eqref{mrho}, the fact that $\supp \varrho \subseteq [-1, 1]$, and the bound $\| \varrho \|_{\mathcal{C}^{m+3} ([-1, 1])} \le C_2$, that 
\begin{flalign*} 
	\big| \Real m_\mu(0) \big| = \pi \big|H \varrho(0) \big| = \bigg| \displaystyle\int_{-1}^1 \frac{\varrho (x) dx}{x} \bigg| = \bigg| \displaystyle\int_{-1}^1 \displaystyle\frac{\big( \varrho (x) - \varrho (0) \big) dx}{x} \bigg| \le 2 \| \varrho \|_{\mathcal{C}_1 ([-1, 1])} \leq 2C_2 \le C_3.
\end{flalign*}

\noindent This verifies the assumptions in  \Cref{p:constructrho}. Thus, by \Cref{p:constructrho}, there is a constant $c_1 = c_1 (\varepsilon, B, m) > 1$ and a probability measure $\mu_0 \in \mathscr{P}_0$ with density $\varrho_0 \in L^1 (\mathbb{R})$, such that $\supp \varrho_0 = [-1, 1]$ and 
	\begin{flalign}
		\label{rho0derivative} 
		\begin{aligned}
		 & (2C_3)^{-1} \le \varrho_0 (x) \le 2C_3, \quad \text{for each $x \in [-1, 1]$}; \qquad \varrho_0 (x) = \varrho (x), \quad \text{for each $x \in [-c_1, c_1]$}; \\
		 & \partial_z^k m_{\mu_0} (0) = r_k + \mathrm{i} \Imaginary \partial_z^k m_{\mu} (0), \quad \text{for each $k \in \llbracket 0, m \rrbracket$}.
		 \end{aligned}
	\end{flalign}

	 For each real number $t > 0$, define the free convolution measure $\mu_t = \mu_0 \boxplus \mu_{\semci}^{(t)} \in \mathscr{P}_0$ and the associated density $\varrho_t \in L^1 (\mathbb{R})$ satisfying $\mu_t (dx) = \varrho_t (x) dx$. Define the constant $c_2 = \frac{c_1}{3C_3}$, and denote the probabilility measure-valued process $\bm{\mu} = (\mu_t)_{[0, c_2]}$; observe that it is the bridge-limiting measure process on $[0, c_2]$ with boundary data $(\mu_0; \mu_{c_2})$, by \Cref{nu0nutlimit}. Denote its height function by $H^{\star} : [0, c_2] \times \mathbb{R} \rightarrow [0, 1]$. We will later define $\mathfrak{H}$ as the restriction of $H^{\star}$ to a certain strip domain.
	 
	 \begin{lem} 
	 	
	 	\label{h0profile}
	 	
		There exist constants $c_3 = c_3 (\varepsilon, B, m) > 0$ and $C_4 = C_4 (\varepsilon, B, m) > 1$ such that the following statements hold. In what follows, we denote $\mathfrak{U}_0 = (0, c_3) \times (-c_3, c_3)$. 
	 	
	 	\begin{enumerate} 
	 		\item For each $z \in \overline{\mathfrak{U}}_0$, we have $\partial_x H^{\star} \le -c_3$ and $\| H^{\star} \|_{\mathcal{C}^{m+2} (\mathfrak{U}_0)} \le C_4$. 
	 		\item The function $H^{\star}$ satisfies the equation \eqref{equationxtb} on $\mathfrak{U}_0$. 
	 		\item The full-$m$ local profile of $H^{\star}$ satisfies $\lim_{z \rightarrow (0, 0)} \bm{Q}_{H^{\star}}^{(m)} (z) = \bm{Q}$. 
	 	\end{enumerate} 
	 	 
	 \end{lem} 
 	
 	\begin{proof}

	The first statement of the lemma follows from \Cref{derivativetm}, whose conditions are verified by \eqref{derivativerho} and the fact that $\| \varrho \|_{\mathcal{C}^{m+1} ([-1, 1])} \le C_3$. This, together with  \Cref{hequation}, implies the second statement of the lemma. To verify the third statement observe, since $H^{\star}$ satisfies \eqref{equationxtb}, its local profile at any point is consistent in the sense of \Cref{admissibleconsistent}. Hence, by \Cref{q01j}, to show $\lim_{z \rightarrow (0, 0)} \bm{Q}_{H^{\star}}^{(m)} (z) = \bm{Q}$, it suffices to show for each $k \in \llbracket 1, m \rrbracket$ that 
	\begin{flalign}
		\label{q0q1}
		\displaystyle\lim_{z \rightarrow (0, 0)} \partial_x^k H^{\star} (z) = q_0^{(k)}, \qquad \text{and} \qquad \displaystyle\lim_{z \rightarrow (0, 0)} \partial_x^{k-1} \partial_t H^{\star} (z) = q_1^{(k)}.
	\end{flalign}

	\noindent The first statement in \eqref{q0q1} holds since $\lim_{z \rightarrow (0, 0)} \partial_x^k H^{\star} (z) = -\partial_x^{k-1} \varrho_0 (0) = -\partial_x^{k-1} \varrho (0) = q_0^{(k)}$, where the first follows from \eqref{htxintegral}; the second from \eqref{rho0derivative}; and the third from \eqref{derivativerho}.
	
	To verify the second statement in \eqref{q0q1}, observe that 
	\begin{flalign*}
		\displaystyle\lim_{z \rightarrow (0, 0)} \partial_x^{k-1} \partial_t H^{\star} (z) & = \bigg( \partial_y^{k-1} \partial_s\displaystyle\int_y^{\infty} \varrho_s (x) dx \bigg) \Bigg|_{s = 0, y = 0} \\
		& = \bigg( \pi \partial_y^{k-1} \displaystyle\int_y^{\infty} \partial_x \big( \varrho_0 (x) H\varrho_0 (x) \big) dx \bigg) \Bigg|_{y=0} \\
		& =  -\pi \bigg( \partial_y^{k-1} \big( \varrho_0 (y) H \varrho_0 (y) \big) \bigg) \Bigg|_{y=0}  = -\pi \displaystyle\sum_{j=0}^{k-1} \binom{k-1}{j} \partial_x^j H \varrho_0 (0) \cdot \partial_x^{k-j-1} \varrho_0 (0),
	\end{flalign*}

	\noindent where first statement follows from  \eqref{htxintegral}; the second from \eqref{trhoty}; the third from performing the integration (and using the fact that $\supp \varrho_0 = [-1, 1]$); and the fourth from taking the $(k-1)$-fold derivative. Together with \eqref{derivativerho}, the second statement of \eqref{mrho}, the last statement of \eqref{rho0derivative}, and \eqref{rj}, this yields 
	\begin{flalign*} 
		\displaystyle\lim_{z \rightarrow (0, 0)} \partial_x^{k-1} \partial_t H^{\star} (z) =  \displaystyle\sum_{j=0}^{k-1} \binom{k-1}{j} q_0^{(k-j)} \partial_x^j \Real m_{\mu_0} (0) = \displaystyle\sum_{j=0}^{k-1} \binom{k-1}{j} q_0^{(k-j)} r_j = q_1^{(k)},
	\end{flalign*}
	
	\noindent confirming the second statement of \eqref{q0q1} and thus the lemma.
	\end{proof} 
	
	Now we can establish \Cref{hconcentrateq}.

\begin{proof}[Proof of \Cref{hconcentrateq}]
	
	Let $G^{\star}$ denote the inverted height function associated with $H^{\star}$ and, recalling the notation of \Cref{h0profile}, fix 
	\begin{flalign}
		\label{c0c4} 
		c_0 = H^{\star} (0, 0); \qquad \nu = \frac{c_3^3}{4C_4^3}; \qquad c_4 = \displaystyle\frac{c_3^2}{2C_4}.
	\end{flalign}

	\noindent For each real number $s \in [0, \nu]$, let $f (s) = G^{\star} (s, c_0 + c_4)$ and let $g (s) = G^{\star} (s, c_0 - c_4)$. Observe that 
	\begin{flalign*} 
		c_0 = H^{\star} (0, 0) \ge - c_3 \displaystyle\int_0^{c_3} \partial_x H^{\star} (0, y) dy \ge c_3^2,
	\end{flalign*} 

	\noindent by the first part of \Cref{h0profile}, so $c_0 \ge c_3^2 \ge 4c_4 > c_4$. Then define the strip domain $\mathfrak{S} = \mathfrak{P}_{f; g} \subset \mathbb{R}^2$ and the function $\mathfrak{H} : \overline{\mathfrak{S}} \rightarrow \mathbb{R}$ by setting $\mathfrak{H} (z) = H^{\star} (z) - c_0 + c_4$ for each $z \in \overline{\mathfrak{S}}$. See the left side of \Cref{f:coupling_2}. Let us verify that the conditions of the proposition hold for this choice of $(\mathfrak{S}, \mathfrak{H})$.
	
	To see that $[0, \nu] \times [-\nu, \nu] \subseteq \mathfrak{P}_{f;g} \subseteq [0, \nu]  \times [-C \nu, C \nu]$, observe again by the first part of \Cref{h0profile}, and using $H^\star(0,0)=c_0$ and $H^\star(0,f(0))=c_0+c_4$, we have $-\frac{c_3}{2C_4} = -c_3^{-1} c_4 \le f(0) \le -C_4^{-1} c_4 = -\frac{c_3^2}{2C_4^2}$. Since $\big| f' (s) \big| = \big| (\partial_t H^{\star}) (s, f(s)) \big| \cdot \big| (\partial_x H^{\star}) (s, f(s)) \big|^{-1}  \le c_3^{-1} C_4$ for $s \in [0, \nu]$ by the first part of \Cref{h0profile}, it follows from \eqref{c0c4} that, for each $s \in [0, \nu]$, 
	\begin{flalign*} 
		-\displaystyle\frac{4 C^2_4 \nu}{c^2_3} \le -\frac{c_3}{2C_4} - \frac{C_4 \nu}{c_3} \le f(s) \le \frac{C_4 \nu}{c_3} - \frac{c_3^2}{2C_4^2} = -\frac{C_4 \nu}{c_3} \le -\nu.
	\end{flalign*} 
where we also used $\nu=c_3^3 (4C_4^3)^{-1} $ from \eqref{c0c4}.  Similar reasoning implies that $\nu \le g(s) \le 4 c_3^{-2} C^2_4 \nu$, which indicates that $[0, \nu] \times [-\nu, \nu] \subseteq \mathfrak{P}_{f; g} \subseteq [0, \nu] \times [-C\nu, C\nu]$ for $C = 4C^2_4 c_3^{-2}$.
	
	To establish the other parts of the proposition, observe that the bounds on $-\partial_x \mathfrak{H}$ and on $\| \mathfrak{H} \|_{\mathcal{C}^{m+1}}$ follow from the first part of \Cref{h0profile}. Moreover, the fact that the full-$m$ local profile of $\mathfrak{H}$ around $(0, 0)$ is given by $\bm{Q}$ follows from the third part of \Cref{h0profile}. 
	
	It therefore remains to show that $\mathfrak{H}$ is $(C_1; c; \delta)$ concentrating in the sense of \Cref{hhestimate0}. That it satisfies \eqref{equationxtb} follows from the second part of \Cref{h0profile}. Now let $n \ge 1$ be an integer; set $n' = \lfloor 2c_4 n \rfloor$ (in what follows, we again omit the floors and ceilings, as they will not affect the proofs), and define the $n'$-tuples $\bm{u}, \bm{v} \in \overline{\mathbb{W}}_{n'}$ by setting 
	\begin{flalign}
		\label{ujvj0}
		u_j = G^{\star} ( 0, c_0 - c_4 + jn^{-1} ), \quad \text{and} \quad v_j = G^{\star} (\nu, c_0 - c_4 + jn^{-1}), \quad \text{for each $j \in \llbracket 1, n' \rrbracket$};
	\end{flalign}

	\noindent and sample $n'$ non-intersecting Brownian bridges $\bm{y} = (y_1, y_2, \ldots , y_{n'}) \in \llbracket 1, n \rrbracket \times \mathcal{C} \big( [0, \nu] \big)$ from the measure $\mathfrak{Q}_{f; g}^{\bm{u}; \bm{v}} (n^{-1})$. We must show that 
	\begin{flalign}
		\label{hndeltay}
		\mathbb{P} \Bigg[ \displaystyle\sup_{z \in \mathfrak{S}} \big| n^{-1} \mathsf{H}^{\bm{y}} (z) - H^{\star} (z) + c_0 - c_4 \big| \le n^{\delta-1} \Bigg] \ge 1 - C_1 e^{-c (\log n)^2}. 
	\end{flalign}

\begin{figure}
\center
\includegraphics[width=0.8\textwidth]{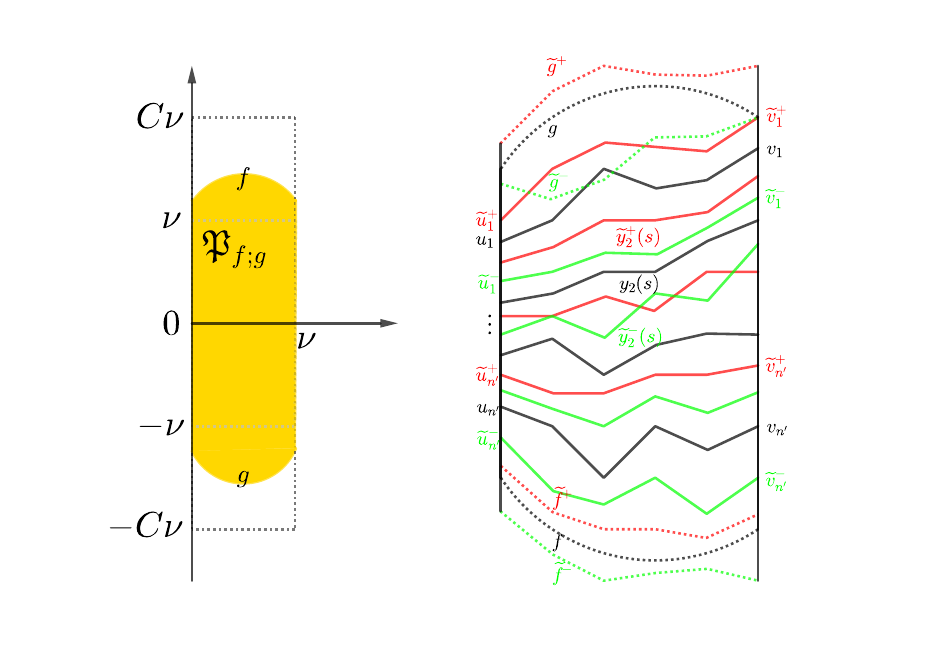}
\caption{Depicted in the left panel is the strip domain $\mathfrak S=\mathfrak P_{f;g}$. Depicted in the right panel is the coupling such that $\widetilde y_j^-\leq y_j\leq \widetilde y_j^+$.}
\label{f:coupling_2}
\end{figure}

	To that end, define the $n$-tuples $\bm{u}^-, \bm{u}^+ \in \overline{\mathbb{W}}_n$ by for each $j \in \llbracket 1, n \rrbracket$ setting 
	\begin{flalign}
		\label{ujvjy}
		& u_j^- = G^{\star} ( 0, jn^{-1} + n^{\delta/2-1}); \qquad u_j^+ = G^{\star} ( 0,  jn^{-1} - n^{\delta/2-1}).
	\end{flalign}	

	\noindent Let $\bm{y}^- (t) = \big( y_1^- (t), y_2^- (t), \ldots , y_n^- (t) \big)$ and $\bm{y}^+ (t) = \big( y_1^+ (t), y_2^+ (t), \ldots , y_n^+ (t) \big)$ denote Dyson Brownian motion run with initial data $\bm{u}^-$ and $\bm{u}^+$, respectively. Since $H^{\star}$ is the height function associated with the measure-valued process $\big( \mu_0 \boxplus \mu_{\semci}^{(t)} \big)_{t \in [0, c_2]}$, \Cref{concentrationequation} gives constants $c = c(\varepsilon, \delta, B, m) > 0$ and $C_1 = C_1 (\varepsilon, \delta, B, m) > 1$ such that, with probability at least $1 - C_1 e^{-c(\log n)^2}$, we have
	\begin{flalign}
		\label{hyhhyh} 
		\begin{aligned}
		& H^{\star} (t, x) - n^{\delta/2-1} - n^{\delta/4-1} \le n^{-1} \mathsf{H}^{\bm{y}^-} (t, x)  \le H^{\star} (t, x) - n^{\delta/2-1} + n^{\delta/4-1}; \\
		& H^{\star} (t, x) + n^{\delta/2-1} - n^{\delta/4 - 1} \le n^{-1} \mathsf{H}^{\bm{y}^+} (t, x) \le H^{\star} (t, x) + n^{\delta/2-1} + n^{ \delta/4-1},
		\end{aligned} 
	\end{flalign} 

	\noindent for any $(t, x) \in \overline{\mathfrak{S}}$. In the below, we restrict to this event. 
	
	Since $f(t) = G^{\star} (t, c_0 - c_4)$ and $g(t) = G^{\star} (t, c_0 + c_4)$, applying \eqref{hyhhyh} with $x$ equal to each of $y_{ (c_0 + c_4) n + 1}^- (t)$, $y_{ (c_0 + c_4) n + 1}^+ (t)$, $y_{ (c_0 - c_4)n}^- (t)$, $y_{ (c_0 - c_4) n}^+ (t)$, $f(t)$, and $g(t)$ (and using the fact that $H^{\star} (t, x)$, $\mathsf{H}^{\bm{y}^-} (t, x)$, and $\mathsf{H}^{\bm{y}^+} (t, x)$ are decreasing in $x$) yields 
	\begin{flalign}
		\label{yfyygy}
		y_{(c_0 + c_4) n + 1}^- (t) \le f(t) \le y_{ (c_0 + c_4) n + 1}^+ (t); \qquad y_{(c_0 - c_4) n}^- (t) \le g(t) \le y_{ (c_0 - c_4)n}^+ (t).
	\end{flalign}

	Now, denote $\mathfrak{c} = c_0 - c_4$, and condition on the $y_j^- (s)$ and $y_j^+ (s)$ with $j \notin \llbracket \mathfrak{c} n + 1, \mathfrak{c} n + n' \rrbracket$. Define the functions $\widetilde{f}^-, \widetilde{f}^+, \widetilde{g}^-, \widetilde{g}^+ : [0, \nu] \rightarrow \mathbb{R}$ by setting 
	\begin{flalign*}
		\widetilde{f}^{\pm} (s) = y_{\mathfrak{c} n + n'+1}^{\pm} (s); \qquad \widetilde{g}^{\pm} (s) = y_{\mathfrak{c} n}^{\pm} (s),
	\end{flalign*} 

	\noindent for each index $\pm \in \{ +, - \}$ and real number $s \in [0, \nu]$. Further define $\widetilde{\bm{y}}^- = (\widetilde{y}_1^-, \widetilde{y}_2^-, \ldots , \widetilde{y}_{n'}^-) \in \llbracket 1, n' \rrbracket \times \mathcal{C} \big( [0, \nu] \big)$ and $\widetilde{\bm{y}}^+ = (\widetilde{y}_1^+, \widetilde{y}_2^+, \ldots , \widetilde{y}_{n'}^+) \in \llbracket 1, n' \rrbracket \times \mathcal{C} \big( [0, \nu] \big)$ by reindexing $\bm{y}^-$ and $\bm{y}^+$, respectively, namely, by setting
	\begin{flalign}
		\label{yjy}
		\widetilde{y}_j^{\pm} (s) = y_{j+\mathfrak{c} n}^{\pm} (s), \qquad \text{for each $\pm \in \{ +, - \}$ and $(j, s) \in \llbracket 1, n' \rrbracket \times [0, \nu]$},
	\end{flalign}

	\noindent and define the $n'$-tuples $\widetilde{\bm{u}}^-, \widetilde{\bm{u}}^+, \widetilde{\bm{v}}^-, \widetilde{\bm{v}}^+ \in \overline{\mathbb{W}}_{n'}$ by setting 
	\begin{flalign}
		\label{j2v}
		\widetilde{u}_j^{\pm} = u_{j+\mathfrak{c}n}^{\pm} = \widetilde{y}_j^{\pm} (0), \quad \text{and} \quad \widetilde{v}_j^{\pm} = \widetilde{y}_j^{\pm} (\nu), \qquad \text{for each $\pm \in \{ +, - \}$ and $j \in \llbracket 1, n' \rrbracket$.} 
	\end{flalign}

	Then, by the second part of \Cref{lambdat}, the laws of $\widetilde{\bm{y}}^-$ and $\widetilde{\bm{y}}^+$ are given by $\mathfrak{Q}_{\tilde{f}^-; \tilde{g}^-}^{\tilde{\bm{u}}^-; \tilde{\bm{v}}^-} (n^{-1})$ and $\mathfrak{Q}_{\tilde{f}^+; \tilde{g}^+}^{\tilde{\bm{u}}^+; \tilde{\bm{v}}^+} (n^{-1})$, respectively; see the right side of \Cref{f:coupling_2}. Observe from \eqref{ujvj0}, \eqref{ujvjy}, and \eqref{yjy} that $\widetilde{\bm{u}}^- \le \bm{u} \le \widetilde{\bm{u}}^+$, since $\partial_y G^{\star} \le 0$. Moreover, \eqref{yfyygy} and \eqref{yjy} imply that $\widetilde{f}^- = \widetilde{y}_{n'+1}^- \le f \le \widetilde{y}_{n'+1}^+ = \widetilde{f}^+$ and $\widetilde{g}^- = \widetilde{y}_0^- \le g \le \widetilde{y}_0^+ = \widetilde{g}^+$. Additionally, by \eqref{ujvj0} and by the $t = \nu$ (or $s = \nu$) case of \eqref{hyhhyh}, \eqref{yjy}, \eqref{j2v}, and the first statement of \Cref{h0profile}, we find that $\bm{\widetilde{v}}^- \le \bm{v} \le \bm{\widetilde{v}}^+$.  Together with the monotone coupling \Cref{monotoneheight}, these indicate that it is possible to couple $\widetilde{\bm{y}}^-$, $\bm{y}$, and $\widetilde{\bm{y}}^+$ in such a way that $\widetilde{y}_j^- (s) \le y_j (s) \le \widetilde{y}_j^+ (s)$ for each $(j, s) \in \llbracket 0, n' \rrbracket \times [0, \nu]$. This, with \eqref{hyhhyh} and \eqref{yjy}, yields \eqref{hndeltay} and thus the proposition.
\end{proof}

\section{Universality of Bulk Statistics} 

\label{StatisticsKernel}

	In this section we establish \Cref{xkernel}. 
	
	\subsection{Proof of \Cref{xkernel}} 
	
	\label{Comparex}
	
	In this section we establish \Cref{xkernel}. Given \Cref{gh}, its proof will be similar to what was done in \cite[Sections 3 and 4]{ULTS}, by comparing non-intersecting Brownian bridges to Dyson Brownian motion (though the affine invariance of our model will simplify the arguments here). Throughout, we adopt the notation from \Cref{xkernel}, and we set 
	\begin{flalign}
		\label{omegann} 
		\omega = 90000( \delta + m^{-1}); \qquad n' = \lfloor n^{1/2} \rfloor; \qquad n'' = \lfloor n^{1/4} \rfloor, \qquad s_0 = n^{\omega/2-1}, \qquad s_1 = n^{3 \omega - 1},
	\end{flalign}

	\noindent so that $\omega \le \frac{1}{20}$ (as $m \ge 2^{25}$ and $\delta \le \frac{1}{5m^2}$). In what follows, we omit all floors and ceilings, as they do not affect the proofs. We also assume (as we may by shifting) that $G(t_0 - s_0, y_0) = 0$. 
	
	Define $\breve{\bm{x}} = (\breve{x}_1, \breve{x}_2, \ldots , \breve{x}_{n'}) \in \llbracket 1, n' \rrbracket \times \mathcal{C} \big( [0, s_0 + s_1] \big)$ by reindexing $\bm{x}$, specifically, by setting 
	\begin{flalign*}
		\breve{x}_j (s) = x_{y_0 n - n'/2 + j} (t_0 - s_0 + s), \qquad \text{for each $(j, s) \in \llbracket 1, n' \rrbracket \times [0, s_0 + s_1]$}.
	\end{flalign*}

	\noindent Also define the $n'$-tuples $\breve{\bm{u}} = (\breve{u}_1, \breve{u}_2, \ldots , \breve{u}_{n'}) \in \overline{\mathbb{W}}_{n'}$ and $\breve{\bm{v}} = (\breve{v}_1, \breve{v}_2, \ldots , \breve{v}_{n}) \in \overline{\mathbb{W}}_{n'}$, and functions $\breve{f}, \breve{g} : [0, s_0 + s_1] \rightarrow \mathbb{R}$ by for each $j \in \llbracket 1, n' \rrbracket$ and $s \in [0, s_0 + s_1]$ setting
	\begin{flalign*}
		 \breve{u}_j = \breve{x}_j (0); \qquad \breve{v}_j = \breve{x}_j (s_0 + s_1); \qquad \breve{f} (s) = \breve{x}_{n'/2 + n''+1} (s); \qquad \breve{g} (s) = \breve{x}_{n'/2-n''-1} (s).
	\end{flalign*}

	 Next, we condition on $\breve{x}_j (s)$ for $(j, s) \notin \big\llbracket  \frac{n'}{2}-n'', \frac{n'}{2} + n'' \rrbracket \times (0, s_0 + s_1)$. Defining the event 
	\begin{flalign}
		\label{a0} 
		\mathscr{A}_0 = \bigcap_{(j, s) \notin \llbracket n'/2 - n'', n'/2+n'' \rrbracket \times (0, s_0 + s_1)} \Bigg\{ \bigg| \breve{x}_j (s) - G \Big(t_0 - s_0 + s, y_0 - \displaystyle\frac{n'}{2n} + \displaystyle\frac{j}{n} \Big) \bigg| \leq n^{\omega/4-1} \Bigg\},
	\end{flalign}

	\noindent we have by \Cref{gh} and the fact that $n^{\omega/4 - 1} \ge \varkappa + n^{2/m+3\delta/2-1}$ (by \eqref{omegann} and the fact that $\varkappa \le n^{\delta - 1}$) that there exists a constant $c_0 = c_0 (\varepsilon, \delta, B_0, m) > 0$ such that 
	\begin{flalign}
		\label{a0probability} 
		\mathbb{P} [\mathscr{A}_0] \ge 1 - c_0^{-1} e^{-c_0 (\log n)^2}.		
	\end{flalign}

	\noindent We restrict to $\mathscr{A}_0$ in what follows.
	
	Next, for each index $\pm \in \{ +, - \}$, let $G^{\pm} : [0, 1] \rightarrow \mathbb{R}$ be a function satisfying the following six properties. 
	\begin{enumerate} 
		\item For all $y \in [0, 1]$, we have $G(t_0 - s_0, y) - n^{\omega/2-1} \le G^- (y) \le G (t_0 - s_0, y)$ and that $G(t_0 - s_0, y)\le G^+ (y) \le G(t_0 - s_0, y) + n^{\omega/2-1}$. 
		\item For all $y \in [0, 1]$ with $|y - y_0| \le \frac{5n''}{6n}$, we have $G(s_0 - t_0, y) - n^{3 \omega/10 - 1} \le G^- (y) \le G(t_0 - s_0, y)$ and that $G(t_0 - s_0, y) \le G^+ (y) \le G(t_0 - s_0, y) + n^{3\delta/10 - 1}$.
		\item For $y \in [0, 1]$ with $|y - y_0| \le \frac{n''}{2n}$, we have $G^{\pm} (y) = G (t_0 - s_0, y)$. 
		\item For all $y \in [0, 1]$ with $|y-y_0| \in \big[ \frac{2n''}{3n}, \frac{5n''}{6n} \big]$, we have $G^{\pm} (t_0 - s_0, y) = G(t_0 - s_0, y) \pm n^{3\omega/10 - 1}$. 
		\item For all $y \in [0, 1]$ with $|y - y_0| \ge \frac{n''}{n}$, we have $G^{\pm} (y) = G (t_0 - s_0, y) \pm n^{\omega/2-1}$. 
		\item For all $y \in [0, 1]$, we have $-\frac{2}{\varepsilon} \le \partial_y G^{\pm} (y) \le -\frac{\varepsilon}{2}$, and $\| G^{\pm} \|_{\mathcal{C}^{m+1}} \le 2 B_0$. 
	\end{enumerate} 

	The existence of such a function follows from the facts that $ n^{\omega-1} \le n^{-3/4} = \frac{n''}{n}$, that $G \in \Adm_{\varepsilon} (\mathfrak{R})$, and that $\| G \|_{\mathcal{C}^{m+1} (\mathfrak{R})} \le B_0$ (by \Cref{fgr}). Then, define the $n'$-tuples $\breve{\bm{u}}^-, \breve{\bm{u}}^+ \in \overline{\mathbb{W}}_{n'}$ by, for each index $\pm \in \{ +, - \}$ and integer $j \in \llbracket 1, n' \rrbracket$,  setting
	\begin{flalign} 
		\label{uj2} 
		& \breve{u}_j^{\pm} = \breve{u}_j,  \quad  \text{if $j \in \Big\llbracket \displaystyle\frac{n'}{2} - \displaystyle\frac{2n''}{3}, \displaystyle\frac{n'}{2} + \displaystyle\frac{2n''}{3} \Big\rrbracket$}; \qquad \breve{u}_j^{\pm} = G^{\pm} \Big( y_0 - \displaystyle\frac{n'}{2n} + \displaystyle\frac{j}{n} \Big), \quad \text{otherwise},
	\end{flalign}

	\noindent observing that $\bm{u}^{\pm} \in \overline{\mathbb{W}}_{n'}$ on $\mathscr{A}_0$ due to \eqref{a0}.
	
	Let $\breve{\bm{w}} (s) = \big( \breve{w}_1 (s), \breve{w}_2 (s), \ldots , \breve{w}_{n'} (s) \big)$ denote Dyson Brownian motion with initial data $\breve{\bm{u}}$, as in \eqref{lambdaequation}, but with the $d\breve{w}_i$ there multiplied by a factor of $n^{-1} n'$ and the $dB_i$ multiplied by a factor of $(n^{-1} n)^{1/2}$; more specifically, they solve (instead of \eqref{lambdaequation}) the stochastic differential equation 
	\begin{flalign*}
		d \breve{w}_i (s) = n^{-1} \displaystyle\sum_{\substack{1 \le j \le n' \\ j \ne i}}\displaystyle\frac{1}{\breve{w}_i^{\pm} (s) - \breve{w}_j^{\pm} (s)} + n^{-1/2} dB_i (s).
	\end{flalign*} 
	
	\noindent Similarly, for each $\pm \in \{ +, - \}$, let $\breve{\bm{w}}^{\pm} (s) = \big( \breve{w}_1^{\pm} (s), \breve{w}_2^{\pm} (s), \ldots , \breve{w}_{n'}^{\pm} (s) \big)$ denote Dyson Brownian motion with initial data $\breve{\bm{u}}^{\pm}$, but with the same rescaling of the $d\breve{w}_i^{\pm} (s)$ as above; see \Cref{f:DBMcouple}. Moreover  define the functions $\breve{f}^-, \breve{f}^+, \breve{g}^-, \breve{g}^+ : [0, s_0 + s_1] \rightarrow \mathbb{R}$ by, for each index $\pm \in \{ +, - \}$ and real number $s \in [0, s_0 + s_1]$, setting 
	\begin{flalign*} 
		\breve{f}^{\pm} (s) = \breve{w}_{n'/2 + n'' + 1}^{\pm} (s); \qquad  \breve{g}^{\pm} (s) = \breve{w}_{n'/2 - n'' - 1}^{\pm} (s).
	\end{flalign*}

	We then have the following proposition, to be established in \Cref{ProofA0} below; it indicates that the middle paths in $\breve{\bm{x}}$ are bounded between those in $\breve{\bm{w}}^-$ and $\breve{\bm{w}}^+$ (after adding a suitable drift).

\begin{figure}
\center
\includegraphics[width=1\textwidth]{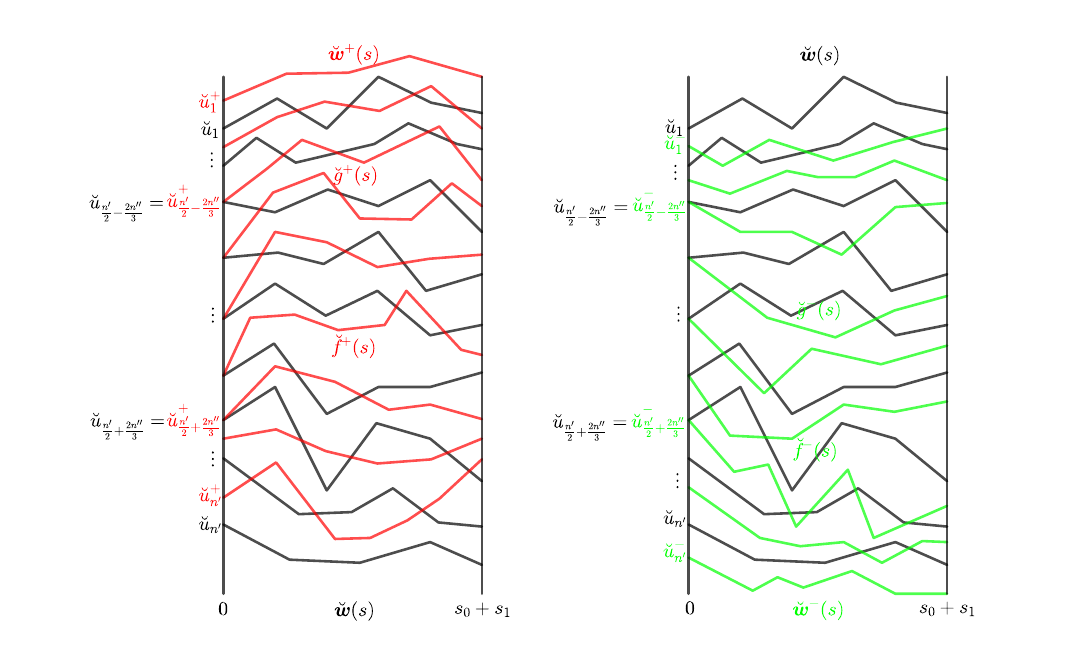}
\caption{Shown above are the paths of $\breve{\bm w}^+$ and $\breve{\bm w}$ in the left panel, and of $\breve{\bm w}^-$ and $\breve{\bm w}$ in the right panel. }
\label{f:DBMcouple}
\end{figure}
	
	\begin{prop}
		
		\label{eventa0} 
		
		Set $\mathfrak{a} = \partial_t G(t_0, y_0)$. There exist constants $c = c(\varepsilon, \omega, m, B_0) > 0$ and $C = C(\varepsilon, \omega, m, B_0) > 1$; a coupling between $\breve{\bm{x}}$, $\breve{\bm{w}}^-$, and $\breve{\bm{w}}^+$; and an event $\mathscr{A}_1$, with $\mathbb{P} [\mathscr{A}_1] \ge 1 - C e^{-c(\log n)^2}$, such that on $\mathscr{A}_1$ we have 
		\begin{flalign*}
			\breve{w}_j^- (s) + (\mathfrak{a} - n^{- \omega}) s \le \breve{x}_j (s) \le \breve{w}_j^+ (s) + (\mathfrak{a}+ n^{-\omega}) s,
		\end{flalign*} 
		
		 \noindent for each $(j, s) \in \big\llbracket \frac{n'}{2} - n'', \frac{n'}{2} + n'' \big\rrbracket \times [0, s_0 + s_1]$.
		
	\end{prop}

	We next have the following three lemmas, whose proofs are both consequences of results in \cite{CLSM,FEUM} and are therefore deferred to \Cref{ProofMotion} below. The first, to be established in \Cref{Proofw} below, indicates that the local statistics of $\breve{\bm{w}}$ are prescribed by the sine process. The second, to be established in \Cref{Proofa2} below, indicates that the middle paths of $\breve{\bm{w}}^-$ and $\breve{\bm{w}}^+$ are close to those of $\breve{\bm{w}}$. The third, to be established in \Cref{Proofwj02} below, states that $\breve{\bm{w}}^-$ and $\breve{\bm{w}}^+$ are likely close in expectation, under any coupling between them. 
	
	\begin{lem}
		
		\label{kernelconvergew}
		
		Let $\mathfrak{a} > 0$ be a (bounded, uniformly in $n$) constant, and denote the $n$-tuple $\bm{w} = \theta n \cdot \big( \breve{\bm{w}} (s_0) + \mathfrak{a} s_0 \big) \in \overline{\mathbb{W}}_n$. We have 
		\begin{flalign*}
			\displaystyle\lim_{n \rightarrow \infty} \displaystyle\int_{\mathbb{R}^k} F (\bm{a}) p_{\bm{w}}^{(k)} (\bm{a}) d \bm{a} = \displaystyle\int_{\mathbb{R}^k} F(\bm{a}) p_{\sin}^{(k)} (\bm{a}) d \bm{a}.
		\end{flalign*}
		
	\end{lem}

	\begin{lem} 
		
	\label{eventa1}
	
	There exist constants $c = c(\varepsilon, \delta, B_0, m) > 0$  and $C = C (\varepsilon, \delta, B_0, m) > 1$; a coupling between $\breve{\bm{w}}^-$, $\breve{\bm{w}}$; and $\breve{\bm{w}}^+$; and an event $\mathscr{A}_2$, with $\mathbb{P}[\mathscr{A}_2] \ge 1 - C n^{-c}$, such that on $\mathscr{A}_2$ we have 
	\begin{flalign}
		\label{wjwj} 
		\begin{aligned}
		& \displaystyle\max_{j \in \llbracket n'/2 - n^{\omega/15000}, n'/2 +  n^{\omega/15000} \rrbracket}  \big| \breve{w}_j (s_0) - \breve{w}_j^- (s_0) \big| \le n^{-1-\omega/5}; \\ 
		& \displaystyle\max_{j \in \llbracket n'/2 - n^{\omega/15000}, n'/2 + n^{\omega/15000} \rrbracket} \big| \breve{w}_j (s_0) - \breve{w}_j^+ (s_0) \big|  \le n^{-1-\omega/5}.
		\end{aligned}
	\end{flalign}
		
	\end{lem}

	\begin{lem}
		
		\label{xy00} 
		
		For any real number $\chi > 0$, there exist constants $c = c(\varepsilon, \delta, B_0, m) > 0$ and $C = C(\chi, \varepsilon, \delta, B_0, m) > 1$ such that, under any coupling between $\breve{\bm{w}}^-$ and $\breve{\bm{w}}^+$, we have 
		\begin{flalign*}
			\displaystyle\max_{j \in \llbracket n'/2 - n^{\omega/15000}, n'/2 + n^{\omega/15000} \rrbracket} \mathbb{E} \Big[ \big| \breve{w}_j^+ (s_0) - \breve{w}_j^- (s_0) \big| \cdot \textbf{\emph{1}}_{|\breve{w}_j^+ (s_0) - \breve{w}_j^- (s_0)| > n^{\chi - 1}} \Big] < C e^{-c (\log n)^2}.
		\end{flalign*} 
		
	\end{lem}

	Given the above four results, we can establish \Cref{xkernel}.

	\begin{proof}[Proof of \Cref{xkernel}]

		 By a Taylor expansion (and the fact that $G(t_0 - s_0, y_0) = 0$), we have 
		\begin{flalign}
			\label{gt0y0} 
			\big| G(t_0, y_0) - \mathfrak{a} s_0 \big| = \big| G(t_0, y_0) - G(t_0 - s_0, y_0) - \mathfrak{a} s_0 \big| \le B_0 s_0^2 \le n^{-1-\omega}.
		\end{flalign} 
	
		\noindent So, it suffices by the continuity of $F$ to show that
		\begin{flalign}
			\label{limitpx} 
			\displaystyle\lim_{n \rightarrow \infty} \displaystyle\int_{\mathbb{R}^k} F(\bm{a}) p_{\theta n \cdot (\breve{\bm{x}} (s_0) - \mathfrak{a} s_0)}^{(k)} (\bm{a}) d \bm{a} = \displaystyle\int_{\mathbb{R}^k} F(\bm{a}) p_{\sin}^{(k)} (\bm{a}) d \bm{a}.
		\end{flalign}
		
		\noindent We next claim that there exist constants $c = c(\varepsilon, \delta, B_0, m) \in \big( 0, \frac{\omega}{30000} \big)$ and $C = C(\varepsilon, \delta, B_0, m) > 1$, and a coupling between $\bm{x}$, $\breve{\bm{w}}^-$, and $\breve{\bm{w}}^+$, such that 
		\begin{flalign}
			\label{pa} 
			\mathbb{P} [\mathscr{A}] \ge 1 - C n^{-c}, \qquad \text{where} \quad \mathscr{A} = \mathscr{A}_0' \cap \mathscr{A}_1 \cap \mathscr{A}_3, 
		\end{flalign} 
	
		\noindent where we have recalled $\mathscr{A}_1$ from \Cref{eventa0} and defined the events
		\begin{flalign*} 
			& \mathscr{A}_0' = \bigcap_{|j - y_0 n| \le n^{\omega/15000}} \Big\{ \big| x_j (s_0) - G(s_0, jn^{-1}) \big| \le n^{\omega/30000 - 1} \Big\}; \\
			& \mathscr{A}_3 = \bigcap_{|j - y_0 n| \le n^{\omega/15000}} \Big\{ \big| \breve{w}_j^- (s_0) - \breve{w}_j^+ (s_0) \big| \le n^{-1-c} \Big\}.
		\end{flalign*}

		Let us establish the theorem assuming \eqref{pa}; throughout this proof, we set $\omega' = \frac{\omega}{30000}$. Since $\lim_{n \rightarrow \infty} \mathbb{P}[\mathscr{A}] = 1$, we may restrict to the event $\mathscr{A}$ in what follows. Setting $\mathfrak{a} = \partial_t G(t_0, y_0)$ and applying the definition of $\mathscr{A}_1$ from \Cref{eventa0} (at $s = s_0 = n^{\omega/2-1}$) gives 
		\begin{flalign}
			\label{wjs00} 
			\breve{w}_j^- (s_0) + \mathfrak{a} s_0 - n^{-1-\omega/2} \le \breve{x}_j (s_0) \le \breve{w}_j^+ (s_0) + \mathfrak{a} s_0 + n^{-1-\omega/2},
		\end{flalign}
	
		\noindent for each integer $j \in \big\llbracket \frac{n'}{2} - n'', \frac{n'}{2}+n'' \big\rrbracket$. Further observe on $\mathscr{A}_0'$ that for $j > \frac{n'}{2} + n^{2\omega'}$ we have $\breve{x}_j (s_0) - \mathfrak{a} s_0 \le G(t_0, y_0 + n^{2\omega'-1})  - \mathfrak{a} s_0 + n^{\omega' - 1} \le n^{-1-\omega} - \varepsilon n^{2\omega'-1} + n^{\omega' - 1} < -n^{\omega' - 1}$ (where we have used \eqref{gt0y0} and the fact from \Cref{fgr} that $\partial_y G \le -\varepsilon$). By similar reasoning, we have on $\mathscr{A}_0$ that for $j < \frac{n'}{2} - n^{2\omega'}$ that $\breve{x}_j (s_0) > n^{\omega' - 1}$. By these two bounds, together with \eqref{wjs00} and the fact that $\big| \breve{w}_j^- (s_0) - \breve{w}_j^+ (s_0) \big| \le n^{-1-c}$ for $j \in \llbracket \frac{n'}{2} - n^{2\omega'}, \frac{n'}{2} + n^{2\omega'} \big\rrbracket$ on $\mathscr{A}$, it follows from the continuity of $F$ that, to verify \eqref{limitpx}, it suffices to show
		\begin{flalign*}
			\displaystyle\lim_{n \rightarrow \infty} \displaystyle\int_{\mathbb{R}^k} F(\bm{a}) p_{\theta n \cdot \breve{\bm{w}}^- (s_0)}^{(k)} (\bm{a}) d \bm{a} = \displaystyle\int_{\mathbb{R}^k} F(\bm{a}) p_{\sin}^{(k)} (\bm{a}) d \bm{a}. 
		\end{flalign*} 		
		
		\noindent This follows from the first statement of \eqref{wjwj} and \Cref{kernelconvergew} (with the continuity of $F$).
		
		It remains to confirm \eqref{pa}. To that end, it suffices by a union bound to show that $\mathbb{P} [\mathscr{A}_0'] \ge 1 - C n^{-c}$, that $\mathbb{P} [\mathscr{A}_1] \ge 1 - C n^{-c}$, and that $\mathbb{P} \big[ \mathscr{A}_1 \cap \mathscr{A}_3^{\complement} \big] \le C n^{-c}$. The first follows by \Cref{gh}, since $\omega' = 3 (m^{-1} + \delta) > 2m^{-1} + 3 \delta$, and the second follows from \Cref{eventa0}.
		
		The third will follow from  \eqref{wjwj}, \Cref{xy00}, and a Markov estimate. To implement this, observe that \eqref{wjwj} and \Cref{xy00} together imply the existence of constants $c_1 = c_1 (\varepsilon, \delta, B_0, m) > 0$ and $C_1 = C_1 (\varepsilon, \delta, B_0, m) > 1$, and a coupling between $\breve{\bm{w}}^-$ and $\breve{\bm{w}}^+$ (possibly different from the one above), such that 
		\begin{flalign*}
			\displaystyle\sum_{|j-n'/2| \le n^{\omega/15000}} \mathbb{E} \Big[ \big| \breve{w}_j^+ (s_0) - \breve{w}_j^- (s_0) \big| \Big] \le C_1 n^{-1-c_1},
		\end{flalign*}
	
		\noindent from which it follows that 
		\begin{flalign}
			\label{wjs01} 
			\displaystyle\sum_{|j-n'/2| \le n^{\omega/15000}} \Big( \mathbb{E} \big[ \breve{w}_j^+ (s_0) \big] - \mathbb{E} \big[ \breve{w}_j^- (s_0) \big] \Big) \le C_1 n^{-1-c_1}.
		\end{flalign}
	
		\noindent Since this bound is independent of the coupling between $\breve{\bm{w}}^+$ and $\breve{\bm{w}}^-$, it also holds for our original one given by \Cref{eventa0}. Next, restricting to the event $\mathscr{A}_1$, we have $\breve{w}_j^- (s_0) \le \breve{w}_j^+ (s_0) + n^{-\omega} s_0 = \breve{w}_j^+ (s_0) + n^{-1-\omega/2}$. Together with \eqref{wjs01}, this implies under our original coupling that 
		\begin{flalign*}
			\displaystyle\sum_{|j - n'/2| \le n^{\omega/15000}} & \mathbb{E} \Big[ \textbf{1}_{\mathscr{A}_1} \cdot \big| \breve{w}_j^+ (s_0) - \breve{w}_j^- (s_0) \big| \Big] \\
			 & \le \displaystyle\sum_{|j-n'/2| \le n^{\omega/15000}} \Big( \mathbb{E} \big[ \breve{w}_j^+ (s_0) \big] - \mathbb{E} \big[ \breve{w}_j^- (s_0) \big] + 2n^{-1-\omega/2} \Big) \le 3C_1 n^{-1-c_2},
		\end{flalign*}
	
		\noindent for some constant $c_2 = c_2 (\varepsilon, \delta, B_0, m) > 0$. By a Markov estimate, we thus have for $c = \frac{c_2}{2}$ that
		\begin{flalign*}
			\mathbb{P} \big[ \mathscr{A}_1 \cap \mathscr{A}_3^{\complement} \big] \le n^{1+c_2/2} \cdot \displaystyle\sum_{|j-n'/2| \le n^{\omega/15000}} \mathbb{E} \Big[ \textbf{1}_{\mathscr{A}_1} \cdot \big| \breve{w}_j^+ (s_0) - \breve{w}_j^- (s_0) \big| \Big] 
			& \le 3C_1 n^{-1-c_2/2}.
		\end{flalign*}
	
		\noindent This gives the third bound above, $\mathbb{P} \big[ \mathscr{A}_1 \cap \mathscr{A}_3^{\complement} \big] \le C n^{-c}$, and thus the theorem. 
	\end{proof}

	\subsection{Proof of \Cref{eventa0}}
	
	\label{ProofA0}

	In this section we establish \Cref{eventa0}; throughout, we recall the notation of that proposition. We first rescale  $s_0$, $s_1$, $\breve{\bm{x}} (s)$, $\breve{\bm{w}}^- (s)$, $\breve{\bm{w}} (s)$, $\breve{\bm{w}}^+ (s)$, $G$, $G^-$, and $G^+$ to form $\widetilde{s}_0$, $\widetilde{s}_1$, $\widetilde{\bm{x}} (s)$, $\widetilde{\bm{w}}^- (s)$, $\widetilde{\bm{w}} (s)$, $\widetilde{\bm{w}}^+ (s)$, $\widetilde{G}$, $\widetilde{G}^-$, and $\widetilde{G}^+$, respectively, defined by setting 
	\begin{flalign}
		\label{syg} 
		\begin{aligned}
		\widetilde{s}_i =  \displaystyle\frac{n s_i}{n'}; \qquad \widetilde{z}_j (s) = \displaystyle\frac{n}{n'} \cdot \breve{z}_j \Big( & \displaystyle\frac{n' s}{n} \Big); \qquad \widetilde{G} (s, y) = \displaystyle\frac{n}{n'} \cdot G \Big(  \displaystyle\frac{n's}{n} + t_0 - s_0, \displaystyle\frac{n' y}{n} + y_0 - \displaystyle\frac{n'}{2n}  \Big), \\
		\widetilde{G}^{\pm} ( & y) = \displaystyle\frac{n}{n'} \cdot G^{\pm} \Big( \displaystyle\frac{n' y}{n} + y_0 - \displaystyle\frac{n'}{2n} \Big),
		\end{aligned} 
	\end{flalign}

	\noindent for any integers $i \in \{0,1\}$ and $j \in \llbracket 1, n' \rrbracket$; indices $z \in \{ x, w^-, w, w^+ \}$ and $\pm \in \{ +, - \}$; and real numbers $s \in [0, \widetilde{s}_0 + \widetilde{s}_1]$ and $y \in \big[ 1-y_0 n^{1/2}, (1-y_0) n^{1/2} \big]$. In this way, $\widetilde{\bm{w}} (s)$ is obtained by applying Dyson Brownian motion \eqref{lambdaequation} to the initial data $\widetilde{\bm{w}} (0)$ for time $s$ (without the additional rescaling described in \Cref{Comparex}); similar statements are true for $\bm{w}^- (s)$ and $\bm{w}^+ (s)$. By \eqref{a0}, and the fact that $(n')^2 = n$, we have 
	\begin{flalign}
		\label{a0g2} 
		\mathscr{A}_0 = \bigcap_{(j, s) \notin \llbracket n'/2 - n'', n'/2 + n'' \rrbracket \times (0, \tilde{s}_0 + \tilde{s}_1)} \Bigg\{ \bigg| \widetilde{x}_j (s) - \widetilde{G} \Big(s, \displaystyle\frac{j}{n'} \Big) \bigg| \le (n')^{\omega/2 - 1} \Bigg\}.
	\end{flalign} 

	\noindent Further observe since $G \in \Adm_{\varepsilon/2} (\mathfrak{R})$; since $-\frac{2}{\varepsilon} \le \partial_y G^{\pm} \le \frac{\varepsilon}{2}$; since $G(t_0-s_0, y_0) = 0$; since $\| G \|_{\mathcal{C}^{m+1} (\mathfrak{R})} \le B_0$; and since $\| G^{\pm} \|_{\mathcal{C}^{m+1}} \le 2B_0$ that, for each $\pm \in \{ +, - \}$ and $k \in \llbracket 1, m+1 \rrbracket$,
	\begin{flalign}
		\label{g0b}
		\begin{aligned} 
		& \widetilde{G} \in \Adm_{\varepsilon/2} \Big( \displaystyle\frac{n}{n'} \cdot \mathfrak{R} \Big); \qquad \qquad -\displaystyle\frac{2}{\varepsilon} \le \partial_y \widetilde{G}^{\pm} (y) \le -\displaystyle\frac{\varepsilon}{2}, \quad \text{for $y \in \big[ 1-y_0 n^{1/2}, (1 - y_0) n^{1/2} \big]$}; \\
		& \widetilde{G} \Big(0, \displaystyle\frac{1}{2} \Big) = \widetilde{G}^{\pm} \Big( 0, \displaystyle\frac{1}{2} \Big) = 0; \qquad  [\widetilde{G}]_k, [\widetilde{G}^{\pm}]_k \le 2 B_0 \Big( \displaystyle\frac{n'}{n} \Big)^{k-1},
		\end{aligned}
	\end{flalign}

	The proof of \Cref{eventa0} will make use of the following lemma, to be established in \Cref{ProofEvent0a1} below, that approximates the locations of the middle paths in $\widetilde{\bm{w}}^- (s)$ and $\widetilde{\bm{w}}^+ (s)$. 
	
	\begin{lem} 
		
		\label{wjs2}
		
		There exists a coupling between $\widetilde{\bm{w}}^-$, $\widetilde{\bm{w}}$, and $\widetilde{\bm{w}}^+$, and a constant $c = c(\varepsilon, \delta, B_0, m) > 0$ such that $\mathbb{P} [\mathscr{A}_1] \ge 1 - c^{-1} e^{-c (\log n)^2}$, where we have defined the event
		\begin{flalign}
			\label{wnw} 
			\begin{aligned}
				\mathscr{A}_1 & = \Bigg\{ \displaystyle\sup_{s \in [0, \tilde{s}_0 + \tilde{s}_1]} \bigg( \displaystyle\max_{j \in \llbracket n'/2 - 2n'', n'/2 + 2n'' \rrbracket} \big| \widetilde{w}_j^- (s) - \widetilde{w}_j^- (0) \big| \bigg) \le (n')^{2\omega/3 - 1} \Bigg\}  \\
				& \qquad \cap \Bigg\{ \displaystyle\sup_{s \in [0, \tilde{s}_0 + \tilde{s}_1]} \bigg( \displaystyle\max_{j \in \llbracket n'/2 - 2n'', n'/2 + 2n'' \rrbracket}  \big| \widetilde{w}_j^+ (s) - \widetilde{w}_j^+ (0) \big| \bigg) \le (n')^{2\omega/3 - 1} \Bigg\} \cap \mathscr{A}_0.
			\end{aligned}
		\end{flalign}
		
	\end{lem}

	Using \Cref{wjs2}, we can establish \Cref{eventa0}.

	\begin{proof}[Proof of \Cref{eventa0}]
		
		Fix $\mathfrak{a} = \partial_t G(t_0, y_0)$, which is uniformly bounded in $[-B_0, B_0]$ by \Cref{fgr}. We claim on the event $\mathscr{A}_1$ from \eqref{wnw} that 
		\begin{flalign}
			\label{wj0wjtfg}
			\begin{aligned} 
				& \breve{w}_j^- (0) \le \breve{u}_j \le \breve{w}_j^+ (0); \\
				& \breve{w}_j^- (s_0 + s_1) + (\mathfrak{a} - n^{-\omega/2}) (s_0 + s_1) \le \breve{v}_j \le \breve{w}_j^+ (s_0 + s_1) + (\mathfrak{a} + n^{-\omega/2}) (s_0 + s_1); \\
				& \breve{f}^- (s) + (\mathfrak{a} - n^{-\omega/2}) s \le \breve{f} (s) \le \breve{f}^+ (s) + (\mathfrak{a} + n^{-\omega/2}) s,  \\
				& \breve{g}^- (s) + (\mathfrak{a} - n^{-\omega/2}) s \le \breve{g} (s) \le \breve{g}^+ (s) + (\mathfrak{a} + n^{-\omega/2}) s,  
			\end{aligned} 
		\end{flalign}
		
		\noindent for each integer $j \in \big\llbracket \frac{n'}{2} - n'', \frac{n'}{2} + n'' \big\rrbracket$ and real number  $s \in [0, s_0 + s_1]$.
		
		Let us quickly establish the proposition assuming \eqref{wj0wjtfg}. To that end, denote the $(2n''+1)$-tuples $\mathring{\bm{u}} = (\breve{u}_{n'/2-n''},  \ldots , \breve{u}_{n'/2+n''}) \in \overline{\mathbb{W}}_{2n''+1}$ and $\mathring{\bm{v}} = (\breve{v}_{n'/2-n''}, \ldots , \breve{v}_{n'/2+n''}) \in \overline{\mathbb{W}}_{2n''+1}$. Similarly, for each index $\pm \in \{ +, - \}$ and real number $s \in [0, s_0 + s_1]$, denote 
		\begin{flalign*} 
			& \mathring{\bm{w}}^{\pm} (s) = \Big( \breve{w}_{n'/2-n''}^{\pm} (s) + (\mathfrak{a} \pm n^{-\omega/2}) s, \ldots , \breve{w}_{n'/2+n''}^{\pm} (s) + (\mathfrak{a} \pm n^{-\omega/2}) s \Big) \in \overline{\mathbb{W}}_{2n''+1}; \\
			& \mathring{f}^{\pm} (s) = \breve{f}^{\pm} (s) + (\mathfrak{a} \pm n^{-\omega/2}) s; \qquad \mathring{g}^{\pm} (s) = \breve{g}^{\pm} (s) + (\mathfrak{a} \pm n^{-\omega/2}) s.
		\end{flalign*} 
		
		\noindent Then, by the second part of \Cref{lambdat} and \Cref{linear}, the laws of the processes $\big( \breve{x}_j (s) \big)$ and $\big( \mathring{w}_j^{\pm} (s) \big)$ for each $\pm \in \{ +, - \}$ and $(j, s) \in \big\llbracket \frac{n'}{2} - n'', \frac{n'}{2} + n'' \big\rrbracket \times [0, s_0 + s_1]$ are given by $\mathfrak{Q}_{\breve{f}; \breve{g}}^{\mathring{\bm{u}}; \mathring{\bm{v}}} (n^{-1})$ and $\mathfrak{Q}_{\mathring{f}^{\pm}; \mathring{g}^{\pm}}^{\mathring{\bm{w}}^{\pm} (0); \mathring{\bm{w}}^{\pm} (s_0 + s_1)} (n^{-1})$, respectively. This, together with \eqref{wj0wjtfg} and \Cref{monotoneheight} yields a coupling on $\mathscr{A}_1$ that $\breve{\bm{w}}^- (s) + (\mathfrak{a} - n^{- \omega/2}) s = \mathring{\bm{w}}^- (s) \le \breve{\bm{x}} (s) \le \mathring{\bm{w}}^+ (s) = \breve{\bm{w}}^+ (s) + (\mathfrak{a} + n^{-\omega/2}) s$ for each $s \in [s_0, s_1]$, verifying the proposition.
		
		The second bound in the first statement of \eqref{wj0wjtfg} follows the fact that $\breve{\bm{w}}^{\pm} (0) = \breve{\bm{u}}^{\pm}$; \eqref{uj2}; the fact that $\mathscr{A}_1 \subseteq \mathscr{A}_0$; and the fact that on $\mathscr{A}_0$ we have (by \eqref{a0}) that $u_j \le G(t_0 - s_0, jn^{-1}) + n^{\omega/4-1} \le G^+ (jn^{-1}) = \breve{u}_j^+$ for $j \notin \big\llbracket \frac{n'}{2} - \frac{2n''}{3}, \frac{n'}{2} + \frac{2n''}{3} \big\rrbracket$. The proof of the first bound there is entirely analogous.  
		
		For the second statement of \eqref{wj0wjtfg}, we only show that $\breve{w}_j^- (s_0 + s_1) + (\mathfrak{a} - n^{-\omega/2}) (s_0 + s_1) \le \breve{v}_j$, as the proof that $\breve{v}_j \le \breve{w}_j^+ (s_0 + s_1) + (\mathfrak{a} + n^{-\omega/2}) (s_0 + s_1)$ is entirely analogous. Since we have restricted to the event $\mathscr{A}_0$ (and since $n = (n')^2$), we have that 
		\begin{flalign*} 
			\frac{n}{n'} \cdot \breve{v}_j \ge \frac{n}{n'} \cdot \bigg( G \Big( t_0 + s_1, y_0 - \frac{n'}{2n} + \frac{j}{n} \Big) - n^{\omega/4-1} \bigg) = \widetilde{G} \Big( \widetilde{s}_0 + \widetilde{s}_1, \displaystyle\frac{j}{n'} \Big) - (n')^{\omega/2-1},
		\end{flalign*} 
	
		\noindent and so it suffices to show that 
		\begin{flalign*}
			\widetilde{w}_j^- (\widetilde{s}_0 + \widetilde{s}_1) \le \widetilde{G} \Big( \widetilde{s}_0 + \widetilde{s}_1, \displaystyle\frac{j}{n'} \Big) - (\mathfrak{a} - n^{-\omega/2}) (\widetilde{s}_0 + \widetilde{s}_1) - (n')^{\omega-1}.
		\end{flalign*}
		
		\noindent Thus, since we have restricted to the event $\mathscr{A}_1$ from \eqref{wnw}; since for $j \in \big\llbracket \frac{n'}{2} - n'', \frac{n'}{2} + n'' \big\rrbracket$ we have
		\begin{flalign*} 
			\widetilde{w}_j^- (0) = \frac{n}{n'} \cdot \breve{u}_j^- \le \widetilde{G} \Big( 0, \frac{j}{n'} \Big) + (n')^{2 \omega / 3 - 1},
		\end{flalign*} 
	
		\noindent by \eqref{uj2}, \eqref{a0g2}, and the definition of $G^-$; and since $\widetilde{s}_0 = (n')^{\omega-1}$ (by \eqref{omegann} and \eqref{syg}), it suffices to show that 
		\begin{flalign*}
			\widetilde{G} \Big( 0, \frac{j}{n'} \Big) \le \widetilde{G} \Big( \widetilde{s}_0 + \widetilde{s}_1, \displaystyle\frac{j}{n'} \Big) -  \mathfrak{a} (\widetilde{s}_0 + \widetilde{s}_1) + (n')^{-\omega} \widetilde{s}_1 - 2 (n')^{\omega-1}.
		\end{flalign*}
	
		\noindent This follows from the fact that 
		\begin{flalign*} 
			\bigg| \widetilde{G} \Big( 0, \displaystyle\frac{j}{n'} \Big) - \widetilde{G} \Big( \widetilde{s}_0 + \widetilde{s}_1, \displaystyle\frac{j}{n'} \Big) - \mathfrak{a} (\widetilde{s}_0 + \widetilde{s}_1) \bigg| \le B_0 \Big(\widetilde{s}_0 + \widetilde{s}_1 + \displaystyle\frac{n''}{n'} \Big)^2 & \le 9B_0  \big( (n')^{12 \omega - 2} + (n')^{-1} \big) \\
			&  \le (n')^{-\omega} \widetilde{s}_1 - 2 (n')^{\omega - 1},
		\end{flalign*} 
	
		\noindent for sufficiently large $n$. Here, in the first inequality we applied a Taylor expansion, the fact that $\mathfrak{a} = \partial_t G(t_0, y_0)$, the fact that $\big| j - \frac{n'}{2} \big| \le n''$, and the fact that $\| G \|_{\mathcal{C}^2 (\mathfrak{R})} \le B_0$; in the second inequality, we used the facts that $\widetilde{s}_0 + \widetilde{s}_1 \le 2 \widetilde{s}_1 \le 2 (n')^{6 \omega - 1}$ and that $n'' = (n')^{1/2}$; and in the third inequality we used the facts that $\omega \le \frac{1}{20}$ (and again that $\widetilde{s}_1 = (n')^{6 \omega - 1}$). 
	
		It therefore remains to establish the last two statements of \eqref{wj0wjtfg}; among these, we only verify the bound $\breve{f}^- \le \breve{f}$ there, as the proofs of the remaining three are entirely analogous. To that end, observe on $\mathscr{A}_1$ that for each $s \in [0, s_0 + s_1]$ we have, denoting $\widetilde{s} = \frac{ns}{n'}$,
		\begin{flalign}
			\label{fsx}
			\begin{aligned} 
			\frac{n}{n'} \cdot \mathring{f}^- (s)  = \widetilde{w}_{n'/2+n''+1}^- (\widetilde{s})  + (\mathfrak{a} - n^{-\omega}) \widetilde{s} & \le \widetilde{w}_{n'/2 + n'' + 1}^- (0) + \mathfrak{a} \widetilde{s}  + (n')^{2\omega/3-1} \\
			& = \widetilde{G}^- \Big( \displaystyle\frac{1}{2} + \displaystyle\frac{n'' + 1}{n'} \Big) + \mathfrak{a} \widetilde{s} + (n')^{2\omega/3-1} \\
			& = \widetilde{G} \Big( 0, \displaystyle\frac{1}{2} + \displaystyle\frac{n'' + 1}{n'} \Big) + \mathfrak{a} \widetilde{s} - (n')^{\omega-1} + (n')^{2\omega/3-1},
			\end{aligned} 
		\end{flalign} 
		
		\noindent where in the first statement we applied the definitions of $\mathring{f}$ and $\breve{f}^-$ (as well as \eqref{syg}); in the second we used the definition of $\mathscr{A}_1$ from \eqref{wnw}; in the third we used the fact that $\breve{w}_{n'/2+n''+1}^- (0) = G^- \big( y_0 + \frac{n''+1}{n} \big)$ (by \eqref{uj2}) and the rescaling \eqref{syg}; and in the fourth we used the definition of $\widetilde{G}^-$ in terms of $\widetilde{G}$. We further have that 
		\begin{flalign*}
			 \widetilde{G} \Big( & 0, \displaystyle\frac{1}{2} + \displaystyle\frac{n'' + 1}{n'} \Big) + \mathfrak{a} \widetilde{s}  - (n')^{\omega-1} + (n')^{2\omega/3-1} \\
			 & \le \widetilde{G} \Big( \widetilde{s}, \displaystyle\frac{1}{2} + \displaystyle\frac{n'' + 1}{n'} \Big)  - (n')^{3\omega/4-1} + \displaystyle\frac{B_0 n'}{n} \Big( \widetilde{s} + \displaystyle\frac{n''+1}{n'} \Big)^2 \\
			& \le \widetilde{x}_{n'/2 + n'' + 1} (\widetilde{s}) - (n')^{3\omega/4-1} + \displaystyle\frac{2B_0 \widetilde{s}^2 n'}{n} + 2 (n')^{\omega/2-1}  \le \displaystyle\frac{n}{n'} \cdot \breve{x}_{n'/2 + n''+1} (s) = \displaystyle\frac{n}{n'} \cdot \breve{f} (s),
		\end{flalign*} 
	
		\noindent where in the first statement we applied a Taylor expansion for $\widetilde{G}$, using the facts that $\mathfrak{a} = \partial_t G(t_0, y_0) = \partial_t G^- (t_0, y_0) = \partial_t \widetilde{G}^- \big( 0, \frac{1}{2} \big)$ and that $[ \widetilde{G}^-]_2 = \frac{n'}{n} \cdot [G^-]_2 \le \frac{2B_0n'}{n}$; in the second we used the fact that we are restricting to the event $\mathscr{A}_0$, \eqref{a0g2}, and the fact that $n' n^{-1} (n'' n'^{-1})^2 = n^{-1}$; in the third we used \eqref{syg} and the fact that $\widetilde{s}^2 \le 4\widetilde{s}_1^2 = 4 (n')^{12 \omega - 2} \le (n')^{-1}$ (as $\omega \le \frac{1}{20}$); and in the fourth we used the definition of $\breve{f}$. Combining this with \eqref{fsx} yields the first inequality in the third statement of \eqref{wj0wjtfg}. As mentioned previously, the proof of the second inequality is entirely analogous, as is the proof of the fourth statement of \eqref{wj0wjtfg}; this establishes the proposition.
	\end{proof}

	\subsection{Proof of \Cref{wjs2}}
	
	\label{ProofEvent0a1}
	
	To establish \Cref{wjs2}, we set some additional notation, which will eventually more precisely prescribe the limiting behaviors of the paths in $\widetilde{\bm{w}} (s)$. To that end, let $\widetilde{H} : [0, \widetilde{s}_0 + \widetilde{s}_1] \times \mathbb{R} \rightarrow [0, 1]$ denote the height function associated with the inverted height function $\widetilde{G}$ (recall \eqref{gty}), given by $\widetilde{H} (t, x) = \sup \big\{ y \in \mathbb{R} : \widetilde{G} (t, y) \ge x \big\}$. In particular, for any $(t, x) \in [0, \widetilde{s}_0 + \widetilde{s}_1] \times [-n^{\omega}, n^{\omega}]$, we have
	\begin{flalign}
		\label{hg2}
		\widetilde{H} (t, x)  \quad \text{is the unique value of $y \in \big[ -(n')^{3 \omega/2}, (n')^{3 \omega/2} \big]$ such that $\widetilde{G} (t, y) = x$},
	\end{flalign}
	
	\noindent where this uniqueness follows from the fact that $\widetilde{G}(s, y)$ is decreasing in $y$ (by \eqref{g0b}), and we implicitly used the fact that any $(s, y) \in [0, \widetilde{s}_0 + \widetilde{s}_1] \times \big[ -(n')^{3 \omega/2}, (n')^{3\omega/2} \big]$ is in the domain of $\widetilde{G}$ since $0 \le y_0 - \frac{(n')^{1+2\omega}}{n} - \frac{n'}{2n} \le y_0 + \frac{(n')^{1+2\omega}}{n} - \frac{n'}{2n} \le 1$ (as $\min \{ y_0, 1 - y_0 \} \ge \dist \big( (t_0, y_0), \partial \mathfrak{R} \big) \ge n^{-\delta}$, as stipulated in \Cref{xkernel}). Next, define the density $\widetilde{\varrho}_0 \in L^1 (\mathbb{R})$ by setting 
	\begin{flalign}
		\label{rhoh} 
		\begin{aligned}
		& \widetilde{\varrho}_0 (x) = -\partial_x \widetilde{H} (0, x), \qquad \text{if $\widetilde{G} (0, 1) \le x \le \widetilde{G} (0, 0)$}; \\ 
		& \widetilde{\varrho}_0 (x) = 0, \qquad \qquad \qquad \quad \text{if $x < \widetilde{G} (0, 1)$ or if $x > \widetilde{G} (0, 0)$.}   
		\end{aligned} 
	\end{flalign}
	
	 Observe in particular that $\int_{-\infty}^{\infty} \widetilde{\varrho}_0 (x) dx = \widetilde{H} \big( 0, \widetilde{G} (0, 1) \big) - \widetilde{H} \big( 0, \widetilde{G}(0, 0) \big) = 1$. Denoting the measure $\widetilde{\mu}_0 = \widetilde{\varrho}_0 (x) dx$, we then find that $\widetilde{\mu}_0 \in \mathscr{P}_0$, that is, it is a probability measure. Denote its free convolution with the rescaled semicircle law (recall \Cref{TransformConvolution}) by $\widetilde{\mu}_t = \widetilde{\mu}_0 \boxplus \mu_{\semci}^{(t)}$ for any $t > 0$; also denote the associated density by $\widetilde{\varrho}_t \in L^1 (\mathbb{R})$, which satisfies $\widetilde{\mu}_t (dx) = \widetilde{\varrho}_t (x) dx$. Denote the classical locations (recall \Cref{gammarho} with $n$ replaced by $n'$) with respect to $\widetilde{\mu}_t$ by $\widetilde{\gamma}_j (t) = \gamma_{j; n'}^{\tilde{\mu}_t}$ for each $j \in \llbracket 1, n' \rrbracket$.
	
	We define the analogous quantities associated with $\widetilde{G}^-$ and $\widetilde{G}^+$ very similarly. In particular, for each index $\pm \in \{ +, - \}$, we define $\widetilde{H}^{\pm} (x) = \sup \big\{ y \in \mathbb{R} : \widetilde{G}^{\pm} (y) \ge x \big\}$, so that for any $x \in [-n^{\omega}, n^{\omega}]$ we have that $\widetilde{H}^{\pm} (x)$ is the unique value of $y \in \big[ -(n')^{3\omega/2}, (n')^{3\omega/2} \big]$ such that $\widetilde{G}^{\pm} (y) = x$. We further define the density $\widetilde{\varrho}_0^{\pm} \in L^1 (\mathbb{R})$ by setting $\widetilde{\varrho}_0^{\pm} (x) = -\partial_x \widetilde{H}^{\pm} (x)$ if $\widetilde{G}^{\pm} (1) \le x \le \widetilde{G}^{\pm} (0)$, and $\widetilde{\varrho}_0^{\pm} (x) = 0$ otherwise. Then, $\widetilde{\mu}_0^{\pm} = \widetilde{\varrho}_0^{\pm} (x) dx$ is a probability measure. Denote its free convolution with the semicircle law by $\widetilde{\mu}_t^{\pm} = \widetilde{\mu}_0^{\pm} \boxplus \mu_{\semci}^{(t)}$ and the associated density by $\widetilde{\varrho}_t^{\pm} \in L^1 (\mathbb{R})$. Also denote the classical locations with respect to $\widetilde{\mu}_t^{\pm}$ by $\widetilde{\gamma}_j^{\pm} (t) = \gamma_{j;n'}^{\tilde{\mu}_t^{\pm}}$.
	
	The next lemma provides bounds on this density $\widetilde{\varrho}_t$, as well as on the classical locations $\widetilde{\gamma}_j (t)$ (and their analogs for $\widetilde{\varrho}_t^{\pm}$ and $\widetilde{\gamma}_j^{\pm} (t)$); it is established in \Cref{ProofEstimate0Lambda} below. 
	
	\begin{lem} 
		
	\label{hestimate0} 
	
	 There exists a constant $C = C(\varepsilon, \delta, B_0, m) > 1$ such that the following hold for each integer $n \ge C$ and index $\pm \in \{ +, - \}$. 
	
	\begin{enumerate} 
		\item If $(t, x) \in [0, n^{-\omega}] \times [-C^{-1}, C^{-1}]$, then $-4 \varepsilon^{-1} \le \widetilde{\varrho}_t (x) \le - \frac{\varepsilon}{4}$ and $-4\varepsilon^{-1} \le \widetilde{\varrho}_t^{\pm} (x) \le -\frac{\varepsilon}{4}$. 
		\item If $i \in \big\llbracket - (n')^{1-\omega}, (n')^{1-\omega} \big\rrbracket$ and $t \in [0, n^{-\omega}]$, then $\big| \widetilde{\gamma}_{n'/2 + i} (t) - \widetilde{\gamma}_{n'/2+i} (0) \big| \le Ct \big( \frac{n'}{n} + \frac{|i|}{n'} + t \big)$ and $\big| \widetilde{\gamma}_{n'/2+i}^{\pm} (t) - \widetilde{\gamma}_{n'/2 + i}^{\pm} (0) \big| \le Ct \big( \frac{n'}{n} + \frac{|i|}{n'} + t\big)$.
		\end{enumerate} 
	\end{lem}

	We also require the below lemma indicating a monotone coupling for Dyson Brownian motion, established in \Cref{ProofEstimate0Lambda} below.
	
	\begin{lem}
		
		\label{lambdamonotone} 
		
		Let $M \ge 1$ and $k \ge 2$ be integers, and fix $k$ $M$-tuples  $\bm{\lambda}^{(1)}, \bm{\lambda}^{(2)}, \ldots , \bm{\lambda}^{(k)} \in \overline{\mathbb{W}}_M$ with $\bm{\lambda}^{(1)} \le \bm{\lambda}^{(2)} \le \cdots \le \bm{\lambda}^{(k)}$. For each integer $j \in \llbracket 1, k \rrbracket$, let $\bm{\lambda}^{(j)} (s) = \big( \lambda_1^{(j)} (s), \lambda_2^{(j)} (s), \ldots , \lambda_M^{(j)} (s) \big)$ denote Dyson Brownian motion run for time $s$ with initial data $\bm{\lambda}^{(j)}$. Then it is possible to couple the $\big\{ \bm{\lambda}^{(j)} (s) \big\}$ such that $\bm{\lambda}^{(1)} (s) \le \bm{\lambda}^{(2)} (s) \le \cdots \le \bm{\lambda}^{(k)} (s)$, for each real number $s \ge 0$. 
	
	\end{lem}

	Now we can establish \Cref{wjs2}.
	
	\begin{proof}[Proof of \Cref{wjs2}]
		
		By \eqref{a0probability}, we may restrict to the event $\mathscr{A}_0$ in what follows. Define the two $n'$-tuples $\bm{\lambda}^- = (\lambda_1^-, \lambda_2^-, \ldots , \lambda_{n'}^-) \in \overline{\mathbb{W}}_{n'}$ and $\bm{\lambda}^+ = (\lambda_1^+, \lambda_2^+, \ldots , \lambda_{n'}^+) \in \overline{\mathbb{W}}_{n'}$ by setting 
		\begin{flalign*} 
			\lambda_j^{\pm} = \widetilde{G}^{\pm} \Big( \frac{2j-1}{n'} \Big), \qquad \text{for each index $\pm \in \{ +, - \}$ and integer $j \in \llbracket 1, n' \rrbracket$}.
		\end{flalign*}
	
		\noindent For each index $\pm \in \{ +, - \}$, let $\bm{\lambda}^{\pm} (s) = \big( \lambda_1^{\pm} (s), \lambda_2^{\pm} (s), \ldots , \lambda_{n'}^{\pm} (s) \big)$ denote Dyson Brownian motion, run for time $s$, with initial data $\bm{\lambda}^{\pm}$. Since we have restricted to the event $\mathscr{A}_0$, we have by \eqref{a0g2}; \eqref{uj2} (with the definitions of $G^-$ and $G^+$); and the fact from \eqref{g0b} that $\partial_y \widetilde{G}^{\pm} \le -\frac{\varepsilon}{2}$, that $\bm{\lambda}^- - 2(n')^{3\omega/5-1} \le \widetilde{\bm{w}}^- (0) \le \bm{\lambda}^- + 2 (n')^{3\omega/5 - 1}$ and $\bm{\lambda}^+ - 2(n')^{3\omega/5 - 1} \le \widetilde{\bm{w}}^+ (0) \le \bm{\lambda}^+ + 2 (n')^{3 \omega / 5 - 1}$. Thus, by \Cref{lambdamonotone}, it is possible to couple the processes $\big(\bm{\lambda}^- (s), \bm{\widetilde{w}}^- (s) \big)$ such that $\bm{\lambda}^- (t) - 2 (n')^{3 \omega/5 - 1} \le \widetilde{\bm{w}}^- (t) \le \bm{\lambda}^+ (t) + 2 (n')^{3 \omega/5-1}$ for each $t \ge 0$; similarly, it is possible to couple $\big( \bm{\lambda}^+ (s), \widetilde{\bm{w}}^+ (s) \big)$ such that $\bm{\lambda}^+ (t) - 2 (n')^{3 \omega/5 - 1} \le \widetilde{\bm{w}}^+ (t) \le \bm{\lambda}^- (t) + 2 (n')^{3 \omega/5-1}$ for each $t \ge 0$. Hence, it suffices to show 
		\begin{flalign}
			\label{a0a12} 
			\begin{aligned} 
			\mathbb{P} \Bigg[ \mathscr{A}_0 \cap \bigcap_{\substack{s \in [0, \tilde{s}_0 + \tilde{s}_1] \\ j \in \llbracket n'/2 - 2n'', n'/2+2n''\rrbracket}} \Big\{ \big| \lambda_j^- (s) - \lambda_j^- (0) \big| \le (n')^{\omega/2-1} \Big\} \Bigg] \ge 1 - c^{-1} e^{-c (\log n)^2}; \\
			\mathbb{P} \Bigg[ \mathscr{A}_0 \cap \bigcap_{\substack{s \in [0, \tilde{s}_0 + \tilde{s}_1] \\ j \in \llbracket n'/2-2n'', n'/2 + 2n''\rrbracket}} \Big\{ \big| \lambda_j^+ (s) - \lambda_j^+ (0) \big| \le (n')^{ \omega/2 - 1} \Big\} \Bigg] \ge 1 - c^{-1} e^{-c (\log n)^2}.
			\end{aligned} 
		\end{flalign}

	 	We only show the first bound in \eqref{a0a12}, as the proof of the second is entirely analogous. Under this notation, \Cref{concentrationequation} applies with the $(n, \bm{\lambda})$ there given by $(n', \bm{\lambda}^-)$ here and the $\mu$ there given by $\widetilde{\mu}_0^-$ here. It therefore yields a constant $c_1 > 0$ such that 
		\begin{flalign*}
			\mathbb{P} [\mathscr{E}_1^-] \ge 1 - c_1^{-1} e^{-c_1 (\log n)^2},
		\end{flalign*} 
	
		\noindent where we have defined the event 
		\begin{flalign*} 
			 \mathscr{E}_1^- =  \bigcap_{j=1}^{n'} \bigcap_{t \in [0, 1]} \bigg\{ \widetilde{\gamma}_{j + \lfloor (\log n)^5 \rfloor}^- (t) - n^{-1} \le \lambda_j^- (t) \le \widetilde{\gamma}_{j- \lfloor (\log n)^5 \rfloor}^- (t) + n^{-1} \bigg\}. 
		\end{flalign*}

		\noindent Observe on $\mathscr{E}_1^-$ that 
		\begin{flalign*}
			\big| \lambda_j^- (s) - \lambda_j^- (0) \big| \le \widetilde{\gamma}_{j-\lfloor (\log n)^5 \rfloor}^- (s) - \widetilde{\gamma}_{j + \lfloor (\log n)^5 \rfloor}^- (0) + 2n^{-1}.
		\end{flalign*}
	
		\noindent This yields on $\mathscr{E}_1^-$ a constant $C_1 = C_1 (\varepsilon, \delta, B_0, m) > 1$ such that, for each pair $(j, s) \in \big\llbracket \frac{n'}{2} - 2n'', \frac{n'}{2} + 2n'' \big\rrbracket \times [0, n^{-\omega}]$, we have
		\begin{flalign*}
			\big| \lambda_j^- (s) - \lambda_j^-  (0) \big| & \le C_1 s \Big( \displaystyle\frac{n'}{n} + \displaystyle\frac{n''}{n'} + s \Big) + \widetilde{\gamma}_{j- \lfloor (\log n)^5 \rfloor}^- (0) - \widetilde{\gamma}_{j + \lfloor (\log n)^5 \rfloor}^- (0) + \displaystyle\frac{2}{n} \\
			& \le 2 C_1 s \Big( \displaystyle\frac{1}{n''} + s \Big) + \displaystyle\frac{4 (\log n)^5}{\varepsilon n'} + \displaystyle\frac{2}{n},
		\end{flalign*}
	
		\noindent where in the first inequality we used the second part of \Cref{hestimate0}; in the second inequality we used \eqref{omegann} and the fact that $\widetilde{\varrho}_0^- = - \partial_x \widetilde{G}^- \ge \frac{\varepsilon}{2}$. Together with the fact that $s \big( \frac{1}{n''} + s \big) \le \frac{1}{n'}$ for sufficiently large $n$ (since $n'' = (n')^{1/2}$, $s \le \widetilde{s}_0 + \widetilde{s}_1 \le 2 (n')^{6\omega-1}$, and $\omega \le \frac{1}{20}$), this implies the first statement of \eqref{a0a12} and thus the lemma.		
	\end{proof}

	\subsection{Proofs of \Cref{hestimate0} and \Cref{lambdamonotone}}
	
	\label{ProofEstimate0Lambda} 
	
	In this section we establish first \Cref{hestimate0} and then \Cref{lambdamonotone}.
	
	\begin{proof}[Proof of \Cref{hestimate0}]
		
		We only establish the estimates in the lemma for $\widetilde{\varrho}_t$ and $\widetilde{\gamma}_j (t)$, as the proofs of the analogous bounds for $\widetilde{\varrho}_t^{\pm}$ and $\widetilde{\gamma}_j^{\pm} (t)$ are entirely analogous. We begin with the first statement of the lemma. To that end observe that, since $\widetilde{G} \in \Adm_{\varepsilon/2} \big( \frac{n}{n'} \cdot \mathfrak{R} \big)$ (by \eqref{g0b}), we have
		\begin{flalign} 
			\label{rho0omega2} 
			-\displaystyle\frac{2}{\varepsilon} \le \widetilde{\varrho}_0 (x) \le -\displaystyle\frac{\varepsilon}{2}, \qquad \text{for each $x \in \big[ \widetilde{G} (0, 1), \widetilde{G} (0, 0) \big]$}.
		\end{flalign} 
		
		\noindent Again by \eqref{g0b}, we also have that $[\widetilde{G}]_k \le \frac{2 B_0 n'}{n}$ for each integer $k \in \llbracket 2, m \rrbracket$, from which it quickly follows from \eqref{hg2}, \eqref{rhoh}, and \eqref{rho0omega2} that there exists a constant $C_1 = C_1 (\varepsilon, \delta, B_0, m) > 1$ with
		\begin{flalign}
			\label{kderivativerho0}
			[\widetilde{\varrho}_0]_k \le \displaystyle\frac{C_1 n'}{n}, \qquad \text{on $\supp \widetilde{\varrho}_0$, for each $k \in \llbracket 1, m-1 \rrbracket$}.
		\end{flalign}
		
		\noindent As such, \Cref{derivativetm} applies to $\widetilde{\varrho}_0$ around any point in the interior of $\supp \widetilde{\varrho}_0$ and yields (since $\widetilde{G} \big( 0, \frac{1}{2} \big) = 0$ and $\partial \widetilde{G} \le -\frac{\varepsilon}{2}$, by \eqref{g0b}) a constant $C_2 = C_2 (\varepsilon, \delta, B_0, m) > 1$ such that 
		\begin{flalign}
			\label{tderivativerho} 
			\displaystyle\sup_{t \in [0, 1/C_2]} \displaystyle\sup_{|x| \le 1/C_2} \big| \partial_t \widetilde{\varrho}_t (x) \big| \le C_2; \qquad \displaystyle\sup_{t \in [0, 1/C_2]} \displaystyle\sup_{|x|\le 1/C_2} \displaystyle\max_{j \in \llbracket 1, m-3 \rrbracket} \big| \partial_t \partial_x^j \widetilde{\varrho}_0 (x) \big| \le C_2.
		\end{flalign}
		
		\noindent The first bound in \eqref{tderivativerho}, with the first two statements of \eqref{g0b}, yields the first statement of the lemma for sufficiently large $n$ (so that $4 C_2 n^{-\omega} \le \varepsilon$).
		
		To establish the second part of the lemma, first observe for $i \in \big\llbracket - (n')^{1-\omega}, (n')^{1-\omega} \big\rrbracket$ that
		\begin{flalign}
			\label{gammaj0}
			\big| \widetilde{\gamma}_{n'/2 + i} (0) \big| = \bigg| \widetilde{G} \Big(0, \displaystyle\frac{1}{2} + \frac{i}{n'} \Big) \bigg| \in \bigg[ -\displaystyle\frac{2 i}{\varepsilon n'}, \displaystyle\frac{2 i}{\varepsilon n'} \bigg] \subseteq \bigg[ - \displaystyle\frac{2}{\varepsilon (n')^{\omega}}, \displaystyle\frac{2}{\varepsilon (n')^{\omega}} \bigg],
		\end{flalign}
	
		\noindent due to the first and third statements of \eqref{g0b}. We then claim that there exists a constant $C_3 = C_3 (\varepsilon, \delta, B_0, m) > 1$ such that (recalling the Hilbert transform from \eqref{transform2})
		\begin{flalign}
			\label{transformrho} 
			\big| H \widetilde{\varrho}_s (y) \big| \le C_3 \Big( \displaystyle\frac{n'}{n} + |y| + s \Big), \qquad \text{for each $(s, y) \in [0, n^{-\omega}] \times \big[- (n')^{-\omega/2}, (n')^{-\omega/2} \big]$}.
		\end{flalign}
		
		\noindent Before verifying \eqref{transformrho} let us confirm the second part of the lemma assuming it. To that end, observe for any $j \in \llbracket 1, n' \rrbracket$ that 
		\begin{flalign*}
			\displaystyle\int_{\tilde{\gamma}_j (0)}^{\infty} \widetilde{\varrho}_t (x) dx - \displaystyle\frac{2j-1}{n'} = \displaystyle\int_{\tilde{\gamma}_j (0)}^{\infty} \widetilde{\varrho}_t (x) dx - \displaystyle\int_{\tilde{\gamma}_j (0)}^{\infty} \widetilde{\varrho}_0 (x) dx & = \displaystyle\int_{\tilde{\gamma}_j (0)}^{\infty} \displaystyle\int_0^t \partial_s \widetilde{\varrho}_s (x) ds dx \\
			& = \pi \displaystyle\int_0^t \displaystyle\int_{\tilde{\gamma}_j (0)}^{\infty} \partial_x \big( \widetilde{\varrho}_s (x) H \widetilde{\varrho}_s (x) \big) dx ds \\
			& = \pi \displaystyle\int_0^t \widetilde{\varrho}_s \big( \widetilde{\gamma}_j (0)\big) \cdot H \widetilde{\varrho}_s \big( \widetilde{\gamma}_j (0) \big) ds,
		\end{flalign*} 
		
		\noindent where the first statement follows from the definition of the classical location $\widetilde{\gamma}_j (0)$; the second from performing the integration in $s$; the third from integrating \eqref{trhoty}; and the fourth from performing the integration in $x$ (using the fact that $\widetilde{\varrho}_s (x) = 0$ for $x$ sufficiently large). This with $j = \frac{n'}{2} + i$ for $i \in \big\llbracket -(n')^{1-\omega}, (n')^{1-\omega} \big\rrbracket$, together with \eqref{transformrho}, \eqref{gammaj0}, and the fact (from the first part of the lemma) that $\big| \widetilde{\varrho}_s (x) \big| \le 4 \varepsilon^{-1}$ for each $(s, x) \in [0, n^{-\omega}] \times [-n^{-\omega/8}, n^{-\omega/8}]$, yields (since $|\gamma_j| \le \frac{3|i|}{\varepsilon n}$, by the second and third statements in \eqref{g0b})
		\begin{flalign*}
			\displaystyle\frac{1}{2} + \displaystyle\frac{i}{n'} - \displaystyle\frac{8 \pi C_3 t}{\varepsilon} \Big( \displaystyle\frac{n'}{n} + \displaystyle\frac{3 |i|}{\varepsilon n'} + t \Big) \le \displaystyle\int_{\tilde{\gamma}_{n'/2 + i} (0)}^{\infty} \widetilde{\varrho}_t (x) dx \le \displaystyle\frac{1}{2} + \displaystyle\frac{i}{n'} + \displaystyle\frac{8 \pi C_3 t}{\varepsilon} \Big( \displaystyle\frac{n'}{ n} + \displaystyle\frac{3 |i|}{\varepsilon n'} + t \Big).
		\end{flalign*} 
		
		\noindent Together with the facts that $\int_{\tilde{\gamma}_{n'/2+i} (t)}^{\infty} \widetilde{\varrho}_t (x) dx = \frac{1}{2} + \frac{2i-1}{n'}$ for each $j \in \llbracket 1, n' \rrbracket$ and that $\widetilde{\varrho}_t (x) \ge \frac{\varepsilon}{2}$ for $|x| \le n^{-\omega/8}$, it follows that 
		\begin{flalign*}
			\big| \widetilde{\gamma}_{n'/2 + i} (t) - \widetilde{\gamma}_{n'/2 + i} (0) \big| \le \displaystyle\frac{30 \pi C_3 t}{\varepsilon^3} \Big( \displaystyle\frac{n'}{n} + \displaystyle\frac{|i|}{n'} + t \Big),
		\end{flalign*}
		
		\noindent which verifies the second part of the lemma.
		
		It then remains to establish \eqref{transformrho}, to which end we will bound $\big| H \widetilde{\varrho}_0 (y) \big|$ and $\big| \partial_t H \widetilde{\varrho}_s (y) \big|$. First observe that $\supp \widetilde{\mu}_0 = \big[ \widetilde{G} (0, 1), \widetilde{G} (0, 0) \big]$ and that
		\begin{flalign}
			\label{g00g01}
			& -2 \varepsilon^{-1} \le \widetilde{G} (0, 1) \le -\displaystyle\frac{\varepsilon}{4} \le \displaystyle\frac{\varepsilon}{4} \le \widetilde{G} (0, 0) \le 2 \varepsilon^{-1}; \qquad \big| \widetilde{G} (0,0) + \widetilde{G} (0, 1) \big| \le \displaystyle\frac{4 B_0 n'}{n},
		\end{flalign}
		
		\noindent where the first statement follows from the first and third statements of \eqref{g0b}, and the second follows from the first and last statement of \eqref{g0b}. We next have that 
		\begin{flalign}
			\label{rho01} 
			\big| H \widetilde{\varrho}_0 (y) \big| = \pi^{-1} \Bigg| \displaystyle\int_{-\infty}^{\infty} \displaystyle\frac{ \widetilde{\varrho}_0 (w) dw}{w-y} \Bigg| = \pi^{-1} \Bigg| \displaystyle\int_{-\infty}^{\infty} \displaystyle\frac{\big( \widetilde{\varrho}_0 (w) - \widetilde{\varrho}_0 (y) \big) dw}{w-y} \Bigg|.
		\end{flalign}
		
		\noindent Thus, for $y \in [-n^{-\omega/8}, n^{-\omega/8}]$, we have
		\begin{flalign}
			\label{rho00} 
			\begin{aligned}
			\big| H \widetilde{\varrho}_0 (y) \big| & \le \pi^{-1} \Bigg| \displaystyle\int_{\tilde{G} (0, 1)}^{\tilde{G} (0, 0)} \displaystyle\frac{\big( \widetilde{\varrho}_0 (w) - \widetilde{\varrho}_0 (y) \big) dw}{w-y} \Bigg| + \pi^{-1} \widetilde{\varrho}_0 (y)  \Bigg| \displaystyle\int_{-\infty}^{\tilde{G}(0, 1)} \displaystyle\frac{dw}{w-y} + \displaystyle\int_{\tilde{G} (0, 0)}^{\infty} \displaystyle\frac{dw}{w - y} \Bigg| \\
			& \le \displaystyle\frac{C_1 n'}{n} \big| \widetilde{G} (0, 0) - \widetilde{G} (0, 1) \big| + \varepsilon^{-1} \Big( \big| \widetilde{G} (0, 0) + \widetilde{G} (0, 1) \big| + 2|y| \Big) \cdot \displaystyle\max_{w \notin [\tilde{G} (0, 1), \tilde{G} (0, 0)]} |w-y|^{-2} \\
			& \le \displaystyle\frac{4 C_1 n'}{\varepsilon n} + \displaystyle\frac{256B_0}{\varepsilon^3} \Big( \displaystyle\frac{n'}{n} + |y| \Big),
			\end{aligned} 
		\end{flalign}
		
		\noindent where the first bound follows from \eqref{rho01} and the  fact that $\supp \widetilde{\varrho}_0 = \big[ \widetilde{G} (0, 1), \widetilde{G} (0, 0) \big]$; the second follows from \eqref{kderivativerho0} and \eqref{rho0omega2}; and the third follows from \eqref{g00g01} (using the fact that $|w-y| \ge \frac{\varepsilon}{8}$ for $|y| \le n^{-\omega}$ and $w \notin \big[ \widetilde{G} (0, 1), \widetilde{G} (0, 0) \big]$, due to the first statement of \eqref{g00g01}).
		
		Next, we claim that there exists a constant $C_4 = C_4 (\varepsilon, \delta, B_0, m) > 1$ such that 
		\begin{flalign}
			\label{rho0s2} 
			\displaystyle\sup_{s \in [0, n^{-\omega}]} \displaystyle\sup_{|y| \le (n')^{-\omega/2}}	\big| \partial_t H \widetilde{\varrho}_s (y) \big| \le C_4.
		\end{flalign}
	
		\noindent Together with \eqref{rho00}, this would yield \eqref{transformrho} and thus the lemma. To that end, observe that 
		\begin{flalign}
			\label{rho03} 
			\big| \partial_t H \widetilde{\varrho}_s (y) \big| = \pi^{-1} \Bigg| \displaystyle\int_{-\infty}^{\infty} \displaystyle\frac{\partial_t \widetilde{\varrho}_s (w) dw}{w-y} \Bigg| = \Bigg| \displaystyle\int_{-\infty}^{\infty} \displaystyle\frac{\partial_x \big( \widetilde{\varrho}_s (w) H \widetilde{\varrho}_s (w) \big) dw}{w-y}\Bigg|,
		\end{flalign}
	
		\noindent where in the second equality we used \eqref{trhoty}. Further observe for $s \in [0, n^{-\omega}]$ that 
		\begin{flalign}
			\label{rho04}
			\supp \widetilde{\varrho}_s \subseteq \supp \widetilde{\mu}_0 + \supp \widetilde{\mu}_{\semci}^{(s)} \subseteq \big[ \widetilde{G} (0, 1) - 2t^{1/2}, \widetilde{G} (0, 0) + 2t^{1/2} \big] \subseteq [-4\varepsilon^{-1}, 4\varepsilon^{-1}],
		\end{flalign}
	
		\noindent where the first holds since $\widetilde{\mu}_s = \widetilde{\mu}_0 \boxplus \mu_{\semci}^{(s)}$, and the second holds by \eqref{rhosct}, \eqref{rho1}, and the fact that $\supp \widetilde{\mu}_0 = \big[ \widetilde{G} (0, 1), \widetilde{G} (0, 0) \big]$; and the third holds by \eqref{g00g01}. Thus, due to \eqref{rho03}, \eqref{rho04}, \eqref{rho0omega2}, and \eqref{kderivativerho0}, to verify \eqref{rho0s2} it suffices to show the existence of a constant $C_5 = C_5 (\varepsilon, \delta, B_0, m) > 1$ such that 
		\begin{flalign}
			\label{rho05} 
			\displaystyle\sup_{s \in [0, n^{-\omega}]} \displaystyle\sup_{|y| \le 1/C_5} \big| \partial_y^k H \widetilde{\varrho}_s (y) \big| \le C_5; \qquad \text{for each $k \in \{ 0, 1, 2 \}$}. 
		\end{flalign}
	
		To confirm \eqref{rho05}, fix $k \in \{ 0, 1, 2 \}$, and observe that
 		\begin{flalign*}
 			\big| \partial_y^k H \widetilde{\varrho}_s (y) \big| = \pi^{-1} \Bigg| \partial_y^k \displaystyle\int_{-\infty}^{\infty} \displaystyle\frac{\widetilde{\varrho}_s (w) dw}{w-y}  \Bigg| & \le \Bigg| \displaystyle\int_{-1/C_2}^{1/C_2} \displaystyle\frac{\partial_y^k \widetilde{\varrho}_s (y+w) dw}{w} \Bigg| + \Bigg| \displaystyle\int_{|w-y| > 1/C_2} \displaystyle\frac{\widetilde{\varrho}_s (w) dw}{(w-y)^{k+1}} \Bigg|  \\
 			& \le \Bigg| \displaystyle\int_{-1/C_2}^{1/C_2} \displaystyle\frac{\partial_y^k \big( \widetilde{\varrho}_s (y+w) - \widetilde{\varrho}_s (y) \big) dw}{w} \Bigg| + (2C_2)^3 \le 10 C_2^3,
 		\end{flalign*}
 	
 		\noindent where in the first statement we used \eqref{transform2}; in the second we changed variables in the first integral by replacing $w$ with $y+w$; in the third we used the facts that $k \le 2$ and that $\widetilde{\varrho}_t$ is a probability measure; and in the fourth we used \eqref{kderivativerho0} and the second bound in \eqref{tderivativerho}. This verifies \eqref{rho05}, which as mentioned above establishes the lemma.
	\end{proof}

	\begin{proof}[Proof of \Cref{lambdamonotone}]
		
		Let $\bm{G} (s) = \bm{G}_M (s)$ denote an $M \times M$ Hermitian Brownian motion as described in \Cref{LambdaEquation}, and for every $j \in \llbracket 1, k \rrbracket$ let $\Lambda^{(j)}$ denote the $M \times M$ diagonal matrix whose $(i, i)$ diagonal entry is given by $\lambda_i^{(j)} (0)$, for each $i \in \llbracket 1, M \rrbracket$. Since $\bm{\lambda}^{(1)} \le \bm{\lambda}^{(2)} \le \cdots \le \bm{\lambda}^{(k)}$, we have $\Lambda^{(1)} \le \Lambda^{(2)} \le \cdots \le \Lambda^{(k)}$, where we write $\Lambda \le \Lambda'$ if $\Lambda' - \Lambda$ is nonnegative definite; in particular, $\Lambda^{(1)} + \bm{G} (s) \le \Lambda^{(2)} + \bm{G} (s) \le \cdots \le \Lambda^{(k)} + \bm{G} (s)$. Together with the fact that (by the first part of \Cref{lambdat}) $\bm{\lambda}^{(j)} (s)$ has the same law as the eigenvalues of $\Lambda^{(j)} + \bm{G} (s)$ for each $j$, jointly in $s \ge 0$, this yields $\bm{\lambda}^{(1)} (s) \le \bm{\lambda}^{(2)} \le \cdots \le \bm{\lambda}^{(k)} (s)$ for each $s \ge 0$.   
	\end{proof}

\section{Regularity Properties of Solutions} 

\label{DerivativesEquation} 

In this section we state and establish various regularity properties of solutions to the equation \eqref{equationxtd} governing the limit shape for non-intersecting Brownian bridges. These include comparison principles (\Cref{maximumboundary} below); derivative bounds (\Cref{derivativef} and \Cref{perturbationbdk} below); and exponential decay of discrepancy estimates (\Cref{equationcompareboundary} below). We also provide a result that explains how perturbing regular solutions to \eqref{equationxtd} preserves their regularity (\Cref{f1f2b} below), thereby showing that \Cref{fgr} is in a sense stable under smooth perturbations of boundary data.

\subsection{Regularity Estimates}

\label{EstimatesEquation0}

In this section we state estimates for solutions to \eqref{equationxtd}. In what follows, we recall the sets of functions $\Adm (\mathfrak{R})$ and $\Adm_{\varepsilon} (\mathfrak{R})$ from \Cref{ConcentrationSmooth0}. For any real number $r > 0$ and point $z \in \mathbb{R}^2$, we recall that $\mathcal{B}_r (z) = \big\{ z' \in \mathbb{R} : |z' - z| < r \big\}$ denotes the open disk of radius $r$ centered at $z$; if $z = (0, 0)$, we abbreviate $\mathcal{B}_r (z) = \mathcal{B}_r$. We begin with the below maximum and comparison principles for solutions of \eqref{equationxtd}, which are quick consequences of the analogous principles for solutions to nonlinear elliptic partial differential equations; it is established in \Cref{EstimateEquation} below.

\begin{lem} 
	
	\label{maximumboundary}
	
	Fix an open set $\mathfrak{R} \subset \mathbb{R}$, and let $F_1, F_2, F \in \Adm (\mathfrak{R}) \cap \mathcal{C}^2 (\mathfrak{R})$ be solutions to \eqref{equationxtd} on $\mathfrak{R}$. 
	\begin{enumerate} 
		\item If $F_1 (z) \le F_2 (z)$ for each $z \in \partial \mathfrak{R}$, then $F_1 (z) \le F_2 (z)$ for each $z \in \mathfrak{R}$. 
		\item We have $\sup_{z \in \mathfrak{R}} \big| F_1 (z) - F_2 (z) \big| \le \sup_{z \in \partial \mathfrak{R}} \big| F_1 (z) - F_2 (z) \big|$. In particular, $\sup_{z \in \mathfrak{R}} \big| F (z) \big| = \sup_{z \in \partial \mathfrak{R}} \big| F (z) \big|$.
	\end{enumerate} 
\end{lem}

We next have the following lemma indicating boundedness of the derivatives for a solution to \eqref{equationxtd}; it is established in \Cref{EstimateEquation} below, again as a consequence of its analog in the context of elliptic partial differential equations. In what follows, $\mathfrak{R}' \subseteq \mathfrak{R}$ is a subset that accounts for boundary regularity, in that $F$ is regular in the interior of $\mathfrak{R}$ but up to the boundary of $\mathfrak{R}'$; in particular, if $\mathfrak{R}'$ is empty, the below lemma is an interior regularity estimate and, if $\mathfrak{R}' = \mathfrak{R}$, then it is a global one. Here, for any sets $\mathcal{S}$ and $\mathcal{S}'$, we let $\mathcal{S} \setminus \mathcal{S}' = \mathcal{S} \setminus (\mathcal{S} \cap \mathcal{S}')$. See the left side of \Cref{f:setting_1}.

\begin{lem}
	
	\label{derivativef} 
	
	For any integer $m \ge 1$; real numbers $r > 0$, $\varepsilon \in (0, 1)$, and $B > 1$; and bounded open subsets $\mathfrak{R}' \subseteq \mathfrak{R} \subset \mathbb{R}^2$ such that $\mathfrak{R}'$ has a smooth boundary, there exists a constant $C = C(\varepsilon, r, B, m, \mathfrak{R}') > 1$ such that the following holds. Let $f \in \mathcal{C}(\partial \mathfrak{R})$ be a function satisfying $\| f\|_0 \le B$, and let $F \in \Adm_{\varepsilon} (\mathfrak{R}) \cap \mathcal{C}^2 (\mathfrak{R})$ be a solution to \eqref{equationxtd} on $\mathfrak{R}$ with boundary data $F |_{\partial \mathfrak{R}} = f$. Assume that there exists a function $\varphi \in \mathcal{C}^{m+1} (\overline{\mathfrak{R}}')$ with $\| \varphi \|_{\mathcal{C}^{m+1} (\mathfrak{R}')}\le B$, such that $f(z) = \varphi (z)$ for each $z \in \partial \mathfrak{R}' \cap \partial \mathfrak{R}$. Letting 
	\begin{flalign} 
		\label{dr} 
		\mathfrak{D}_r = \big\{ z \in \mathfrak{R} : \dist (z, \partial \mathfrak{R} \setminus \partial \mathfrak{R}') > r \big\},
	\end{flalign} 

	\noindent we have $\| F \|_{\mathcal{C}^m (\mathfrak{D}_r)} \le C$. 
\end{lem}

	The following lemma states that the solutions of \eqref{equationxtd} are real analytic (that is, they have an absolutely convergent power series expansion at each point in the interior of their domains); it is shown in \Cref{EstimateEquation} below, as a quick consequence of classical results stating real analyticity for solutions of elliptic partial differential equations.
		
	\begin{lem} 
		
		\label{sumd}
		
		Fix a real number $\varepsilon\in (0,1)$ and some open set $\mathfrak{R} \subset \mathbb{R}^2$; let $F \in \Adm_{\varepsilon} (\mathfrak{R}) \cap \mathcal{C}^2 (\mathfrak{R})$ denote a solution to \eqref{equationxtd} on $\mathfrak{R}$. Then, $F$ is real analytic $\mathfrak{R}$.  
	\end{lem} 

	\begin{figure}
	\center
\includegraphics[width=1\textwidth]{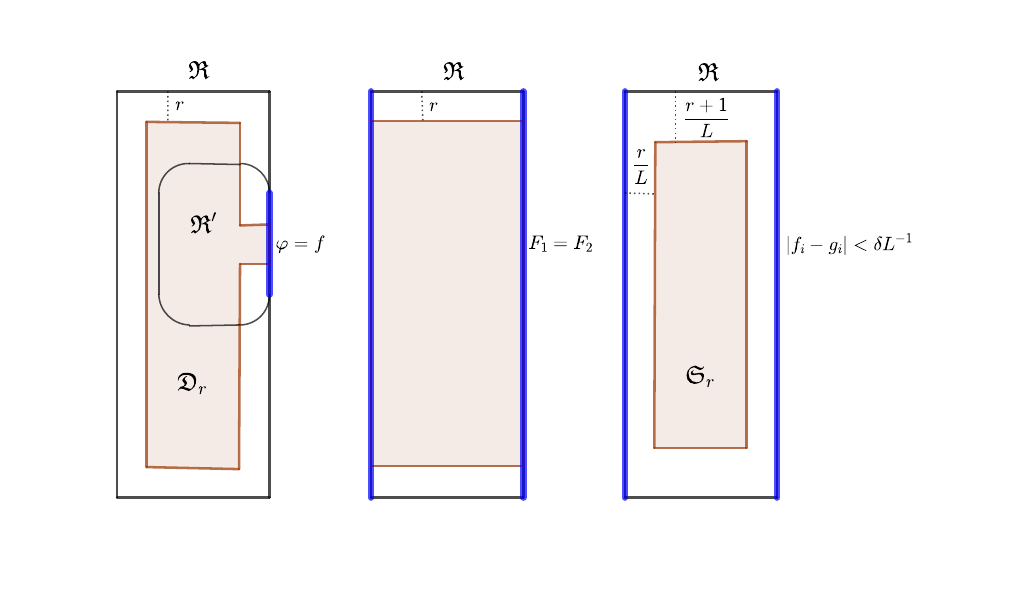}

\caption{Shown on the left is depiction for \Cref{derivativef}. Shown in the middle is a depiction for \Cref{equationcompareboundary} stating, if $F_1(z)=F_2(z)$ for each $z$ on the blue sides of $\mathfrak R$, then $F_1-F_2$ is exponentially small in the shaded region. Shown on the right is a depiction for \Cref{f1f2b}. }
\label{f:setting_1}
	\end{figure}

We next have the following lemma bounding the effect of a $\mathcal{C}^0$ boundary perturbation on the derivatives of a solution to \eqref{equationxtd}; its proof is similar to \cite[Proposition 2.13]{ULTS} and is given in \Cref{BoundaryPerturbationVariational} below.

\begin{lem} 
	
	\label{perturbationbdk} 
	
	For any integer $m \ge 1$; real numbers $r > 0$, $\varepsilon \in (0, 1)$, and $B > 1$; and bounded open subsets $\mathfrak{R}' \subseteq \mathfrak{R} \subset \mathbb{R}^2$ such that $\mathfrak{R}'$ has a smooth boundary, there exists a constant $C = C(\varepsilon, r, B, m, \mathfrak{R}') > 1$ such that the following holds. For each index $i \in \{ 1, 2 \}$, let $F_i \in \Adm_{\varepsilon} (\mathfrak{R}) \cap \mathcal{C}^2 (\mathfrak{R})$ denote a solution to \eqref{equationxtd} such that $\sup_{z \in \partial \mathfrak{R}} \big| F_i (z) \big|\le B$. Assume that there exist functions $\varphi_i \in \mathcal{C}^{m+5} (\overline{\mathfrak{R}})$ such that $\| \varphi_i \|_{\mathcal{C}^{m+5} (\mathfrak{R})} < B$ and $F_i (z) = \varphi_i (z)$ for each $z \in \partial \mathfrak{R}' \cap \partial \mathfrak{R}$. Then, letting $\mathfrak{D}_r \subseteq \mathfrak{R}$ be as in \eqref{dr}, we have  
	\begin{flalign*}
		\| F_1 - F_2 \|_{\mathcal{C}^m (\mathfrak{D}_r)} \le C \big(  \| F_1 - F_2 \|_{\mathcal{C}^0 (\mathfrak{R})} + \| \varphi_1 - \varphi_2 \|_{\mathcal{C}^{m+2} (\mathfrak{R})} \big).  
	\end{flalign*}
	
\end{lem}

The following result is established in \Cref{ProofCompare1} below. It states that, given two solutions $F_1, F_2$ to \eqref{equationxtd} on a tall rectangle of height $1$ and width $L^{-1}$ (with $L$ bounded below), whose boundary data match on its west and east boundaries, $|F_1 - F_2|$ decays exponentially in $L$ in the middle of the rectangle. As such, it can be viewed as indicating that the ``discrepancy'' between $F_1$ and $F_2$ (originating from their different boundary data along the north and south sides of the rectangle) decay exponentially around the middle of the rectangle. See the middle of \Cref{f:setting_1}.

\begin{prop}
	
	\label{equationcompareboundary}
	
	For any real numbers $\varepsilon, r \in \big( 0, \frac{1}{4} \big)$ and $B > 1$, there exists a constant $c = c(\varepsilon, r, B) > 0$ such that the following holds. Let $L > 0$ be a real number, and define the open rectangle $\mathfrak{R} = (0, L^{-1}) \times (-1, 1) \subset \mathbb{R}^2$. Let $F_1, F_2 \in \Adm_{\varepsilon} (\mathfrak{R}) \cap \mathcal{C}^5 (\overline{\mathfrak{R}})$ be two solutions to \eqref{equationxtd} such that $\| F_i \|_{\mathcal{C}^5 (\mathfrak{R})} \le B$ for each $i \in \{ 1, 2 \}$. If $F_1 (t, x) = F_2 (t, x)$ for any $(t, x) \in \partial \mathfrak{R}$ with $t \in \{ 0, L^{-1} \}$, then
		\begin{flalign*}
			\big| F_1 (t, x) - F_2 (t, x) \big| \le c^{-1} e^{-cL^{1/8}}, \qquad \text{for any $(t, x) \in [0, L^{-1}] \times [r-1, 1-r]$}.
		\end{flalign*}
	
\end{prop}

We conclude this section with the following lemma, which is a consequence of the above bounds and is established in \Cref{ProofF1F2Small} below; it essentially states the following. Fix a solution $F$ to \eqref{equationxtd}, bounded in $\mathcal{C}^m$ for some integer $m$ on a rectangle $\mathfrak{R}$, as well some boundary data $g_0$ and $g_1$ on the two vertical sides on the rectangle; assume that $g_0$ and $g_1$ are small perturbations of $F$ (restricted to those boundaries). Then, it is possible to find a solution $G$ to \eqref{equationxtd} on a slightly shorter rectangle $\mathfrak{S}$, whose boundary data on the vertical sides of the rectangle are given by $g_0$ and $g_1$ (the first condition in the below lemma), and that is close to $F$ (quantified through the second and third conditions of the below lemma). The second part of the lemma states that $F$ and $G$ are close in any $\mathcal{C}^k$ norm in the interior of $\mathfrak{S}$, and the third part states that $F$ and $G$ are close in $\mathcal{C}^{m-5}$ (that is, fewer derivatives than the original assumed bound on $F$) up to the boundary of $\mathfrak{S}$. See the right side of \Cref{f:setting_1}.

\begin{lem}
	
	\label{f1f2b}
	
	For any integers $m, k \ge 7$, and real numbers $\varepsilon > 0$; $r \in \big( 0, \frac{1}{4} \big)$; and $B > 1$, there exist constants $\delta = \delta(\varepsilon, B) > 0$, $C_1 = C_1 (\varepsilon, r, B, k) > 1$, and $C_2 = C_2 (\varepsilon, B, m) > 1$ such that the following holds. Fix a real number $L > 2$, and define the open rectangles 
	\begin{flalign*} 
		\mathfrak{R} = \Big( 0, \displaystyle\frac{1}{L} \Big) \times (-1, 1); \qquad \mathfrak{S}_r = \Big( \displaystyle\frac{r}{L}, \displaystyle\frac{1-r}{L} \Big) \times \Big( \displaystyle\frac{r+1}{L} - 1, 1 - \displaystyle\frac{r+1}{L} \Big); \qquad \mathfrak{S} = \mathfrak{S}_0. 
	\end{flalign*} 
	
	\noindent Let $F \in \Adm_{\varepsilon} (\mathfrak{R}) \cap \mathcal{C}^m (\overline{\mathfrak{R}})$ denote a solution to \eqref{equationxtd} such that $\| F \|_{\mathcal{C}^m (\mathfrak{R})} \le B$, and define the functions $f_0, f_1 : [-1, 1] \rightarrow \mathbb{R}$ by setting $f_i (x) = F (iL^{-1}, x)$ for each $(i, x) \in \{ 0, 1 \} \times [-1, 1]$. Further let $g_0, g_1 : [-1, 1] \rightarrow \mathbb{R}$ denote two functions such that $\| g_i \|_{\mathcal{C}^m (-1, 1)} \le B$ and $\big| g_i (x) - f_i (x) \big| \le \delta L^{-1}$ for each $(i, x) \in \{ 0, 1\} \times [-1, 1]$. Then, there exists a solution $G \in \Adm_{\varepsilon/2} (\mathfrak{S}) \cap \mathcal{C}^{m-5} (\overline{\mathfrak{S}})$ to \eqref{equationxtd} satisfying the following three properties.
	
	\begin{enumerate} 
		\item For each $i \in \{ 0, 1 \}$ and $x \in \big[ L^{-1} - 1, 1 - L^{-1} \big]$, we have $G(iL^{-1}, x) = g_i (x)$. 
		\item We have $\| F - G \|_{\mathcal{C}^k (\mathfrak{S}_r)} \le C_1 L^k \cdot \big( \| f_0 - g_0 \|_{\mathcal{C}^0} + \| f_1 - g_1 \|_{\mathcal{C}^0}  \big)$. 
		\item We have  $\| F - G \|_{\mathcal{C}^{m-5} (\mathfrak{S})} \le C_2 L^{m-5} \cdot \big( \| f_0 - g_0 \|_{\mathcal{C}^{m-3}} + \| f_1 - g_1 \|_{\mathcal{C}^{m-3}} \big)$ and $\| F - G \|_{\mathcal{C}^{m-5} (\mathfrak{S})} \le C_2 L^{m-5} \cdot \big( \| f_0 - g_0 \|_{\mathcal{C}^0}^{3/m} + \| f_1 - g_1 \|_{\mathcal{C}^0}^{3/m} \big)$.
	\end{enumerate}
	
\end{lem}

\subsection{Proof of \Cref{equationcompareboundary}}

\label{ProofCompare1} 

In this section we establish \Cref{equationcompareboundary}. In what follows, we fix real numbers $\varepsilon > 0$, $L > 1$, and $L_0 \in \big[ \frac{L}{2B}, 2BL \big]$. We also denote the rectangular open set $\mathfrak{R}_0 = (0, 1) \times (-L_0, L_0)$ and let $F_1, F_2 \in \Adm_{\varepsilon} (\overline{\mathfrak{R}_0}) \cap \mathcal{C}^2 (\overline{\mathfrak{R}_0})$ denote two solutions to the equation \eqref{equationxtd}.

The next lemma bounds $\big| F_1 (z) - F_2 (z) \big|$, for a point $z \in \mathfrak{R}_0$ such that $F_1 = F_2$ on the part of $\partial \mathfrak{R}_0$ near $z$; we establish it in \Cref{ProofFExponential} below.

\begin{lem} 
	
	\label{gzgzm2}
	
	For any real number $B > 1$, there exist constants $c = c(\varepsilon, B) > 0$ and $C = C(\varepsilon, B) > 1$ such that the following holds. Fix a point $z = (t_0, x_0) \in \mathfrak{R}_0$, such that $[0, 1] \times [x_0 - 2L^{1/2}, x_0 + 2L^{1/2}] \subseteq \mathfrak{R}_0$. Let $A \le B$ be a real number, and assume the below three statements hold. 
	
	\begin{enumerate} 
		\item For any indices $i \in \{ 1, 2\}$ and $j, k \in \{ t, x \}$, and point $w \in \mathfrak{R}_0$, we have $\big| \nabla F_i (w) \big| \le B$ and $\big| \partial_j \partial_k F_i (w) \big| \le BL^{-1}$; we also have $\| F_i \|_{\mathcal{C}^5 (\mathfrak{R}_0)} - \| F_i \|_{\mathcal{C}^0 (\mathfrak{R}_0)} \le  B$.
		
		\item For any point $w \in \partial \mathfrak{R}_0$ with $|w - z| \le 2L^{1/2}$, we have $F_1 (w) = F_2 (w)$.
		
		\item For any point $w \in \partial \mathfrak{R}_0$, we have $\big| F_1 (w) - F_2 (w) \big| \le A$. 
	\end{enumerate} 
	
	\noindent Then, $\big| F_1 (z) - F_2 (z) \big| \le C ( L^{-1/8} A + e^{-cL^{1/8}})$. 
	
\end{lem}

We can quickly prove \Cref{equationcompareboundary} by repeated application of \Cref{gzgzm2}.

\begin{proof}[Proof of \Cref{equationcompareboundary}]
	
	We may assume throughout this proof that $L$ is sufficiently large (in a way only dependent on $\varepsilon$ and $B$), for otherwise the lemma holds by altering the values of $c$ and $C$ (using the fact that $\big| F_1 (t, x) - F_2 (t, x) \big| \le 2B$).  Define the rescaled domain $\widehat{\mathfrak{R}} = L \cdot \mathfrak{R} = (0, 1) \times (-L, L)$ and, for each index $i \in \{ 1, 2 \}$, set $\widehat{F}_i \in \mathcal{C}^2 (\widehat{\mathfrak{R}})$ by for each $w \in [0, 1] \times [-L, L]$ letting 
	\begin{flalign*} 
		\widehat{F}_i (w) = L\cdot F_i (L^{-1} w), \quad \text{so} \quad \nabla \widehat{F}_i (w) = \nabla \widehat{F}_i (L^{-1} w), \quad \text{and}\quad  \partial_j \partial_k \widehat{F}_i (w)  = L^{-1} \cdot \partial_i \partial_j F_i (L^{-1} w),
	\end{flalign*} 
	
	\noindent for any indices $j, k \in \{ t, x \}$. In particular, the $\widehat{F}_i$ continue to solve \eqref{equationxtd}.
	
	It therefore suffices to show that $\big| \widehat{F}_i (t, x) - \widehat{F}_2 (t, x) \big| \le C e^{-c L^{1/8}}$ for each $(t, x) \in [0, 1] \times \big[ (r-1) L, (1-r) L \big]$. Denoting the open rectangle $\widehat{\mathfrak{R}}_m = (0, 1) \times (mL^{2/3} - L, L - mL^{2/3})$ for any integer $m \ge 1$, we will more generaly show by induction on $m \in \llbracket 0, L^{1/8} \rrbracket$ that there exist constants $c = c(\varepsilon, B) > 0$ and $C = C(\varepsilon, B) > 1$ such that
	\begin{flalign}
		\label{f1f2m} 
		\big| \widehat{F}_1 (t, x) - \widehat{F}_2 (t, x) \big| \le 2 B L^{-m/9} + C m e^{-cL^{1/8}}, \quad \text{for any $(t, x) \in \widehat{\mathfrak{R}}_m$},
	\end{flalign} 
	
	\noindent from which the lemma would follow upon taking $m = \lfloor L^{1/8} \rfloor$; see the left side of \Cref{f:StochasticProcess}.
	
	Let us first show that the $\widehat{F}_i$ satisfy the first and third assumptions of \Cref{gzgzm2} (with the $A$ there equal to $2B$ here). Since $\| F_i \|_{\mathcal{C}^5 (\overline{\mathfrak{R}})} \le B$, we have that $\big| \nabla \widehat{F}_i (w) \big| \le B$ and $\big| \partial_j \partial_k \widehat{F}_i (w) \big| \le BL^{-1}$ for each $w \in \widehat{\mathfrak{R}}$ and that $\| \widehat{F}_i \|_{\mathcal{C}^5 (\hat{\mathfrak{R}}_m)} - \| \widehat{F}_i \|_{\mathcal{C}^0 (\hat{\mathfrak{R}}_m)} \le B$; this verifies the first assumption. Moreover, since $F_1 (t, x) = F_2 (t, x)$ for $t \in \{ 0, L^{-1} \}$, since $\big| \nabla F_i (w) \big| \le B$, and since $\mathfrak{R} = (0, L^{-1}) \times (-1, 1)$ we have $\big| F_1 (w) - F_2 (w) \big| \le 2BL^{-1}$ for each $w \in \overline{\mathfrak{R}}$. Hence, $\big| \widehat{F}_1 (w) - \widehat{F}_2 (w) \big| \le 2B$ for each $w \in \widehat{\mathfrak{R}}$, verifying the third assumption of \Cref{gzgzm2}. This in particular confirms \eqref{f1f2m} at $m = 0$. 
	
	We therefore assume in what follows that \eqref{f1f2m} holds at some value of $m \in \llbracket 0, L^{1/8} - 1 \rrbracket$ and verify that it also holds with $m$ replaced by $m+1$. For any integer $k \in \llbracket 0, L^{1/8} \rrbracket$, set 
	\begin{flalign*} 
		\varsigma_k = \sup_{w \in \mathfrak{R}_k} \big| \widehat{F}_1 (w) - \widehat{F}_2 (w) \big|.
	\end{flalign*} 
	
	\noindent Then the second assumption of \Cref{gzgzm2} applies, with the $(\mathfrak{R}_0, F_1, F_2, z)$ there given by the $(\widehat{\mathfrak{R}}_m, \widehat{F}_1, \widehat{F}_2, z)$ with any $z \in \widehat{\mathfrak{R}}_{m+1}$. Since $L - (m+1) L^{2/3} \ge \frac{L}{2B} + 2L^{1/2}$ for $m \in \llbracket 0, L^{1/8} \rrbracket$, \Cref{gzgzm2} therefore yields constants $c_1 = c_1 (\varepsilon, B) > 0$ and $C_1 = C_1 (\varepsilon, B) > 1$ such that
	\begin{flalign*} 
		\varsigma_{m+1} = \displaystyle\sup_{w \in \widehat{\mathfrak{R}}_{m+1}} \big| \widehat{F}_1 (w) - \widehat{F}_2 (w) \big| & \le C_1 L^{-1/8} \cdot \displaystyle\sup_{w \in \widehat{\mathfrak{R}}_m} \big| \widehat{F}_1 (w) - \widehat{F}_2 (w) \big| + C_1 e^{-c_1 L^{1/8}}  \\
		& = C_1 L^{-1/8} \varsigma_m + C_1 e^{-c_1 L^{1/8}} \le L^{-1/9} \varsigma_m + C_1 e^{-c_1 L^{1/8}},
	\end{flalign*} 
	
	\noindent where in the last inequality we used the fact that $L$ is sufficiently large. Applying this with \eqref{f1f2m}, with the $(c, C)$ there equal to $(c_1, C_1)$ here yields
	\begin{flalign*} 
		\varsigma_{m+1} \le 2BL^{-(m+1)/9} + C_1 m L^{-1/9} e^{-c_1 L^{1/8}} + C_1 e^{-c_1 L^{1/8}} \le 2BL^{-(m+1)/9} + C_1 (m+1) e^{-c_1 L^{1/8}}.
	\end{flalign*} 
	
	\noindent This verifies \eqref{f1f2m} (with the $m$ there replaced by $m+1$ here), establishing the proposition.
\end{proof}

To establish \Cref{gzgzm2}, we use a probabilistic interpretation of the $F_i$. To that end, for any real numbers $u \in \mathbb{R}$ and $v \in [-\varepsilon^{-1}, \varepsilon]$, we recall from \Cref{eigenvalues2} the $2 \times 2$ matrix $\bm{D} = \bm{D} (u, v)$. Let $W : \mathbb{R}_{\ge 0} \rightarrow \mathbb{R}^2$ denote a standard two-dimensional Brownian motion. For any point $z \in \mathfrak{R}_0$ and index $i \in \{ 1, 2 \}$, define the process $Y_i = Y_i^z : [0, \tau_i] \rightarrow \mathbb{R}^2$ through the stochastic differential equation
\begin{align}
	\label{wyi}
	dY_i (s) = \bm{D} \Big( \nabla F_i \big( Y_i (s) \big) \Big)^{1/2} \cdot dW(s), \qquad \text{with initial data $Y_i (0) = z$},
\end{align}

\noindent so that $Y_1$ and $Y_2$ are coupled through the same Brownian motion $W(s)$; see the right side of \Cref{f:StochasticProcess} for a depiction. Here, the hitting time $\tau_i = \tau_{i;z}$ is given by 
\begin{flalign*}
	\tau_i = \inf \big\{ t \ge 0 : Y_i (s) \in \partial \mathfrak{R}_0 \big\}.
\end{flalign*}

\noindent That the stochastic differential equation \eqref{wyi} is well-defined follows from the fact that the coefficients $\bm{D} \big( \nabla F_i (t, x) \big)^{1/2}$ of the equation are uniformly Lipschitz (as $F_i \in \mathcal{C}^2 (\overline{\mathfrak{R}_0})$).

	\begin{figure}
	\center
\includegraphics[width=0.8\textwidth]{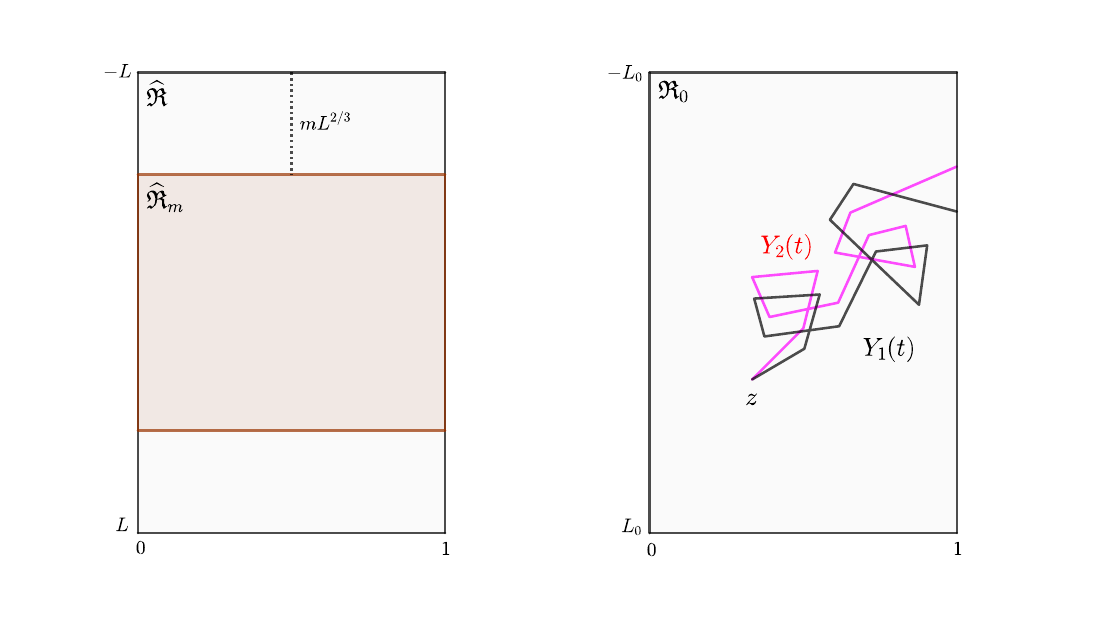}

\caption{Shown to the left is the rectangle $\widehat{\mathfrak R}_m$ used in the proof of \Cref{equationcompareboundary}. Shown to the right are the coupled stochastic processes $Y_1(s)$ and $Y_2(s)$ as in \eqref{wyi}.}
\label{f:StochasticProcess}
	\end{figure}

\begin{lem} 
	
	\label{gexpectationtt0}
	
	For each index $i \in \{ 1 ,2 \}$, the process $F_i \big( Y_i (s) \big)$ is a martingale. In particular,
	\begin{flalign}
		\label{gwt}
		dF_i \big( Y_i (s) \big) = \nabla F_i \big( Y_i (s) \big) \cdot \bm{D} \Big( \nabla F_i \big( Y_i (s) \big) \Big)^{1/2} \cdot d W(s).
	\end{flalign}
	
\end{lem} 

\begin{proof} 
	
	Since $\tau_i = 0$ if $z \in \partial \mathfrak{R}_0$, we may assume that $z \in \mathfrak{R}_0$. Then, denoting the Hessian of $F_i$ by 
	\begin{flalign*}
		\bm{H} (z) = \left[ \begin{array}{cc} \partial_{tt} F_i (z) & \partial_{tx} F_i (z) \\ \partial_{xt} F_i (z) & \partial_{xx} F_i (z) \end{array} \right],
	\end{flalign*}
	
	\noindent we have by It\^{o}'s lemma that
	\begin{align*}
		d F_i \big( Y_i (s) \big) &=\nabla F_i \big( Y_i (s) \big) \cdot \bm{D} \Big( \nabla F_i \big( Y_i (s) \big) \Big)^{1/2} \cdot d W (s) + \frac{1}{2} \Tr \bigg(  \bm{D} \Big( \nabla F_i \big( Y_i (s) \big) \Big) \bm{H} \Big( F_i \big( Y_i (s) \big) \Big) \bigg) \\
		&=\nabla F_i \big( Y_i (s) \big) \cdot \bm{D} \Big( \nabla F_i \big( Y_i (s) \big) \Big)^{1/2} \cdot d W(s),
	\end{align*}
	
	\noindent where in the last equality we used the fact that $F_i$ satisfies \eqref{equationxtd} (with the $(t, y)$ there equal to $Y_i (s)$ here). This yields \eqref{gwt} and thus that $F_i \big( Y_i (s) \big)$ is a martingale.
\end{proof} 

Hence, we must analyze properties of the process $Y_i$, to which end we have the following lemma.

\begin{lem} 
	
	\label{testimate0} 
	
	There exists a constant $c = c (\varepsilon) > 0$ such that, for any real number $A > 1$, index $i \in \{ 1, 2 \}$, and point $z = (t_0, x_0) \in \mathfrak{R}_0$, we have
	\begin{flalign}
		\label{s2} 
		\begin{aligned}
			& \mathbb{P} [ \tau_{i;z} \ge A] \le c^{-1} e^{-cA}; \qquad \mathbb{P} \bigg[ \displaystyle\sup_{s \in [0, \tau_{i;z}]} \big| Y_i (s) - z \big| \ge A \Big] \le c^{-1} e^{-cA}.
		\end{aligned}
	\end{flalign}

\end{lem} 	 

\begin{proof}
	
	Denote $W(s) = \big( W_1 (s), W_2 (s) \big) \in \mathbb{R}$, so that $W_1, W_2 : \mathbb{R}_{\ge 0} \rightarrow \mathbb{R}$ are two independent standard Brownian motions; also let $Y_i (s) = \big( Y_{i; 1} (s), Y_{i;2} (s) \big) \in \mathbb{R}^2$. Then, by \eqref{wyi} and the fact that $\bm{D}(u, v) = \diag (1, \pi^2 v^{-4})$, we have $Y_{i; 1} (s) - t_0 = W_1 (s)$, for each $s \in [0, \tau_{i;z}]$. In particular we have for any $s\leq \min\{\tau_{1;z}, \tau_{2;z}\}$ that
	\begin{align}\label{e:firstcor}
	Y_{1;1}(s)=Y_{2;1}(s)=W_1(s) +t_0.
	\end{align}
	Since $\mathfrak{R}_0 \subset [0, 1] \times \mathbb{R}$, it follows that $\tau_{i;z} \ge A$ can only hold if $W_i (s) = Y_{i;1} (s) \in [-1, 1]$, for each $s \in [0, A]$. Thus,
	\begin{flalign}
		\label{asw}
		\begin{aligned}
			\mathbb{P} [ \tau_{i;z} \ge A] & \le \mathbb{P} \Bigg[ \bigcap_{s \in [0, A]} \big\{ W_1 (s) \in [-1, 1] \big\} \Bigg] = \mathbb{P} \Bigg[ \bigcap_{k=1}^{\lfloor A \rfloor} \Big\{ \big| W_1 (k) - W_1 (k-1) \big| \le 2 \Big\} \Bigg].
		\end{aligned} 
	\end{flalign}
	
	\noindent Now observe that there exists a constant $c_1 > 1$ such that $\mathbb{P} \big[ W_1 (k) - W_1 (k-1) \le 2 \big] < 1-c_1$, for any $k \ge 1$. Hence,
	\begin{flalign*} 
		\mathbb{P} \Bigg[ \bigcap_{k=1}^{\lfloor A \rfloor} & \Big\{ \big| W_1 (k) - W_1 (k - 1) \big| \le 2 \Big\} \Bigg] = \displaystyle\prod_{k=1}^{\lfloor A \rfloor} \mathbb{P} \Big[ \big| W_1 (k) - W_1 (k-1) \big| \le 2 \Big] \le  (1-c_1)^{\lfloor A \rfloor},
	\end{flalign*} 
	
	\noindent which together with \eqref{asw} yields the first bound in \eqref{s2}.
	
	Next, since $F_i \in \Adm_{\varepsilon} (\mathfrak{R}_0)$, we have $\partial_x F_i (t, x) \in [-\varepsilon^{-1}, -\varepsilon]$ for each $(t, x) \in \mathfrak{R}_0$; it follows that $\big| \bm{D} \big(\nabla F_i (t, x) \big)^{1/2}  \big| \le 5 \varepsilon^{-2}$ for each $(t, x) \in \mathfrak{R}_0$. By the definition \eqref{wyi} of $Y_i$ (and the fact that $Y_i$ is a time-changed Brownian motion by \cite[Theorem 3.4.6]{MSC}), this yields  
	\begin{flalign*}
		\mathbb{P} \bigg[ \displaystyle\sup_{s \in [0, \tau_i]} \big| Y_i (s) - z \big| \ge A \bigg] & \le \mathbb{P} \bigg[ \displaystyle\sup_{s \in [0, A]} \big| W(s) \big| \ge 10 \varepsilon^{-2} A \bigg] + \mathbb{P} [\tau_i \ge A] \\ 
		& \le \mathbb{P} \bigg[ \displaystyle\sup_{s \in [0, 1]} \big| W(s) \big| \ge 10 \varepsilon^{-2} A^{1/2} \bigg] + c^{-1} e^{-cA},
	\end{flalign*}
	
	\noindent where in the last inequality we used the first bound in \eqref{s2}. This, together with the fact (see \cite[Chapter 4, Equation (3.40)]{MSC}) that there is a constant $c_2 = c_2 (\varepsilon) > 0$ such that $\sup_{s \in [0, 1]} \big| W(s) \big| \le 10 \varepsilon^{-2} A^{1/2}$ holds with probability at least $1 - c_2^{-1} e^{-c_2 A}$, implies the second bound in \eqref{s2}.	
\end{proof}

\subsection{Proof  \Cref{gzgzm2}}

\label{ProofFExponential}

In this section we establish \Cref{gzgzm2}.

\begin{proof}[Proof of \Cref{gzgzm2}]
	
	We may assume throughout this proof that $L > 10$ (for otherwise taking the constant $C = 2$ in the lemma gives $\big| F_1 (z) - F_2 (z) \big| \le \sup_{z' \in \partial \mathfrak{R}_0} \big| F_1 (z') - F_2 (z') \big| \le A \le C L^{-1/8} A$, by \Cref{maximumboundary}). Fix $M = L^{1/4}$; abbreviate $\tau_i = \tau_{i;z}$ for each $i \in \{ 1, 2 \}$; set $\tau = \min \{ \tau_1, \tau_2 \}$; and define the event 
	\begin{flalign}
		\label{aa01} 
		&\mathscr{A} = \{\tau \le M \} \cap \bigg\{ \displaystyle\max_{i \in \{ 1, 2 \}} \displaystyle\sup_{s \in [0, \tau_i]} \big| Y_i (s) - z \big| \le M \bigg\}, \qquad \text{so that $\mathbb{P} \big[ \mathscr{A}^{\complement} \big] \le c_1^{-1} e^{-c_1 M}$},
	\end{flalign}
	
	\noindent for some constant $c_1 = c_1 (\varepsilon) > 0$, by \Cref{testimate0}. Denote $W(s) = \big( W_1 (s), W_2 (s) \big)$ and $Y_i (s) = \big( Y_{i; 1} (s), Y_{i; 2} (s) \big)$, for each index $i \in \{ 1, 2 \}$ and real number $s \ge 0$. Recall from \eqref{e:firstcor} that we have $Y_{1; 1} (s) = W_1 (s) + t_0 = Y_{2; 1} (s)$ for each $s \in [0, \tau]$. Since on the event $\mathscr{A}$ we have $Y_{i_0} (\tau) \in \partial \mathfrak{R}_0$ for some $i_0 \in \{ 1, 2 \}$ while $\big| Y_{i_0; 2} (\tau) - x_0 \big| \le L^{1/4} \le L^{1/2}$, and since $(0, 1) \times [x_0 - 2L^{1/2}, x_0 + 2L^{1/2}] \subseteq \mathfrak{R}_0$, it follows that $Y_{i_0; 1} (\tau) \in \{ 0, 1 \}$ (that is, the $t$-coordinate of $Y_{i_0} (\tau)$ causes it to be on $\partial \mathfrak{R}_0$). Hence, 
	\begin{flalign}
		\label{y11y21a} 
		Y_{1, 1} (\tau) = Y_{2, 1} (\tau) \in \{ 0, 1 \}, \quad \text{and thus} \quad \tau_1 = \tau_2, \quad \text{on the event $\mathscr{A}$}.
	\end{flalign}

	By \eqref{wyi}, we have 
	\begin{flalign}
		\label{y1y2w} 
		d \big( Y_1 (t) - Y_2 (t) \big) &= \bigg( \bm{D} \Big( \nabla F_1 \big( Y_1 (t) \big) \Big)^{1/2} - \bm{D} \Big( \nabla F_2 \big( Y_2 (t) \big) \Big)^{1/2} \bigg) \cdot d W(t).
	\end{flalign}
	
	\noindent Thus, for any real numbers $t > 0$ and $p > 1$, the Burkholder--Davis--Gundy inequality \cite[Theorem 3.2]{DFIM} (see also \cite[Theorem 11.2.1]{PTIIM}) yields 
	\begin{flalign}
		\label{integraly1y2} 
		\mathbb{E} \bigg[ \displaystyle\sup_{s \in [0, t]} \big| Y_1 (s) - Y_2 (s) \big|^{2p} \bigg] \le (72p)^{2p} \cdot \mathbb{E} \Bigg[ \bigg(\displaystyle\int_0^t \Big\| \bm{D} \big( \nabla F_1 (Y_1 (s) \big)^{1/2} - \bm{D} \big( \nabla F_2 (Y_2 (s)) \big)^{1/2} \Big\| ds \bigg)^{2p} \Bigg].
	\end{flalign}
	
	To bound the integral on the right side of \eqref{integraly1y2}, first observe that since $\partial_x F_1 (w), \partial_x F_2 (w) \in [-\varepsilon^{-1}, -\varepsilon]$ that the $\bm{D} (\nabla F_i)$ are uniformly smooth (in a way dependent on $\varepsilon$) in $\nabla F_i$. Thus, there exists a constant $C_1 = C_1 (\varepsilon) > 0$ such that 
	\begin{flalign}
		\label{dg1g2} 
		\begin{aligned}
			\bigg\| \bm{D} \Big( \nabla F_1 \big( & Y_1 (t) \big) \Big)^{1/2} - \bm{D} \Big( \nabla F_2 \big( Y_2 (t) \big) \Big)^{1/2} \bigg\| \\
			& \le  C_1  \Big| \nabla F_1 \big( Y_1 (t) \big) - \nabla F_2 \big( Y_2 (t) \big) \Big| \\
			& \le C_1 \bigg( \Big| \nabla F_1 \big( Y_1 (t) \big) - \nabla F_1 \big( Y_2 (t) \big) \Big| + \Big| \nabla F_1 \big( Y_2 (t) \big) - \nabla F_2 \big( Y_2 (t) \big) \Big| \bigg) \\
			& \le  C_1 BL^{-1} \big| Y_1 (t) - Y_2 (t) \big| + C_1 \Big| \nabla F_1 \big( Y_2 (t) \big) - \nabla F_2 \big( Y_2 (t) \big) \Big|,
		\end{aligned}
	\end{flalign} 
	
	\noindent where in the last bound we used the first assumption of the lemma $\big| \partial_j \partial_k F_i (w) \big| \le BL^{-1}$. Since we deterministically have $Y_1 (t) - Y_2 (t) \le 4L + 1 \le 5L$ (as $\mathfrak{R}_0 \subseteq [0, 1] \times [-2L, 2L]$) and $\nabla F_i (w) \in [-\varepsilon^{-1}, -\varepsilon]$, we therefore have that
	\begin{flalign}
		\label{expectationy1y2a2}
		\begin{aligned} 
			\mathbb{E} \Bigg[ \textbf{1}_{\mathscr{A}^{\complement}} \cdot \bigg( \displaystyle\int_0^t \Big\| \bm{D} \big( \nabla F_1 (Y_1 (s)) \big) - \bm{D} \big( \nabla F_2 (Y_2(s)) \big) \Big\| ds \bigg)^{2p} \Bigg] & \le \mathbb{P} \big[ \mathscr{A}^{\complement} \big] \cdot (C_1 t)^{2p} (5B + \varepsilon^{-1})^{2p} \\
			& \le (C_2 t)^{2p} e^{-c_1 M},
		\end{aligned}
	\end{flalign}
	
	\noindent for some constant $C_2 = C_2 (\varepsilon, B) > 1$, where in the last inequality we applied \eqref{aa01}.

	To bound the second term on the right side of \eqref{dg1g2}, fix some neighborhood $\mathfrak{R}_0' \subseteq \mathfrak{R}_0$ of $Y_2 (t)$ with smooth boundary such that $[0, 1] \times \big[ Y_2 (t) - 1, Y_2 (t) + 1 \big] \subseteq \mathfrak{R}_0' \subseteq [0, 1] \times \big[ Y_2 (t) - 2, Y_2 (t) + 2 \big]$. On the event $\mathscr{A}$, we have (since $M = L^{1/2}$) that $Y_2 (t) \in [0, 1] \times [x_0 - 2L^{1/2}, x_0 + 2L^{1/2}] \subseteq \mathfrak{R}_0$, and so we may take $\mathfrak{R}_0'$ to be of some fixed shape (for example, a rectangle of uniformly bounded height and width, with smoothened corners) up to translation, independently of $\varepsilon$, $A$, $B$, and $L$. Then,
	\begin{flalign}
		\label{f1y2f2y2} 
		\begin{aligned}
			\textbf{1}_{\mathscr{A}} \cdot \Big| \nabla F_1 \big( Y_2 (t) \big) - \nabla F_2 \big( Y_2 (t) \big) \Big| & \le \displaystyle\sup_{w \in \mathfrak{R}_0'} \big| \nabla F_1 (w) - \nabla F_2 (w) \big| \\
			& \le C_3 \displaystyle\sup_{w \in \partial \mathfrak{R}_0'} \big| F_1 (w) - F_2 (w) \big| \le C_3 \displaystyle\sup_{w \in \partial \mathfrak{R}_0} \big| F_1 (w) - F_2 (w) \big| \le C_3 A,
		\end{aligned} 
	\end{flalign}
	
	\noindent for some some constant $C_3 = C_3 (\varepsilon, B) > 1$, where in the second inequality we applied (the $m = 1$ case of) \Cref{perturbationbdk} (applied with both $\varphi_i = F_1 |_{\mathfrak{R}_0'}$, which satisfy $\varphi_1 - \varphi_2 = 0$ and $\| \varphi_i \|_{\mathcal{C}^5 (\mathfrak{R}_0')} - \| \varphi_i \|_{\mathcal{C}^0 (\mathfrak{R}_0')} \le \| F_1 \|_{\mathcal{C}^5 (\mathfrak{R}_0)} - \| F_1 \|_{\mathcal{C}^0 (\mathfrak{R}_0)} \le B$), in the third we applied \Cref{maximumboundary}, and in the fourth we applied the third assumption of the lemma. Inserting this, \eqref{dg1g2}, and \eqref{expectationy1y2a2} into \eqref{integraly1y2} gives
	\begin{flalign*}
		\mathbb{E} \bigg[ \displaystyle\sup_{s \in [0, t]} \big| & Y_1 (s) - Y_2 (s) \big|^{2p} \bigg] \\
		& \le (72p)^{2p} \Bigg( (C_2 t)^{2p} e^{-c_1 M} + (2 C_1 C_3 At)^{2p} + (2C_1 BL^{-1})^{2p} \bigg( \displaystyle\int_0^t \big| Y_1 (s) - Y_2 (s) \big| ds \bigg)^{2p} \Bigg).
	\end{flalign*} 
	
	By Markov's inequality, we deduce the existence of a constant $C_4 = C_4 (\varepsilon, B) > 1$ such that 
	\begin{flalign*}
		\mathbb{P} \Bigg[ \displaystyle\sup_{s \in [0, t]} \big| Y_1 (s) - Y_2 (s) \big| \ge C_4 p t (A + e^{-c_1 M/2p}) + C_4 p L^{-1} \displaystyle\int_0^t \big| Y_1 (s) - Y_2(s) \big| ds \Bigg] \le 4^{-p}.
	\end{flalign*}
	
	\noindent Taking $p = L^{1/8}$ gives (after replacing $c_1$ by $\frac{c_1}{2}$) with probability at least $1 - 4^{-L^{1/8}}$ that 
	\begin{flalign}
		\label{y1s2} 
		\displaystyle\sup_{s \in [0, t]} \big| Y_1 (s) - Y_2 (s) \big| \le C_4 L^{3/8} (A + e^{-c_1 L^{1/8}}) + C_4 L^{-7/8} \displaystyle\int_0^t \big| Y_1 (s) - Y_2 (s) \big| ds, 
	\end{flalign} 
	
	\noindent for any $t \in [0, L^{1/4}]$. Together with Gronwall's inequality,\footnote{Here, we implicitly take a union bound so that \eqref{y1s2} holds for all $t$ in a $2^{-L^{1/8}}$-mesh of $[0, L^{1/8}]$, and apply the discrete Gronwall inequality \cite[Equation (8)]{DIA}.} this yields a constant $C_5 = C_5 (\varepsilon, B) > 1$ such that 
	\begin{flalign}
		\label{y1sy2sa}
		\mathbb{P} \bigg[ \displaystyle\sup_{s \in [0, L^{1/2}]} \big| Y_1 (s) - Y_2 (s) \big| \ge C_5 L^{3/8} (A + e^{-c_1 L^{1/8}}) \bigg] \le 2^{-L^{1/8}}.
	\end{flalign}
	
	Since by \Cref{gexpectationtt0} the process $F_i \big( Y_i (s) \big)$ is a martingale, we have
	\begin{flalign}
		\label{f1zf2zy} 
		F_1 (z) - F_2 (z) = \mathbb{E} \Big[ F_1 \big( Y_1 (\tau) \big) - F_2 \big( Y_2 (\tau) \big) \Big].
	\end{flalign}
	
	\noindent Setting $z_0 = z - (t_0, 0) = (0, x_0) \in \partial \mathfrak{R}_0$, Taylor expanding $F_1$ and $F_2$, and using the fact that $\big| \partial_j \partial_k F_i (w) \big| \le BL^{-1}$, we find that 
	\begin{flalign*}
		\bigg| \mathbb{E} \Big[ F_1 \big( Y_1 (& \tau)  \big) - F_2 \big( Y_2 (\tau) \big) \Big] - \mathbb{E} \Big[ F_1 \big( Y_1 (\tau) \big) - F_2 \big( Y_1 (\tau) \big) \Big] - \mathbb{E} \Big[ \nabla F_2 (z_0) \cdot \big( Y_1 (\tau) - Y_2 (\tau) \big) \Big] \\
		& \quad - \mathbb{E} \bigg[ \Big( \nabla F_2 \big( Y_2 (\tau) \big) - \nabla F_2 (z_0) \Big) \cdot \big( Y_1 (\tau) - Y_2 (\tau) \big) \bigg] \bigg| \le \displaystyle\frac{B}{2L} \cdot \mathbb{E} \Big[ \big| Y_1 (\tau) - Y_2 (\tau) \big|^2 \Big].
	\end{flalign*}
	
	\noindent Together with \eqref{f1zf2zy}, the fact that $\mathbb{E} \big[ Y_1 (\tau) \big] = z = \mathbb{E} \big[ Y_2 (\tau) \big]$ (as $Y_1$ and $Y_2$ are martingales by \eqref{wyi}), and again the bound $\big| \partial_j \partial_k F_i (w) \big| \le BL^{-1}$, this yields 
	\begin{flalign*}
		\big| F_1 (z) - F_2 (z) \big| & = \bigg| \mathbb{E} \Big[ F_1 \big( Y_1 (\tau) \big) - F_2 \big( Y_2 (\tau) \big) \Big] \bigg| \\
		& \le \bigg| \mathbb{E} \Big[ F_1 \big( Y_1 (\tau) \big) - F_2 \big( Y_1 (\tau) \big) \Big] \bigg| + \displaystyle\frac{B}{L} \cdot \mathbb{E} \Big[ \big| Y_2 (\tau) - z_0 \big| \cdot \big| Y_2 (\tau) - Y_1 (\tau) \big| \Big] \\
		& \qquad + \displaystyle\frac{B}{2L} \cdot \mathbb{E} \Big[ \big| Y_1 (\tau) - Y_2 (\tau) \big|^2 \Big],
	\end{flalign*} 
	
	\noindent and hence
	\begin{flalign}
		\label{f1f2z}
		\begin{aligned}
			\big| F_1 (z) - F_2 (z) \big| & \le \bigg| \mathbb{E} \Big[ F_1 \big( Y_1 (\tau) \big) - F_2 \big( Y_1 (\tau) \big) \Big] \bigg| \\ 
			& \qquad + \displaystyle\frac{2B}{L} \cdot \mathbb{E} \bigg[ \Big( \big| z_0 - Y_1 (\tau) \big| + \big| Y_2 (\tau) - z_0 \big| \Big) \cdot \big| Y_1 (\tau) - Y_2 (\tau) \big| \bigg].
		\end{aligned} 
	\end{flalign}
	
	Since $F_1 \big( Y_1 (\tau) \big) = F_2 \big( Y_1 (\tau) \big)$ on the event $\mathscr{A}$ (by \eqref{y11y21a}), and since $\big| F_1 \big( Y_1 (\tau) \big) - F_2 \big( Y_2 (\tau) \big) \big| \le 10B L$ (as $\big| \nabla F_1 (w) \big| \le B$, $\big|\nabla F_2 (w) \big| \le B$, and the diameter of $\mathfrak{R}_0$ is at most $5L$), it follows from \eqref{aa01} that   
	\begin{flalign}
		\label{f1f2z1} 
		\bigg|	\mathbb{E} \Big[ F_1 \big( Y_1 (\tau) \big) - F_2 \big( Y_1 (\tau) \big) \Big] \bigg| = \mathbb{E} \bigg[ \textbf{1}_{\mathscr{A}^{\complement}} \cdot \Big| F_1 \big( Y_1 (\tau) \big) - F_2 \big( Y_1 (\tau) \big) \Big| \bigg] \le 10 c_1^{-1} B L e^{-c_1 L^{1/8}}
	\end{flalign}
	
	\noindent Moreover, from \eqref{y1sy2sa}, \eqref{aa01}, and the fact that $\big| Y_i (\tau) - z_0 \big| \le M + 1 \le L^{1/2}$ for each $i \in \{ 1, 2 \}$ on $\mathscr{A}$ (and again the fact that the diameter of $\mathfrak{R}_0$ is at most $5L$), we deduce the existence of constants $c_2 = c_2 (\varepsilon, B) > 0$ and $C_6 = C_6 (\varepsilon, B) > 1$ such that  
	\begin{flalign}
		\label{f1f2z0} 
		\begin{aligned}
			\mathbb{E} & \bigg[ \Big( \big| Y_1 (\tau) - z_0 \big| + \big| Y_2 (\tau) - z_0 \big| \Big) \cdot \big| Y_1 (\tau) - Y_2 (\tau) \big| \bigg] \\ 
			& \le 2L^{1/2} \cdot C_5 L^{3/8} (A + e^{-c_1 L^{1/8}})  +\big( c_1^{-1} e^{-c_1 L^{1/8}} + 2^{-L^{1/8}} ) \cdot 2 (5L)^2 \le C_6 L^{7/8} A + C_6 e^{-c_2 L^{1/8}},
		\end{aligned} 
	\end{flalign}

	\noindent where we have used the fact that, on $\mathscr{A}$, we have $|Y_1 (\tau) - z_0 \big| + \big| Y_2 (\tau) - z_0 \big| \le 2M = 2L^{1/2}$. Inserting \eqref{f1f2z1} and \eqref{f1f2z0} into \eqref{f1f2z} yields 
	\begin{flalign*}
		\big| F_1 (z) - F_2 (z) \big| \le 2 C_6 BL^{-1/8} A + 2 C_6 B L^{-1} e^{-c_2 L^{1/8}} + 10 c_1^{-1} BL e^{-c_1 L^{1/8}},
	\end{flalign*} 
	
	\noindent which implies the lemma.
\end{proof}

\appendix

\section{Concentration Bounds for the Height Function}

\label{ProofCurves0} 

In this section we establish a concentration bound for non-intersecting Brownian bridges; it (and its proof) is analogous to the one that appears in the context of random tilings (see \cite[Theorem 21]{LSRT}, for example). For any family of continuous curves $\bm{x} = (x_1, x_2, \ldots , x_n) \in \llbracket 1, n \rrbracket \times \mathcal{C} \big( [a, b] \big)$, we recall from \eqref{htx} that the associated height function $\mathsf{H} = \mathsf{H}^{\bm{x}} : [a, b] \times \mathbb{R} \rightarrow \mathbb{R}$ is defined by $\mathsf{H} (t, w) = \# \big\{ j \in \llbracket 1, n \rrbracket : x_j (t) > w \big\}$. In what follows, we also recall the measure $\mathfrak{Q}_{f; g}^{\bm{u}; \bm{v}}$ from \Cref{qxyfg}.

\begin{lem}
	
	\label{concentrationbridge}
	
	Let $n \ge 1$ be an integer; $T > 0$ and $r, B \ge 1$ be real numbers; $\bm{u}, \bm{v} \in \overline{\mathbb{W}}_n$ be $n$-tuples; and $f, g : [0, T] \rightarrow \overline{\mathbb{R}}$ be measurable functions with $f \le g$. Sample non-intersecting Brownian bridges $\bm{x} \in \llbracket 1, n \rrbracket \times \mathcal{C} \big( [0, T] \big)$ from the measure $\mathfrak{Q}_{f; g}^{\bm{u}; \bm{v}}$. Fix real numbers $t \in [0, T]$ and $w \in \big[ f(t), g(t) \big]$. Denoting the event $\mathscr{E} = \big\{ \mathsf{H} (t, w) \le B \big\}$, there exists a deterministic number $\mathfrak{Y} = \mathfrak{Y} (\bm{u}; \bm{v}; f; g; T; t; w; B) \ge 0$ such that
	\begin{align}
		\label{e:concentration}
		\mathbb{P} \Big[ \big| \mathsf{H} (t, w) - \mathfrak{Y} \big| \ge rB^{1/2} \Big] \le 2 e^{-r^2/4} + 2 \cdot \mathbb{P} \big[\mathscr{E}^{\complement} \big].
	\end{align}
	
	\noindent In particular, setting $B = n$, we have $\mathbb{P} \big[ \big| \mathsf{H} (t, w) - \mathfrak{Y} \big| \ge rn^{1/2} \big] \le 2e^{-r^2/4}$. 
	
\end{lem}

To establish \Cref{concentrationbridge}, we proceed through a discretization. For any integer $\mathsf{S} \ge 0$, a \emph{discrete path} is a function $\mathsf{p} : \llbracket0, \mathsf{S} \rrbracket \rightarrow \mathbb{Z}$ such that $\mathsf{p} (t) - \mathsf{p} (t-1) \in \{ 0, 1 \}$, for each $t \in \llbracket 1, \mathsf{S} \rrbracket$. A \emph{discrete path ensemble} on $\llbracket 0, \mathsf{S} \rrbracket$ is an ordered sequence of discrete paths $\bm{\mathsf{p}} = (\mathsf{p}_1, \mathsf{p}_2, \ldots , \mathsf{p}_k)$, with $\mathsf{p}_j : \llbracket 0, \mathsf{S} \rrbracket \rightarrow \mathbb{Z}$ for each $j \in \llbracket 1, k \rrbracket$. Given integer sequences $\bm{u} = (u_1, u_2, \ldots , u_k) \in \mathbb{Z}^k$ and $\bm{v} = (v_1, v_2, \ldots , v_k) \in \mathbb{Z}^k$, we say that $\bm{\mathsf{p}}$ \emph{starts} at $\bm{u}$ and \emph{ends} at $\bm{v}$ if $\mathsf{p}_j (0) = u_j$ and $\mathsf{p}_j (\mathsf{S}) = v_j$ for each $j \in \llbracket 1, k \rrbracket$, respectively; observe that this is only possible if $u_j \le v_j \le u_j + \mathsf{S}$ for each $j$. Given two functions $f, g : \llbracket 0, \mathsf{S} \rrbracket \rightarrow \mathbb{Z}$, we further say that the ensemble $\bm{\mathsf{p}}$ \emph{remains above $f$} and \emph{remains below $g$} if $f(s) < \mathsf{p}_j (s)$ and $\mathsf{p}_j (s) < g(s)$ for each $(j, s) \in \llbracket 1, n \rrbracket \times \llbracket 0, \mathsf{S} \rrbracket$, respectively. See the left side of \Cref{f:Heightconcentration} for a depiction.

	\begin{figure}
	\center
\includegraphics[width=0.8\textwidth]{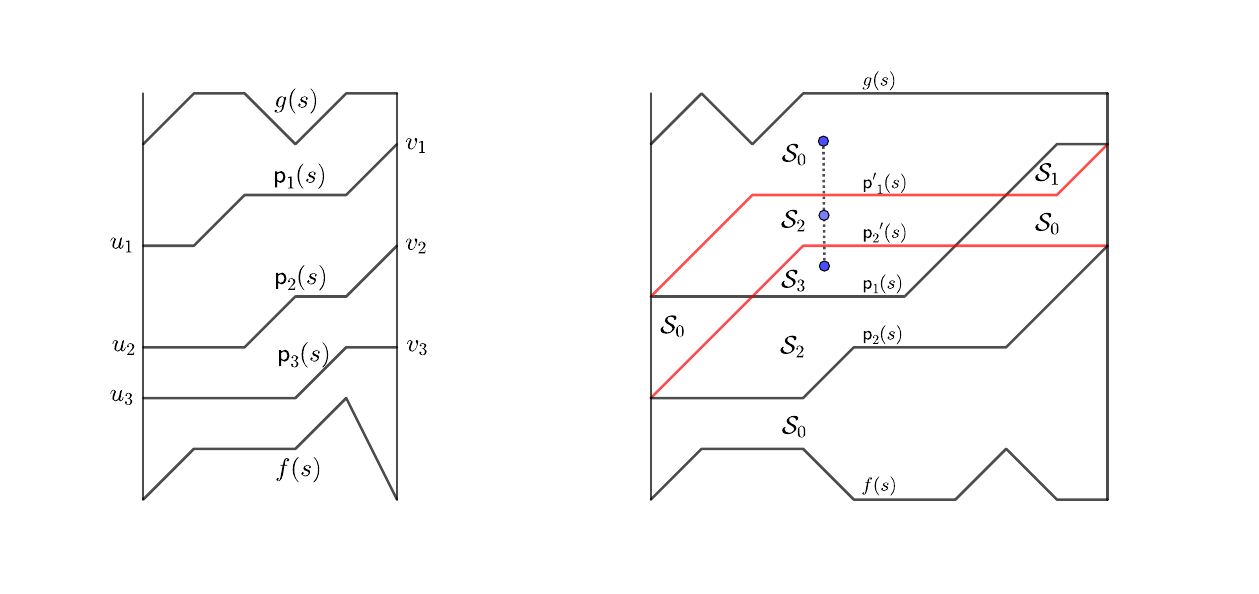}

\caption{Shown to the left is a discrete path ensemble ${\bm{\mathsf p}}=({\mathsf p}_1, {\mathsf p}_2, {\mathsf p}_3)$. Shown to the right is a level set decomposition $\mathscr{D} =\mathcal S_0\cup \mathcal{S}_1 \cup \cdots \cup \mathcal S_3$; there, $(3,2,0)$ is an exit sequence of  $\mathcal{S}_3$.}
\label{f:Heightconcentration}
	\end{figure}

Denoting $\bm{\mathsf{p}}(t) = \big( \mathsf{p}_1 (t), \mathsf{p}_2 (t), \ldots , \mathsf{p}_k (t) \big)$ for any $t$, we additionally say that the ensemble $\bm{\mathsf{p}}$ is \emph{non-intersecting} if $\bm{\mathsf{p}} (t) \in \mathbb{W}_k$, for each $t \in \llbracket 0, \mathsf{S} \rrbracket$. Given a non-intersecting path ensemble $\bm{\mathsf{p}} = (\mathsf{p}_1, \mathsf{p}_2, \ldots , \mathsf{p}_k)$ on $\llbracket 0, \mathsf{S} \rrbracket$, we define its height function $\mathsf{H}^{\bm{\mathsf{p}}} : \llbracket 0, \mathsf{S} \rrbracket \times \mathbb{Z} \rightarrow \mathbb{Z}_{\ge 0}$ by, for any $(t, w) \in \llbracket 0, \mathsf{S} \rrbracket \times \mathbb{Z}$, setting 
\begin{flalign*}
	\mathsf{H}^{\bm{\mathsf{p}}} (t, w) = \# \big\{ j \in \llbracket 1, k \rrbracket : \mathsf{p}_j (t) > w \big\}.
\end{flalign*}	

Given the above notation, we can state the discrete variant of \Cref{concentrationbridge} that quickly implies it. Observe that, in the below concentration bound, the error estimates and probabilities do not depend on the size $\mathsf{S}$ of the domain; this is in contrast to analogous statements in the context of tilings (for example, in \cite[Theorem 21]{LSRT}, the error in the height function grows at least as $\mathsf{S}^{1/2}$), though the proof is similar.

\begin{lem}
	
	\label{concentrationpaths} 
	
	Let $n, \mathsf{S} \ge 1$ be integers; $r, B \ge 1$ be real numbers; $\bm{u}, \bm{v} \in \mathbb{W}_n \cap \mathbb{Z}^n$ be two $n$-tuples of integers; and $f, g : \llbracket 0, \mathsf{S} \rrbracket \rightarrow \mathbb{Z}$ be functions with $f \le g$. Let $\bm{\mathsf{p}} = (\mathsf{p}_1, \mathsf{p}_2, \ldots , \mathsf{p}_n)$ denote a uniformly random discrete path ensemble conditioned to start at $\bm{u}$; end at $\bm{v}$; remain above $f$ and below $g$; and to be non-intersecting (assuming at least one such ensemble exists). Fixing $(t, w) \in \llbracket 0, \mathsf{S} \rrbracket \times \mathbb{Z}$, and denoting the event $\mathscr{E} (\bm{\mathsf{p}}) = \big\{ \mathsf{H}^{\bm{\mathsf{p}}} (t, w) \le B \big\}$, there exists a deterministic real number $\mathfrak{Y} = \mathfrak{Y} (\bm{u}; \bm{v}; f; g; t; w; B) \in \llbracket 0, n \rrbracket$ such that 
	\begin{flalign*}
		\mathbb{P} \Big[ \big| \mathsf{H}^{\bm{\mathsf{p}}} (t, w) - \mathfrak{Y} \big| \ge r B^{1/2} \Big] \le 2 e^{-r^2/4} + 2 \cdot \mathbb{P} \big[ \mathscr{E} (\bm{\mathsf{p}})^{\complement} \big].
	\end{flalign*}

\end{lem}

\begin{proof}
	
	This proof follows the discussion in \cite[Section 3.2]{UERT}. Let $\bm{\mathsf{p}}' = (\mathsf{p}_1', \mathsf{p}_2', \ldots , \mathsf{p}_n')$ denote a non-intersecting discrete path ensemble with the same law as, but independent from, $\bm{\mathsf{p}}$. Denote $\mathscr{D} = \mathscr{D}(\mathsf{S}) = \llbracket 0, \mathsf{S} \rrbracket \times \mathbb{Z} \subset \mathbb{Z}^2$; for any subset $\mathcal{S} \subseteq \mathscr{D}$, the boundary $\partial \mathcal{S}$ of $\mathcal{S}$ denotes the set of vertices $(s, y) \in \mathcal{S}$ that are adjacent\footnote{Here, two vertices $(s, y)$ and $(s', y')$ are adjacent if $(s-s', y-y') \in \big\{ (0, 1), (1, 0), (1, 1), (0, -1), (-1, 0), (-1, -1) \big\}$ (in this way, we view vertices as on a triangular lattice, instead of on a square one).} to a vertex in $\mathbb{Z}^2 \setminus \mathcal{S}$. Define the function $F : \mathscr{D} \rightarrow \mathbb{Z}$ by setting $F (s, y) = \mathsf{H}^{\bm{\mathsf{p}}} (s, y) - \mathsf{H}^{\bm{\mathsf{p}}'} (s, y)$, for each $(s, y) \in \mathscr{D}$. 
	
	A family $\bm{\mathcal{S}} = (\mathcal{S}_0, \mathcal{S}_2, \ldots , \mathcal{S}_k)$ of mutually disjoint subsets of $\mathscr{D}$ satisfying $\bigcup_{j=0}^k \mathcal{S}_j = \mathscr{D}$ is a \emph{level set decomposition} of $\mathscr{D}$ with respect to $F$ if the following two properties hold. First, for any two vertices $(s, y), (s', y') \in \mathcal{S}_j$ of the same subset, $F(s, y) = F(s', y')$. Second, for any two adjacent vertices $(s, y) \in \mathcal{S}_j$ and $(s', y') \in \mathcal{S}_{j'}$ with $j \ne j'$ (namely, in different subsets), $F(s, y) \ne F(s', y')$; observe in this case that $\big| F (s,y) - F(s', y') \big| = 1$, as we either have that both $\mathsf{H}^{\bm{\mathsf{p}}} (s, y) - \mathsf{H}^{\bm{\mathsf{p}}} (s', y') \in \{ 0, 1 \}$ and $\mathsf{H}^{\bm{\mathsf{p}}'} (s, y) - \mathsf{H}^{\bm{\mathsf{p}}'} (s', y') \in \{ 0, 1 \}$ hold or that both $\mathsf{H}^{\bm{\mathsf{p}}} (s, y) - \mathsf{H}^{\bm{\mathsf{p}}} (s', y') \in \{ 0, -1 \}$ and $\mathsf{H}^{\bm{\mathsf{p}}'} (s, y) - \mathsf{H}^{\bm{\mathsf{p}}'} (s', y') \in \{ 0, - 1\}$ hold. We set $F(\mathcal{S}_j) = F(s, y)$, for any $(s, y) \in \mathcal{S}_j$. In the below, we condition on the level set decomposition $\bm{\mathcal{S}}$ of $\mathscr{D}$ with respect to $F$.
	
	Further observe that $F(s, y) = 0$ for any $(s, y) \in \llbracket 0, \mathsf{S} \rrbracket \times \mathbb{Z}$ that either satisfies $(s, y) \in \partial \mathscr{D}$ (since $\bm{\mathsf{p}}$ and $\bm{\mathsf{p}}'$ both start and end at the same $n$-tuples $\bm{u}$ and $\bm{v}$, respectively), or that satisfies $y < f(s)$ (as then $\mathsf{H}^{\bm{\mathsf{p}}} (s, y) = n = \mathsf{H}^{\bm{\mathsf{p}}'} (s, y)$) or $y \ge g(s)$ (as then $\mathsf{H}^{\bm{\mathsf{p}}} (s, y) = 0 = \mathsf{H}^{\bm{\mathsf{p}}'} (s, y)$). Since this family of vertices is connected, there is one infinite set in $\bm{\mathcal{S}}$, which we let be $\mathcal{S}_0$. It has the property that $\mathscr{D} \setminus \mathcal{S}_0$ is finite, so the number of elements $k+1$ in $\bm{\mathcal{S}}$ is finite.
	
	Next, given $i, j \in \llbracket 1, n \rrbracket$ with $i \ne j$, we say that $\mathcal{S}_i$ is \emph{adjacent} to $\mathcal{S}_j$ if there exists a vertex in $\mathcal{S}_i$ that is adjacent to a vertex in $\mathcal{S}_j$. In this case, we say that $\mathcal{S}_i$ is \emph{exterior adjacent} to $\mathcal{S}_j$ if $\mathcal{S}_i$ is contained in the infinite connected component of $\mathscr{D} \setminus \mathcal{S}_j$; we otherwise call $\mathcal{S}_i$ \emph{interior adjacent} to $\mathcal{S}_j$. In particular, $\mathcal{S}_0$ is exterior adjacent to any $\mathcal{S}_j$ adjacent to it. It is quickly verified (see \cite[Lemma 12]{UERT}) that, for every $j \in \llbracket 1, k \rrbracket$, there is at most one element of $\bm{\mathcal{S}}$ that is exterior adjacent to $\mathcal{S}_j$. See the right side of \Cref{f:Heightconcentration} for a depiction.

		An \emph{exit sequence} of any $\mathcal{S}_j \in \bm{\mathcal{S}}$ is an ordered sequence $(i_0, i_1, \ldots , i_m)$ of integers in $\llbracket 0, k \rrbracket$ so that $i_0 = j$, $i_m = 0$, and $\mathcal{S}_{i_a}$ is exterior adjacent to $\mathcal{S}_{i_{a-1}}$, for each $a \in \llbracket 1, m \rrbracket$; here, $m = m(j) = m(j; \bm{\mathcal{S}})$ depends on $j$ and $\bm{\mathcal{S}}$. Then, it is quickly verified (see the proof of \cite[Lemma 14]{UERT}) that $F(\mathcal{S}_j)$ has the same law as $\sum_{a=1}^m \xi_a$, where the $\xi_a$ are mutually independent $\{ -1, 1 \}$-Bernoulli random variables with $\mathbb{P} [\xi_a = 1] =\frac{1}{2} = \mathbb{P} [\xi_a = -1]$. Hence, by Chernoff's inequality, 
	\begin{flalign}
		\label{fsjfsjs} 
		\mathbb{P} \bigg[ \Big| F(\mathcal{S}_j) - \mathbb{E} \big[ F(\mathcal{S}_j) \big] \Big| \ge rm^{1/2} \bigg| \bm{\mathcal{S}} \bigg] \le 2 e^{-r^2/2}. 
	\end{flalign}

	Now fix $j \in \llbracket 0, m \rrbracket$ to be such that $(t, w) \in \mathcal{S}_j$. Observe on the event $\mathscr{E} = \mathscr{E} (\bm{\mathsf{p}}) \cap \mathscr{E} (\bm{\mathsf{p}}') = \big\{ \mathsf{H}^{\bm{\mathsf{p}}} (t, w) \le B \big\} \cap \big\{ \mathsf{H}^{\bm{\mathsf{p}}'} (t, w) \le B \big\}$ that $m(j; \bm{\mathcal{S}}) \le 2B$. Indeed, define a sequence $(w_0, w_1, \ldots , w_m)$ by first setting $w_0 = w$ and then defining $w_{i+1}$ for $i \in \llbracket 0, m - 1 \rrbracket$ inductively as follows. Let $j_i \in \llbracket 0, k \rrbracket$ be such that $(t, w_i) \in \mathcal{S}_{j_i}$ (so $j_0 = j$), and let $w_{i+1} > w_i$ be the maximum integer such that $(t, w_{i+1}) \notin \mathcal{S}_{j_i}$ but $(t, w_{i+1} - 1) \in \mathcal{S}_{j_i}$ (assuming $j_i \ne 0$; if instead $j_i=0$ then  the sequence stops). Then, $(j_0, j_1, \ldots , j_m)$ is an exit sequence for $\mathcal{S}_j$. Moreover, for each $i \in \llbracket 0, m - 1 \rrbracket$, we either have $\mathsf{H}^{\bm{\mathsf{p}}} (t, w_{i+1}) = \mathsf{H}^{\bm{\mathsf{p}}} (t, w_i) - 1$ or $\mathsf{H}^{\bm{\mathsf{p}}'} (t, w_{i+1}) = \mathsf{H}^{\bm{\mathsf{p}}'} (t, w_i) - 1$, since $F (t, w_{i-1}) \ne F (t, w_i)$ and since $\mathsf{H}^{\bm{\mathsf{p}}} (t, y)$ and $\mathsf{H}^{\bm{\mathsf{p}}'} (t, y)$ are nonincreasing in $y$. On the event $\mathscr{E}$, at most $B$ such decrements can occur for $\mathsf{H}^{\bm{\mathsf{p}}}$ and for $\mathsf{H}^{\bm{\mathsf{p}}'}$, meaning that $m \le 2B$.
	
	Combining this estimate $\textbf{1}_{\mathscr{E}} \cdot m(j; \bm{\mathcal{S}}) \le 2B$ with \eqref{fsjfsjs} and a union bound, we obtain 
	\begin{flalign*}
		\mathbb{P} \bigg[ \Big| F(\mathcal{S}_j) - \mathbb{E} \big[ F(\mathcal{S}_j) \big] \Big| \ge r(2B)^{1/2} \bigg] & \le \mathbb{P} \bigg[ \textbf{1}_{\mathscr{E}} \cdot \Big| F(\mathcal{S}_j) - \mathbb{E} \big[ F(\mathcal{S}_j) \big] \Big| \ge r(2B)^{1/2} \bigg| \bm{\mathcal{S}} \bigg] + \mathbb{P} \big[ \mathscr{E}^{\complement} \big] \\
		& \le 2 e^{-r^2/2} + 2 \cdot \mathbb{P} \big[ \mathscr{E} (\bm{\mathsf{p}})^{\complement} \big],
	\end{flalign*}
	
	\noindent which gives the lemma.
\end{proof} 

\begin{proof}[Proof of \Cref{concentrationbridge}]
	
	The second statement follows from the first, since $\mathsf{H} (t, w) \le n$ holds deterministically. The first statement follows from \Cref{concentrationpaths} and the fact that non-intersecting random discrete paths converge to non-intersecting Brownian motions (namely, the $\big( T^{1/2} \mathsf{S}^{-1/2} \cdot \mathsf{p}_j ( \lfloor s T^{-1}\mathsf{S} \rfloor) \big)$ from \Cref{concentrationpaths} over $(j, s) \in \llbracket 1, n \rrbracket \times [0, T]$ converges in law to the $\big(x_j (s) \big)$ over $(j, s) \in \llbracket 1, n \rrbracket \times [0, T]$ from \Cref{concentrationbridge}).
\end{proof}

\section{Proofs of Results From \Cref{Estimates}}

\label{Proof2}

\subsection{Proof of \Cref{estimatexj3}}

\label{ProofContinuous}

In this section we establish \Cref{estimatexj3}. We begin with the following lemma that bounds non-intersecting Brownian bridges constrained to lie below a linear lower boundary; see the left side of \Cref{f:coupling3} for a depiction. In the below, we let $\bm{0}_n = (0, 0, \ldots , 0) \in \overline{\mathbb{W}}_n$ (where $0$ appears with multiplicity $n$).

	\begin{figure}
	\center
\includegraphics[width=0.8\textwidth]{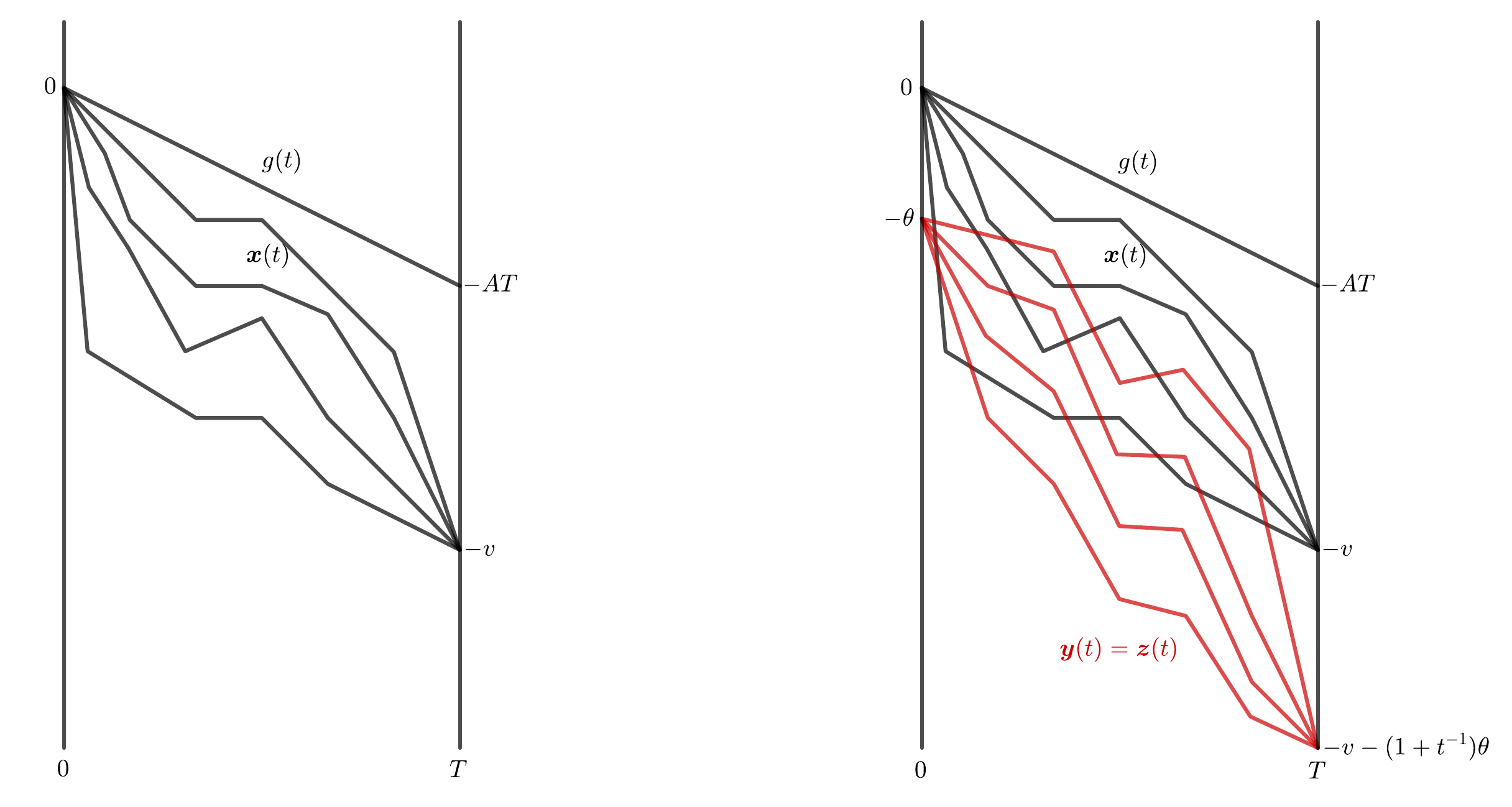}

\caption{The left panel depicts \Cref{xjflinear}. The right panel depicts its proof by coupling $\bm{x}$ with the non-intersecting Brownian bridges $\bm{z}$. }
\label{f:coupling3}
	\end{figure}
	
\begin{lem}
	
	\label{xjflinear}
	
	There exist constants $c > 0$ and $C > 1$ such that the following holds. Fix an integer $n \ge 1$; real numbers $A \ge 0$, $T > 0$, and $v \ge A T$; and define $g : [0, T] \rightarrow \mathbb{R}$ by setting $g(x) = -Ax$, for each $x \in [0, T]$. Let $\bm{v} = (-v, -v, \ldots , -v) \in \overline{\mathbb{W}}_n$ (where $-v$ appears with multiplicity $n$), and sample non-intersecting Brownian bridges $\bm{x} = (x_1, x_2, \ldots , x_n) \in \llbracket 1, n \rrbracket \times \mathcal{C} \big( [0, T] \big)$ under the measure $\mathfrak{Q}_{-\infty; g}^{\bm{0}_n; \bm{v}}$. Then, for any real numbers $t \in [0, T]$ and $B \ge 1$, we have
	\begin{flalign*}
		\mathbb{P} \Big[ x_n (t)   \le -vt T^{-1} - B t^{1/2} \log |9 t^{-1} T| \Big] \le C e^{Cn - cB^2 n}.
	\end{flalign*}
\end{lem}

\begin{proof}
	
	In view of the invariance of non-intersecting Brownian bridges under affine transformations and diffusive scaling (recall \Cref{linear} and \Cref{scale}), we may assume that $T = 1$. Then, set $\theta = \frac{B}{8} \cdot t^{1/2} \log |9 t^{-1}|$ and $v' = (t^{-1} + 1) \theta + v$; denote $\bm{u} = (-\theta, -\theta, \ldots , -\theta) \in \overline{\mathbb{W}}_n$ and $\bm{v}' = (-v', -v', \ldots , -v') \in \overline{\mathbb{W}}_n$. Sample non-intersecting Brownian bridges $\bm{y} = (y_1, y_2, \ldots , y_n) \in \llbracket 1, n \rrbracket \times \mathcal{C} \big( [0, T] \big)$ and $\bm{z} = (z_1, z_2, \ldots , z_n) \in \llbracket 1, n \rrbracket \times \mathcal{C} \big( [0, T] \big)$ from the measures $\mathfrak{Q}_{-\infty; g}^{\bm{u}; \bm{v}'}$ and $\mathfrak{Q}^{\bm{u}; \bm{v}'}$, respectively; see the right side of \Cref{f:coupling3}. By \Cref{uvv}, we may couple $\bm{x}$ and $\bm{y}$ so that 
	\begin{flalign}
		\label{yjsxjstheta} 
		\big| y_j (s) - x_j (s) \big| \le (t^{-1}s + 1) \theta , \qquad \text{for each $(j, s) \in \llbracket 1, n \rrbracket \times [0, 1]$}.
	\end{flalign}

	Now, let us compare $\bm{y}$ and $\bm{z}$. To that end, first observe from \Cref{estimatexj2} that there exist constants $c \in (0, 1)$ and $C_1 > 1$ such that
	\begin{flalign}
		\label{zjsat0} 
		\begin{aligned}
			& \mathbb{P} \Bigg[ \bigcap_{j=1}^n \bigcap_{s \in [0, 1]} \Big\{  z_j (s) \ge - \displaystyle\frac{B}{8} \cdot s^{1/2} \log |9 s^{-1}|  - (t^{-1} s + 1) \theta - vs \Big\} \Bigg] \ge 1 - C_1 e^{C_1 n - c B^2 n}; \\
			& \mathbb{P} \Bigg[ \bigcap_{j=1}^n \bigcap_{s \in [0, 1]} \Big\{  z_j (s) \le \displaystyle\frac{B}{8} \cdot s^{1/2} \log |9 s^{-1}| - (t^{-1} s + 1) \theta - vs \Big\} \Bigg] \ge 1 - C_1 e^{C_1 n - c B^2 n}.
		\end{aligned} 
	\end{flalign}
	
	\noindent Observe that $\frac{B}{8} \cdot s^{1/2} \log |9 s^{-1}| - (t^{-1} s + 1) \theta - vs \le - As$, for each $s \in [0, 1]$. Indeed, if $s \le t$, this follows from the facts that $v \ge A$ and $t^{1/2} \log |9 t^{-1}| \ge s^{1/2} \log |9 s^{-1}|$ when $0 < s \le t \le 1$; for $s > t$, this follows from the facts that $v \ge A$ and $t^{-1/2} \log |9t^{-1}| \ge s^{-1/2} \log |9s^{-1}|$ when $0 < t \le s \le 1$. Hence, the second statement of \eqref{zjsat0} indicates that $z_j (s)$ remains below $g$ with probability at least $1 - C_1 e^{C_1 n - c B^2}$. It follows that we may couple $\bm{z}$ and $\bm{y}$ to coincide on an event of this probability $1 - C_1 e^{C_1 n - c B^2}$. Together with \eqref{yjsxjstheta}, the first part of \eqref{zjsat0}, and a union bound, this gives
	\begin{flalign*}
		\mathbb{P} \Big[ x_n (t) \ge - 4\theta - vt - \displaystyle\frac{B}{8} \cdot t^{1/2} \log |9t^{-1}| \Big] \ge 1 - 2C_1 e^{C_1 n - c B^2 n},
	\end{flalign*} 
	
	\noindent which yields the lemma, since $\theta = \frac{B}{8} \cdot t^{1/2} \log |9t^{-1}|$. 
\end{proof}

We then deduce the following variant of \Cref{estimatexj3} that proves high probability H\"{o}lder bounds for a single choice of $(j, s, t)$, which quickly implies \Cref{estimatexj3}. We detail the proof of the first part and outline that of the second, which is similar.

\begin{lem}
	
	\label{2estimatex} 
	
	Adopt the notation of \Cref{estimatexj3}.
	
	\begin{enumerate}
		\item \label{xjts1} Under the assumption of part \ref{afg} of \Cref{estimatexj3}, we have
		\begin{flalign*}
			\mathbb{P} \Big[ \big| x_j (t+s) - x_j (t) - sT^{-1}  (v_j - u_j) \big| \ge s^{1/2} \big( B \log | 2s^{-1} T | + \kappa^{-1} (A+B) \big) \Big]  \le C e^{Cn^2 - cB^2 n}. 
		\end{flalign*}
		
		\item \label{xjts2} Under the assumption of part \ref{afg2} of \Cref{estimatexj3}, we have
		\begin{flalign*}
			\mathbb{P} \Big[ \big| x_j (t + s) - x_j (t) - sT^{-1} (v_j - u_j) \big| \ge s^{1/2} \big( B \log |9s^{-1} T| + 2 \big( A + B) \big) \Big] \le C e^{Cn^2 - cB^2 n}.
		\end{flalign*}
	\end{enumerate}
\end{lem}

\begin{proof}[Proof of \Cref{2estimatex}(\ref{xjts1})]
	
	By \Cref{scale} and \Cref{linear}, we may assume that $(a, b) = (0, 1)$ and that $(u_j, v_j) = (0, 0)$. Since $(sT^{-1})^{1/2} \le 1$ (as $s \le T$), it suffices to show that 
	\begin{flalign}
		\label{xj12} 
		\begin{aligned}
			& \mathbb{P} \big[ x_j (t+s) - x_j (t) \ge Bs^{1/2} \log |2s^{-1}| + \kappa^{-1} s (A+B) \big] \le C e^{Cn^2 - cB^2 n}; \\
			& \mathbb{P} \big[  x_j (t+s) - x_j (t) \le  - Bs^{1/2} \log |2s^{-1}| - \kappa^{-1} s (A+B) \big] \le C e^{Cn^2 - cB^2 n}. 
		\end{aligned} 
	\end{flalign}
	
	\noindent We only establish the first bound in \eqref{xj12}, as then that of the second would follow by symmetry (namely, by reflecting the bridges into the line $\big\{ t = \frac{1}{2} \big\}$). To that end, we first claim that there exist constants $c_1 \in (0, 1)$ and $C_1 > 1$ such that
	\begin{flalign}
		\label{xjrab}
		\mathbb{P} \Bigg[ \displaystyle\sup_{r \in [0, 1]} x_j (r) \le A + B \Bigg] \ge 1 - C_1 e^{C_1 n - c_1 B^2 n}.
	\end{flalign}
	
	\noindent To see this, define $\bm{u}', \bm{v}' \in \overline{\mathbb{W}}_n$ by setting 
	\begin{flalign*} 
		u_k' = u_k + \displaystyle\frac{B}{2} + A; \qquad v_k' = v_k + \displaystyle\frac{B}{2} + A,
	\end{flalign*} 
	
	\noindent for each $k \in \llbracket 1, n \rrbracket$. Then, sample non-intersecting Brownian bridges $\bm{y} = (y_1, y_2, \ldots , y_n) \in \llbracket 1, n \rrbracket \times \mathcal{C} \big( [0, 1] \big)$ and $\bm{y}' = (y_1', y_2', \ldots , y_n') \in \llbracket 1, n \rrbracket \times \mathcal{C} \big( [0, 1] \big)$ according to the measures $\mathfrak{Q}_f^{\bm{u}'; \bm{v}'}$ and $\mathfrak{Q}^{\bm{u}'; \bm{v}'}$, respectively. Then, by \Cref{monotoneheight}, we may couple $\bm{x}$ and $\bm{y}$ so that 
	\begin{flalign}
		\label{xjyjr01}
		x_k (r) \le y_k (r), \qquad \text{for each $(k, r) \in \llbracket 1, n \rrbracket \times [0, 1]$}.
	\end{flalign}
	
	\noindent Moreover, by \Cref{estimatexj2}, there exist constants $c_1 \in (0, 1)$ and $C_2 > 1$ such that
	\begin{flalign}
		\label{ynyj} 
		\begin{aligned}
			& \mathbb{P} \Bigg[ \displaystyle\inf_{r \in [0, 1]} y_n' (r) \ge \displaystyle\sup_{r \in [0, 1]} f(r) + \displaystyle\frac{B}{4} \Bigg] \ge 1 - C_2 e^{C_2 n - c_1 B^2 n}; \\
			& \mathbb{P} \Bigg[ \displaystyle\sup_{r \in [0, 1]} y_j' (r) \le A + B \Bigg] \ge 1 - C_2 e^{C_2 n - c_1 B^2 n},
		\end{aligned} 
	\end{flalign}
	
	\noindent where in the first inequality we used that $\min \{ u_n', v_n' \} = \min \{ u_n, v_n \} + A + \frac{B}{2} \ge \sup_{r \in [0, 1]} f(r) + \frac{B}{2}$, and in the second inequality we used the fact that $u_j' = A + \frac{B}{2} = v_j'$ (recalling that $u_j=v_j=0$).
	
	The first bound in \eqref{ynyj} indicates that $\bm{y}'$ lies above $f$ with probability at least $1 - C_2 e^{C_2 n - c_1 B^2 n}$. Hence, we may couple $\bm{y}$ and $\bm{y}'$ so that they coincide with probability at least $1 - C_2 e^{C_2 n - c_1 B^2 n}$. This, together with the second bound in \eqref{ynyj}, \eqref{xjyjr01}, and a union bound then yields \eqref{xjrab}.

	\begin{figure}
	\center
\includegraphics[width=0.8\textwidth]{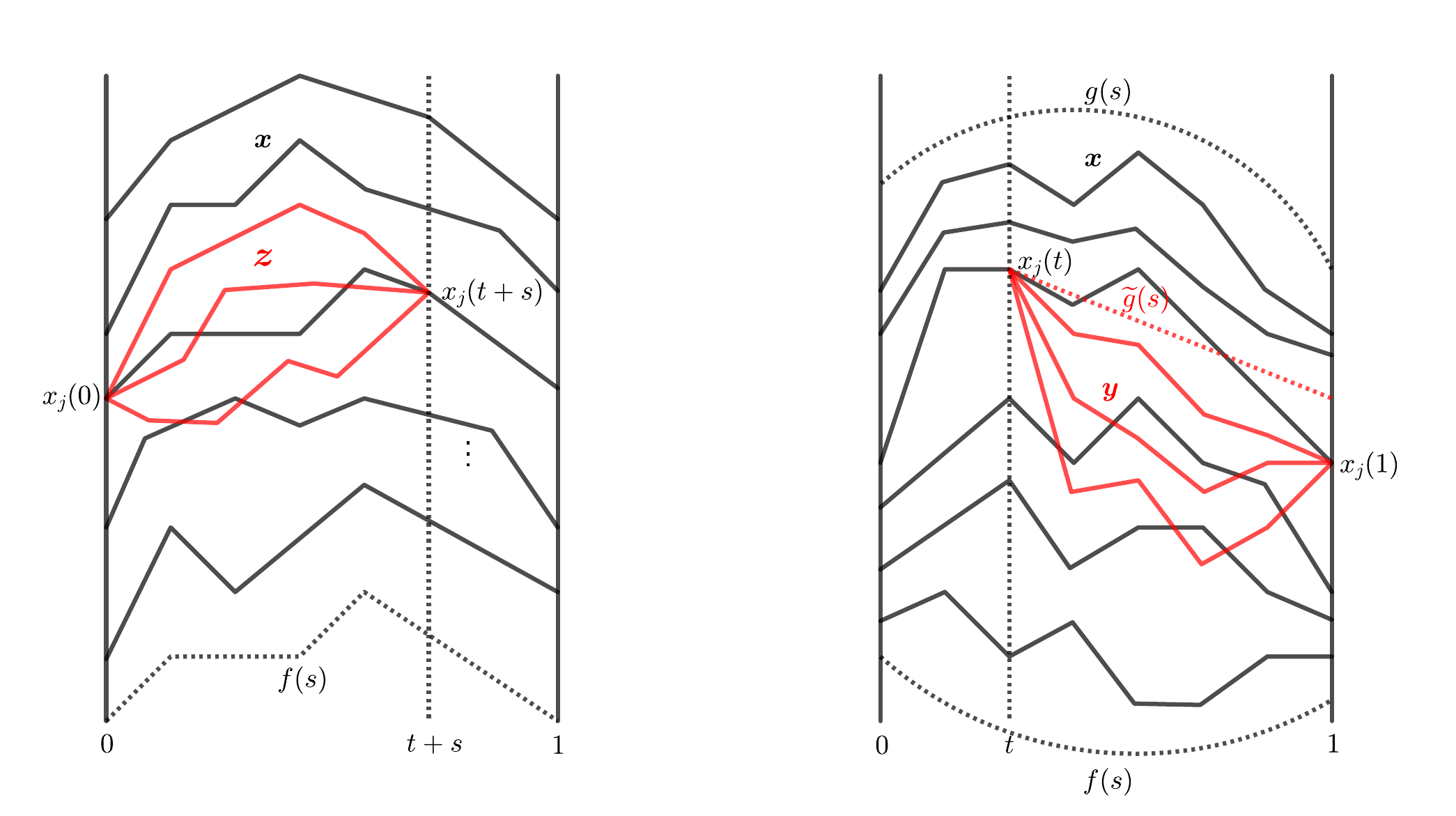}

\caption{The left panel depicts the proof of \Cref{2estimatex}(\ref{xjts1}). The right panel depicts that of \Cref{2estimatex}(\ref{xjts2}).}
\label{f:Holder2}
	\end{figure}

	Now, to show the first bound in \eqref{xj12}, we condition on $x_k (r)$ for $(k, r) \in \llbracket j+ 1, n \rrbracket \times [0, 1]$ and for $(k, r) \in \llbracket 1, n \rrbracket \times [t+s, 1]$, and set $\bm{u}'', \bm{v}'' \in \overline{\mathbb{W}}_{n-j+1}$ by letting $u_k'' = x_j (0) = 0$ and $v_k'' = x_j (t+s)$, for each $k \in \llbracket 1, j \rrbracket$. Sample non-intersecting Brownian bridges $\bm{z} = (z_1, z_2, \ldots , z_j) \in \llbracket 1, j \rrbracket \times \mathcal{C} \big( [0, t+s] \big)$ from the measure $\mathfrak{Q}^{\bm{u}''; \bm{v}''} (n^{-1})$. Then, \Cref{monotoneheight} yields a coupling between $\bm{x}$ and $\bm{z}$ such that $x_k (r) \ge z_k (r)$, for each $(k, r) \in \llbracket 1, j \rrbracket \times [0, 1]$; see the left side of \Cref{f:Holder2}. Thus,
	\begin{flalign*}
		\mathbb{P} \bigg[ x_j (t) \le x_j (& t+s) - B s^{1/2} \log |2s^{-1}| - \displaystyle\frac{s(A+B)}{\kappa} \bigg] \\
		& \le \mathbb{P} \bigg[ z_j (t) \le z_j (t + s) - B s^{1/2} \log |2s^{-1}|  - \displaystyle\frac{s(A+B)}{\kappa} \bigg] \le 2 C_1 e^{C_1 n - c_1 B^2 n},
	\end{flalign*}
	
	\noindent where in the last equality we used \Cref{estimatexj2} and the facts that $s + t \ge t \ge \kappa$ and that
	\begin{flalign*} 
		\displaystyle\frac{v_j'' - u_j''}{s+t} = (s+t)^{-1} \cdot  z_j (t + s)  \le (s+t)^{-1} \cdot \max_{r \in [0, 1]} x_j (r) \le \displaystyle\frac{A+B}{\kappa},
	\end{flalign*}
	
	\noindent on the event $\big\{ \max_{r \in [0, 1]} x_j (r) \le A +B \big\}$. This, with \eqref{xjrab}, establishes the first bound in \eqref{xj12} and thus the first statement of the lemma. 
\end{proof} 

\begin{proof}[Proof of \Cref{2estimatex}(\ref{xjts2})]
	
	Since the proof of the second part of the lemma is similar to that of the first part, we only outline it. We may again assume that $(a, b) = (0, 1)$ and $(u_j, v_j) = (0, 0)$. We may further assume by symmetry that $t \le \frac{1}{2}$. It suffices (again since $(sT^{-1})^{1/2} \le 1$) to show that 
	\begin{flalign}
		\label{xjtsxjt}
		\begin{aligned}
			& \mathbb{P} \big[ x_j (t+s) - x_j (t) \ge Bs^{1/2} \log |9s^{-1}| + 2 s (A+B) \big] \le C e^{Cn^2 - cB^2 n}; \\
			& \mathbb{P} \big[  x_j (t+s) - x_j (t) \le  - Bs^{1/2} \log |9s^{-1}| - 2 s (A+B) \big] \le C e^{Cn^2 - cB^2 n},
		\end{aligned} 
	\end{flalign}
	
	\noindent and by symmetry it suffices to show the latter. Following the derivation of \eqref{xjrab}, we again find that there exist constants $c_1 \in (0, 1)$ and $C_1 > 1$ such that
	\begin{flalign}
		\label{xjrab2}
		\mathbb{P} \Bigg[ \displaystyle\sup_{r \in [0, 1]} x_j (r) \le A + B \Bigg] \ge 1 - C_1 e^{C_1 n - c_1 B^2 n}.
	\end{flalign}
	
	Now condition on $x_k (r)$ for $(k, r) \in \llbracket j+1,n \rrbracket \times [0, 1]$ and for $(k, r) \in \llbracket 1, n \rrbracket \times [0, t]$;  set $\widetilde{g}: [t,1] \rightarrow \mathbb{R}$ by letting $\widetilde{g} (r) = x_j(t) + (t-r) A$, for each $r \in [t, 1]$; and set $\bm{u}', \bm{v}' \in \overline{\mathbb{W}}_j$ by letting $u_k' = x_j (t)$ and $v_k' =\min\{ x_j (1), \widetilde{g}(1)\}  = \min\{0, \widetilde{g}(1)\}$ (where we used that $x_j(1)=v_j=0$)  for each $k \in \llbracket 1, j \rrbracket$.  Then sample non-intersecting Brownian bridges $\bm{y} = (y_1, y_{2}, \ldots , y_j) \in \llbracket 1,j \rrbracket \times \mathcal{C} \big( [t, 1] \big)$ from the measure $\mathfrak{Q}_{-\infty; \tilde{g}}^{\bm{u}'; \bm{v}'} (n^{-1})$; see the right side of \Cref{f:Holder2}. Since $\big| \partial_t g(r) \big| \le A$ for each $r \in [t, 1]$, and since $g(t) \ge x_j (t) = u_j'$, we have $g \ge \widetilde{g}$. Thus, \Cref{monotoneheight} yields a coupling between $\bm{x}$ and $\bm{y}$ such that 
	\begin{flalign} 
		\label{xryr2} 
		x_k (r) \ge y_k (r), \qquad \text{for each $(k, r) \in \llbracket 1, j \rrbracket \times [t, 1]$}.
	\end{flalign}
	
	Applying \Cref{xjflinear}  to $\bm{y}$ yields constants $c_1 \in (0, 1)$ and $C_2 > 1$ such that
	\begin{flalign}
		\label{yjtsuv}
		\mathbb{P} \bigg[ y_j (t+s) \ge y_j (t) - \displaystyle\frac{(u_j' -v_j')s}{1-t} - Bs^{1/2} \log |9s^{-1}| \bigg] \ge 1 - C_2 e^{C_2 n - c_1 B^2 n}.
	\end{flalign}
	
	\noindent If $v_j'=\widetilde{g}(1)=u_j'-A(1-t)$, then $u_j'-v_j'= A(1-t)\leq A$. If instead $v_j'=0$, and  \eqref{xjrab2} implies 
	$\mathbb{P} [ u_j' \le A + B ] \ge 1 - C_1 e^{C_1 n - c_1 B^2 n}$; since $t \le \frac{1}{2}$, we have in either case that
	\begin{flalign*}
		\mathbb{P} \bigg[ \displaystyle\frac{u_j'-v_j'}{1-t} \le 2(A+B) \bigg] \ge 1 - C_1 e^{C_1 n - cB^2 n}.
	\end{flalign*}
	
	\noindent Together with \eqref{yjtsuv}, \eqref{xryr2}, and a union bound, this gives the second bound in \eqref{xjtsxjt}, which implies the second statement of the lemma.
\end{proof}

\begin{proof}[Proof of \Cref{estimatexj3} (Outline)]
	
	The proof follows quickly from \Cref{2estimatex} and (the proof of) \cite[Lemma 3.3]{BPLE}, very similarly to the proof of \cite[Proposition 3.5]{BPLE} given \cite[Lemmas 3.2 and 3.3]{BPLE}, so further details are omitted.
\end{proof}

\subsection{Proofs of \Cref{rhotestimatek}, and \Cref{derivativetm}}

\label{ProofRho}

In this section we establish first \Cref{rhotestimatek}, and next \Cref{derivativetm}.

\begin{proof}[Proof of \Cref{rhotestimatek}]
	
	By \cite[Lemma 3]{FCSD}, we have $\big| m_t^{\mu} (z) \big| \le t^{-1/2}$ for any $z \in \overline{\Lambda}_{t; \mu}$. By \eqref{mrho}, this yields $\varrho_t (x) \le \pi^{-1} t^{-1/2}$, for any $x \in \mathbb{R}$, confirming the first statement of the lemma. 
	
	To verify the latter, fix $x \in \mathbb{R}$ and set $z \in \partial \Lambda_t$ such that $M(z) = z - tm_0 (z) = x \in \mathbb{R}$ (where such a $z$ exists by \Cref{mz}). Then, \eqref{mrho} and \eqref{mt} give 
	\begin{flalign} 
		\label{rhotm0z} 
		\varrho_t (x) = \pi^{-1} \Imaginary m_t \big( z - tm_0 (z) \big) = \pi^{-1} \Imaginary m_0 (z).
	\end{flalign} 
	
	\noindent Let $z_0 \in \partial \Lambda_t$ be such that $M(z_0) = x_0$. Repeatedly applying the relation $\partial_x = \big( \frac{\partial_z}{\partial_x} \big)^{-1} \partial_z = (1 - tm_0')^{-1} \partial_z$, it follows that there exists a constant $C = C (k) > 1$ such that
	\begin{flalign}
		\label{derivativekrho}
		\begin{aligned}
		\big| \partial_x^k \varrho_t (x_0) \big| & \le C \cdot  \displaystyle\max_{k' \in \llbracket 1, k \rrbracket} \big| \partial_z^{k'} m_0 (z_0) \big|^{2k/k'} \cdot \displaystyle\max_{k'' \in \llbracket 0, 2k \rrbracket} \big| 1  - tm_0' (z_0) \big|^{-k''}.
		\end{aligned}
	\end{flalign}
	
	\noindent To bound the first term on the right side, observe since $|z_0-y|\ge \Imaginary z_0$ for any $y \in \mathbb{R}$ that
	\begin{align}\label{e:b1}
		\big| \partial_z^k m_0(z_0) \big| \le k! \displaystyle\int_{-\infty}^{\infty} \frac{\varrho_0(y) dy}{|z_0 - y|^{k+1}}
		\le \frac{k!}{ (\Imaginary z_0)^{k-1}}\displaystyle\int_{-\infty}^{\infty} \frac{\varrho_0 (y) dy}{|z_0 - y|^{2}}
		=k! \cdot \frac{\Imaginary m_0(z_0)}{(\Imaginary z_0)^{k}}\le \frac{k!}{t (t\delta\pi)^{k-1}},
	\end{align}
	
	\noindent where in the last two statements we used the facts that 
	\begin{flalign}
		\label{m0z0z0} 
		\begin{aligned} 
			& \qquad \Imaginary m_0 (z_0) = \pi \varrho_t (x_0) \ge \pi \delta; \qquad \Imaginary z_0 = t \Imaginary m_0 (z_0); \\
			& \displaystyle\int_{-\infty}^{\infty} \displaystyle\frac{\varrho_0 (y) dy}{|z_0 - y|^2} = \displaystyle\frac{1}{\Imaginary z_0} \cdot \Imaginary \displaystyle\int_{-\infty}^{\infty} \displaystyle\frac{\varrho_0 (y) dy}{z_0 - y} = \displaystyle\frac{\Imaginary m_0 (z_0)}{\Imaginary z_0} = t^{-1},
		\end{aligned}
	\end{flalign} 
	
	\noindent Here, the first statement follows from \eqref{rhotm0z} and the assumption $\varrho_t (x_0) \ge \delta$; and the second follows from the fact that $M(z_0) = z_0 - tm_0 (z_0) = x_0 \in \mathbb{R}$; and the third follows from the definition of $m_0$.
	
	To estimate the last term on the right side of \eqref{derivativekrho}, observe that
	\begin{align}\label{e:b2}
		\big| 1-tm'_0(z_0) \big|&\ge 1-t\Real m'_0(z_0),
	\end{align}
	
	\noindent and that
	\begin{align*}
		\Real m'_0(z_0) = \Real \displaystyle\int_{-\infty}^{\infty} \frac{\varrho_0(y) dy}{(y-z_0)^2} & = \displaystyle\int_{-\infty}^{\infty} \big(|y-z_0|^2-2(\Imaginary z_0)^2 \big) \cdot \displaystyle\frac{\varrho_0(y)d y}{|y-z_0|^4} \\
		& =\frac{1}{t}-2 (\Imaginary z_0)^2\displaystyle\int_{-\infty}^{\infty} \frac{\varrho_0 (y) dy}{|y-z_0|^4},
	\end{align*}
	
	\noindent where in the last equality we used the third statement of \eqref{m0z0z0}. Inserting this into \eqref{e:b2}, we find
	\begin{align}\begin{split}\label{e:b3}
			\big| 1-tm'_0(z_0) \big| & \ge 2t (\Imaginary z_0)^2\displaystyle\int_{-\infty}^{\infty} \frac{\varrho_0(y)d y}{|y-z_0|^4}\ge 2t (\Imaginary z_0)^2\left(\displaystyle\int_{-\infty}^{\infty} \frac{\varrho_0(y)dy}{|y-z_0|^2}\right)^2\ge \frac{2 (\Imaginary z_0)^2}{t} \ge 2 t (\pi \delta)^2,
	\end{split}\end{align}
	
	\noindent where in the second statement we used the fact that $\varrho_0$ is a probability density; in the third we used the third statement of \eqref{m0z0z0}; and in the fourth we used the first two statements of \eqref{m0z0z0}. The lemma then follows from inserting \eqref{e:b1} and \eqref{e:b3} into \eqref{derivativekrho}.
\end{proof}

\begin{proof}[Proof of \Cref{derivativetm}]
	
	Denote the Stieltjes transform of $\mu_0$ by $m_0 = m_{\mu_0}$. We will use \eqref{derivativekrho} to estimate the $x$-derivatives of $\varrho_t$, to which end we must upper bound derivatives of $m_0$ and lower bound $|1 - t m_0'|$. To implement this, we claim that there exists a constant $C_1 = C_1 (k) > 1$ such that, for any integer $j \in \llbracket 0, k-1 \rrbracket$,
	\begin{flalign}
		\label{zkm0z}
		\begin{aligned}
			& \big| \partial_z^j m_0 (z) \big| \le C_1 B^{j+2} (\delta \mathfrak{c})^{-j-1}, \qquad \text{for any $z \in \overline{\mathbb{H}}$ such that $|\Real z| \le \displaystyle\frac{\mathfrak{c}}{2}$}; \\
			& \big| 1 - tm_0' (z) \big| \ge \displaystyle\frac{t \delta \mathfrak{c}}{2 \Imaginary z}, \qquad \qquad \quad \text{for any $z \in \overline{\mathbb{H}}$ such that $|\Real z| \le \displaystyle\frac{\mathfrak{c}}{2}$ and $\Imaginary z \in [0, \mathfrak{c}]$}.
		\end{aligned}
	\end{flalign}
	
	\noindent  To see this, fix $j \in \llbracket 0, k-1 \rrbracket$ and observe that \eqref{mz0} implies for $|\Real z| < \frac{\mathfrak{c}}{2}$ that
	\begin{flalign}
		\label{mzm} 
		\begin{aligned} 
			\big| \partial_z^j m_0 (z) \big|
			=j! \Bigg| \int_{-\infty}^{\infty} \frac{\varrho_0(x)d x}{(x-z)^{j+1}} \Bigg|
			& \le j! \Bigg| \int_{-\mathfrak{c}}^{\mathfrak{c}} \frac{\varrho_0(x)d x}{(x-z)^{j+1}} \Bigg| + 2^{j+1} j! \mathfrak{c}^{-j-1}  \\
			& \le  \bigg| \int_{-\mathfrak{c}}^{\mathfrak{c}} \frac{\partial_x^j \varrho_0 (x)d x}{x-z} \bigg|
			+ 2^{j+1} j! \mathfrak{c}^{-j-1} + 2^{j+1} (j+1)! B^{j+1} (\delta \mathfrak{c})^{-j}.
		\end{aligned} 
	\end{flalign} 
	
	\noindent Here, to obtain the second statement we used the facts that $|\Real z| \le \frac{\mathfrak{c}}{2}$ and that $\varrho_0$ is the density of a probability measure. To obtain the third, we integrated by parts $j$ times, using the estimate on the boundary terms, 
	\begin{flalign*} 
		\displaystyle\sum_{i=0}^j i! \bigg| \frac{\partial_x^{j-i} \varrho_0 (\mathfrak{c})}{(z - \mathfrak{c})^i} - \frac{ \partial_x^{j-i} \varrho_0 (-\mathfrak{c})}{(z + \mathfrak{c})^i} \bigg| \le \displaystyle\frac{2^j}{\mathfrak{c}^j} \displaystyle\sum_{i=0}^j i! \Big( \big| \partial_x^{j-i} \varrho_0 (\mathfrak{c}) \big| + \big| \partial_x^{j-i} \varrho_0 (-\mathfrak{c})\big| \Big) \le \displaystyle\frac{2^{j+1} (j+1)! B^{j+1}}{(\delta \mathfrak{c})^{j}},
	\end{flalign*} 
	
	\noindent where the first bound holds again since $|\Real z| \le \frac{\mathfrak{c}}{2}$ and the second holds since $\big| \partial_x^i \varrho_0 (x) \big| \le B^{i+1} \delta^{-i}$ for any $i \in \llbracket 0, k \rrbracket$. Then,
	\begin{flalign*}
		\big| \partial_z^j m_0 (z) \big| & \le \int_{-\mathfrak{c}}^{\mathfrak{c}} \bigg| \frac{\partial_x^j \varrho_0 (x) - \partial_x^j \varrho_0 (\Real z)}{z-x} \bigg| dx + \big| \partial_x^j \varrho_0 (\Real z) \big| \cdot \Bigg| \displaystyle\int_{-\mathfrak{c}}^{\mathfrak{c}} \displaystyle\frac{dx}{x-z} \Bigg| + \mathcal{O} \big( B^{j+1} (\delta \mathfrak{c})^{-j} \big) \\
		& \le \displaystyle\int_{-\mathfrak{c}}^{\mathfrak{c}} \bigg| \displaystyle\frac{\partial_x^j \varrho_0 (x) - \partial_x^j \varrho_0 (\Real z)}{x-z} \bigg| dx + \mathcal{O} \big( B^{j+1} (\delta \mathfrak{c})^{-j} \big) \\
		& \le 2 \mathfrak{c} \cdot \displaystyle\sup_{|x| \le \mathfrak{c}} \big| \partial_x^{j+1} \varrho_0 (x) \big| + \mathcal{O} \big( B^{j+1} \big( \mathfrak{c} \delta)^{-j-1} \big) \le \mathcal{O} \big( B^{j+2} (\mathfrak{c} \delta)^{-j-1} \big),
	\end{flalign*}
	
	\noindent where in the first and second estimates we used \eqref{mzm}, the bound $\big| \partial_x^j \varrho_0 (x) \big| \le B^j \delta^{-j}$ for $|x| \le \mathfrak{c}$ (as $|\Real z| \le \mathfrak{c}$), and the fact that $\big| \PV \int_{-\mathfrak{c}}^{\mathfrak{c}} (x-z)^{-1} dx \big| \le 2$ (as $|\Real z| \le \frac{\mathfrak{c}}{2}$); in the fourth, we used the same bound on derivatives of $\varrho_0$, but with $j$ replaced by $j+1$. Here, the implicit constant only depends on $j$, verifying the first statement of \eqref{zkm0z}.
	
	To verify the second, observe for $|\Real z| \le \frac{\mathfrak{c}}{2}$ and $\Imaginary z \in [0, \mathfrak{c}]$ that
	\begin{align*}
		\big| 1-tm'_0(z) \big|\ge 2t (\Imaginary z)^2\int_{-\infty}^{\infty} \frac{\varrho_0(x)d x}{|x-z|^4}& \ge 
		2t \delta (\Imaginary z)^2\int_{-\mathfrak{c}}^{\mathfrak{c}} \frac{ d x}{|x-z|^4} \\
		& \ge 2 t \delta (\Imaginary z)^2 \displaystyle\int_{-\mathfrak{c}/2}^{\mathfrak{c}/2} \displaystyle\frac{dx}{\big( |x| + \Imaginary z \big)^4} \\
		& = \displaystyle\frac{4}{3} \cdot t \delta (\Imaginary z)^2 \bigg( (\Imaginary z)^{-3} - \Big( \displaystyle\frac{\mathfrak{c}}{2} + \Imaginary z \Big)^{-3}  \bigg) 
		\ge \frac{t\delta \mathfrak{c}}{2 \Imaginary z},
	\end{align*}
	
	\noindent where in the first statement we used the first bound of in \eqref{e:b3}; in the second we used the fact that $\varrho_0 (x) \ge \delta$ for $|x| \le \mathfrak{c}$; in the third we changed variables from $x$ to $x - \Real z$ and used the bound $|\Real z |\le \frac{\mathfrak{c}}{2}$; in the fourth we performed the integration; and in the fifth we used the fact that $\Imaginary z \le \mathfrak{c}$. This verifies the second statement of \eqref{zkm0z}.
	
	To apply \eqref{zkm0z}, let $C_2 = C_2 (k) > 10$ denote a sufficiently large constant (to be determined later) and fix real numbers 
	\begin{flalign}
		\label{x0t2} 
		x_0, t \in \mathbb{R}, \qquad \text{such that} \quad |x_0| \le \frac{\mathfrak{c}}{4}, \quad \text{and} \quad t \in \Big[ 0, \frac{(\delta \mathfrak{c})^{20}}{C_2 B^{20}} \Big].
	\end{flalign}
	
	\noindent By \Cref{mz}, there exists a complex number $z_0 = z_0 (x_0, t) \in \partial \Lambda_t$ such that $z_0 - t m_0 (z_0) = x_0$. We claim that $|\Real z_0| \le \frac{\mathfrak{c}}{2}$. To see this, first observe that $z_0 = z_0 (x_0, t)$ is continuous in $t$. Indeed, since the map $M(z) = z - t m_0 (z)$ is a homeomorphism from $\overline{\Lambda}_t$ (recall \eqref{mtlambdat}) to $\overline{\mathbb{H}}$ by \Cref{mz} that is uniformly continuous in $t$, its inverse $M^{-1}$ is also continuous in $t$. Thus, $z_0 (x_0, t) = M^{-1} (x_0)$ is continuous in $t$. Now, assume to the contrary that $\big| \Real z_0 (x, t) \big| > \frac{\mathfrak{c}}{2}$. Then since $z_0 (x_0, 0) = x_0 \in \big[ -\frac{\mathfrak{c}}{4}, \frac{\mathfrak{c}}{4} \big]$, there exists some $s \in [0, t]$ such that $\big| \Real z_0 (x_0, s) \big| = \frac{\mathfrak{c}}{2}$. Since $|x_0| \le \frac{\mathfrak{c}}{4}$, it follows that $\big| t m_0 (z_0) \big| = |z_0 - x_0| \ge \frac{\mathfrak{c}}{4}$ at $z_0 = z_0 (x_0, s)$. This contradicts the fact (from the $j = 0$ case of \eqref{zkm0z} and \eqref{x0t2}) that 
	\begin{flalign*} 
		\big| t m_0 (z_0) \big| \le \frac{(\delta \mathfrak{c})^{20}}{C_2 B^{20}} \cdot C_1 B^2 (\delta \mathfrak{c})^{-1} \le \frac{\mathfrak{c}}{8},
	\end{flalign*} 
	
	\noindent for $C_2 > C_1$, and so $|\Real z_0| \le \frac{\mathfrak{c}}{2}$. Further observe that $\Imaginary z_0 \le \frac{\mathfrak{c}}{2}$, since otherwise the third statement of \eqref{m0z0z0} would yield (as $C_2 > 4$)
	\begin{align*}
		\frac{1}{t}=\int_{-\infty}^{\infty}\frac{\varrho_0(x) dx }{|x-z_0|^2} < 4 \mathfrak{c}^{-2} \int_{-\infty}^{\infty} \varrho_0(x)d x = 4 \mathfrak{c}^{-2} < \displaystyle\frac{C_2 B^{20}}{(\delta \mathfrak{c})^{20}},
	\end{align*}
	
	\noindent which contradicts \eqref{x0t2}. 
	
	Therefore, 
	\begin{align*}
		\frac{1}{t}=\int_{-\infty}^{\infty}\frac{\varrho_0(x) dx}{|x-z_0|^2}\le B \int_{-\mathfrak{c}}^{\mathfrak{c}} \frac{dx}{|x-z_0|^2}+ 4 \mathfrak{c}^{-2} \le \frac{B\pi}{\Imaginary z}+ 4 \mathfrak{c}^{-2},
	\end{align*}
	
	\noindent where in the first statement we used the third part of \eqref{m0z0z0}; in the second we used the facts that $\sup_{|x| \le \mathfrak{c}} \varrho (x) \le B$, that $\varrho$ is a probability density, and that $|\Real z_0| \le \frac{\mathfrak{c}}{2}$; and in the third we used the equality $\int_{-\infty}^{\infty} |x-z|^{-2} dx = \pi (\Imaginary z)^{-1}$. Since $t \le \frac{\mathfrak{c}^2}{8}$, it follows that $t (\Imaginary z)^{-1} \ge (2 B \pi)^{-1}$. Inserting this into the second bound of \eqref{zkm0z} (which applies as $|\Real z_0 | \le \frac{\mathfrak{c}}{2}$ and $\Imaginary z_0 \le \frac{\mathfrak{c}}{2}$) gives
	\begin{flalign}
		\label{tm0z2}
		\big|1-tm_0'(z_0) \big| \ge \displaystyle\frac{\delta \mathfrak{c}}{4 \pi B}.
	\end{flalign} 
	
	\noindent Inserting the first bound of \eqref{zkm0z} (at $z = z_0$, which applies since $|\Real z_0| \le \frac{\mathfrak{c}}{2}$) and \eqref{tm0z2} into \eqref{derivativekrho}, it follows that there exists a constant $C_3 = C_3 (k) > 1$ such that 
	\begin{flalign}
		\label{rhotxderivativetx}
		\big| \partial_x^j \varrho_t (x_0) \big| \le C_3 (B^8 \delta^{-6} \mathfrak{c}^{-6})^{j+1}, \qquad \text{for any $j \in \llbracket 0, k-1 \rrbracket$}.
	\end{flalign}

	Before proceeding to bound the $t$-derivative of $\varrho_t$, let us quickly establish the first statement of \eqref{rhotxtx}. To that end, observe that 
	\begin{flalign*}
		\big| \varrho_t (x_0) - \varrho_0 (x_0) \big| & = \pi^{-1} \big| \Imaginary m_t (x_0) - \Imaginary m_0 (x_0) \big| \\
		& = \pi^{-1} \Big| \Imaginary m_t \big( z_0 - tm_0 (z_0) \big) - \Imaginary m_0 (x_0) \Big| \\
		& = \pi^{-1} \big| \Imaginary m_0 (z_0) - \Imaginary m_0 (x_0) \big| \\
		&  \le C_3 B^{16} \delta^{-12} \mathfrak{c}^{-12} \cdot |z_0 - x_0| \le C_3 B^{16} \delta^{-12} \mathfrak{c}^{-12} \cdot t \big| m_0 (z_0) \big| \le \displaystyle\frac{\delta}{4},
	\end{flalign*}
	
	\noindent where in the first statement, we applied \eqref{mrho}; in the second, the definition of $z_0$; in the third statement, \eqref{mt}; in the fourth, the first bound in \eqref{zkm0z} at $j = 1$; in the fifth, the fact that $z_0 - x_0 = t m_0 (z_0)$; and in the sixth, the bound on $t$ from \eqref{x0t2} (and making $C_2 > 4C_3$ sufficiently large) with first bound in \eqref{zkm0z} at $j = 0$. Since $\varrho_0 (x_0) \ge \delta$ (as $|x| \le \frac{\mathfrak{c}}{4} \le \mathfrak{c}$), this implies the first statement of \eqref{rhotxtx}.
	
	To prove the second statement of \eqref{rhotxtx}, observe by \eqref{rhotxderivativetx} that it suffices to bound the $t$-derivatives of $\varrho_t$. Recall from \eqref{trhoty} that $\partial_t \varrho_t(x)= \pi \partial_x \big( \varrho_t(x)H\varrho_t(x) \big)$,  where $H\varrho_t(x)$ is the Hilbert transform of $\varrho_t$ (recall \eqref{transform2}). For any real number $x\in \big[-\frac{\mathfrak{c}}{4},\frac{\mathfrak{c}}{4} \big]$ and integer $j \in \llbracket 0, k-2 \rrbracket$, we can then bound the derivatives of Hilbert transform by
	\begin{align}\begin{split}\label{e:Hderi}
			\big| \partial_x^j H\varrho_t(x) \big|	&= \pi^{-1} \Bigg| \partial_x^j \bigg( \PV \displaystyle\int_{-\infty}^{\infty} \frac{\varrho_t(x-y) dy}{y} \bigg) \Bigg| \\
			& =  \pi^{-1} \Bigg| \partial_x^j \bigg( \PV \displaystyle\int_{-\mathfrak{c}/2}^{\mathfrak{c}/2} \displaystyle\frac{\varrho_t (x-y) dy}{y} \bigg) + \partial_x^j \displaystyle\int_{|y| > \mathfrak{c}/2} \displaystyle\frac{\varrho_t (x - y) dy}{y} \Bigg| \\
			& =  \pi^{-1} \Bigg| \partial_x^j \bigg( \PV \int_{-\mathfrak{c}/2}^{\mathfrak{c}/2} \frac{\big( \varrho_t(x-y)-\varrho_t(x) \big) dy}{y}  \bigg) + \partial_x^j \displaystyle\int_{|y-x| > \mathfrak{c}/2} \displaystyle\frac{\varrho_t (y)dy}{x-y} \Bigg| \\
			& \le \mathfrak{c} \cdot \displaystyle\sup_{|y| \le \mathfrak{c}} \big| \partial_x^{j+1} \varrho_t (y) \big| + 2^j j! \mathfrak{c}^{-j}  = \mathcal{O} \big( (B^8 \delta^{-6} \mathfrak{c}^{-6})^{j+2} \big),
	\end{split} \end{align}
	
	\noindent where in the first and second statements we used \eqref{transform2}; in the third we used the fact that $\PV (a) = 0$ for any constant $a$ in the first integral, and we changed variables from $y$ to $x-y$ in the second; in the fourth we used the fact that $\big| \partial_x^j (x-y)^{-1} \big| \le j! 2^{j+1} \mathfrak{c}^{-j-1}$ for $|x-y| \ge \frac{\mathfrak{c}}{2}$ and that $\varrho_t$ is a probability measure; and in the fifth we applied \eqref{rhotxderivativetx}. Combining this with \eqref{rhotxderivativetx} and \eqref{trhoty} then yields the second bound in \eqref{rhotxtx}.
\end{proof}

\subsection{Proof of \Cref{uvv}}

\label{DiscreteLinear} 

In this section we establish \Cref{uvv}, to which end we discretize in time (leaving space continuous). For any integer $T \ge 1$, a \emph{($T$-step) Gaussian walk} starting at $u \in \mathbb{R}$ is a probability measure on $(T+1)$-tuples $\big( x(0), x(1), \ldots , x(T) \big) \in \mathbb{R}^T$ with $x(0) = u$ such that, for each $j \in \llbracket 1, T \rrbracket$, the jump $x(j) - x(j-1)$ is a centered Gaussian random variable of variance $1$. A \emph{($T$-step) Gaussian bridge} from $u$ to $v$ is a Gaussian walk starting at $u$, conditioned to end at $v$ (so $x(T) = v$). The following definition, similar to \Cref{qxyfg}, provides notation for non-intersecting Gaussian bridges.

\begin{definition}
	
	\label{gxyfg}
	
	Fix integers $T, n \ge 1$; two $n$-tuples $\bm{u}, \bm{v} \in \overline{\mathbb{W}}_n$; and two functions $f, g: \llbracket 0, T \rrbracket \rightarrow \overline{\mathbb{R}}$. Let $\mathfrak{G}_{f; g}^{\bm{u}; \bm{v}}$ denote the law on sequences $\bm{x} (t) = \big( x_1 (t), x_2 (t), \ldots , x_n (t) \big)$, with $t \in \llbracket 0, T \rrbracket$, given by $n$ independent $T$-step Gaussian walks, conditioned on satisfying the following three properties.
	\begin{enumerate} 
		\item The $x_j$ do not intersect, that is, $\bm{x} (t) \in \mathbb{W}_n$ for each $t \in \llbracket 1, T-1 \rrbracket$. 
		\item The $x_j$ start at $u_j$ and end at $v_j$, that is, $x_j (0) = u_j$ and $x_j (T) = v_j$ for each $j \in \llbracket 1, n \rrbracket$.
		\item The $x_j$ are bounded below by $f$ and above by $g$, that is, $f < x_j < g$ for each $j \in \llbracket 1, n \rrbracket$. 
	\end{enumerate} 

	\noindent We assume $f(0) \le u_n \le u_1 \le g(0)$ and $f(T) \le v_n \le v_1 \le g(T)$, even when not stated explicitly. 
	
\end{definition}

Analogous to \Cref{linear}, the following states that non-intersecting Gaussian walks are invariant under affine transformations.

\begin{rem}
	\label{discretelinear}
	
	Nonintersecting Gaussian walks satisfy the following invariance property under affine transformations. Adopt the notation of \Cref{gxyfg}, and fix real numbers $\alpha, \beta \in \mathbb{R}$. Define the $n$-tuples $\widetilde{\bm{u}}, \widetilde{\bm{v}} \in \mathbb{W}_n$ and functions $\widetilde{f}, \widetilde{g} : \llbracket 0, n \rrbracket \rightarrow \overline{\mathbb{R}}$ by setting 
	\begin{flalign*} 
		& \widetilde{u}_j = u_j + \alpha, \quad \text{and} \quad \widetilde{v}_j = v_j + T \beta + \alpha, \qquad \qquad \qquad \quad \text{for each $j \in \llbracket 0, n \rrbracket$}; \\
		& \widetilde{g}(t) = f(t) + t \beta + \alpha, \quad \text{and} \quad \widetilde{g} (t) = g(t) + t \beta + \alpha, \qquad \text{for each $t \in \llbracket 0, T \rrbracket$}.
	\end{flalign*} 
	
	\noindent Sampling $\widetilde{\bm{x}} = \big( \widetilde{x}_1 (t), \widetilde{x}_2 (t), \ldots , \widetilde{x}_n (t) \big)$ under $\mathfrak{G}_{\tilde{f}, \tilde{g}}^{\tilde{\bm{u}}; \tilde{\bm{v}}}$, there is a coupling between $\widetilde{\bm{x}}$ and $\bm{x}$ such that $\widetilde{x}_j (t) = x_j (t) + \beta t + \alpha$ for each $t \in \llbracket 0, T \rrbracket$ and $j \in \llbracket 1, n \rrbracket$. 
	
	As in \Cref{linear}, this follows from the analogous affine invariance of a single Gaussian random bridge, together with the fact that affine transformations do not affect the non-intersecting property.

\end{rem}

The following lemma states height monotonicity (the analog of \Cref{monotoneheight}) for non-intersecting Gaussian bridges. Its proof is very similar to that of \Cref{monotoneheight} (from \cite{PLE}) and is thus omitted.

\begin{lem}
	
	\label{monotoneheightdiscrete}

	Fix integers $T, n \ge 1$; four $n$-tuples $\bm{u}, \widetilde{\bm{u}}, \bm{v}, \widetilde{\bm{v}} \in \mathbb{W}_n$; and functions $f, \widetilde{f}, g, \widetilde{g}: \llbracket 0, T \rrbracket \rightarrow \overline{\mathbb{R}}$ with $f \le g$ and $\widetilde{f} \le \widetilde{g}$. Sample non-intersecting Gaussian bridges $\bm{x} (t)$ and $\widetilde{\bm{x}} (t)$ from the measures $\mathfrak{G}_{f; g}^{\bm{u}; \bm{v}}$ and $\mathfrak{G}_{\tilde{f}; \tilde{g}}^{\tilde{\bm{u}}; \tilde{\bm{v}}}$, respectively. If 
	\begin{flalign*}
		f \le \widetilde{f}; \quad  g \le \widetilde{g}; \quad \bm{u} \le \widetilde{\bm{u}}; \quad  \bm{v} \le \widetilde{\bm{v}},
	\end{flalign*}
	
	\noindent then there exists a coupling between $\bm{x}$ and $\widetilde{\bm{x}}$ such that $\bm{x} (t) \le \widetilde{\bm{x}} (t)$, for each $t \in \llbracket 0, T \rrbracket$. 
	
\end{lem}

The following lemma is a discrete (in time) variant of \Cref{uvv} and quickly implies it.

\begin{lem} 
	
	\label{fgk} 
	
	Adopt the notation and assumptions of \Cref{monotoneheightdiscrete}, and fix some real number $B \ge 0$. If $u_j \le \widetilde{u}_j \le u_j + B$ and $v_j \le \widetilde{v}_j \le v_j + B$ for all $j \in \llbracket 1, n \rrbracket$, then the following two statements hold.

	\begin{enumerate} 
		\item If $f(t) \le \widetilde{f}(t) \le f(t) + B$ and $g(t) \le \widetilde{g}(t) \le g(t) + B$ for each $t \in \llbracket 0, T \rrbracket$, then there is a coupling between $\bm{x}$ and $\widetilde{\bm{x}}$ so that $x_j (t) \le\widetilde{x}_j (t) \le x_j (t) + B$ for each $(t, j) \in \llbracket 0, T \rrbracket \times \llbracket 1, n \rrbracket$. 
		\item If $\bm{u} = \widetilde{\bm{u}}$, and $f(t) \le \widetilde{f} (t) \le f(t) + \frac{tB}{T}$ and $g(t) \le \widetilde{g}(t) \le g(t) + \frac{tB}{T}$ for each $t \in \llbracket 0, T \rrbracket$, then there is a coupling between $\bm{x}$ and $\widetilde{\bm{x}}$ so that $x_j (t) \le\widetilde{x}_j (t) \le x_j (t) + \frac{tB}{T}$ for each $(t, j) \in \llbracket 0, T \rrbracket \times \llbracket 1, n \rrbracket$.
	\end{enumerate} 
	
\end{lem}

\begin{proof}[Proof of \Cref{uvv}]
	
	This follows from \Cref{fgk} with the convergence of Gaussian walks to Brownian bridges; we omit further details.
\end{proof} 

Similarly to the proof \cite[Lemma 2.6 and Lemma 2.7]{PLE}, that of \Cref{fgk} makes use of a local Glauber dynamic, which in our setting is defined as follows.

\begin{definition}
	\label{dynamic2} 
	
	Fix integers $T, n \ge 1$ and two functions $f, g : \llbracket 0, T \rrbracket \rightarrow \mathbb{R}$. For $t \in \llbracket 0, T \rrbracket$, let $\bm{y} (t) = \big( y_1 (t), y_2 (t), \ldots , y_n (t) \big) \in \mathbb{W}_n$ be a family of $n$ non-intersecting $T$-step walks. The \emph{Glauber dynamics} is the discrete-time Markov chain whose state $\mathcal{P}^k \bm{y} (t) = \big( \mathcal{P}^k y_1 (t), \mathcal{P}^k y_2 (t), \ldots , \mathcal{P}^k y_n (t) \big)$ at time $k \ge 0$ is determined as follows. Below, for each $t \in \llbracket 0, T \rrbracket$, let $y_0 (t) = g(t)$ and $y_{n+1} (t) = f(t)$. 
	
	Set $\mathcal{P}^0 \bm{y} = \bm{y}$. For $k \ge 1$, let $(a, b, d) \in \llbracket 0, T-2 \rrbracket \times \llbracket 1, n \rrbracket \times \mathbb{Z}_{\ge 0} $ be such that $k = dn(T-1) + an + b$. For $(t, j) \ne (a+1, b)$, set $\mathcal{P}^k y_j (t) = \mathcal{P}^{k-1} y_j (t)$. For $(t, j) = (a+1, b)$, set $\mathcal{P}^k y_b (a+1)$ to be the middle ($t=1$) value of a $2$-step Gaussian random bridge, conditioned to start at $\mathcal{P}^{k-1} y_b (a)$; end at $\mathcal{P}^{k-1} y_b (a+2)$; be above $\mathcal{P}^{k-1} y_{b + 1} (a+1)$; and be below $\mathcal{P}^{k-1} y_{b-1} (a+1)$. 
	
\end{definition}

\begin{rem}
	
	\label{gmeasurep}
	
	It follows from the Gibbs property (for non-intersecting Gaussian bridges) that $\mathfrak{G}_{f; g}^{\bm{y} (0); \bm{y}(T)}$ is a stationary measure for the Glauber dynamics $\mathcal{P}$.
	
\end{rem}

The following lemma indicates that the Glauber dynamics converge to the stationary measure $\mathfrak{G}_{f; g}^{\bm{y}(0); \bm{y}(T)}$ in the long-time limit; its proof is given in \Cref{Proof0} below. 

\begin{lem}
	
	\label{converge0} 
	
	Adopt the notation of \Cref{dynamic2}, and assume that $\max_{t \in \llbracket 0, T \rrbracket} \big( |f(t)| + |g(t)| \big) < \infty$. Then, the law of $\mathcal{P}^{2n(T-1)k} \bm{y}$ converges as $k$ tends to $\infty$ to $\mathfrak{G}_{f; g}^{\bm{y}(0); \bm{y}(T)}$. 
	
\end{lem} 

Now let us briefly outline the proof of \Cref{fgk}.

\begin{proof}[Proof of \Cref{fgk} (Outline)]
	
	We only discuss the second statement of the lemma, as the proof of the first is entirely analogous. We will assume throughout that $\big| f(t) \big| + \big| g(t) \big|  < \infty$ for each $t \in \llbracket 1, T-1 \rrbracket$, as the cases when $f(s) = -\infty$ or $g(s) = \infty$ for some $s \in \llbracket 1, T-1\rrbracket$ can then be obtained from a limiting procedure (in particular, one can consider a sequence $f_k$ of lower boundaries such that $f_k (s)$ always remains finite but tends to $-\infty$ as $k$ tends to $\infty$, and similarly for $g$).
	
	Let us first confirm the lemma when $(T, n) = (2, 1)$, in which case we will verify the following more general result. Fix a real number $A \ge 0$ and assume that 
	\begin{flalign*}
		& x_1 (0) \le \widetilde{x}_1 (0) \le  x_1 (0) +A; \qquad  x_1 (0) \le \widetilde{x}_1 (2) \le  x_1 (2) + A + B; \\
		& f(1) \le \widetilde{f}(1) \le f(1) + A + \displaystyle\frac{B}{2}; \quad g(1) \le \widetilde{g} (1) \le g(1) + A + \displaystyle\frac{B}{2}.
	\end{flalign*}

	\noindent Then, there exists a coupling between $x_1$ and $\widetilde{x}_1$ such that 
	\begin{flalign}
		\label{x1aa} 
		x_1 (1) \le \widetilde{x}_1 (1) \le x_1 (1) + A + \displaystyle\frac{B}{2}.
	\end{flalign}

	\noindent Indeed, this is a consequence of the fact (which follows quickly from \Cref{monotoneheightdiscrete} and \Cref{discretelinear}, or can also be verified directly) that $\mathbb{P} \big[ x_1 (1) \ge r \big] \le \mathbb{P} \big[ \widetilde{x}_1 (1) \ge r \big] \le \mathbb{P} \big[ x_1 (1) + A + \frac{B}{2} \ge r \big]$, for any real number $r \in \mathbb{R}$. This shows \eqref{x1aa}, whose $A = 0$ special case yields the second part of \Cref{fgk} at $(T, n) = (2, 1)$.  
	 
	 Thus, we assume in what follows that either $T \ge 3$ or $n \ge 2$. Fix two ensembles $\bm{y}(t) = \big( y_1 (t), y_2 (t), \ldots , y_n (t) \big) \in \mathbb{W}_n$ and $\widetilde{\bm{y}} (t) = \big( \widetilde{y}_1 (t), \widetilde{y}_2 (t), \ldots , \widetilde{y}_n (t) \big) \in \mathbb{W}_n$ of $n$ non-intersecting $T$-step walks, with starting points $\bm{y} (0) = \bm{u} = \widetilde{\bm{y}} (0)$, and ending points $\bm{y} (T) = \bm{v}$ and $\widetilde{\bm{y}} (T) = \widetilde{\bm{v}}$, such that, for each $t \in \llbracket 0, T \rrbracket$, 
	\begin{flalign*}
		\bm{y}(t) \le \widetilde{\bm{y}}(t) \le \bm{y} (t) + \displaystyle\frac{tB}{T}; \qquad f(t) \le y_n (t) \le y_1 (t) \le g(t); \qquad \widetilde{f}(t) \le \widetilde{y}_1 (t) \le \widetilde{y}_n (t) \le \widetilde{g} (t).
	\end{flalign*}
	
	\noindent Such initial data exist due to the assumptions of the second part of the lemma.
	
	We next run the Glauber dynamics on $\bm{y}$ (with boundaries $(f, g)$) and $\widetilde{\bm{y}}$ (with boundaries $(\widetilde{f}, \widetilde{g})$), coupled in the following way. If we update the pair $\big( \mathcal{P}^{k-1} y_j (t), \mathcal{P}^{k-1} \widetilde{y}_j (t) \big)$, then using \eqref{x1aa} (with the $(A; B)$ there given by $\big( \frac{(t-1)B}{T}; \frac{2B}{T} \big)$ here), we couple $\mathcal{P}^k y_j (t) \le \mathcal{P}^k \widetilde{y}_j (t) \le \mathcal{P}^k y_j (t) + \frac{tB}{T}$. By induction on $k$, this ensures that $\mathcal{P}^k y_j (t) \le \mathcal{P}^k \widetilde{y}_j (t) \le \mathcal{P}^k y_j (t) + \frac{tB}{T}$ holds for each $(t, j) \in \llbracket 0, T \rrbracket \times \llbracket 1, n \rrbracket$ and $k \ge 0$. Letting $k$ tend to $\infty$ and taking any limit point of this coupling, the lemma then follows from \Cref{converge0}.
\end{proof}

\subsection{Proof of \Cref{converge0}}

\label{Proof0}

In this section we establish \Cref{converge0}. We begin by recalling a general statement on convergence to stationarity of infinite-dimensional Markov chains. 

Let $\Omega$ be a measurable space with $\sigma$-algebra $\mathcal{F}$; let $\mathscr{P} (\Omega)$ denote the space of probability measures on $(\Omega, \mathcal{F})$. Given any two probability measures $\nu_1, \nu_2 \in \mathscr{P} (\Omega)$, recall that the total variation distance between them is defined by
\begin{flalign}
	\label{nu1dnu2} 
	d_{\TV} (\nu_1, \nu_2) = \displaystyle\sup_{A \in \mathcal{F}} \big| \nu_1 (A) - \nu_2 (A) \big|.
\end{flalign}

\begin{assumption} 
	
	\label{vx} 
	
	Adopting the above notation, let $\mathsf{K} : \Omega \times \mathcal{F} \rightarrow \mathbb{R}_{\ge 0}$ be a Markov transition kernel. For any function $\varphi : \Omega \rightarrow \mathbb{R}_{\ge 0}$ and measure $\mu$ on $\Omega$, define the function $\mathfrak{\mathsf{K}} \varphi : \Omega \rightarrow \mathbb{R}_{\ge 0}$ and measure $\mathsf{K} \mu$ on $\Omega$ by setting
	\begin{flalign*} 
		\mathsf{K} \varphi (x) 	= \displaystyle\int_{\Omega} \varphi (y) \mathsf{K} (x, dy); \qquad \mathsf{K} \mu (A) = \displaystyle\int_{\Omega} \mathsf{K} (x, A) \mu (dx),
	\end{flalign*} 
	
	\noindent for any $x \in \Omega$ and measurable set $A \in \mathcal{F}$. Assume that there exist constants $\alpha \in (0, 1)$, $\gamma \in (0, 1)$, $B \ge 0$, and $R > \frac{2B}{1 - \gamma}$; a potential function $V : \Omega \rightarrow \mathbb{R}_{\ge 0}$; and a probability measure $\nu$ on $\Omega$, such that the following two conditions hold.
	
	\begin{enumerate} 
		\item For each $x \in \Omega$, we have $\mathsf{K}V (x) \le \gamma V (x) + B$, for each $x \in \Omega$.
		\item For each $x \in \Omega$ with $V(x) \le R$, and any measurable set $A \in \mathcal{F}$, we have $\mathsf{K} (x, A) \ge \alpha \nu (A)$. 
	\end{enumerate} 
	
\end{assumption} 

The following result provides a convergence theorem for such Harris chains \cite{ESMP}. It appears in \cite{CSS}, though it is stated as written below in \cite{ET}.

\begin{lem}[{\cite[Theorem 1.2]{ET}}]
	
	\label{convergeprocess}
	
	Adopt \Cref{vx}, and fix a measure $\mu$ on $\Omega$. Then, the Markov process defined by $\mathsf{K}$ has a unique stationary measure $\mu_0$, and $\lim_{m \rightarrow \infty} \| \mathsf{K}^m \mu - \mu_0 \|_{\TV} = 0$. 
	
\end{lem} 

We next apply \Cref{convergeprocess} to the Glauber dynamics $\mathcal{P}$ of \Cref{dynamic2}. Throughout the remainder of this section, we adopt the notation of that definition. These include the functions $f, g : \llbracket 0, T \rrbracket \rightarrow \mathbb{R}$; the family $\bm{y}$ of $n$ non-intersecting $T$-step walks $\bm{y} (t) = \big( y_1 (t), y_2 (t), \ldots , y_n (t) \big) \in \mathbb{W}_n$ (over $t \in \llbracket 0, T \rrbracket$), and the associated Markov operator $\mathcal{P}$ (which we also interpret as a kernel) for the Glauber dynamics. We also set $\bm{u} = \bm{y}(0)$ and $\bm{v} = \bm{y}(T)$, which are fixed throughout the dynamics. Then the state space for the Glauber dynamics $\mathcal{P}$ can be viewed as 
\begin{flalign*}
	\Omega_0 = \Big\{ \big( \bm{y} (t) \big)_{t \in \llbracket 1, T - 1 \rrbracket} \in (\mathbb{W}_n)^{T-1} : \displaystyle\min_{t \in \llbracket 1, T-1 \rrbracket} \big( y_n (t) - f(t) \big) > 0, \displaystyle\min_{t \in \llbracket 1, T-1 \rrbracket} \big( g(t) - y_1 (t) \big) > 0 \Big\}.
\end{flalign*}

\noindent The associated potential function $V_0$ is defined by
\begin{flalign*}
	V_0 (\bm{y}) = \displaystyle\max_{j \in \llbracket 1, n \rrbracket } \displaystyle\max_{t \in \llbracket 1, T-1 \rrbracket } \Big( \big| y_j (t) \big| + 1 \Big).
\end{flalign*} 

\noindent We then have the following two lemmas verifying \Cref{vx} for the Glauber dynamics.

\begin{lem}
	
	\label{vy1} 
	
	There exist constants $\gamma \in (0, 1)$ and $B = B(\bm{u}, \bm{v}, f, g) \ge 0$ such that, for any family $\bm{y}$ of $n$ non-intersecting $T$-step walks with $\bm{y} (0) = \bm{u}$ and $\bm{y} (T) = \bm{v}$, we have $\mathcal{P} V_0 (\bm{y}) \le \gamma V_0 (\bm{y}) + B$.
	
\end{lem}

\begin{proof}
	
	This follows from taking $\gamma = \frac{1}{2}$ and $B = \max_{t \in \llbracket 0, T \rrbracket} \big( | f(t)| + |g(t)| \big) + 1$, which applies since $f(t) \le y_j (t) \le g(t)$ holds for any $(t, j) \in \llbracket 0, T\rrbracket \times \llbracket 1, n\rrbracket$. In particular, $\mathcal{P} V_0 (\bm{y}) \le B \le \gamma V_0 (\bm{y}) + B$ holds deterministically. 
\end{proof}

\begin{lem}
	
	\label{vy2} 
	
	There exists a constant $\alpha = \alpha (\bm{u}, \bm{v}, f, g) > 0$ such that the following holds. Letting $\nu_0$ denote the Lebesgue measure on the set $\Omega_0$, we have  $\mathcal{P}^{2n(T-1)} (\bm{y}, A) \ge \alpha \nu_0 (A)$, for each $\bm{y} \in \Omega_0$ and any measurable subset $A \subseteq (\mathbb{W}_n)^{T-1}$. 
	
\end{lem}

\begin{proof}[Proof (Outline)]

	For any integer $T' \ge 2$; two $n$-tuples $\bm{u}', \bm{v}' \in \mathbb{W}_n$; and two functions $f', g' : \llbracket 0, T' \rrbracket \rightarrow \overline{\mathbb{R}}$, the density of the measure $\mathfrak{G}_{f'; g'}^{\bm{u}'; \bm{v}'}$ on sequences $\bm{\mathsf{x}}(t) = \big( \mathsf{x}_1 (t), \mathsf{x}_2 (t), \ldots , \mathsf{x}_n (t) \big)$ is 
	\begin{flalign}
		\label{g0} 
		C \cdot \textbf{1}_{\bm{\mathsf{x}} \in \Omega_0} \cdot \displaystyle\prod_{j = 1}^n \Bigg( \textbf{1}_{\mathsf{x}_j (0) = u_j'} \textbf{1}_{\mathsf{x}_j (T') = v_j'} \displaystyle\prod_{t=1}^{T'} \exp \bigg( -\displaystyle\frac{1}{2} \big( \mathsf{x}_j (t) - \mathsf{x}_j (t-1) \big)^2 \bigg) \displaystyle\prod_{t = 1}^{T'-1} d \mathsf{x}_j (t) \Bigg),
	\end{flalign}
	
	\noindent for some normalization constant $C = C (f', \bm{u}', \bm{v}') > 0$. Observe that there exist some constant $c_1 = c_1 (f', g', \bm{u}', \bm{v}') > 0$ such that $C > c_1$ (that is, the normalization constant is bounded below), since the property $\bm{\mathsf{x}} \in \Omega_0$ is an open condition. Further observe that the density \eqref{g0} is uniformly bounded (since $\big| f(t) \big| + \big| g(t) \big|$ is). Thus, there exists a constant $c_1 = c_1 (f'; \bm{u}'; \bm{v}') > 0$ such that the measure $\mathfrak{G}_{f'}^{\bm{u}'; \bm{v}'}$ is absolutely continuous with respect to the Lebesgue measure on this set, and its Radon--Nikodym derivative is bounded above by $c_1^{-1}$ and below by $c_1$.

	Adopting the notation of \Cref{dynamic2}, we fix an integer $k \ge 1$ and apply this in the case $(n', T') = (1, 2)$; $(\bm{u}', \bm{v}') = (u_1', v_1') = \big( \mathcal{P}^{k-1} y_b (a), \mathcal{P}^{k-1} y_b (a+2) \big)$; and $(f', g') = \big( f'(1), g'(1) \big) = \big( \mathcal{P}^{k-1} y_{b+1} (a+1), \mathcal{P}^{k-1} y_{b-1} (a+1) \big)$. We then deduce (using the fact that the density \eqref{g0} is bounded below by a constant multiple of the Lebesgue measure on $\Omega_0$) that there is a constant $c_2 = c_2 (\bm{u}, \bm{v}, f, g, R) > 0$ such that 
	\begin{flalign}
		\label{y2} 
		\mathbb{P} \big[ \mathcal{P}^k y_b (a+1) \in S_k \big] \ge c_2 \displaystyle\int_{S_k} dy,
	\end{flalign} 

	\noindent for any measurable set $S_k  \subseteq \big\{ x \in \mathbb{R} : \mathcal{P}^{k-1} y_{b+1} (a+1) \le x \le \mathcal{P}^{k-1} y_{b-1} (a+1)  \big\}$. Applying this $n(T-1)$ times yields a constant $c_3 = c_3 (\bm{u}, \bm{v}, f, g, R) > 0$ such that 
	\begin{flalign}
		\label{kt1np}
		\mathbb{P} \big[ \mathcal{P}^{k+n(T-1)} \bm{y} \in S \big] \ge c_3 \cdot \nu_0 (S),
	\end{flalign}

	\noindent for any measurable set 
	\begin{flalign*}
		S \subseteq \Big\{ \bm{x} = \big( \bm{x} (t) \big)_{t \in \llbracket 1, T-1\rrbracket} \in (\mathbb{W}_n)^{T-1} : \mathcal{P}^k y_{j+1} (t) \le x_j (t) \le g(t), \quad \text{for all $(t, j)$} \Big\}.
	\end{flalign*}

	\noindent Applying \eqref{kt1np} twice, first at $k = 0$ and then at $k = n(T-1)$, yields the lemma.
\end{proof}

Given these lemmas, we can quicky establish \Cref{converge0}.

\begin{proof}[Proof of \Cref{converge0}]
	
	That $\mathfrak{G}_f^{\bm{u}; \bm{v}}$ is stationary for $\mathcal{P}$ follows from \Cref{gmeasurep}. Thus, the lemma follows from \Cref{convergeprocess}, using \Cref{vy1} and \Cref{vy2} to verify \Cref{vx}.
\end{proof}

\subsection{Proof of \Cref{estimatexj}} 

\label{Proof00} 

In this section we establish \Cref{estimatexj}. We begin with the following concentration bound, which follows quickly from combining results of \cite{RMCE,LLCTE}, for the eigenvalues of a GUE random matrix.

\begin{lem}[{\cite[Corollary 5]{RMCE},\cite[Corollary 1.6]{LLCTE}}]
	
	\label{estimatexj0} 
	
	There exists a constant $C > 1$ such that the following holds. Letting $n \ge 1$ be an integer and $(\lambda_1, \lambda_2, \ldots , \lambda_n) \in \overline{\mathbb{W}}_n$ denote the eigenvalues of an $n \times n$ GUE matrix $\bm{G} = \bm{G}_n$, we have 
	\begin{flalign*}
		\mathbb{P} \Bigg[ \bigcup_{j = 1}^n \Big\{  \big| \lambda_j - \gamma_{\semci; n} (j) \big| \ge (\log n)^8 \cdot n^{-2/3} \cdot \min \{ & j, n-j+1 \}^{-1/3} \Big\} \Bigg] \le C e^{- (\log n)^5}.
	\end{flalign*} 
\end{lem}

Now we can establish \Cref{estimatexj}.

\begin{proof}[Proof of \Cref{estimatexj}]
	
	By \Cref{linear}, we may assume throughout that $u = v = 0$. We may also assume by shifting that $a = -b = T$, so that $(a, b) = (-T, T)$ for some $T \le n^D$.
	
	We next equate the law of $\bm{x} (t)$ with the eigenvalues of a suitably scaled GUE matrix. To that end, let $\bm{\lambda} (t) = \big( \lambda_1 (t), \lambda_2 (t), \ldots , \lambda_n (t) \big) \in \mathbb{W}_n$ denote Dyson Brownian motion with starting data $\bm{0}_n$. From the second part of \Cref{lambdat}, $\bm{\lambda} (t)$ prescribes the law of $n$ non-intersecting Brownian motions (of variance $n^{-1}$) run for time $t$, conditoned to never intersect; from the first part of \Cref{lambdat}, it also prescribes the eigenvalues of an $n \times n$ GUE matrix, scaled by $t^{1/2}$. 
	
	Next, recall that a Brownian bridge $B(t)$ on the interval $[-T, T]$, conditioned to start and end at $0$ (that is, $B(-T) = B(T) = 0$), has the same law as $\frac{T-t}{(2T)^{1/2}} \cdot W \big( \frac{T+t}{T-t} \big)$, where $W : [0, \infty) \rightarrow \mathbb{R}$ is a Brownian motion. This, implies (see also \cite[Equation (2)]{BPLE}) that  
	\begin{flalign} 
		\label{xjlambdaj} 
		\text{$x_j (t)$ has the same law as $(2T)^{-1/2} (T^2  - t^2)^{1/2} \cdot \lambda_j$};
	\end{flalign} 
	
	Given \eqref{xjlambdaj}, the proofs of both parts of the lemma are very similar, obtained by applying \Cref{estimatexj0} on a suitable mesh of $t$; we therefore only detail the proof of the first statement of the lemma. To that end, let $C_0 > 1$ be the constant from \Cref{estimatexj0}. For any integer $j \in \llbracket 1, n \rrbracket$ and real number $t \in [-T, T]$, define the events 
	\begin{flalign*}
		& \mathscr{G} (t)  = \bigcup_{j = 1}^n \bigg\{  \Big| x_j (t) - (2T)^{-1/2} (  T^2 - t^2)^{1/2} \cdot \gamma (j) \Big|  \ge (\log n)^8 \cdot T^{1/2} n^{-2/3} \cdot \min \{  j, n-j+1 \}^{-1/3} \bigg\}; \\ 
		& \mathscr{G}_0  = \bigcup_{j=1}^n \bigcup_{a \le t < t+s \le b} \Big\{ \big| x_j (t+s) - x_j (t) \big| \ge s^{1/2} n^5 \log | 4s^{-1} T| \Big\}; \quad \mathscr{G}_1  = \bigcup_{j= -\lfloor Tn^{20(D+1)} \rfloor}^{\lfloor Tn^{20(D+1)} \rfloor} \mathscr{G} \Big( \displaystyle\frac{j}{n^{20 (D+1)}} \Big),
	\end{flalign*} 
	
	\noindent where we have abbreviated $\gamma (j) = \gamma_{\semci; n} (j)$. By \eqref{xjlambdaj}, \Cref{estimatexj0}, (the $B = n^5$ case of) \Cref{estimatexj2}, a union bound, and the fact that $T \le n^D$, there exist constants $c > 0$ and $c_1 = c_1 (D) > 1$ such that
	\begin{flalign}
		\label{g0g1}
		\mathbb{P} [\mathscr{G}_0 \cap \mathscr{G}_1] \le 3 C_0 n^{20(D+1)} T e^{ - c_1 (\log n)^5} + c^{-1} e^{n/c - cn^{10}} \le 2 c_1^{-1} e^{-c (\log n)^5},
	\end{flalign}
	
	\noindent for sufficiently large $n$. Moreover observe, since the function $(T - t^2 T^{-1})^{1/2}$ is $\frac{1}{2}$-H\"{o}lder on $[-T, T]$ with constant $2^{1/2}$, on the event $\mathscr{G}_0^{\complement} \cap \mathscr{G}_1^{\complement}$ we have for each $j \in \llbracket 1, n \rrbracket$ that
	\begin{flalign}
		\label{xjt1}
		\begin{aligned} 
			&  \displaystyle\sup_{t \in [a, b]} \bigg|  x_j (t) -  \Big( \displaystyle\frac{T}{2}- \frac{t^2}{2T} \Big)^{1/2} \cdot \gamma (j) \bigg|  \\
			& \quad \le  \displaystyle\max_{n^{20(D+1)} t \in \llbracket -Tn^{20 (D+1)}, Tn^{20 (D+1)} \rrbracket} \bigg| x_j (t) - (2T)^{-1/2} (T^2 - t^2)^{1/2} \cdot \gamma (j) \bigg|   \\
			& \qquad + \displaystyle\sup_{t \in [a, b]} \Bigg( \sup_{s \in [0, n^{-20(D+1)}]}  \big| x_j (t + s) - x_j (t) \big| \Bigg) \\ 
			& \qquad + \displaystyle\sup_{t \in [a, b]} \Bigg( \displaystyle\sup_{s \in [0, n^{-20(D+1)}]} \Big| \big( T^2 - (t+s)^2 \big)^{1/2} - (T^2 - t^2)^{1/2} \Big| \Bigg) \\
			& \quad \le (\log n)^8  \cdot T^{1/2} n^{-2/3} \cdot \min \{ j, n - j + 1 \}^{-1/3} + 2 n^{-10D - 5} \cdot \log \big( 10(T+1) Tn^{20(D+1)} \big) \\
			& \quad \le 2 (\log n)^8  \cdot T^{1/2} n^{-2/3} \cdot \min \{ j, n - j + 1 \}^{-1/3} + n^{-D}, 
		\end{aligned}
	\end{flalign}
	
	\noindent for sufficiently large $n$. The lemma then follows from \eqref{g0g1} and \eqref{xjt1}.		
\end{proof}

\section{Proofs of Results From \Cref{Limit1}}

\label{ProofLimit1} 

\subsection{Proofs of \Cref{gtabab}, \Cref{aab}, and \Cref{mutrhot}}

\label{ProofContinuous1}

In this section we establish first \Cref{gtabab}, then \Cref{aab}, and next \Cref{mutrhot}.

\begin{proof}[Proof of \Cref{gtabab}]
	
	Denote the bridge-limiting measure process on $[\widetilde{a}, \widetilde{b}]$ with boundary data $(\mu_{\tilde{a}}, \mu_{\tilde{b}})$ by $(\widetilde{\mu}_t^{\star})_{t \in [\tilde{a}, \tilde{b}]}$. Recalling the ensemble $\bm{x}^n \in \llbracket 1, n \rrbracket \times \mathcal{C} \big( [a, b] \big)$, condition on $\bm{x} (t)$ for $t \notin [\widetilde{a}, \widetilde{b}]$. Then, the law of the restriction $\bm{y}^n = \bm{x}^n |_{[\tilde{a}, \tilde{b}]} \in \llbracket 1, n \rrbracket \times \mathcal{C} \big( [\widetilde{a}, \widetilde{b}] \big)$ of $\bm{x}$ to $[\widetilde{a}, \widetilde{b}]$ is given by $\mathfrak{Q}^{\bm{x}^n(\tilde{a}); \bm{x}^n(\tilde{b})} (A n^{-1})$. Then \Cref{rhot} implies that, with probability $1-o(1)$ (that is, tending to $1$ as $n$ tends to $\infty$), 
	\begin{flalign} 
		\label{ymu1} 
		\text{the measure $A \cdot \emp \big( \bm{y}^n (t) \big)$ converges under the L\'{e}vy metric to $\mu_t^{\star}$, uniformly in $t \in [\widetilde{a}, \widetilde{b}]$}.
	\end{flalign} 
	
	In particular, \eqref{ymu1} holds for $t \in \{ \widetilde{a}, \widetilde{b} \}$, and so (with probability $1-o(1)$), the ensemble $\bm{y}$ satisfies the conditions of \Cref{rhot}, with $(a, b, \mu_a, \mu_b)$ there equal to $(\widetilde{a}, \widetilde{b}, \mu_{\tilde{a}}, \mu_{\tilde{b}})$ here. Hence, with probability $1-o(1)$, the measure $A \cdot \emp \big( \bm{y} (t) \big)$ converges to $\widetilde{\mu}_t^{\star}$, uniformly in $t \in [\widetilde{a}, \widetilde{b}]$. Together with \eqref{ymu1}, this yields the lemma.
\end{proof} 

\begin{proof}[Proof of \Cref{aab}]
	
	If $\bm{x}^n = (x_1^n, x_2^n, \ldots , x_n^n) \in \llbracket 1, n \rrbracket \times \mathcal{C} \big( [a, b] \big)$ is as in \Cref{rhot}, then by \Cref{scale} the ensemble $\widetilde{\bm{x}}^n = (\widetilde{x}_1^n, \widetilde{x}_2^n, \ldots , \widetilde{x}_n^n) \in \llbracket 1, n \rrbracket \times \mathcal{C} \big( [0, 1] \big)$ defined by setting 
	\begin{flalign*}
		\widetilde{x}_j (t) = A^{-1/2} (b-a)^{-1/2} \cdot x_j \big( a + (b-a) t \big), \qquad \text{for each $(j, t) \in \llbracket 1, n \rrbracket \times [0, 1]$},
	\end{flalign*}
	
	\noindent is as in \Cref{rhot}, with the $(a, b; A; \mu_a, \mu_b)$ there equal to $(0, 1; 1; \widetilde{\mu}_0, \widetilde{\mu}_1)$ here. Then, recalling that $\nu_t^n = \emp \big( \bm{x}^n (t) \big)$ and denoting $\widetilde{\nu}_t = \emp \big( \widetilde{\bm{x}}^n (t) \big)$ for each $t \in [0, 1]$, we have 
	\begin{flalign*}
		\widetilde{\nu}_t^n (I) = \nu_{a + t(b-a)}^n \big( A^{1/2} (b-a)^{1/2} \cdot I \big), \qquad \text{for any interval $I \subseteq \mathbb{R}$}.
	\end{flalign*}
	
	\noindent Letting $n$ tend to $\infty$ and using \Cref{rhot}, this verifies the first statement of \eqref{muhgaab}. Together with \eqref{htxintegral} and \eqref{gty}, this also quickly implies the second and third statements of \eqref{muhgaab}.
\end{proof}

\begin{proof}[Proof of \Cref{mutrhot}]

	The first statement of the lemma is due to \cite[Corollary 2.8(b)]{FOAMI}, and the second to \cite[Theorem 2.7]{FOAMI} (see also \cite[Proposition 2.6(6)]{LDSI}). For the third, the estimates on $\varrho_t^{\star}$ and $\partial_x \varrho_t^{\star}$ follow from the second part of the lemma, together with \Cref{rhotestimatek}.
	
	For the fourth, let us first show that $\varrho_t (x)$ is $\frac{1}{7}$-H\"{o}lder in $x$. In particular, we claim that there exists a constant $C_1 > 1$ such that, for any $t \in (0, 1)$ and $x, x' \in \mathbb{R}$ with $|x-x'|$ sufficiently small (in a way only dependent on $t$), we have 
	\begin{flalign}
		\label{rhotxx}
		\big| \varrho_t^{\star} (x') - \varrho_t^{\star} (x) \big| \le C_1 |x-x'|^{1/7} \big( t(1-t) \big)^{-1}.
	\end{flalign}
	
	Assume to the contrary that this is false, so there exist $t \in (0, 1)$ and $x, x' \in \mathbb{R}$ such that \eqref{rhotxx} does not hold; we may suppose by symmetry that $x' > x$. Then, since $\varrho_t^{\star} (x) \ge 0$ and $\varrho_t^{\star} (x') \ge 0$, we either have $\varrho_t^{\star} (x) > C_1 (x'-x)^{1/7} (t-t^2)^{-1}$ or $\varrho_t^{\star} (x') \ge C_1 (x'-x)^{1/7} (t-t^2)^{-1}$; let us assume the former holds. By the third part of the lemma, we have for some constant $C_2 > 1$ that
	\begin{flalign*}
		\big| \partial_x \varrho_t^{\star} (y) \big| \le C_2 (x'-x)^{-6/7} \big( t(1-t) \big)^2, \qquad \text{whenever $\varrho_t^{\star} (y) \ge \big( t(1-t) \big)^{-1} (x'-x)^{1/7}$}.
	\end{flalign*}
	
	\noindent Integrating this bound over a maximal subset of $[x, x']$ on which $\varrho_t^{\star} (y) \ge ( t-t^2)^{-1} |x-x'|^{1/7}$, we find for any such $y_0$ that
	\begin{flalign}
		\label{rhoy0x}
		\big| \varrho_t^{\star} (y_0) - \varrho_t^{\star} (x) \big| \le C_2 |y_0-x| \cdot (x'-x)^{-6/7} \big( t (1- t) \big)^2 \le C_2 (x'-x)^{1/7} \big( t(1-t) \big)^2.
	\end{flalign}
	
	\noindent This implies for $C_1 > C_2 + 2$ that $\varrho_t^{\star} (y) \ge |x-x'|^{1/7} (t-t^2)^{-1}$ for any $y \in [x, x']$, so \eqref{rhoy0x} applies at $y_0 = x'$. This yields \eqref{rhotxx}, which is a contradiction, so \eqref{rhotxx} holds for any $(t, x) \in (0, 1) \times \mathbb{R}$.
	
	Now fix any point $(t_0, x_0) \in (0, 1) \times \mathbb{R}$ and sequence $(t_1, x_1), (t_2, x_2), \ldots \in (0, 1) \times \mathbb{R}$ with $\lim_{j \rightarrow \infty} (t_j, x_j) = (t_0, x_0)$. For any real number $\delta > 0$, we have by the continuity of the measures $\mu_t^{\star}$ in $t$ that 
	\begin{flalign*} 
		\lim_{j \rightarrow \infty} \int_{x_0-\delta}^{x_0+\delta} \varrho_{t_j} (x) dx = \int_{x_0 - \delta}^{x_0 + \delta} \varrho_{t_0}(x) dx.
	\end{flalign*}
	
	\noindent By \eqref{rhotxx}, this implies for sufficiently large $j$ (dependent on $\delta$) that
	\begin{flalign*}
		\big| \varrho_{t_0} (x_0) - \varrho_{t_j} (x_j) \big| & \le (2 \delta)^{-1} \Bigg( \displaystyle\int_{x_0 - \delta}^{x_0 + \delta} \Big( \big| \varrho_{t_0} (x) - \varrho_{t_0} (x_0) \big| + \big| \varrho_{t_j} (x_j) - \varrho_{t_j} (x_0) \big| \Big) dx \Bigg) \\
		& \qquad + (2 \delta)^{-1} \Bigg| \displaystyle\int_{x_0 - \delta}^{x_0 + \delta} \varrho_{t_0} (x) dx - \displaystyle\int_{x_0 - \delta}^{x_0 + \delta} \varrho_{t_j} (x) dx \Bigg| \le 4 C_1 \big( t_0 (1-t_0) \big)^{-1} \delta^{1/7} + \delta ,
	\end{flalign*}
	
	\noindent which upon letting $\delta$ tend to $0$ yields the continuity of $\varrho_t (x)$ in $(t, x)$.
\end{proof}

\subsection{Proofs of \Cref{p:solution}, \Cref{hurho}, and \Cref{muamubrho}}

\label{ProofDerivativeH}

In this section we establish first \Cref{p:solution}, then \Cref{hurho}, and next \Cref{muamubrho}. 

\begin{proof}[Proof of \Cref{p:solution}]
	
	Expanding the second and first equations in \eqref{equationrhou} yields
	\begin{flalign}
		\label{trhouxrho}
		\partial_t \varrho + u \partial_x \varrho + \varrho \partial_x u = 0; \qquad u \partial_t \varrho + \varrho \partial_t u + 2 u \varrho \partial_x u + u^2 \partial_x \varrho - \pi^2 \varrho^2 \partial_x \varrho = 0.
	\end{flalign}
	
	\noindent The first equation of \eqref{trhouxrho} is equivalent to the imaginary part of \eqref{ftfx}. Moreover, subtracting the second equation of \eqref{trhouxrho} from the first one multiplied by $u$ then yields 
	\begin{flalign*}
		\varrho (\partial_t u + u \partial_x u - \pi^2 \varrho \partial_x \varrho) = 0.
	\end{flalign*}
	
	\noindent Since $\varrho_t (x) > 0$ for $(t, x) \in \Omega$, this implies $\partial_t u + u \partial_x u - \pi^2 \varrho \partial_x \varrho = 0$ on $\Omega$, which is equivalent to the real part of \eqref{ftfx}; this verifies \eqref{ftfx}.
\end{proof}

\begin{proof}[Proof of \Cref{hurho}]
	
	Let us first show that $H^{\star}$ is locally Lispchitz; that this holds in directions parallel to the $x$-axis follows from the definition \eqref{htxintegral} and the bound on $\varrho_t^{\star}$ from the third part of \Cref{mutrhot}. It remains to show that $H^{\star}$ is locally Lipschitz in directions parallel to the $t$-axis. 
	
	To that end, fix points $(t_1, x_0), (t_2, x_0) \in (0, 1) \times \mathbb{R}$ with $t_1 < t_2$; by the second part of \Cref{mutrhot}, there exists some real number $x_1 > x_0$ such that $\varrho_t (x) = 0$ for each $x > x_1 - 1$ and $t \in (0, 1)$. For any real number $\varepsilon > 0$, let $\varphi_{\varepsilon} : (0, 1) \rightarrow \mathbb{R}$ denote a smooth approximation of the indicator function $\textbf{1}_{t \in [t_1, t_2]} \cdot \textbf{1}_{x \in [x_0, x_1]}$. Explicitly, fix a smooth function $\psi : \mathbb{R}^2 \rightarrow \mathbb{R}_{\ge 0}$ with $\supp \psi \subseteq [-1, 1] \times [-1, 1]$ and with total integral $\int_{-\infty}^{\infty} \int_{-\infty}^{\infty} \psi (t, x) dt dx = 1$, and define $\psi_{\varepsilon} : \mathbb{R}^2 \rightarrow \mathbb{R}$ by setting $\psi_{\varepsilon} (t, x) = \varepsilon^{-2} \psi (\varepsilon^{-1} t, \varepsilon^{-1} x)$ for each $(t, x) \in \mathbb{R}^2$. Then, set
	\begin{flalign*}
		\varphi_{\varepsilon} (t, x) = \displaystyle\int_{t_1}^{t_2} \displaystyle\int_{x_0}^{x_1} \psi_{\varepsilon} (t-s, x-w) dw ds.
	\end{flalign*}
	
	Since $(u^{\star}, \varrho^{\star})$ weakly solves the equation $\partial_t \varrho^{\star} = \partial_x (u^{\star} \varrho^{\star})$, we have
	\begin{flalign}
		\label{equationhintegral} 
		\displaystyle\int_0^1 \displaystyle\int_{-\infty}^{\infty} \varrho_t^{\star} (x) \partial_t \varphi_{\varepsilon} (t, x) dx dt + \displaystyle\int_0^1 \displaystyle\int_{-\infty}^{\infty} u_t^{\star} (x) \varrho_t^{\star} (x) \partial_x \varphi_{\varepsilon} (t, x) dx dt = 0.
	\end{flalign}
	
	\noindent The continuity of $\varrho$ (due to the fourth part of \Cref{mutrhot}), and the facts that $\psi_{\varepsilon}$ has total integral equal to $1$ and that $\supp \psi_{\varepsilon} \subseteq [-\varepsilon, \varepsilon] \times [-\varepsilon, \varepsilon]$, imply
	\begin{flalign}
		\label{differencehequation} 
		\begin{aligned}
			\displaystyle\lim_{\varepsilon \rightarrow 0} & \displaystyle\int_0^1 \displaystyle\int_{-\infty}^{\infty} \varrho_t^{\star} (x) \partial_t \varphi_{\varepsilon} (t, x) dx dt \\
			& = \displaystyle\lim_{\varepsilon \rightarrow 0} \displaystyle\int_0^1 \displaystyle\int_{-\infty}^{\infty} \displaystyle\int_{x_0}^{x_1} \varrho_t^{\star} (x) \big( \psi_{\varepsilon} (t - t_1, x-w) - \psi_{\varepsilon} (t-t_2, x-w) \big) dw dx dt \\
			& = \displaystyle\int_{x_0}^{x_1} \big( \varrho_{t_1}^{\star} (x) - \varrho_{t_2}^{\star} (x) \big) dx = H_{t_1}^{\star} (x_0) - H_{t_2}^{\star} (x_0), 
		\end{aligned} 
	\end{flalign}
	
	\noindent where we used the fact that $H_{t_j}^{\star} (x_1) = 0$ for each $j \in \{ 1, 2 \}$, since $\varrho_{t_j}^{\star} (x) = 0$ for $x > x_1$. 
	
	Next, \cite[Corollary 2.8(c)]{FOAMI} states that, for any constant $\omega \in \big( 0, \frac{1}{2} \big)$, there exists a constant $C = C(\omega, \mu_0, \mu_1) > 1$ satisfying
	\begin{flalign}
		\label{urhoc}
		\displaystyle\sup_{x \in \mathbb{R}} \displaystyle\sup_{t \in [\omega, 1-\omega]} \big| u_t^{\star} (x) \varrho_t^{\star} (x) \big| < C.
	\end{flalign}
	
	\noindent We further have 
	\begin{flalign}
		\label{integralrhou0} 
		\begin{aligned}
			\displaystyle\lim_{\varepsilon \rightarrow 0} & \displaystyle\int_0^1 \displaystyle\int_{-\infty}^{\infty} u_t^{\star} (x) \varrho_t^{\star} (x) \partial_x \varphi_{\varepsilon} (t, x) dx dt  \\ 
			& = -\displaystyle\lim_{\varepsilon \rightarrow 0} \displaystyle\int_0^1 \displaystyle\int_{-\infty}^{\infty} \displaystyle\int_{t_1}^{t_2} \displaystyle\int_{x_0}^{x_1} u_t^{\star} (x) \varrho_t^{\star} (x) \partial_w \psi_{\varepsilon} (t-s, x-w) dw ds dx dt  \\ 
			& = \displaystyle\lim_{\varepsilon \rightarrow 0} \displaystyle\int_0^1 \displaystyle\int_{-\infty}^{\infty} \displaystyle\int_{t_1}^{t_2} u_t^{\star} (x) \varrho_{t}^{\star} (x) \big( \psi_{\varepsilon} (t-s, x-x_0) - \psi_{\varepsilon} (t-s, x-x_1) \big) ds dx dt \\
			& = \displaystyle\lim_{\varepsilon \rightarrow 0} \displaystyle\int_0^1 \displaystyle\int_{-\infty}^{\infty} \displaystyle\int_{t_1}^{t_2} u_{t+s}^{\star} (x + x_0) \varrho_{t+s}^{\star} (x + x_0) \psi_{\varepsilon} (t, x) ds dx dt, 
		\end{aligned} 
	\end{flalign}
		
	\noindent where in first equality we used the fact that $\partial_x \psi_{\varepsilon} (t-s, x-w) = -\partial_w \psi_{\varepsilon} (t-s, x-w)$; in the second we performed the integration; and in the third we used the fact that $u_t^{\star} (x) \varrho_t^{\star} (x) \psi_{\varepsilon} (t-s, w-x_1) = 0$ for sufficiently small $\varepsilon$ (as $\varrho_t^{\star} (x) = 0$ for $x \ge x_1 - 1$ and $\supp \psi_{\varepsilon} \subseteq [-\varepsilon, \varepsilon] \times [-\varepsilon, \varepsilon]$) and changed variables from $(t, x)$ to $(t+s, x+x_0)$. Now, since $\psi_{\varepsilon}$ is nonnegative with total integral equal to $1$, inserting \eqref{integralrhou0} and \eqref{differencehequation} into \eqref{equationhintegral}, and using \eqref{urhoc} yields $\big| H_{t_1}^{\star} (x_0) - H_{t_2}^{\star} (x_0) \big| \le C_1$ for some constant $C_1 = C_1 (t_1, t_2, \mu_1, \mu_2) > 1$. In particular, $H_t^{\star}$ is locally Lipschitz. 
	
	It remains to verify \eqref{hurhoequation}; its first statement follows from \eqref{htxintegral}, so it remains to confirm its second. To that end, the Lebesgue differentiation theorem yields, for almost all $(s, x_0) \in (0, 1) \times \mathbb{R}$,  
	\begin{flalign*}
		\displaystyle\lim_{\varepsilon \rightarrow 0} \displaystyle\int_0^1 \displaystyle\int_{-\infty}^{\infty} u_{t+s}^{\star} (x + x_0) \varrho_{t+s}^{\star} (x + x_0) \psi_{\varepsilon} (t, x) dx dt = u_s^{\star} (x_0) \varrho_s^{\star} (x_0).
	\end{flalign*}
	
	\noindent With \eqref{urhoc}, this implies for almost every $x_0 \in \mathbb{R}$ that 
	\begin{flalign*}
		\displaystyle\lim_{\varepsilon \rightarrow 0} & \displaystyle\int_0^1 \displaystyle\int_{-\infty}^{\infty} \displaystyle\int_{t_1}^{t_2} u_{t+s}^{\star} (x+x_0) \varrho_{t+s}^{\star} (x+x_0) \psi_{\varepsilon} (t, x) ds dx dt = \displaystyle\int_{t_1}^{t_2} u_s^{\star} (x_0) \varrho_s^{\star} (x_0) ds.
\end{flalign*}

	\noindent Together this, \eqref{equationhintegral}, \eqref{differencehequation}, and \eqref{integralrhou0} gives for almost all $x_0 \in \mathbb{R}$ that $H_{t_2}^{\star} (x_0) - H_{t_1}^{\star} (x_0) = \int_{t_1}^{t_2} u_s^{\star} (x_0) \varrho_s^{\star} (x_0) ds$, which establishes the second part of \eqref{hurhoequation}.
\end{proof}

\begin{proof}[Proof of \Cref{muamubrho}]
	
	We only establish the first part of the lemma, as the proofs of the second and third are entirely analogous. To do this, it suffices to show that $H^{\star}$ is smooth on $\Omega$, as from \eqref{gty} this quickly implies that $G^{\star}$ is smooth on $\Omega^{\inv}$. 
	
	To that end, fix some point $(t_0, x_0) \in \Omega$; we must show that $H^{\star}$ is smooth in a neighborhood of $(t_0, x_0)$. Let $a', b' \in (a, b)$ be real numbers such that $a' < t_0 < b'$. Then, \Cref{gtabab} implies that $(\mu_t)_{t \in [a', b']}$ is the bridge-limiting measure process on $[a', b']$ with boundary data $(\mu_{a'}; \mu_{b'})$. Furthermore, \Cref{mutrhot} implies that the densities $\varrho_{a'}$ of $\mu_{a'}$ and $\varrho_{b'}$ of $\mu_{b'}$ are compactly supported and uniformly bounded. Hence, \Cref{gtabab}, the second part of \Cref{t:LDP}, and \Cref{aab1} (with the $(a, b; \mu_a, \mu_b)$ there equal to $(a', b'; \mu_a, \mu_b)$ here) imply that $\varrho_t^{\star} (x)$ and $u_t^{\star} (x)$ are smooth for $(t, x)$ in a neighborhood of $(t_0, x_0)$. Together with \eqref{hurhoequation} (again with \Cref{aab1}), this implies that $H^{\star}$ is smooth in a neighborhood of $(t_0, x_0)$, verifying the lemma. 
\end{proof}

\section{Proofs of Results From \Cref{StatisticsKernel}} 

\label{ProofMotion}

\subsection{Proof of \Cref{kernelconvergew}}

\label{Proofw} 

In this section we establish \Cref{kernelconvergew}. We first require the following definition from \cite{CLSM}, which defines a class of initial data for which it is possible to show quick convergence of Dyson Brownian motion. In what follows, we recall the Stieltjes transform $m_{\nu}$ associated with a measure $\nu$ from \eqref{mz0}, and the empirical measure $\emp (\bm{a})$ associated with a sequence $\bm{a}$ from \eqref{aemp}. 

\begin{definition}[{\cite[Definition 2.1]{CLSM}}]
	
	\label{regulare}
	
	Fix real numbers $E_0 \in \mathbb{R}$ and $\zeta > 0$. For each positive real number $n$, let $r = r_n$ and $R = R_n$ be two parameters satisfying $n^{\zeta - 1} \le r \le n^{-\zeta}$ and $n^{\zeta} r \le R \le n^{-\zeta}$. We call an $n$-tuple $\bm{d} = (d_1, d_2, \ldots , d_n) \in \mathbb{R}^n$ \emph{$(r, R; \zeta)$-regular with respect to $E_0$} if the following holds. Denoting $\nu = \nu^n = \emp (\bm{d})$, there exist constants $c \in (0, 1)$ and $C > 1$ (both independent of $n$) such that $\max_{j \in \llbracket 1, n \rrbracket} |d_j| \le n^C$ and 
	\begin{flalign*}
		c \le \Imaginary m_{\nu^n} (E + \mathrm{i} \eta)  \le C, \qquad \text{for all $n \ge C$ and} \qquad (E, \eta) \in (E-R, E + R) \times (r, 10).
	\end{flalign*}

\end{definition}

\begin{lem} 
	
	\label{regularrho}
	
	On the event $\mathscr{A}_0$, the $n'$-tuples $\widetilde{\bm{w}}^- (0)$, $\widetilde{\bm{w}} (0)$, and $\widetilde{\bm{w}}^+ (0)$ are $\big( (n')^{\omega/2-1}; (n')^{-\omega/2}; \frac{\omega}{2} \big)$-regular with respect to $0$.  
\end{lem} 

\begin{proof}
	
	We only show the lemma for $\widetilde{\bm{w}} (0)$, as the proof is entirely analogous for $\widetilde{\bm{w}}^- (0)$ and $\widetilde{\bm{w}}^+ (0)$; throughout, we restrict to the event $\mathscr{A}_0$ and denote $\widetilde{\nu} = \widetilde{\nu}^n = \emp \big(\widetilde{\bm{w}} (0) \big)$. Then, for any $z \in \mathbb{H}$, 
	\begin{flalign*}
		m_{\nu} (z) = \displaystyle\frac{1}{n'} \displaystyle\sum_{j=1}^{n'} \displaystyle\frac{1}{\widetilde{w}_j (0) - z}.
	\end{flalign*} 
	
	\noindent Let $z = E + \mathrm{i} \eta$ for some real numbers $E$ and $\eta$ with $|E| \le (n')^{-\omega/2}$ and $(n')^{\omega/2-1} \le \eta \le 10$. Then, 
	\begin{flalign*}
		\big| m_{\nu} (z) \big| & \le \Bigg| \displaystyle\frac{1}{n'} \displaystyle\sum_{j=1}^{n'} \displaystyle\frac{1}{\widetilde{G} (0, jn'^{-1}) - z} \Bigg| + \displaystyle\frac{1}{n'} \displaystyle\sum_{j=1}^{n'} \displaystyle\frac{\big| \widetilde{G} (0, jn'^{-1}) - \widetilde{w}_j (0) \big|}{\big| \widetilde{w}_j (0) - z \big| \cdot \big| \widetilde{G} (0, jn'^{-1}) - z \big|} \\
		& \le \Bigg| \displaystyle\frac{1}{n'} \displaystyle\sum_{j=1}^{n'} \displaystyle\frac{1}{\widetilde{G} (0, jn'^{-1}) - z} \Bigg| + \displaystyle\frac{1}{n' \eta^2} \displaystyle\max_{j \in \llbracket 1, n' \rrbracket} \bigg| \widetilde{G} \Big( 0, \frac{j}{n'} \Big) - \widetilde{w}_j (0) \bigg| \\
		& \le \Bigg| \displaystyle\frac{1}{n'} \displaystyle\sum_{j=1}^{n'} \displaystyle\frac{1}{\widetilde{G} (0, jn'^{-1}) - z} \Bigg| + \eta^{-2} (n')^{\omega/2-2} \le \Bigg| \displaystyle\frac{1}{n'} \displaystyle\sum_{j=1}^{n'} \displaystyle\frac{1}{\widetilde{G} (0, jn'^{-1}) - z} \Bigg| + (n')^{-\omega/2},
	\end{flalign*} 
	
	\noindent where in the second bound we used the fact that $\big| \widetilde{w}_j (0) - z \big| \cdot \big| \widetilde{G} (0, \frac{j}{n'}) - z \big| \ge (\Imaginary z)^2 = \eta^2$; for the third we used the fact that we are restricting to $\mathscr{A}_0$, together with \eqref{a0g2}; and in the fourth we used the fact that $\eta \ge (n')^{\omega/2 - 2}$. 
	
	Now let $j_0 \in \llbracket 1, n' \rrbracket$ be the integer (if it exists) such that $\widetilde{G} \big(0, \frac{j_0}{n'} \big) \le E < \widetilde{G} \big( 0, \frac{j_0 - 1}{n'} \big)$. Observe from \eqref{g0b} that $\partial_y \widetilde{G} \le -\frac{\varepsilon}{2}$ and $\widetilde{G} (0, \frac{1}{2}) = 0$. Therefore, since $|E| \le (n')^{-\omega/2}$, it follows for sufficiently large $n$ that $j_0$ exists and 
	\begin{flalign*} 
		\Big| j_0 - \frac{n'}{2} \Big| \le 2 \varepsilon^{-1} (n')^{-\omega/2}.
	\end{flalign*} 
	
	We will assume in what follows that $j_0 \ge \frac{n'}{2}$, as the proof in the alternative case when $j_0 < \frac{n'}{2}$ is entirely analogous. Then, 
	\begin{flalign*}
		\Bigg| & \displaystyle\frac{1}{n'} \displaystyle\sum_{j=1}^{n'} \displaystyle\frac{1}{\widetilde{G} (0, jn'^{-1}) -  z} \Bigg| \\
		& \le \Bigg| \displaystyle\frac{1}{2 n'} \displaystyle\sum_{j = 2j_0 - n'}^{n'} \bigg( \displaystyle\frac{1}{\widetilde{G} (0, jn'^{-1}) - z} + \displaystyle\frac{1}{\widetilde{G} (0, 2j_0 n'^{-1} - jn'^{-1}) - z} \bigg) \Bigg|  + \Bigg| \displaystyle\frac{1}{n'} \displaystyle\sum_{j=1}^{2j_0 - n' - 1} \displaystyle\frac{1}{\widetilde{G} (0, jn'^{-1}) - z} \Bigg| \\ 
		& \le \Bigg| \displaystyle\frac{1}{n'} \displaystyle\sum_{j=2j_0-n'}^{n'} \displaystyle\frac{\widetilde{G} (0, jn'^{-1}) + \widetilde{G} (0, 2j_0 n'^{-1} - jn'^{-1}) - 2z}{\big( \widetilde{G} (0, jn'^{-1}) - z\big) \big( \widetilde{G} (0, 2j_0 n'^{-1} - jn'^{-1}) - z \big)} \Bigg| + (n')^{-\omega/8},
	\end{flalign*}
	
	\noindent where in the second bound we used the fact that for $j < 2j_0 - n' \le 4 \varepsilon^{-1} (n')^{1-\omega/2}$ we have 
	\begin{flalign*} 
		\bigg| \widetilde{G} \Big(0, \frac{j}{n'} \Big) - z \bigg| & \ge  \widetilde{G} \Big( 0, \displaystyle\frac{j}{n'} \Big) - E \\
		&  \ge \widetilde{G} \Big(0, \frac{j}{n'} \Big) - \widetilde{G} \Big( 0, \frac{j_0-1}{n'} \Big) \ge  \frac{\varepsilon (j_0 - j-1)}{2 n'} \ge \frac{\varepsilon}{4} - \displaystyle\frac{10(n')^{-\omega/2}}{\varepsilon} \ge \frac{\varepsilon}{8} \ge (n')^{-\omega/8},
	\end{flalign*}
	
	\noindent by \eqref{g0b}. We  additionally have for sufficiently large $n$ that 
	\begin{flalign}
		\label{gj1}
		\begin{aligned}
			\bigg| \widetilde{G} \Big( 0, \frac{j}{n'} \Big) + \widetilde{G} \Big( 0, \displaystyle\frac{2j_0 - j}{n'} \Big) - 2z \bigg| & \le \bigg| \widetilde{G} \Big( 0, \displaystyle\frac{j}{n'} \Big) + \widetilde{G} \Big( 0, \displaystyle\frac{2j_0 - j}{n'} \Big) - 2 \widetilde{G} \Big( 0, \displaystyle\frac{j_0}{n'} \Big) \bigg| + 2 \eta + \displaystyle\frac{4}{\varepsilon n'} \\
			& \le \displaystyle\frac{2 B_0 n'}{n} \cdot \Big( \displaystyle\frac{j_0 - j}{n'} \Big)^2 + 3 \eta = \displaystyle\frac{2 B_0 (j-j_0)^2}{nn'} + 3\eta, 
		\end{aligned} 
	\end{flalign}
	
	\noindent where in the first inequality we used the facts that $z = E + \mathrm{i} \eta$ with $\widetilde{G} \big( 0, \frac{j}{n'} \big) \le E < \widetilde{G} \big( 0, \frac{j-1}{n'} \big)$ and $-2 \varepsilon^{-1} \le \partial_y \widetilde{G} \le -\frac{\varepsilon}{2}$ (by \eqref{g0b}), and in the second we used the facts that $[\widetilde{G}]_2 \le \frac{2 B_0 n'}{n}$ (again by \eqref{g0b}) and that $\eta \ge (n')^{\omega-1}$. We also have that 
	\begin{flalign}
		\label{gj2} 
		\bigg| \widetilde{G} \Big( 0, \displaystyle\frac{j}{n'}  \Big) - z \bigg| \ge \displaystyle\frac{\varepsilon|j-j_0|}{4n'} + \displaystyle\frac{\eta}{4}; \qquad \bigg| G \Big( 0, \displaystyle\frac{2j_0 - j}{n'} \Big) - z \bigg| \ge \displaystyle\frac{\varepsilon |j-j_0|}{4n'} + \displaystyle\frac{\eta}{4},
	\end{flalign}
	
	\noindent where in both inequalities we again used the facts that $\partial_y \widetilde{G} \le -\frac{\varepsilon}{2}$. Together, \eqref{gj1} and \eqref{gj2} yield 
	\begin{flalign*}
		\displaystyle\frac{1}{n'} & \Bigg| \displaystyle\sum_{j=2j_0 - n'}^{n'} \displaystyle\frac{\widetilde{G} (0, jn'^{-1}) + \widetilde{G} (0, 2j_0 n'^{-1} - jn'^{-1}) - 2z}{\big( \widetilde{G} (0, jn'^{-1}) - z \big) \big( \widetilde{G} (0, 2j_0 n'^{-1} - jn'^{-1}) - z \big)} \Bigg| \\
		& \le \displaystyle\frac{1}{n'} \displaystyle\sum_{j=2j_0-n'}^{n'} \displaystyle\frac{\frac{2B_0|j-j_0|^2}{nn'} + 3 \eta}{\frac{\varepsilon^2 |j-j_0|^2}{16(n')^2} + \frac{\eta^2}{16}} \\ 
		& \le 48 B_0 \varepsilon^{-2} \displaystyle\sum_{j=0}^{n'} \displaystyle\frac{n^{-1} j^2 + \eta n'}{j^2 + (\eta n')^2} \le 48 B_0 \varepsilon^{-2} \Bigg( \displaystyle\sum_{j=0}^{\eta n'} \Big( \displaystyle\frac{j^2}{n (\eta n')^2} + \displaystyle\frac{1}{\eta n'} \Big) + \displaystyle\sum_{j=\eta n'}^{n'} \Big( \displaystyle\frac{1}{n} + \displaystyle\frac{\eta n'}{j^2} \Big) \Bigg) \le C,
	\end{flalign*}
	
	\noindent for some constant $C = C(\varepsilon, \delta, B_0) > 1$; this yields the lemma. 
\end{proof}

The following result of \cite{FEUM}\footnote{It was stated there for Dyson Brownian motion associated with the Gaussian Orthogonal Ensemble (GOE), but it applies equally well to that associated with the GUE, as stated here.} (see also earlier results from \cite{CLSM}) states that, if one runs Dyson Brownian motion on initial data that is $(r; R; \zeta)$-regular with respect to a point $E_0 \in \mathbb{R}$, then its local statistics after run for some time $t$ between $r$ and $R^2$ are given by the sine kernel. In what follows, we recall notation on free convolutions from \Cref{TransformConvolution}.

\begin{lem}[{\cite[Theorem 2.2]{FEUM}}]
	
	\label{convergelocal}
	
	For any integer $k \ge 1$, real numbers $\zeta, \varsigma > 0$, and compactly supported, smooth function $F \in \mathcal{C}^{\infty} (\mathbb{R}^k)$, there exist constants $c = c(k, \zeta, \varsigma) > 0$ and $C = C(k, \zeta, \varsigma, F) > 1$ such that the following holds. Adopt the notation and assumptions of \Cref{regulare}; fix some real number $t \in [n^{\varsigma} r, n^{-\varsigma} R^2]$; define the measure $\nu_t = \nu \boxplus \varrho_{\semci}^{(t)}$; and denote the the density of $\nu_t$ by $\varrho_t \in L^1 (\mathbb{R})$.  Fix some real number $E \in \big[ E_0 - \frac{R}{2}, E_0 + \frac{R}{2} \big]$; let $\bm{d} (t) = \big( d_1 (t), d_2 (t), \ldots , d_n (t) \big)$ denote Dyson Brownian motion, run for time $t$, with initial data $\bm{d}$; and denote its $k$-point correlation function by $p_{\bm{d} (t)}^{(k)}$ (recall \eqref{kp}). Then, we have 
	\begin{flalign}
		\label{pmt}
		\begin{aligned}
			& \Bigg|   \displaystyle\int_{\mathbb{R}^k} F (\bm{a}) p_{\bm{d}(t)}^{(k)} \left( E + \displaystyle\frac{\bm{a}}{n \varrho_t (E)} \right) d \bm{a} - \displaystyle\int_{\mathbb{R}^k} F (\bm{a}) p_{\sin}^{(k)} (\bm{a}) d \bm{a} \Bigg| < C n^{-c}. 
		\end{aligned}
	\end{flalign}

\end{lem}

Now we can outline the proof of \Cref{kernelconvergew}.

\begin{proof}[Proof of \Cref{kernelconvergew} (Outline)]
	
	We only address the case when $\mathfrak{a} = 0$, as the proof is entirely analogous if $\mathfrak{a} \ne 0$. Recalling the notation from \Cref{ProofA0} (and, in particular, \eqref{syg}), denote the measure $\nu_0 = (n')^{-1} \sum_{j=1}^{n'} \delta_{\tilde{w}_j (0)} \in \mathscr{P}_0$. For any real number $s \ge 0$, denote its free convolution with the semicircle law by $\nu_s = \nu_0 \boxplus \mu_{\semci}^{(s)}$, and denote the associated density by $\varrho_s \in L^1 (\mathbb{R})$, which satisfies $\nu_s (dx) = \varrho_s (x) dx$. 
	
	Observe that $\widetilde{s}_0 = (n')^{\omega-1}$ (by \eqref{omegann} and \eqref{syg}) and that $\widetilde{G} \big( \widetilde{s}_0, \frac{1}{2} \big) = 0$ (by \eqref{g0b}). This, together with \Cref{regularrho} and \Cref{convergelocal} (with the parameters $(n; E_0, E; t; r, R; \zeta, \varsigma)$ there equal to $\big( n'; 0, 0, \widetilde{s}_0; (n')^{\omega/2-1}, (n')^{-\omega/2}; \frac{\omega}{2}, \frac{\omega}{4} \big)$ here) then imply that 
	\begin{flalign*}
		\displaystyle\lim_{n \rightarrow \infty} \displaystyle\int_{\mathbb{R}^k} F(\bm{a}) p_{\tilde{\bm{w}} (\tilde{s}_0)}^{(k)} \bigg( \displaystyle\frac{\bm{a}}{n' \varrho_{\tilde{s}_0} (0)} \bigg) d \bm{a} = \displaystyle\int_{\mathbb{R}^k} F(\bm{a}) p_{\sin}^{(k)} (\bm{a}) d \bm{a}.
	\end{flalign*} 

	\noindent It follows from applying the change of variables replacing $\bm{a}$ with $\frac{n'}{n} \cdot \bm{a}$, and using \eqref{syg}, and  that 
	\begin{flalign*}
		\displaystyle\lim_{n \rightarrow \infty} \displaystyle\int_{\mathbb{R}^k} F (\bm{a}) p_{\breve{\bm{w}} (s_0)}^{(k)} \bigg( \displaystyle\frac{\bm{a}}{n \varrho_{\tilde{s}_0} (0)} \bigg) d \bm{a} = \displaystyle\int_{\mathbb{R}^k} F (\bm{a}) p_{\sin}^{(k)} (\bm{a}) d \bm{a}. 
	\end{flalign*} 	

	\noindent Since $\theta = \partial_y G(t_0, y_0) = \partial_y \widetilde{G} \big( \widetilde{s}_0, \frac{1}{2} \big)$ and since $\lim_{n \rightarrow \infty} \partial_x \widetilde{G} \big( \widetilde{s}_0, \frac{1}{2} \big) = \partial_x \widetilde{G} \big( 0, \frac{1}{2} \big)$ (as $\lim_{n \rightarrow \infty} \widetilde{s}_0 = 0$ and $[\widetilde{G}]_1 = [G]_1 \le B_0$), it therefore suffices to show that $\lim_{n \rightarrow \infty} \varrho_{\tilde{s}_0} (0) = \partial_x \widetilde{G} \big( 0, \frac{1}{2} \big)$.
	
	Stated alternatively, since the approximate density of $\widetilde{w}_j (0)$ around $0$ is $\partial_x \widetilde{G} \big( 0, \frac{1}{2} \big)$, this amounts to showing that the effect of applying the free convolution with $\mu_{\semci}^{\tilde{s}_0}$ for $\widetilde{s}_0$ small has a negligible effect on this density. This is a standard  statement about free convolutions that has been proven and applied numerous times in the random matrix literature; we therefore do not repeat its proof here and refer, for example, to \cite[Lemmas 3.3 and 3.4]{BUSRM} for very similar arguments.
\end{proof}

\subsection{Proof Outline of \Cref{eventa1}} 

\label{Proofa2}

In this section we provide a brief outline for the proof of \Cref{eventa1}, using results of \cite{CLSM}.

\begin{proof}[Proof of \Cref{eventa1} (Outline)]
	
	We only outline the proof of the first statement of \eqref{wjwj}, as that of the second is entirely analogous. Following \cite[Equation (4.10)]{CLSM}, we define the processes $\bm{\xi} (s) = \big( \xi_1 (s), \xi_2 (s), \ldots , \xi_{n'} (s) \big)$ and $\bm{\xi}^- (s) = \big( \xi_1^- (s), \xi_2^- (s), \ldots , \xi_{n'}^- (s) \big)$ by for each pair $(j, s) \in \llbracket 1, n' \rrbracket \times [0, n s_0]$  setting 
	\begin{flalign}
		\label{xi11} 
		\xi_j (s) = n \cdot \breve{w}_j \Big( \frac{s}{n} \Big); \qquad \xi_j^- = n \cdot \breve{w}_j^- \Big( \frac{s}{n} \Big),
	\end{flalign}

	\noindent which satisfy the stochastic differential equations
	\begin{flalign}
		\label{xi1} 
		\partial_t \xi_i (t) = \displaystyle\sum_{\substack{1 \le j \le n' \\ j \ne i}} \displaystyle\frac{dt}{\xi_i (t) - \xi_j (t)} + dB_i (t); \qquad \partial_t \xi_i^- (t) = \displaystyle\sum_{\substack{1 \le j \le n' \\ j \ne i}} \displaystyle\frac{dt}{\xi_i^- (t) - \xi_j^- (t)} + dB_i (t).
	\end{flalign}
	
	We further define for any distinct integers $j, k \in \llbracket 1, n' \rrbracket$ the quantity 
	\begin{flalign*} 
		\Delta_{jk} = (2 \cdot \textbf{1}_{j < k} - 1) \cdot (n')^{-500}.
	\end{flalign*} 

	\noindent Then following \cite[Equation (4.16)]{CLSM} we define the processes $\widehat{\bm{\xi}} = \big( \widehat{\xi}_1 (s), \widehat{\xi}_2 (s), \ldots , \widehat{\xi}_{n'} (s) \big)$ and $\widehat{\bm{\xi}}^- (s) = \big( \widehat{\xi}_1^- (s), \widehat{\xi}_2^- (s), \ldots , \widehat{\xi}_{n'}^- (s) \big)$ obtained by ``regularizing'' the denominators on the right sides of \eqref{xi1}, namely, as the (strong) solutions to the stochastic differential equations 
	\begin{flalign}
		\label{xi2} 
		\partial_t \widehat{\xi}_i (t) = \displaystyle\sum_{\substack{1 \le j \le n' \\ j \ne i}} \displaystyle\frac{dt}{\xi_i (t) - \xi_j (t) + \Delta_{ij}} + dB_i (t); \qquad \partial_t \widehat{\xi}_i^- (t) = \displaystyle\sum_{\substack{1 \le j \le n' \\ j \ne i}} \displaystyle\frac{dt}{\xi_i^- (t) - \xi_j^- (t) + \Delta_{ij}} + dB_i (t),
	\end{flalign} 

	\noindent with initial data $\widehat{\bm{\xi}} (0) = \bm{\xi} (0)$ and $\widehat{\bm{\xi}}^- (0) = \bm{\xi}^- (0)$. It then follows from\footnote{Let us mention that there are two differences between the setup of \cite[Lemma 4.4]{CLSM} and that here. First, \cite{CLSM} considers the $\beta = 1$ (GOE) case of Dyson Brownian motion, while here we consider the $\beta = 2$ (GUE) case. Second, the differential equations \eqref{xi1} and \eqref{xi2} in \cite{CLSM} have an additional $-\frac{1}{2n'} \cdot \xi_j (t)$ term on their right sides. However, it is quickly verified that neither point affects the proofs or statements of the results from \cite{CLSM}.} \cite[Lemma 4.4]{CLSM} (together with \Cref{regularrho}) that $\widehat{\bm{\xi}}$ and $\widehat{\bm{\xi}}^-$ likely are close to $\bm{\xi}$ and $\bm{\xi}^-$, respectively, namely, with probability at least $1 - (n')^{90}$ we have
	\begin{flalign}
		\label{xi10} 
		\displaystyle\sup_{\substack{t \in [0, (n')^{1-3\omega}] \\ j \in \llbracket n'/2 - (n')^{1-\omega}, n'/2 + (n')^{1-\omega} \rrbracket}} \Big( \big| \xi_j (t) - \widehat{\xi}_j (t) \big| + \big| \xi_j^- (t) - \widehat{\xi}_j^- (t) \big| \Big) \le (n')^{-15}.
	\end{flalign}

	Now set\footnote{The choice of $K$ in \cite{CLSM} was slightly different, but it can quickly be verified that the below statements will valid for $K = (n')^{1/6}$ as well. Indeed, fixing $K = (n')^{1/6}$ correponds in \cite{CLSM} to imposing that $\widetilde{\bm{w}}$ be $\big( n^{-2/3-o(1)}, n^{-1/3-o(1)}; \frac{1}{3} + o(1) \big)$-regular around $0$; this is guaranteed \Cref{regularrho} (which implies that $\widetilde{\bm{w}}$ is even more regular than this).} $K = (n')^{1/6}$. Following \cite[Equation (4.24)]{CLSM}, we define the processes of $(2K+1)$-tuples $\widetilde{\bm{\xi}} (s) = \big( \widetilde{\xi}_{n'/2-K} (s), \ldots , \widetilde{\xi}_{n'/2+K} (s) \big)$ and $\widetilde{\bm{\xi}}^- (s) = \big( \widetilde{\xi}_{n'/2-K}^- (s), \ldots , \widetilde{\xi}_{n'/2+K}^- (s) \big)$ by only retaining the middle $2K+1$ of the paths in \eqref{xi2}, namely, as the (strong) solutions to the stochastic differential equations
	\begin{flalign}
		\label{xi3} 
		\begin{aligned}
		& \partial_t \widetilde{\xi}_i (t) = \displaystyle\sum_{\substack{n'/2-K \le j \le n'/2+K \\ j \ne i}} \displaystyle\frac{dt}{\xi_i (t) - \xi_j (t) + \Delta_{ij}} + dB_i (t); \\
		& \partial_t \widetilde{\xi}_i^- (t) = \displaystyle\sum_{\substack{n'/2-K \le j \le n'/2+K \\ j \ne i}} \displaystyle\frac{dt}{\xi_i^- (t) - \xi_j^- (t) + \Delta_{ij}} + dB_i (t),
		\end{aligned} 
	\end{flalign}

	\noindent with initial data $\widetilde{\xi}_i (0) = \widehat{\xi}_i (0)$ and $\widetilde{\xi}_i^- (0) = \widehat{\xi}_i^- (0)$, for each $i \in \big\llbracket \frac{n'}{2} - K, \frac{n'}{2} + K \big\rrbracket$. Observe from \eqref{uj2} and the fact that $G \big( t_0 - s_0, y_0 -\frac{n'}{2n} + \frac{j}{n} \big) = G^- \big(y_0 - \frac{n'}{2n} + \frac{j}{n} \big)$ for $j \in \big\llbracket \frac{n'}{2} - K, \frac{n'}{2} + K \big\rrbracket$ that 
	\begin{flalign}
		\label{xi4} 
		\widetilde{\bm{\xi}} (0) = \widetilde{\bm{\xi}}^- (0).
	\end{flalign}

	Now define the processes of $(2K+1)$-tuples $\widehat{\bm{\vartheta}} (s) = \big( \widehat{\vartheta}_{n'/2-K} (s), \ldots , \widehat{\vartheta}_{n'/2+K} (s) \big)$ and $\widetilde{\bm{\vartheta}} (s) = \big( \widetilde{\vartheta}_{n'/2-K} (s), \ldots , \widetilde{\vartheta}_{n'/2+K} (s) \big)$ by setting
	\begin{flalign*}
		\widehat{\vartheta}_i (s) = \widehat{\xi}_i (s) - \widehat{\xi}_i^- (s), \quad \text{and} \quad \widetilde{\vartheta}_i (s) = \widetilde{\xi}_i (s) - \widetilde{\xi}_i^- (s); \qquad \text{for each $j \in \Big\llbracket \displaystyle\frac{n'}{2} - K, \displaystyle\frac{n'}{2} + K \Big\rrbracket$}.
	\end{flalign*}

	\noindent 
	To compare $\widehat{\bm{\vartheta}} (s) $ and $\widetilde{\bm{\vartheta}} (s)$ first observe, by \eqref{uj2} and the fact that $\big| G(t_0 - s_0, y) - G^- (y) \big| \le n^{\omega/2-1}$ (and that $n' = n^{1/2}$), that we have $\big| \xi_j (0) - \xi_j^- (0) \big| \le (n')^{\omega}$ for each $j \in \llbracket 1, n' \rrbracket$. Together with \Cref{lambdamonotone}, it follows that $\big| \xi_j (s) - \xi_j^- (s) \big| \le (n')^{\omega}$ for each $s > 0$. Hence, subtracting \eqref{xi3} from \eqref{xi2} and integrating over $t \in [0, n s_0]$, we obtain
	\begin{flalign}
		\label{xi5} 
		\begin{aligned}
		\big| \widetilde{\vartheta}_i (ns_0) - \widehat{\vartheta}_i (ns_0) \big| & = \Big| \big( \widetilde{\xi}_i (n s_0) - \widetilde{\xi}_i^- (n s_0) \big) - \big( \widehat{\xi}_i (ns_0) - \widehat{\xi}_i^- (ns_0) \big) \Big| \\
		& \le ns_0 \cdot \displaystyle\sup_{t \in [0, ns_0]} \Bigg| \displaystyle\sum_{j : |j-n'/2| > K} \bigg( \displaystyle\frac{1}{\xi_i (t) - \xi_j (t) + \Delta_{ij}} - \displaystyle\frac{1}{\xi_i^- (t) - \xi_j^- (t) + \Delta_{ij}} \bigg) \Bigg| \\
		& \le n s_0 \cdot \displaystyle\sup_{t \in [0, n s_0]}  \displaystyle\sum_{j : |j-n'/2| > K} \displaystyle\frac{\big| \xi_i (t) - \xi_i^- (t) \big| + \big| \xi_j (t) - \xi_j^- (t) \big|}{\big| \xi_i (t) - \xi_j (t) + \Delta_{ij}\big| \big| \xi_i^- (t) - \xi_j^- (t) + \Delta_{ij} \big|}, \\
		& \le 2 (n')^{2 \omega} \displaystyle\sup_{t \in [0, ns_0]} \sum_{j : |j-n'/2| > K} \displaystyle\frac{1}{\big| \xi_i (t) - \xi_j (t) \big| \big| \xi_i^- (t) - \xi_j^- (t) \big|},
		\end{aligned} 
	\end{flalign} 

	\noindent  for any $i \in \big\llbracket \frac{n'-K}{2}, \frac{n'+K}{2} \big\rrbracket$, where in the second bound we used the fact that $ns_0 = (n')^{\omega}$. By \Cref{concentrationequation}, and the first statement of \Cref{hestimate0}, there exists a constant $c > 0$ such that with probability $1 - c^{-1} e^{-c (\log n)^2}$ we have $\big| \xi_i (t) - \xi_j (t) \big| \ge \frac{\varepsilon}{8} |i-j| \ge (n')^{-\omega} K$ and $\big| \xi_i (t) - \xi_j (t) \big| \ge \frac{\varepsilon}{8} |i-j| \ge (n')^{-\omega} K$ whenever $\big| j - \frac{n'}{2} \big| > K$ and $i \in \big\llbracket \frac{n'-K}{2}, \frac{n'+K}{2} \big\rrbracket$. Inserting this into \eqref{xi5} yields 
	\begin{flalign}
		\label{xi9} 
		\big|  \widetilde{\vartheta}_i (ns_0) - \widehat{\vartheta}_i (ns_0)  \big| \le 2 (n')^{4\omega} K^{-2} = 2 (n')^{4\omega - 1/3}, \qquad \text{for each $i \in \Big\llbracket \displaystyle\frac{n'-K}{2}, \displaystyle\frac{n'+K}{2} \Big\rrbracket$},
	\end{flalign}

	\noindent where we have recalled that $K = (n')^{1/6}$. 
	
	Following \cite[Equation (4.38)]{CLSM}, we next define the process $\mathring{\bm{\vartheta}} (s) = \big( \mathring{\vartheta}_{n'/2-K} (s), \ldots , \mathring{\vartheta}_{n'/2+K} (s) \big)$ as the solution to the differential equation 
	\begin{flalign}
		\label{xi6}
		\partial_t \mathring{\vartheta}_i (t) = \displaystyle\sum_{\substack{n'/2 - K \le j \le n'/2 + K \\ j \ne i}} \displaystyle\frac{\big( \mathring{\vartheta}_i (s) - \mathring{\vartheta}_j (s) \big) dt}{\big( \xi_i (t) - \xi_j (t) + \Delta_{ij} \big)\big( \xi_i^- (t) - \xi_j^- (t) + \Delta_{ij} \big)},
	\end{flalign}

	\noindent with initial data $\mathring{\bm{\vartheta}} (0) = \widetilde{\bm{\vartheta}} (0) = (0, 0, \ldots , 0)$, where the last equality follows from \eqref{xi4}. Combining this initial condition with the explicit form of \eqref{xi6}, it follows
	\begin{flalign}
		\label{xi7} 
		\mathring{\vartheta}_i (s) = 0, \qquad \text{for each $i \in \Big\llbracket \displaystyle\frac{n'}{2} - K, \displaystyle\frac{n'}{2} + K \Big\rrbracket$}.
	\end{flalign} 
	
	Next, \cite[Lemma 4.8]{CLSM}\footnote{It was stated there when the $\bm{\xi}^-$ paths come from applying Dyson Brownian motion to the eigenvalues of a GOE random matrix, but it is quickly verified that the proof and statement of that result applies equally well for the $\bm{\xi}^-$ here (satisfying the regularity hypothesis shown in \Cref{regularrho}).} implies (with a union bound over $2n^{\omega/15000} + 1 = 2(n')^{\omega/7500} + 1 \le K^{\omega/1000}$ indices) that, with probability at least $1 - K^{-\omega/1000}$ we have
	\begin{flalign}
		\label{xi8} 
		\displaystyle\max_{j \in \llbracket n'/2 - n^{\omega/15000}, n'/2 + n^{\omega/15000} \rrbracket} \big| \mathring{\vartheta}_j (ns_0) - \widetilde{\vartheta}_j (ns_0) \big| \le K^{4 \omega - 1/2}.
	\end{flalign}
	
	\noindent Combining \eqref{xi8}, \eqref{xi7}, \eqref{xi9}, and \eqref{xi10} yields with probability at least $1 - n^{-\omega / 15000}$ that $\xi_j (ns_0) - \xi_j^- (ns_0) \le K^{4\omega - 1/2} + 2(n')^{4\omega-1/3} + (n')^{-15} \le (n')^{\omega - 1/12}$. Together with \eqref{xi11} (and the fact that $\omega \le \frac{1}{20}$), this yields the lemma. 
\end{proof}

\subsection{Proof Outline of \Cref{xy00}} 

\label{Proofwj02} 

In this section we provide a brief outline for the proof of \Cref{xy00}.

\begin{proof}[Proof of \Cref{xy00} (Outline)]
	
	Throughout this proof, we fix an integer $j \in \big\llbracket \frac{n'}{2} - n^{\omega/15000}, \frac{n'}{2} + n^{\omega/15000} \big\rrbracket$; we may also assume that $\chi < \omega$. It suffices to show that, for any real number $A \ge 1$, 
	\begin{flalign}
		\label{wj20}
		\mathbb{P} \Big[ \big| \breve{w}_j^- (s_0) \big| \ge An^5 \Big] < Ce^{-cAn}; \qquad  \mathbb{P} \Big[ \big| \breve{w}_j^+ (s_0) \big| \ge An^5  \Big]  < Ce^{-cAn},
	\end{flalign}
	
	\noindent and that, for any coupling between $\breve{\bm{w}}^-$ and $\breve{\bm{w}}^+$,
	\begin{flalign}
		\label{wjs02}
		\mathbb{P} \Big[ \big| \breve{w}_j^+ (s_0) - \breve{w}_j^- (s_0) \big| > n^{\chi-1} \Big] \le Ce^{-c(\log n)^2}.
	\end{flalign}
	
	\noindent Indeed, given these two bounds, it would then follow that 
	\begin{flalign*}
		\mathbb{E} \Big[ \big| \breve{w}_j^+ (s_0) - \breve{w}_j^+ (& s_0) \big| \cdot \textbf{1}_{|\breve{w}_j^+ (s_0) - \breve{w}_j^- (s_0)  \big| > n^{\chi-1}} \Big] \\
		& \le 2n^{10} \cdot \mathbb{P} \Big[ \big| \breve{w}_j^+ (s_0) - \breve{w}_j^- (s_0) \big| > n^{\chi-1} \Big] + \mathbb{E} \Big[ \big| \breve{w}_j^+ (s_0) \big| \cdot \textbf{1}_{|\breve{w}_j^+ (s_0)| > n^{10}} \Big] \\
		& \qquad + \mathbb{E} \Big[ \big| \breve{w}_j^- (s_0) \big| \cdot \textbf{1}_{|\breve{w}_j^- (s_0)| > n^{10}} \Big] \\
		& \le 2Cn^{10} e^{-c(\log n)^2} + n^5 \displaystyle\int_{n^5}^{\infty} \bigg( \mathbb{P} \Big[ \big| \breve{w}_j^+ (s_0) \big| \ge n^5 A \Big] + \mathbb{P} \Big[ \big| \breve{w}_j^- (s_0) \big| \ge n^5 A \Big] \bigg) dA \\
		& \le 2Cn^{10} e^{-c(\log n)^2} + 2 c^{-1} C n^5 e^{-cn},		
	\end{flalign*} 
	
	\noindent verifying the lemma. The bound \eqref{wj20} is a quick consequence of \Cref{estimatexj2} with the fact that Dyson Brownian motion $\bm{\lambda}(s_0)$, run for time $s_0 \ge 0$ with initial data given by  some $n$-tuple $\bm{a}$, has the same law as $(1 + s_0) \cdot \widetilde{\bm{\lambda}} \big( \frac{s_0}{1+s_0} \big)$, where $\widetilde{\bm{\lambda}}$ is a family of $n$ non-intersecting Brownian bridges on $[0, 1]$ sampled from the measure $\mathfrak{Q}^{\bm{a}; \bm{0}_n}$ (where $\bm{0}_n = (0, 0, \ldots , 0)$ consists of $n$ zeroes).

	To establish \eqref{wjs02}, it suffices to exhibit (deterministic) real numbers $\theta_j^-, \theta_j^+ \in \mathbb{R}$ such that 
	\begin{flalign}
		\label{wjs03}
		\mathbb{P} \Big[ \big| \breve{w}_j^+ (s_0) - \theta_j^+ \big| > n^{\chi/2-1} \Big] \le C e^{-c(\log n)^2}; \qquad \mathbb{P} \Big[ \big| \breve{w}_j^- (s_0) - \theta_j^- \big| > n^{\chi/2-1} \Big] \le Ce^{-c(\log n)^2}.
	\end{flalign}
	
	\noindent Indeed this, \Cref{eventa1}, a union bound, would yield a constant $c_1 = c_1 (\chi, \varepsilon, \delta, B_0, m) > 1$ such that 
	\begin{flalign*} 
		\mathbb{P} \big[ |\theta_j^+ - \theta_j^-| > n^{3\chi/4-1} \big] \ge \mathbb{P} \big[ |\theta_j^+ - \theta_j^-| > 2n^{\chi/2-1} + n^{-c_1 - 1} \big] \ge 1 - c_1^{-1} n^{-c_1} > 0.
	\end{flalign*} 
	
	\noindent Since $\theta_j^+$ and $\theta_j^-$ are deterministic, it follows that $|\theta_j^+ - \theta_j^-| \le n^{3\chi/4-1}$. Together with \eqref{wjs03}, a union bound, and the fact that $n^{\chi-1} > 2n^{\chi/2-1} + n^{3\chi/4-1}$ for sufficiently large $n$, this implies 
	\begin{flalign*} 
		\mathbb{P} \Big[ \big| \breve{w}_j^+ (s_0) - \breve{w}_j^+ (s_0) \big| > n^{\chi-1} \Big] & \le \mathbb{P} \Big[ \big| \breve{w}_j^+ (s_0) - \theta_j^+ \big| > n^{\chi/2-1} \Big] + \mathbb{P} \Big[ \big| \breve{w}_j^- (s_0) - \theta_j^- \big| > n^{\chi/2-1} \Big] \\ 
		& \le 2Ce^{-c(\log n)^2},
	\end{flalign*} 
	
	\noindent thereby confirming \eqref{wjs02}.
	
	It remains to establish \eqref{wjs03}; let us only explain the first bound there, since the proof of the second is entirely analogous. This will follow from the concentration bound \Cref{concentrationequation}. In particular, denote the measure $\nu_0^+ = (n')^{-1} \sum_{j=1}^{n'} \textbf{1}_{n \breve{u}_j^+ / n'}$. For any real number $s \ge 0$, also define its free convolution $\nu_s^+ = \nu_0^+ \boxplus \mu_{\semci}^{(s)}$ with the semicircle distribution; let $\varrho_s^+$ denote its density with respect to Lebesgue measure, and for each $k \in \llbracket 1, n' \rrbracket$ let $\gamma_k^+ (s) = \gamma_{k;n'}^{\nu_s^+}$ denote its classical location (recall \Cref{gammarho}). By \Cref{concentrationequation} (and \Cref{mtscalebeta}), denoting $\widetilde{s}_0 = \frac{ns_0}{n'} = n^{\omega-1}$ (by \eqref{omegann}), there exist constants $c_2 > 0$ and $C_1 > 1$ such that 
	\begin{flalign*}
		\mathbb{P} \bigg[ \displaystyle\frac{n'}{n} \cdot \gamma_{j + \lfloor (\log n)^5 \rfloor}^+ (\widetilde{s}_0) - n^{-2} \le \breve{w}_j^+ (s_0) \le \displaystyle\frac{n'}{n} \cdot \gamma_{j - \lfloor (\log n)^5 \rfloor}^+ (\widetilde{s}_0) + n^{-2} \Big] \ge 1 - C_1 e^{-(\log n)^2}.
	\end{flalign*}
	
	\noindent Taking $\theta_j^+ = \frac{n'}{n} \cdot \gamma_j (\widetilde{s}_0)$, to prove \eqref{wjs03} it therefore suffices to show for sufficiently large $n$ that
	\begin{flalign*}
		\gamma_{j-\lfloor (\log n)^5 \rfloor}^+ (\widetilde{s}_0) - \gamma_{j+\lfloor (\log n)^5 \rfloor}^+ (\widetilde{s}_0) \le \displaystyle\frac{(\log n)^{10}}{n'}.
	\end{flalign*}
	
	\noindent This follows from the fact that $\varrho_{s_0}^+ (x) \ge \frac{\varepsilon}{2}$ for $x \in \big[ \gamma_j^+ (\widetilde{s}_0) - (n')^{2\omega}, \gamma_j^+ (\widetilde{s}_0) + (n')^{2\omega} \big]$, which can be deduced from the fact that $\widetilde{s}_0 = n^{1-\omega}$ is small (by \eqref{omegann}); that $\frac{n}{n'} \cdot \bm{\breve{w}}^+ (0) = \frac{n}{n'} \cdot \bm{\breve{u}}^+$ approximates $G$ to within $(n')^{\omega/2-1} \ll \widetilde{s}_0$, for indices $k \in \big\llbracket \frac{n' - n''}{2}, \frac{n' + n''}{2} \big\rrbracket$ (by \eqref{uj2} and \eqref{a0}); and that $G \in \Adm_{\varepsilon} (\mathfrak{R})$. This last statement is a standard one about free convolutions that has been proven and applied numerous times in the random matrix literature; we therefore do not repeat its proof here and refer, for example, to \cite[Lemmas 3.3 and 3.4]{BUSRM} for very similar arguments.
\end{proof}

\section{Proofs of Results From \Cref{DerivativesEquation}} 

\label{ProofEstimateEquation}

\subsection{Proofs of \Cref{maximumboundary}, \Cref{derivativef}, and \Cref{sumd}} 

\label{EstimateEquation}

In this section we first recall several results from \cite{EDSO} about regularity properties of uniformly elliptic partial differential equations and then use them to establish \Cref{maximumboundary}, \Cref{derivativef}, and \Cref{sumd}. To that end, we fix an integer $d \ge 1$ and an open subset $\mathfrak{R} \subseteq \mathbb{R}^d$, and we recall the notation from \Cref{ConcentrationSmooth0}. For any function $f \in \mathcal{C} (\mathfrak{R})$; real number $\alpha \in (0, 1)$; and integer $k \in \mathbb{Z}_{\ge 0}$, define $[f]_{\alpha} = [f]_{\alpha; \mathfrak{R}}$, $[f]_{k, \alpha} = [f]_{k, \alpha; \mathfrak{R}}$, and $\| f \|_{\mathcal{C}^{k, \alpha} (\mathfrak{R})} = \| f \|_{\mathcal{C}^{k, \alpha} (\overline{\mathfrak{R}})} = \| f \|_{k, \alpha} = \| f \|_{k, \alpha; \mathfrak{R}}$ by
\begin{flalign*}
	[f]_{\alpha} = \displaystyle\sup_{\substack{y, z \in \mathfrak{R} \\ y \ne z}} \displaystyle\frac{\big| f(y) - f(z) \big|}{|y - z|^{\alpha}}; \qquad [f]_{k, \alpha} = \displaystyle\max_{\substack{\gamma \in \mathbb{Z}_{\ge 0}^d \\ |\gamma| = k}} [ \partial_{\gamma} f ]_{\alpha}; \qquad \| f \|_{\mathcal{C}^{k, \alpha} (\mathfrak{R})} = \| f \|_k + [f]_{k, \alpha}. 
\end{flalign*} 

Define the H\"{o}lder space $\mathcal{C}^{k, \alpha} (\mathfrak{R})$ to be the set of $f \in \mathcal{C}^k (\mathfrak{R})$ such that $[f]_{k, \alpha; \mathfrak{R}_0} < \infty$, for any compact subset $\mathfrak{R}_0 \subset \mathfrak{R}$. Also define $\mathcal{C}^{k, \alpha} (\overline{\mathfrak{R}}) = \big\{ f \in \mathcal{C}^k (\mathfrak{R}): [f]_{k, \alpha; \mathfrak{R}} < \infty \big\}$. Then $\| f \|_{k, \alpha}$ is a norm on $\mathcal{C}^{k, \alpha} (\overline{\mathfrak{R}})$. By restriction, we have that $\mathcal{C}^{k, \alpha} (\mathfrak{R}) \subseteq \mathcal{C}^{k, \alpha} (\mathfrak{U})$ for any subset $\mathfrak{U} \subseteq \mathfrak{R}$. Observe that, for any $f, g \in \mathcal{C}^{0, \alpha} (\mathfrak{R})$, these norms satisfy 
\begin{flalign}
	\label{fgestimatefg}
	\begin{aligned}
		\| f g \|_{\mathcal{C}^{0, \alpha} (\mathfrak{R})} & \le  \| f \|_{\mathcal{C}^{0, \alpha} (\mathfrak{R})} \| g \|_0.
	\end{aligned}
\end{flalign}

We further say that the boundary $\partial \mathfrak{R}$ of some open set $\mathfrak{R} \subset \mathbb{R}^d$ is $\mathcal{C}^{k, \alpha}$, if for any point $z_0 \in \partial \mathfrak{R}$, there exists a real number $r> 0$ and an injective function $\psi = \psi_{z_0; r} : \mathcal{B}_r (z_0) \rightarrow \mathbb{R}^d$ such that 
\begin{align*}
\mathfrak{D}_{z_0; r; \psi; \mathfrak{R}} 
= \psi \big( \mathcal{B}_r (z_0) \cap \mathfrak{R} \big) \subset \mathbb{R}_{> 0} \times \mathbb{R}^{d-1}; 
\qquad
\psi \big( \mathcal{B}_r (z_0) \cap \partial \mathfrak{R} \big) \subset \{ 0 \} \times \mathbb{R}^{d-1}; 
\end{align*}
where 
$
\psi \in \mathcal{C}^{k, \alpha} \big( \mathcal{B}_r (z_0) \big)
$
and
$
\psi^{-1} \in \mathcal{C}^{k, \alpha} (\mathfrak{D}_{z_0; r; \psi;\mathfrak{R}}). 
$
We call this boundary $\mathcal{C}^k$ if it is $\mathcal{C}^{k, 0}$, and we call it smooth if it is $\mathcal{C}^k$ for each integer $k \ge 1$.

The following lemma provides the \emph{Schauder estimates} for solutions to uniformly elliptic partial differential equations with H\"{o}lder continuous coefficients; as in \Cref{derivativef} and \Cref{perturbationbdk}, we state a form of it that addresses interior ($\mathfrak{R}'$ empty) and boundary regularity ($\mathfrak{R}' = \mathfrak{R}$) simultaneously. Although these bounds hold in all dimensions $d \ge 2$, we only state them for $d = 2$. 

\begin{lem}[{\cite[Theorem 6.6]{EDSO}}]
	
	\label{aijuestimates}
	
	Fix real numbers $r > 0$, $B > 1$, and $\alpha \in (0, 1)$, and bounded open sets $\mathfrak{R}' \subseteq \mathfrak{R} \subset \mathbb{R}^2$ such that $\mathfrak{R}'$ has $\mathcal{C}^{2,\alpha}$ boundary. There exists a constant $C = C(r, B, \alpha, \mathfrak{R}') > 1$ such that the following holds. For each $j, k \in \{ 1, 2 \}$, let $a_{jk}, b_j \in \mathcal{C}^{0, \alpha} (\mathfrak{R})$ denote functions such that $a_{xy} = a_{yx}$ and
	\begin{flalign*}
		\displaystyle\max_{j, k \in \{ 1, 2 \}} \| a_{jk} \|_{0, \alpha} < B; \qquad \displaystyle\max_{j \in \{ 1, 2 \}} \| b_j \|_{0, \alpha} < B; \qquad \displaystyle\inf_{z \in \mathfrak{R}} \displaystyle\sum_{j, k \in \{ 1, 2 \}} a_{jk} (z) \xi_i \xi_j \ge B^{-1} |\xi|^2,
	\end{flalign*}
	
	\noindent all hold for any $\xi = (\xi_1, \xi_2) \in \mathbb{R}^2$. Let $g \in \mathcal{C}^{0, \alpha} (\mathfrak{R})$ denote a function, and suppose that $F \in \mathcal{C}^{2, \alpha} (\mathfrak{R})$ satisfies the elliptic partial differential equation 
	\begin{flalign}
		\label{aijubiuf}
		\displaystyle\sum_{j, k \in \{ 1, 2 \}} a_{jk} (z) \partial_j \partial_k F (z) + \displaystyle\sum_{j \in \{ 1, 2 \}} b_j (z) \partial_j F (z) = g, \quad \text{for each $z \in \mathfrak{R}$.}
	\end{flalign}
	
	\noindent Assume that there exists some function $\varphi \in \mathcal{C}^{2,\alpha} (\overline{\mathfrak{R}})$ such that $F(z) = \varphi(z)$, for each $z \in \partial \mathfrak{R}' \cap \partial \mathfrak{R}$. Then, recalling the set $\mathfrak{D}_r \subseteq \mathfrak{R}$ from \eqref{dr}, we have 
	\begin{flalign*} 
		\| F \|_{\mathcal{C}^{2,\alpha} (\mathfrak{D}_r)} \le C \big( \| F \|_{\mathcal{C}^0 (\mathfrak{R})} + | \varphi \|_{\mathcal{C}^{2,\alpha} (\mathfrak{R})} + \| g \|_{\mathcal{C}^{0, \alpha} (\mathfrak{R})} \big).
	\end{flalign*}

\end{lem}

We will be interested in properties of solutions to certain families of non-linear, uniformly elliptic partial differential equations. The following result provides a comparison principle for such solutions. 

\begin{lem}[{\cite[Theorem 10.1]{EDSO}}]
	
	\label{aijcomparison} 
	
	Fix a constant $B > 0$ and a bounded open subset $\mathfrak{R} \subset \mathbb{R}^2$. For each $j, k \in \{ 1, 2 \}$, let $a_{jk} \in \mathcal{C}^1 (\mathbb{R}^2)$ denote functions such that $a_{xy} = a_{yx}$ and
	\begin{flalign}
		\label{xijxikbsum}
		B^{-1} |\xi|^2 \le \displaystyle\sum_{j, k \in \{ 1, 2 \}} a_{jk} (w) \xi_j \xi_k \le B |\xi|^2, \qquad \text{for any $\xi = (\xi_1, \xi_2) \in \mathbb{R}^2$ and $w \in \mathbb{R}^2$.}
	\end{flalign}
	
	\noindent Moreover, for each $i \in \{ 1, 2 \}$, let $f_i: \partial \mathfrak{R} \rightarrow \mathbb{R}$ denote a continuous function, and suppose that $F_i: \overline{\mathfrak{R}} \rightarrow \mathbb{R}$ with $F_i \in \mathcal{C}^2 (\mathfrak{R}) \cap \mathcal{C} (\overline{\mathfrak{R}})$ is a solution to the partial differential equation
	\begin{flalign}
		\label{ajkfiequation1} 
		\displaystyle\sum_{j, k \in \{ 1, 2 \}} a_{jk} \big( \nabla F (z) \big) \partial_j \partial_k F (z) = 0, \quad \text{for each $z \in \mathfrak{R}$},
	\end{flalign}
	
	\noindent with boundary data $F_i |_{\partial \mathfrak{R}} = f_i$. If $f_1 (z) \ge f_2 (z)$ for each $z \in \partial \mathfrak{R}$, then $F_1 (z) \ge F_2 (z)$ for each $z \in \mathfrak{R}$. 
	
\end{lem}

The following lemma provides interior and boundary regularity estimates (which are stated simultaneously, as in \Cref{derivativef}) for solutions $F$ to equations of the form \eqref{ajkfiequation1}; it is known from \cite[Corollary 6.7]{EDSO}, \cite[Theorem 8.29]{EDSO}, \cite[Equation (13.41)]{EDSO}, together with \cite[Chapter 3.6]{EDE}.

\begin{lem}[{\cite{EDE, EDSO}}]
	
	\label{aijgradientfestimate} 
	
	For a fixed integer $m \in \mathbb{Z}_{> 1}$; real numbers $r, B \in \mathbb{R}_{> 0}$; and bounded open subsets $\mathfrak{R}' \subseteq \mathfrak{R} \subset \mathbb{R}^2$ with $\mathcal{C}^{m+1}$ boundaries, there exists a constant $C = C (r, B, m, \mathfrak{R}') > 1$ such that the following holds. For each $j, k \in \{ 1, 2 \}$, let $a_{jk} \in \mathcal{C}^{m - 1} (\mathbb{R}^2)$ denote functions with $a_{xy} = a_{yx}$ and
	\begin{flalign}
		\label{aijxiestimate}
		\displaystyle\max_{j, k \in \{ 1, 2 \}} \| a_{jk} \|_{\mathcal{C}^{m - 1} (\mathbb{R}^2)} \le B; \qquad \displaystyle\inf_{z \in \mathbb{R}^2} \displaystyle\sum_{j, k \in \{ 1, 2 \}} a_{jk} (z) \xi_j \xi_k \ge B^{-1} |\xi|^2,
	\end{flalign}
	
	\noindent for any $\xi = (\xi_1, \xi_2) \in \mathbb{R}^2$. Then for any continuous function $f \in \mathcal{C} (\partial \mathfrak{R})$ such that $\| f \|_0 \le B$, the following two statements hold. 
	\begin{enumerate} 
		\item There exists a unique $F \in \mathcal{C}^2 (\mathfrak{R}) \cap \mathcal{C} (\overline{\mathfrak{R}})$ satisfying the partial differential equation \eqref{ajkfiequation1} with boundary data $ F|_{\partial \mathfrak{R}} = f$. 
		\item Suppose there exists a function $\varphi \in \mathcal{C}^{m+1} (\overline{\mathfrak{R}}')$ such that $\| \varphi | \|_{\mathcal{C}^{m+1} (\mathfrak{R}')} \le B$ and $f(z) = \varphi (z)$ for each $z \in \partial \mathfrak{R}' \cap \partial \mathfrak{R}$. Letting $\mathfrak{D}_r = \big\{ z \in \mathfrak{R} : d ( z, \partial \mathfrak{R} \setminus \partial \mathfrak{R}') > r \big\}$, we have $\| F \|_{\mathcal{C}^m (\mathfrak{D}_r)} \le C$. 
	\end{enumerate}

\end{lem}

From the above statements, we can quickly establish \Cref{maximumboundary}, \Cref{derivativef}, and \Cref{sumd}.

\begin{proof}[Proof of \Cref{maximumboundary}]

	Observe that the second statement of the lemma follows from the first. Indeed, set $G^- (z) = F_2 (z) - \sup_{w \in \partial \mathfrak{R}} \big| F_1 (z) - F_2 (z)\big|$ and $G^+ (z) = F_2 (z) + \sup_{z \in \partial \mathfrak{R}} \big| F_1 (z) - F_2 (z) \big|$, for each $z \in \overline{\mathfrak{R}}$. Then, $G^-$ and $G^+$ solve \eqref{equationxtd} on $\mathfrak{R}$ (since $F_2$ does), and $G^- \le F_1 \le G^+$ on $\partial \mathfrak{R}$; hence, $G^- \le F_1 \le G^+$ on $\mathfrak{R}$, implying $\sup_{z \in \mathfrak{R}} \big| F_1 (z) - F_2 (z) \big| \le \sup_{z \in \partial \mathfrak{R}} \big| F_1 (z) - F_2 (z) \big|$. Taking $(F_1, F_2) = (F, 0)$ yields $\sup_{z \in \mathfrak{R}} \big| F(z) \big| \le \sup_{z \in \partial \mathfrak{R}} \big| F(z) \big|$. 
	
	Now we establish the first statement of the lemma, which would follow from \Cref{aijcomparison}, except that \eqref{xijxikbsum} might not be valid for a uniform constant $B > 1$. To circumvent this, we restrict to a subset of $\mathfrak{R}$. More precisely, assume to the contrary that $F_1 (z_0) > F_2 (z_0)$ for some $z_0 \in \mathfrak{R}$. Then, since $F_1, F_2 \in \mathcal{C} (\overline{\mathfrak{R}})$ and $F_1 |_{\partial \mathfrak{R}} \le F_2 |_{\partial \mathfrak{R}}$, there exists some real number $\varepsilon > 0$ and open subset $\mathfrak{R}'$ with $\overline{\mathfrak{R}'} \subset \mathfrak{R}$, such that $F_1 (z) - \varepsilon \le F_2 (z)$ for each $z \in \partial \mathfrak{R}'$, and $F_1 (z) - \varepsilon > F_2 (z)$ for each $z\in \mathfrak R'$. 
	
	Observe that $F_1 - \varepsilon$ solves \eqref{equationxtd}. Furthermore, since $F_1, F_2 \in \mathcal{C}^2 (\mathfrak{R}) \cap \Adm (\mathfrak{R})$, we have by compactness that $\big| \partial_y F_j (z) \big| + \big| \partial_y F_j (z) \big|^{-1} \le B$ is uniformly bounded by some constant $B > 1$, over $z \in \overline{\mathfrak{R}'}$. It follows from the explicit form \eqref{uvd} of the coefficients $\mathfrak{d}_{ij}$ that \eqref{xijxikbsum} holds (for a larger value of $B$, given by $\pi^2 B^4$). Applying \Cref{aijcomparison} then indicates that $F_1 (z) - \varepsilon \le F_2 (z)$ for all $z \in \mathfrak{R}'$. This is a contradiction for $z = z_0$, confirming the lemma.
\end{proof}		

\begin{proof}[Proof of \Cref{derivativef}]
	
	For each $i, j \in \{ t, y \}$, let $\widetilde{\mathfrak{d}}_{ij} : \mathbb{R}^2 \rightarrow \mathbb{R}$ be smooth functions such that $\widetilde{\mathfrak{d}}_{ij} (u, v) = \mathfrak{d}_{ij} (u, v)$ whenever $v \in [-\varepsilon^{-1}, -\varepsilon]$, and such that \eqref{aijxiestimate} holds for some constant $\widetilde{B} = \widetilde{B}(\varepsilon) > 1$. Then, \Cref{aijgradientfestimate} yields a constant $C = C(\varepsilon, r, B, m, \mathfrak{R}') > 1$ such that there is a unique solution $\widetilde{F}$ to the equation 
	\begin{flalign}
		\label{dij2} 
		\displaystyle\sum_{i, j \in \{ t, y \}} \widetilde{\mathfrak{d}}_{ij} \big( \nabla \widetilde{F} (z) \big) \partial_i \partial_j \widetilde{F} (z) = 0, \qquad \text{for $z \in \mathfrak{R}$, with $\widetilde{F} |_{\partial \mathfrak{R}} = f$},
	\end{flalign}
	
	\noindent and that this $\widetilde{F}$ satisfies $\| \widetilde{F} \|_{\mathcal{C}^m (\mathfrak{D}_r)} \le C$. Observe since $\partial_x F (z) \in [-\varepsilon^{-1}, \varepsilon]$ for each $z \in \mathfrak{R}$ (as $F \in \Adm_{\varepsilon} (\mathfrak{R})$); since $F$ satisfies \eqref{equationxtd}; and since $\mathfrak{d}_{ij} (u, v) = \widetilde{\mathfrak{d}}_{ij} (u, v)$ for $v \in [-\varepsilon^{-1}, \varepsilon]$, that $F$ also satisfies \eqref{dij2}. Hence, $F = \widetilde{F}$, and so $\| F \|_{\mathcal{C}^m (\mathfrak{D}_r)} = \| \widetilde{F} \|_{\mathcal{C}^m (\mathfrak{D}_r)} \le C$.		
\end{proof}

We can also establish \Cref{sumd} as a quick consequence of results from \cite{RSNEPS}.

\begin{proof}[Proof of \Cref{sumd}]

	This is a special case of \cite[Theorem 1]{RSNEPS}; let us explain how to match our setup to the one there. First some point $z_0 \in \mathfrak{R}$; it suffices to show that $F$ is real analytic in some neighborhood of $z_0$. We may assume by translation in what follows that $z_0 = 0$; let $r > 0$ be a real number such that $\mathcal{B}_{4r} \subseteq \mathfrak{R}$. By \Cref{derivativef}, there exists a constant $C > 1$ such that $\| F \|_{\mathcal{C}^2 (\mathcal{B}_{2r})} \le C$ and $F \in \mathcal{C}^3 (\mathcal{B}_{2r})$. 
	
	Now, define the function $\Phi : \mathbb{R}^3 \rightarrow \mathbb{R}$ by setting $\Phi(v_1, v_2, v_3)=v_2+\pi^2 v_1^{-4}  v_3$; also define the set $\mathfrak{E} = \big\{ (v_1,v_2,v_3)\in \mathbb{R}^3: \varepsilon\leq-v_1\leq \varepsilon^{-1}, |v_2| \le C, |v_3|\leq C \big\}$. Then we can reformulate the equation \eqref{equationxtd} as $\Phi(\partial_y F, \partial_{tt}F, \partial_{yy}F)=0$; since $F \in \Adm_{\varepsilon} (\mathfrak{R})$ and $\| F \|_{\mathcal{C}^2 (\mathcal{B}_{2r})} \le C$, we also have $\big (\partial_y F(z), \partial_{tt}F(z), \partial_{yy}F(z) \big)\in \mathfrak E$ for any $z\in \mathcal{B}_{2r}$. On $\mathfrak E$, $\Phi$ is real analytic; in particular, for any point $(v_1, v_2, v_3)\in \mathfrak E$ and integers $k_1, k_2, k_3\geq 0$, we have
	\begin{align*}
		\left|\partial_{v_1}^{k_1}\partial_{v_2}^{k_2}\partial_{v_3}^{k_3}\Phi(v_1, v_2, v_3)\right|\leq \pi^2 C \varepsilon^{-k-4}(k+3)!, 
	\end{align*}

	\noindent where we have denoted $k = k_1 + k_2 + k_3 \ge 0$. Together with the fact that $F \in \mathcal{C}^3 (\mathcal{B}_{2r})$, this verifies the assumptions in \cite[Theorem 1]{RSNEPS} (in particular, see \cite[Equation (5)]{RSNEPS}), with the $M_k$ there equal to $k!$ here. It follows that there exists some constant $A > 1$ such that, for any point $(t, y) \in \mathcal{B}_r$ and integers $k_1, k_2\geq 0$, we have
	\begin{align}
		\big| \partial_t^{k_1}\partial_y^{k_2}F(t, y) \big|\leq A^{k_1+k+2} (k_1+k_2)! \le A^2 \cdot (2A)^{k_1 + k_2} k_1! k_2!,
	\end{align}

	\noindent  and hence $F$ is real analytic on $\mathcal{B}_r$.
\end{proof}

\subsection{Proof of \Cref{perturbationbdk}}

\label{BoundaryPerturbationVariational}

In this section we establish \Cref{perturbationbdk}. To that end, we show the following proposition (which quickly implies \Cref{perturbationbdk}) bounding the derivatives of $F_1 - F_2$ assuming $F_1$ and $F_2$ satisfy the same nonlinear, elliptic partial differential equation with close boundary data. As in \Cref{perturbationbdk}, we state its interior and boundary forms simultaneously; its proof will be similar to that of \cite[Proposition A.5]{ULTS}, which provides its interior variant (at $m=1$).          

\begin{lem}
	
	\label{ellipticperturbation} 
	
	For any integer $m \ge 1$; real numbers $r > 0$ and $B > 1$; and bounded open subsets $\mathfrak{R}' \subseteq \mathfrak{R} \subset \mathbb{R}^2$ such that $\mathfrak{R}'$ has $\mathcal{C}^{m+5}$ boundary, there exists a constants $C = (B, m, r, \mathfrak{R}') > 1$ such that the following holds. For each $j, k \in \{ 1, 2 \}$, let $a_{jk} \in \mathcal{C}^{m+4} (\mathbb{R}^2)$ denote functions satisfying $a_{xy} = a_{yx}$ and 
	\begin{flalign}
		\label{aijxiestimate2}
		\displaystyle\max_{j, k \in \{ 1, 2 \}} \| a_{jk} \|_{\mathcal{C}^{m+4} (\mathbb{R}^2)} \le B; \qquad \displaystyle\inf_{z \in \mathbb{R}^2} \displaystyle\sum_{j, k \in \{ 1, 2 \}} a_{jk} (z) \xi_j \xi_k \ge B^{-1} |\xi|^2,
	\end{flalign}
	
	\noindent for any $\xi = (\xi_x, \xi_y) \in \mathbb{R}^2$. For each $i \in \{ 1, 2 \}$, let $f_i \in \mathcal{C} (\partial \mathfrak{R})$ denote a continuous function such that $\sup_{z \in \partial \mathfrak{R}} \big| f_i (z) \big| \le B$, and let $F_i \in \mathcal{C}^{m+4} (\mathfrak{R}) \cap \mathcal{C} (\overline{\mathfrak{R}})$ denote the solution to 
	\begin{flalign}
		\label{ajkfiequation}
		\displaystyle\sum_{j, k \in \{ 1, 2 \}} a_{jk} \big( \nabla F_i (x) \big) \partial_j \partial_k F_i (z) = 0, \quad \text{for each $z \in \mathfrak{R}$},
	\end{flalign} 
	
	\noindent with boundary data $F_i |_{\partial \mathfrak{R}} = f_i$. Assume that there exist functions $\varphi_i \in \mathcal{C}^{m+5} (\overline{\mathfrak{R}})$ such that $\| \varphi_i \|_{\mathcal{C}^{m+5} (\mathfrak{R})} < B$ and $F_i (z) = \varphi_i (z)$ for each $z \in \partial \mathfrak{R}' \cap \partial \mathfrak{R}$. Then, letting $\mathfrak{D}_r \subseteq \mathfrak{R}$ be as in \eqref{dr}, we have 
		\begin{flalign}
			\label{estimatef1f2elliptic}
			\| F_1 - F_2 \|_{\mathcal{C}^m (\mathfrak{D}_r)}  \le C \big( \| F_1 - F_2 \|_{\mathcal{C}^0 (\mathfrak{R})} + \| \varphi_1 - \varphi_2 \|_{\mathcal{C}^{m+2} (\mathfrak{R})} \big). 
		\end{flalign}

\end{lem}

\begin{proof}[Proof of \Cref{perturbationbdk} (Outline)]
	
	This follows from \Cref{ellipticperturbation}, very similarly to how \Cref{derivativef} followed from \Cref{aijgradientfestimate}; we omit further details.
\end{proof} 

To establish \Cref{ellipticperturbation}, we will use the following lemma, which is an interpolation estimate that bounds derivatives of a function in terms of its higher (and lower) derivatives; its proof is similar to that of \cite[Lemma 6.35]{EDSO}. 

\begin{lem} 
	
	\label{fjderivativesestimate} 
	
	For any integers $k > j > 0$ and bounded open subset $\mathfrak{R} \subset \mathbb{R}^2$ with $\mathcal{C}^{k+1}$ boundary, there exists a constant $C = C(k, \mathfrak{R}) > 1$ such that the following holds. Let $F \in \mathcal{C}^{k} (\mathfrak{R})$ denote a function, and suppose that $0 < A \le B$ are real numbers such that $A \ge  \| F \|_0$ and $B \ge [F]_{k; 0; \mathfrak{R}}$. Then, 
	\begin{flalign}
		\label{gradientf1f2estimate1}
		[F]_{j, 0; \mathfrak{R}} \le C A^{1 - j / k} B^{j / k}. 
	\end{flalign}
	
\end{lem} 

\begin{proof} 
	
	Since the proof of this lemma is similar to that of \cite[Lemma 6.35]{EDSO}, we only establish it in the case $(j, k) = (1, 2)$; the proof for general $k > j \ge 1$ is entirely analogous. 
	
	Fix some constant $C_0 = C_0 (\mathfrak{R}) > 1$ to be determined later, and assume to the contrary that there exists some $i \in \{ 1, 2 \}$ and $z_0 \in \mathfrak{R}$ such that $\big| \partial_i F (z_0) \big| > C_0 (AB)^{1 / 2}$. Assume that $i = x$ and $\partial_i F (z_0) > C_0 (AB)^{1 / 2}$, as the proofs in the alternative cases are entirely analogous. Further set $w = (1, 0) \in \mathbb{R}^2$ and $w' = (2^{-1/2}, 2^{-1/2}) \in \mathbb{R}^2$. Since the boundary of $\mathfrak{R}$ is $\mathcal{C}^{k+1}$, there exists a constant $t_0 = t_0 (\mathfrak{R}) > 0$ such that either $z_0 + tw \in \mathfrak{R}$ for each $t \in [0, t_0]$; $z_0 - tw \in \mathfrak{R}$ for each $t \in [0, t_0]$; $z_0 + t w' \in \mathfrak{R}$ for each $t \in [0, t_0]$; or $z_0 - tw' \in \mathfrak{R}$ for each $t \in [0, t_0]$. Let us assume that the third case holds, as the proofs in the remaining ones are very similar. Then, since $[F]_{2, 0; \mathfrak{R}} \le B$, we have $\big| \partial_x F (z_0 + t w') - \partial_x F (z_0) \big| \le Bt$ for each $t \in [0, t_0]$. 
	
	Thus, $\partial_x F (z_0 + t w') > 2^{-1/2} C_0 (AB)^{1/2} - Bt$ whenever $t \in [0, t_0]$, and so integration over $t$ yields $\big| F(z_0) - F(z_0 + t w') \big| > 2^{-1/2} C_0 (AB)^{1/2} t - \frac{Bt^2}{2}$, for any $t \in [0, t_0]$. Selecting $t = \big( \frac{A}{B} \big)^{1/2} t_0 \le t_0$ yields $\| F (z_0) - F(z_0 + t_0 w') \big| > \big( 2^{-1/2} C_0 t_0 - \frac{t_0^2}{2} \big) A$. In particular, selecting $C_0$ sufficiently large so that $2^{-1/2} C_0 t_0 - t_0^2 > 3$, this implies that $\| F \|_{0; \mathfrak{R}} > A$, which is a contradiction, verifying \eqref{gradientf1f2estimate1}.
\end{proof} 

Now we can establish \Cref{ellipticperturbation}.

\begin{proof}[Proof of \Cref{ellipticperturbation}]

	The proof of this proposition is similar to that of \cite[Proposition A.5]{ULTS}, obtained by replacing the interior regularity, Schauder, and interpolation estimates (given by \cite[Lemma A.3]{ULTS}, \cite[Lemma A.1]{ULTS}, and \cite[Lemma A.6]{ULTS}, respectively), by the boundary ones (given by \Cref{aijgradientfestimate}, \Cref{aijuestimates}, and \Cref{fjderivativesestimate}, respectively). We only address the case $m = 1$ here (so that it suffices to bound $|\nabla F_1 - \nabla F_2|$), as bounding the difference of the higher derivatives follows similarly by differentiating the equations below. Throughout, we assume that $\mathfrak{R}' = \mathfrak{R}$ for notational simplicity; the proof in the case of more general $\mathfrak{R}'$ is entirely analogous.  
	
	Denoting
	\begin{flalign}
		\label{f1bf2}
		\varsigma = \max \Big\{ \| F_1 - F_2 \|_{\mathcal{C}^0 (\mathfrak{R})}, \| \varphi_1 - \varphi_2 \|_{\mathcal{C}^3 (\mathfrak{R})} \Big\}
	\end{flalign}
	
	\noindent we have 
	\begin{flalign}
		\label{f1f2estimate}
		\| F_1 - F_2 \|_{\mathcal{C}^0 (\mathfrak{R})} \le \varsigma; \qquad \displaystyle\max_{i \in \{ 1, 2 \}} \displaystyle\sup_{z \in \mathfrak{R}} \big| F_i (z) \big| \le B,
	\end{flalign}
	
	\noindent where the second bound follows from \Cref{aijcomparison} and the fact that $\big| f_i (z) \big| \le B$, for each $i \in \{1, 2 \}$ and $z \in \partial \mathfrak{R}$. 
	Moreover, the $m = 5$ case of \Cref{aijgradientfestimate} yields a constant $C_1 = C_1 (B, \mathfrak{R}) > 1$ such that 
	\begin{flalign}
		\label{fi3bestimate}
		\displaystyle\max_{i \in \{ 1, 2 \}} \| F_i \|_{\mathcal{C}^5 (\mathfrak{R})} < C_1.
	\end{flalign}
	
	\noindent Thus \eqref{gradientf1f2estimate1} (with the $F$ there equal to $F_1 - F_2$ here; the $A$ there equal to $\varsigma$ here; the $B$ there equal to $C_1$ here; and the $(j, k)$ there equal to either $(1, 5)$ or $(2, 5)$), together with \eqref{f1f2estimate} and \eqref{fi3bestimate}, yields a constant $C_2 = C_2 (B, \mathfrak{R}) > 1$ such that
	\begin{flalign}
		\label{gradientf1f2estimate2}
		\begin{aligned}
			\displaystyle\sup_{z \in \mathfrak{R}} \big| \nabla F_1 (z) - \nabla F_2 (z) \big| & \le C_2 \varsigma^{4 / 5}; \qquad \| F_1 - F_2 \|_{\mathcal{C}^2 (\mathfrak{R})} \le C_2 \varsigma^{3 / 5}.
		\end{aligned} 
	\end{flalign}
	
	Thus, we may assume in what follows that $\varsigma < \delta$, for some sufficiently small constant $\delta = \delta (\varepsilon, B) > 0$ (to be specified later), for otherwise the result follows from the fact that  $\sup_{z \in \mathfrak{R}} \big| \nabla F_1 (z) -\nabla F_2 (z) \big|\leq C_2\varsigma^{4/5}$ from \eqref{gradientf1f2estimate2}.  To improve the first bound in \eqref{gradientf1f2estimate2} to \eqref{estimatef1f2elliptic}, define $g: \mathfrak{R} \rightarrow \mathbb{R}$ by setting
	\begin{flalign*} 
		g(z) = \varsigma^{-1 / 2} \big( F_2 (z) - F_1 (z) \big), \qquad \text{for each $z \in \overline{\mathfrak{R}}$.}
	\end{flalign*} 
	
	\noindent Then the first bound in \eqref{f1f2estimate} and \eqref{gradientf1f2estimate2} together yield
	\begin{flalign}
		\label{estimateg1}
		\| g \|_{\mathcal{C}^0 (\mathfrak{R})} \le \varsigma^{1 / 2}; \qquad \| g \|_{\mathcal{C}^1 (\mathfrak{R})} \le 2 C_2 \varsigma^{3 / 10}; \qquad \| g \|_{\mathcal{C}^2 (\mathfrak{R})} \le C_2 \varsigma^{1 / 10}. \qquad .
	\end{flalign}
	
	\noindent Next, setting $i = 2$ and using the fact that $F_2 = F_1 + \varsigma^{1 / 2} g$ in \eqref{ajkfiequation}, we obtain 
	\begin{flalign}
		\label{2ajkfiequation}
		\displaystyle\sum_{1 \le j, k \le 2} a_{jk} \big( \nabla F_1 (z) + \varsigma^{1 / 2} \nabla g(z) \big) \big( \partial_j \partial_k F_1 (z) + \varsigma^{1 / 2} \partial_j \partial_k g (z) \big) = 0.
	\end{flalign}
	
	\noindent Then, subtracting the $i = 1$ case of \eqref{ajkfiequation} from \eqref{2ajkfiequation}, we deduce that
	\begin{flalign}
		\label{ghequationmunukappa}
		\begin{aligned} 
			& \displaystyle\sum_{j, k \in \{ 1, 2 \} } \Big( \mu_{jk} (z) \partial_j \partial_k g (z) + \nu_{jk} (z) \kappa_{jk} (z) \cdot \nabla g(z) \Big) \\
			& \qquad \quad = - \varsigma^{1 / 2} \displaystyle\sum_{j, k \in \{ 1, 2 \} 	} \Big(  \big( \nabla g(z) \cdot \kappa_{jk} (z) \big) \partial_j \partial_k g (z) + \nu_{jk} (z) h_{jk} (z) + \varsigma^{1 / 2} h_{jk} (z) \partial_j \partial_k g (z) \Big),
		\end{aligned} 
	\end{flalign} 
	
	\noindent where for each $j, k \in \{ 1, 2 \}$, we have defined $\mu_{jk}, \nu_{jk}, h_{jk} \in \mathcal{C} (\mathfrak{R})$ and $\kappa_{jk}: \mathfrak{R} \rightarrow \mathbb{R}^2$ by setting
	\begin{flalign}
		\label{hdefinition}
		\begin{aligned}
			& \mu_{jk} (z) =  a_{jk} \big( \nabla F_1 (z) \big); \qquad \kappa_{jk} (z) = \nabla a_{jk} \big( \nabla F_1 (z) \big); \qquad \nu_{jk} (z) = \partial_j \partial_k F_1 (z);  \\
			& h_{jk} (z) = \varsigma^{-1} \Big( a_{jk} \big( \nabla F_1 (z) + \varsigma^{1 / 2} \nabla g (z) \big) - a_{jk} \big( \nabla F_1 (z) \big) - \varsigma^{1 / 2} \nabla g (z) \cdot \nabla a_{jk} \big( \nabla F_1 (z) \big) \Big), \\
		\end{aligned} 
	\end{flalign} 
	
	\noindent for each $z \in \mathfrak{R}$. Now, observe that \eqref{fgestimatefg}, \eqref{fi3bestimate}, and the fact that $\| a_{jk} \|_{\mathcal{C}^4 (\mathbb{R}^2)} \le B$ together imply the existence of a constant $C_3 = C_3 (B, \mathfrak{R}) > 1$ such that 
	\begin{flalign}
		\label{munukappa} 
		\| \mu_{jk} \|_{\mathcal{C}^{0, 1 / 2} (\mathfrak{R})} \le C_3; \qquad \| \kappa _{jk} (z) \|_{\mathcal{C}^{0, 1 / 2} (\mathfrak{R})} \le C_3; \qquad \| \nu_{jk} (z) \|_{\mathcal{C}^{0, 1 / 2} (\mathfrak{R})} \le C_3.
	\end{flalign} 
	
	Since the $a_{jk}$ satisfy \eqref{aijxiestimate2}, it follows from \eqref{ghequationmunukappa}, \eqref{f1bf2}, and the $\alpha = \frac{1}{2}$ case of the Schauder estimate \Cref{aijuestimates} that there exists a constant $C_4 = C_4 (B, \mathfrak{R}) > 1$ such that
	\begin{flalign*}
		\| g \|_{\mathcal{C}^{2, 1/2} (\mathfrak{R})} & \le C_4 \| g \|_{\mathcal{C}^0 (\mathfrak{R})} + C_4 \varsigma^{1 / 2} \displaystyle\max_{j, k \in \{ 1, 2 \} } \Big\|  \big( \nabla g(z) \cdot \kappa_{jk} (z) \big) \partial_j \partial_k g (z) \Big\|_{\mathcal{C}^{0,1/2} (\mathfrak{R})} \\
		& \qquad  + C_4 \varsigma^{1 / 2} \displaystyle\max_{j, k \in \{ 1, 2 \} }  \big\| \nu_{jk} (z) h_{jk} (z) \big\|_{\mathcal{C}^{0,1/2} (\mathfrak{R})} + C_4 \varsigma^{1/2} \\
		& \qquad + C_4 \varsigma  \displaystyle\max_{j, k \in \{ 1, 2 \} } \big\| h_{jk} (z) \partial_j \partial_k g (z) \big\|_{\mathcal{C}^{0,1/2} (\mathfrak{R})}.
	\end{flalign*} 
	
	\noindent Thus, by \eqref{fgestimatefg} and \eqref{munukappa}, there exists a constant $C_5 = C_5 (B, \mathfrak{R}) > 1$ such that
	\begin{flalign}
		\label{gestimategh}
		\begin{aligned}
			\| g \|_{\mathcal{C}^{2, 1/2} (\mathfrak{R})} & \le C_5 \varsigma^{1 / 2} \Big( \| g \|_{\mathcal{C}^{2, 1/2} (\mathfrak{R})} \| g \|_{\mathcal{C}^1 (\mathfrak{R})} + \| g \|_{1, 1 / 2; \mathfrak{R}} \| g \|_{\mathcal{C}^2 (\mathfrak{R})}   + \| h \|_{\mathcal{C}^{0,1/2} (\mathfrak{R})} + 1 \Big) \\
			& \quad + C_5 \varsigma \Big( \| h \|_{\mathcal{C}^0 (\mathfrak{R})} \| g \|_{\mathcal{C}^{2, 1/2} (\mathfrak{R})} + \| h \|_{\mathcal{C}^{0,1/2} (\mathfrak{R})} \| g \|_{\mathcal{C}^2 (\mathfrak{R})} \Big)+ C_5 \| g \|_{\mathcal{C}^0 (\mathfrak{R})}.
		\end{aligned} 
	\end{flalign} 
	
	\noindent Now, \eqref{estimateg1} yields a constant $C_6 = C_6 (B, \mathfrak{R}) > 1$ such that 
	\begin{flalign}
		\label{ghestimates}
		\begin{aligned}
			\max \big\{ \| g \|_{\mathcal{C}^1 (\mathfrak{R})}, \| g \|_{1, 1 / 2; \mathfrak{R}}, \| g & \|_{\mathcal{C}^2 (\mathfrak{R})}, \| h \|_{\mathcal{C}^0 (\mathfrak{R})}, \| h \|_{\mathcal{C}^{0,1/2} (\mathfrak{R})} \big\} \\
			& \quad \le \max \big\{ \| g \|_{\mathcal{C}^2 (\mathfrak{R})}, \| h \|_{\mathcal{C}^1 (\mathfrak{R})} \big\} <  C_6,
		\end{aligned}
	\end{flalign}
	
	\noindent where the bound on $\| h \|_{\mathcal{C}^1 (\mathfrak{R})}$ follows from differentiating the definition \eqref{hdefinition} of $h$; \eqref{fi3bestimate}; \eqref{estimateg1}; the fact that $\| a_{jk} \|_{\mathcal{C}^4 (\mathbb{R}^2)} \le B$ for each $j, k \in \{ 1, 2 \}$; and a Taylor expansion.
	
	Inserting the first bound of \eqref{estimateg1}; \eqref{ghestimates}; and the fact that $\| g \|_{\mathcal{C}^{2, 1/2} (\mathfrak{R})} \ge \| g \|_{\mathcal{C}^2 (\mathfrak{R})}$ into \eqref{gestimategh} then yields a constant $C_7 = C_7 (B, \mathfrak{R}) > 1$ such that 
	\begin{flalign*}
		(1 - C_7 \varsigma^{1 / 2}) \| g \|_{\mathcal{C}^{2, 1/2} (\mathfrak{R})} \le C_7 \varsigma^{1 / 2}.
	\end{flalign*}
	
	\noindent  Thus, if $\delta < \min \big\{ \frac{1}{4 C_7^2}, 1 \big\}$ and $\varsigma < \delta$, we obtain that
	\begin{flalign}
		\label{f1f2estimate2}
		\| F_1 - F_2 \|_{\mathcal{C}^1 (\mathfrak{R})} \le \| F_1 - F_2 \|_{\mathcal{C}^{2, 1/2} (\mathfrak{R})} = \varsigma^{1 / 2} \| g \|_{\mathcal{C}^{2, 1/2} (\mathfrak{R})} \le 2 C_7 \varsigma,
	\end{flalign}
	
	\noindent from which the proposition follows.
\end{proof}

\subsection{Proof of \Cref{f1f2b}}

\label{ProofF1F2Small} 

In this section we establish \Cref{f1f2b}.

\begin{proof}[Proof of \Cref{f1f2b}]
	
	We first perform a rescaling, to which end we define the open rectangles 
	\begin{flalign*} 
		\widehat{\mathfrak{R}} = L \cdot \mathfrak{R} = (0, 1) \times (& -L,L); \qquad \widehat{\mathfrak{S}}_r = L \cdot \mathfrak{S}_r = (r,1-r) \times (r+1-L, L-r-1); \\
		& \quad \widehat{\mathfrak{S}} = L \cdot \mathfrak{S} = (0, 1) \times (1-L, L-1),
	\end{flalign*}
	
	\noindent as well as the functions $\widehat{F} \in \mathcal{C}^m (\widehat{\mathfrak{R}})$ and $\widehat{f}_i, \widehat{g}_i : [-L, L] \rightarrow \mathbb{R}$ for each $i \in \{ 0, 1 \}$ by, for any $w \in [0, 1] \times [-L, L]$ and $x \in [-L, L]$, setting
	\begin{flalign}
		\label{fgi2} 
		\widehat{F} (w) = L \cdot F (L^{-1} w); \quad \widehat{f}_i (x) = L \cdot f_i (L^{-1} x); \quad \widehat{g}_i (x) = L \cdot g_i (L^{-1} x), \quad \text{so $\nabla \widehat{F}(w) = \nabla F(L^{-1} w)$},
	\end{flalign}
	
	\noindent which in particular implies that $\widehat{F} \in \Adm_{\varepsilon} (\widehat{\mathfrak{R}})$. Further define the functions $\widehat{H} \in \mathcal{C}^m (\widehat{\mathfrak{R}})$ and $\widehat{h}_i : [-L, L] \rightarrow \mathbb{R}$ for each $i \in \{ 0, 1 \}$ by, for any $(t, x) \in [0, 1] \times [-L, L]$, setting 
	\begin{flalign}
		\label{hh}
		\widehat{h}_i (x) = \widehat{g}_i (x) - \widehat{F} (i, x) = \widehat{g}_i (x) - \widehat{f}_i (x); \qquad \widehat{H} (t, x) = \widehat{F} (t, x) + (1-t) \widehat{h}_0 (x) + t \widehat{h}_1 (x),
	\end{flalign}
	
	\noindent so that $\widehat{H} (i, x) = \widehat{g}_i (x)$ for each $i \in \{ 0, 1 \}$ and $x \in [-L, L]$. The scalings \eqref{fgi2} (together with the assumptions of the lemma) then imply for any index $i \in \{ 0, 1 \}$ and integer $j \in \llbracket 0, m \rrbracket$ that
	\begin{flalign}
		\label{fghfh}  
		[\widehat{F}]_j \le BL^{1-j}; \quad [\widehat{g}_i]_j \le BL^{1-j}; \quad [\widehat{H}]_j \le 3B L^{1-j}; \quad \| \widehat{F} - \widehat{H} \|_{\mathcal{C}^0 (\widehat{\mathfrak{R}})} \le \| \widehat{h}_0 \|_{\mathcal{C}^0} + \| \widehat{h}_1 \|_{\mathcal{C}^0} \le 2 \delta.
	\end{flalign}
	
	It suffices to show the existence of a solution $\widehat{G} \in \Adm_{\varepsilon/2} (\widehat{\mathfrak{S}})$ to the equation \eqref{equationxtd} such that 
	\begin{flalign}
		\label{g2} 
		\begin{aligned}
			& \widehat{G} (i, x) = \widehat{g}_i (x), \quad \text{for each $i \in \{ 1, 2 \}$ and $x \in [1-L, L-1]$}; \\
			& \| \widehat{F} - \widehat{G} \|_{\mathcal{C}^k (\widehat{\mathfrak{S}}_r)} \le C_1 \cdot \big( \| \widehat{h}_0 \|_{\mathcal{C}^0} + \| \widehat{h}_1 \|_{\mathcal{C}^0} \big); \\ 
			& \| \widehat{F} - \widehat{G} \|_{\mathcal{C}^{m-5} (\widehat{\mathfrak{S}})} \le C_2 \cdot \big( \| \widehat{h}_0 \|_{\mathcal{C}^{m-3}} + \| \widehat{h}_1 \|_{\mathcal{C}^{m-3}} \big),
		\end{aligned} 
	\end{flalign}
	
	\noindent for some constants $C_1 = C_1 (\varepsilon, r, B, k) > 1$ and $C_2 = C_2 (\varepsilon, B, m) > 1$. Then, the function $G \in \Adm_{\varepsilon/2} (\mathfrak{S})$ defined by $G(w) = L^{-1} \cdot \widehat{G} (Lw)$ for each $w \in [0, L^{-1}] \times [ L^{-1} -1, 1- L^{-1}]$ would satisfy the conditions of the lemma (where here, we again used the fact that $\nabla G (w) = \nabla \widehat{G} (Lw)$, as in \eqref{fgi2}; to deduce the last bound in the third part of the lemma, we also used $\|\widehat h_i\|_{\mathcal C^{m-3}}\leq L\|h_i\|_{\mathcal C^{m-3}}$ for each $i\in \{0,1\}$, $\|F-G\|_{\mathcal C^{m-5}}\leq L^{m-6} \|\widehat F-\widehat G\|_{\mathcal C^{m-5}}$, and the interpolation estimate \Cref{fjderivativesestimate} to deduce that $\| f_i - g_i \|_{\mathcal{C}^{m-3}} \le C_2' \| f_i - g_i \|_{\mathcal{C}^0}^{3/m}$ for each $i \in \{ 0, 1 \}$ and some constant $C_2' = C_2' (\varepsilon, B, m) > 1$, which holds since $\| f_i - g_i \|_{\mathcal{C}^m} \le \| f_i \|_{\mathcal{C}^m} + \| g_i \|_{\mathcal{C}^m} \le 2B$).

		\begin{figure}
	\center
\includegraphics[scale = .45]{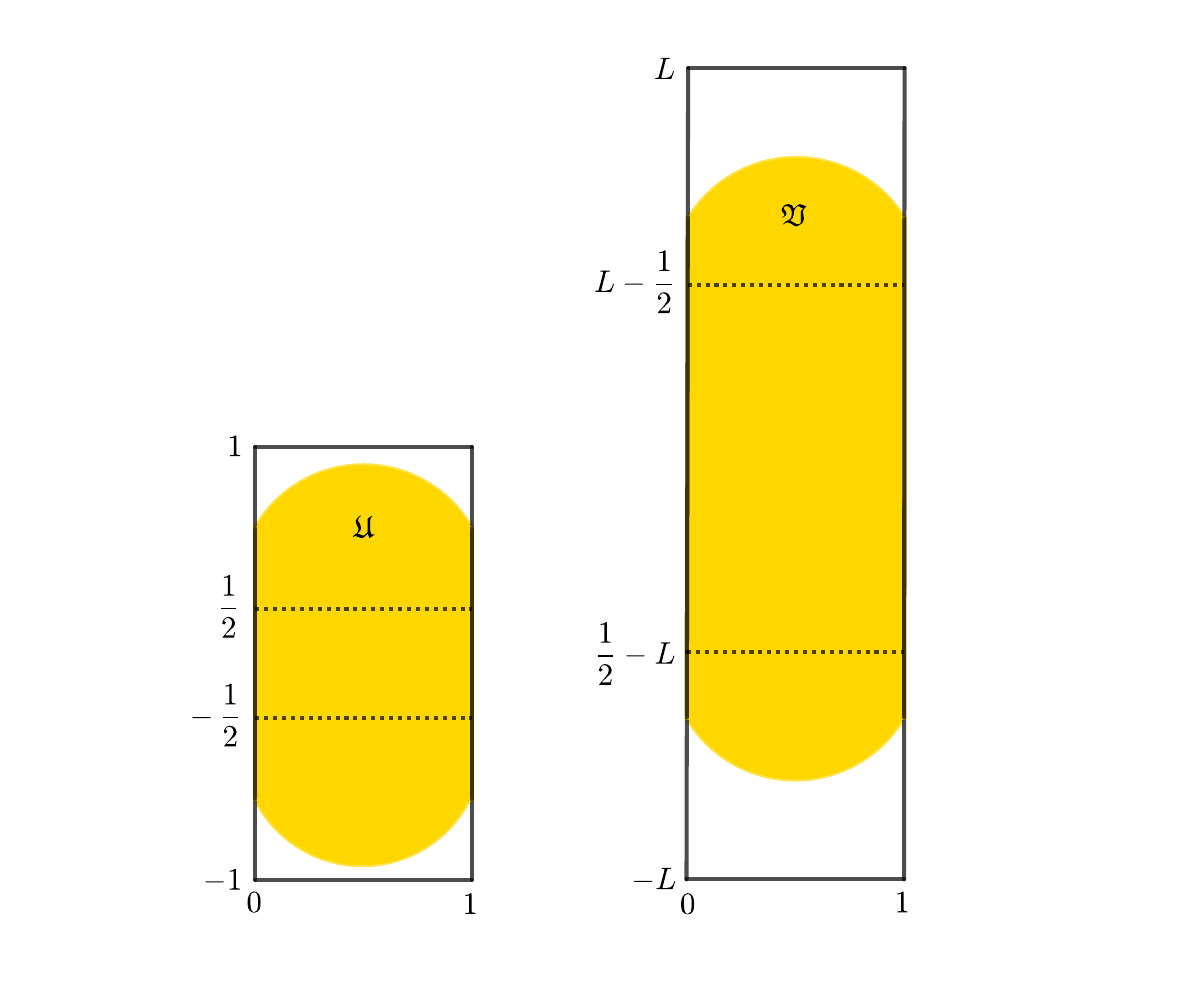}

\caption{Shown on the left is the open set $(0,1)\times \big[-\frac{1}{2}, \frac{1}{2} \big] \subset\mathfrak U\subset (0,1)\times (-1,1)$. Shown on the right is the open set $\mathfrak V$ as in \eqref{v}. }
\label{f:PDE}
	\end{figure}

	To that end, let us fix an open set $\mathfrak{U} \subset (0, 1) \times (-1, 1)$ with smooth boundary, such that $(0, 1) \times \big[ -\frac{1}{2}, \frac{1}{2} \big] \subset \mathfrak{U}$; see the left side of \Cref{f:PDE}. For any real number $a \in \mathbb{R}$, we also denote its vertical translate $\mathfrak{U}_a = \mathfrak{U} + (0, a)$, and we define the domain $[0, 1] \times [1-L,L-1] \subset \mathfrak{V} \subset [0, 1] \times [-L, L]$ by
	\begin{flalign}
		\label{v}
		\mathfrak{V} = \big( [0, 1] \times [1-L, L-1] \big) \cup \mathfrak{U}_{1-L} \cup \mathfrak{U}_{L-1}.
	\end{flalign}
	
	\noindent See the right side of \Cref{f:PDE}. As in the proof of \Cref{derivativef}, for any indices $i, j \in \{ t, x \}$ we also fix fuctions $\widetilde{\mathfrak{d}}_{ij} \in \mathcal{C}^{m+3} (\mathbb{R}^2)$ so that $\widetilde{\mathfrak{d}}_{ij} (u, v) = \mathfrak{d}_{ij} (u, v)$ whenever $v \in \big[-\frac{2}{\varepsilon}, -\frac{\varepsilon}{2} \big]$, and such that \eqref{aijxiestimate} holds for some constant $\widetilde{B} = \widetilde{B} (\varepsilon) > 1$. 
	
	Then, \Cref{aijgradientfestimate} yields a solution $\widehat{G} \in \mathcal{C}^{m-1} (\overline{\mathfrak{V}})$ to the equation
	\begin{flalign}
		\label{ghequation}
		\displaystyle\sum_{i,j \in \{t, x\}} \widetilde{\mathfrak{d}}_{ij} \big( \nabla \widehat{G} (z) \big) \partial_i \partial_j \widehat{G} (z) = 0, \qquad \text{for $z \in \mathfrak{V}$, with $\widehat{G}|_{\partial \mathfrak{V}} = \widehat{H} |_{\partial \mathfrak{V}}$}. 
	\end{flalign}
	
	\noindent We first verify that $\widehat{G}$ satisfies the three properties in \eqref{g2}, and then we confirm that $\widehat{G} \in \Adm_{\varepsilon/2} (\mathfrak{V}) \subseteq \Adm_{\varepsilon/2} (\widehat{\mathfrak{S}})$ (where the last inclusion holds by restricting $\widehat{G}$, as $\widehat{\mathfrak{S}} \subset \mathfrak{V}$) and that $\widehat{G}$ solves the original equation \eqref{equationxtd}. To confirm the first condition in \eqref{g2}, observe from the definition \eqref{v} of $\mathfrak{V}$ that $\{ 0, 1 \} \times [1-L,L-1] \subset \partial \mathfrak{V}$, so for any $(i, x) \in \{ 0, 1 \} \times [1-L,L-1]$ we have $\widehat{G} (i, x) = \widehat{H} (i, x) = \widehat{g}_i (x)$.
	
	The second and third statements will follow from a suitable application of \Cref{ellipticperturbation}. To implement this, first observe that 
	\begin{flalign}
		\label{fgdelta}
		\displaystyle\sup_{w \in \mathfrak{V}} \big| \widehat{F} (w) - \widehat{G} (w) \big| \le \displaystyle\sup_{w \in \partial \mathfrak{V}} \big| \widehat{F} (w) - \widehat{G} (w) \big| = \displaystyle\sup_{w \in \partial \mathfrak{V}} \big| \widehat{F} (w) - \widehat{H} (w) \big| \le \| \widehat{h}_0 \|_{\mathcal{C}^0} + \| \widehat{h}_1 \|_{\mathcal{C}^0} \le 2 \delta.
	\end{flalign}
	
	\noindent Here, the first inequality follows from the comparison principle \Cref{aijcomparison} (applied twice, as in the proof of \Cref{maximumboundary}, first with $(F_1, F_2)$ there equal to $(\widehat{F}, \widehat{G} - \sup_{w \in \partial \mathfrak{V}} \big| \widehat{F} (w) - \widehat{G} (w) \big|)$ and next with $(F_1, F_2)$ there equal to $(\widehat{F}, \widehat{G} + \sup_{w \in \partial \mathfrak{V}} \big| \widehat{F} (w) - \widehat{G} (w) \big|)$ here); the second follows from the boundary data \eqref{ghequation} for $\widehat{G}$; and the third and fourth follow from the last bounds in \eqref{fghfh}. 
	
	Now, fix some $w_0 = (t_0, x_0) \in \widehat{\mathfrak{S}}_r$ so that $\mathcal{B}_r (t_0, x_0) \subseteq \mathfrak{V}$, and apply \Cref{ellipticperturbation} with the $(F_1, F_2)$ there equal to $\big( \widehat{F} - \widehat{F} (0, x_0); \widehat{G} - \widehat{F} (0, x_0) \big)$ here; $(B, m, r; \mathfrak{R})$ there equal to $\big( 8B+\widetilde{B}, k; \frac{r}{2}; \mathfrak{U}_{x_0} \big)$ here; and $\mathfrak{R}'$ there empty here (so we may take the $(\varphi_1, \varphi_2)$ there to be $0$ here). This yields a constant $C_1 = C_1 (\varepsilon,r, B, k) > 1$ (recalling that $\widetilde{B}$ only depends on $\varepsilon$) such that
	\begin{flalign*}
		\big\| \widehat{F} - \widehat{G} \|_{\mathcal{C}^k (\mathcal{B}_{r/2} (w_0))} \le C_1 \cdot \displaystyle\sup_{w \in \partial \mathfrak{U}_{x_0}} \big| \widehat{F} (w) - \widehat{G}(w) \big| \le C_1 \Big( \| \widehat{h}_0 \|_{\mathcal{C}^0} + \| \widehat{h}_1 \|_{\mathcal{C}^0} \Big), 
	\end{flalign*}
	
	\noindent where the last bound follows from \eqref{fgdelta}. Ranging over $w_0 \in \widehat{\mathfrak{S}}_r$ gives the second part of \eqref{g2}.
	
	To establish the third part of \eqref{g2}, fix some $w_0 \in (t_0, x_0) \in \mathfrak{V}$, and let $s \in [1-L, L-1]$ be such that $\mathcal{B}_{r/4} (w_0) \subset \mathfrak{U}_s$ (which is guaranteed to exist by \eqref{v}). Then, apply \Cref{ellipticperturbation} with the $(F_1, F_2)$ equal to $\big( \widehat{F} - \widehat{F} (0, x_0), \widehat{G} - \widehat{F} (0, x_0) \big)$ here; $(B, m, r; \mathfrak{R}, \mathfrak{R}')$ there equal to $\big( 8B + \widetilde{B}, m-5, \frac{r}{8}; \mathfrak{U}_s, \mathfrak{U}_s \big)$ here; and $\varphi_1 (t, x)$ and $\varphi_2 (t, x)$ there equal to $\widehat{F} (t, x) - \widehat{F} (0, x_0)$ and $\widehat{H}(t, x) - \widehat{F} (0, x_0)$ here, respectively. This yields a constant $C_2 = C_2 (\varepsilon, r, B, m) > 1$ (which also depends on the open set $\mathfrak{U}$, but we omit this dependence since $\mathfrak{U}$ is fixed) such that 
	\begin{flalign*}
		\big\| \widehat{F} - \widehat{G} \big\|_{\mathcal{C}^{m-5} (\mathfrak{U}_s)} \le C_2 \bigg( \displaystyle\sup_{w \in \partial \mathfrak{U}_s} \big| \widehat{F} (w) - \widehat{G} (w) \big| + \| \widehat{F} - \widehat{H} \|_{\mathcal{C}^{m-3}} \bigg) \le 2 C_2 \Big( \| \widehat{h}_0 \|_{\mathcal{C}^{m-3}} + \| \widehat{h}_1 \|_{\mathcal{C}^{m-3}} \Big),
	\end{flalign*}
	
	\noindent where in the last bound we used \eqref{hh} and \eqref{fgdelta}. Ranging over all $w \in \mathfrak{V}$ yields
	\begin{flalign}
		\label{fgh0h1m}
		\big\| \widehat{F} - \widehat{G} \big\|_{\mathcal{C}^{m-5} (\mathfrak{V})} \le 2 C_2 \Big( \| \widehat{h}_0 \|_{\mathcal{C}^{m-3}} + \| \widehat{h}_1 \|_{\mathcal{C}^{m-3}} \Big),
	\end{flalign}
	
	\noindent which since $\mathfrak{S} \subset \mathfrak{V}$ implies the third part of \eqref{g2}.
	
	It remains to verify that $\widehat{G} \in \Adm_{\varepsilon/2} (\mathfrak{V}) \subseteq \Adm_{\varepsilon/2} (\widehat{\mathfrak{S}})$ and that $\widehat{G}$ satisfies \eqref{equationxtd}. Observe that it suffices to prove the former, as it implies the latter since $\widehat{G}$ satisfies \eqref{ghequation}, and $\widetilde{\mathfrak{d}}_{ij} (\nabla \widehat{G}) = \mathfrak{d}_{ij} (\nabla \widehat{G})$ on $\mathfrak{V}$ if $\widehat{G} \in \Adm_{\varepsilon/2} (\mathfrak{V})$. To show the former, observe from \eqref{fgh0h1m}, \eqref{fgdelta}, and the ($(F; j, k; \mathfrak{R}) = (\widehat{F} - \widehat{G}; 1, 2; \mathfrak{U}_s)$ case of the) interpolation estimate \Cref{fjderivativesestimate} that there exists a constant $C_3 = C_3 (\varepsilon, B, m) > 1$ (also implicitly dependent on $\mathfrak{U}$) such that the following holds. For sufficiently small $\delta = \delta (\varepsilon, B, m) > 0$ we have
	\begin{flalign*}
		\displaystyle\sup_{w \in \mathfrak{V}} \big| \nabla \widehat{F} (w) - \nabla \widehat{G} (w) \big| \le C_3 \delta^{1/2}.
	\end{flalign*}
	
	\noindent Together with the fact that $\widehat{F} \in \Adm_{\varepsilon} (\widehat{\mathfrak{R}})$, it follows (after decreasing $\delta$ again if necessary) that $\widehat{G} \in \Adm_{\varepsilon/2} (\mathfrak{V})$, thereby establishing the lemma.
\end{proof}


\begin{thebibliography}{10}
	
	\bibitem{PWNUC}
	M.~Adler, P.~L. Ferrari, and P.~van Moerbeke.
	\newblock Airy processes with wanderers and new universality classes.
	\newblock {\em Ann. Probab.}, 38(2):714--769, 2010.
	
	\bibitem{ULTS}
	A.~Aggarwal.
	\newblock Universality for lozenge tiling local statistics.
	\newblock To appear in \emph{Ann. Math.}, preprint, ar{X}iv:1907.09991.
	
	\bibitem{UERT}
	A.~Aggarwal and V.~Gorin.
	\newblock Gaussian unitary ensemble in random lozenge tilings.
	\newblock {\em Probab. Theory Related Fields}, 184(3-4):1139--1166, 2022.
	
	\bibitem{U}
	A.~Aggarwal and J.~Huang.
	\newblock In preparation.
	
	\bibitem{EDE}
	L.~Ambrosio, A.~Carlotto, and A.~Massaccesi.
	\newblock {\em Lectures on elliptic partial differential equations}, volume~18
	of {\em Appunti. Scuola Normale Superiore di Pisa (Nuova Serie) [Lecture
		Notes. Scuola Normale Superiore di Pisa (New Series)]}.
	\newblock Edizioni della Normale, Pisa, 2018.
	
	\bibitem{IRM}
	G.~W. Anderson, A.~Guionnet, and O.~Zeitouni.
	\newblock {\em An introduction to random matrices}, volume 118 of {\em
		Cambridge Studies in Advanced Mathematics}.
	\newblock Cambridge University Press, Cambridge, 2010.
	
	\bibitem{LDSI}
	S.~Belinschi, A.~Guionnet, and J.~Huang.
	\newblock Large deviation principles via spherical integrals.
	\newblock {\em Probability and Mathematical Physics}, 3(3):543--625, 2022.
	
	\bibitem{FCSD}
	P.~Biane.
	\newblock On the free convolution with a semi-circular distribution.
	\newblock {\em Indiana University Mathematics Journal}, pages 705--718, 1997.
	
	\bibitem{LRMES}
	P.~M. Bleher and A.~B.~J. Kuijlaars.
	\newblock Large {$n$} limit of {G}aussian random matrices with external source.
	{III}. {D}ouble scaling limit.
	\newblock {\em Comm. Math. Phys.}, 270(2):481--517, 2007.
	
	\bibitem{LLCTE}
	P.~Bourgade, K.~Mody, and M.~Pain.
	\newblock Optimal local law and central limit theorem for {$\beta$}-ensembles.
	\newblock {\em Comm. Math. Phys.}, 390(3):1017--1079, 2022.
	
	\bibitem{LSRMES}
	E.~Br\'{e}zin and S.~Hikami.
	\newblock Level spacing of random matrices in an external source.
	\newblock {\em Phys. Rev. E (3)}, 58(6):7176--7185, 1998.
	
	\bibitem{USCG}
	E.~Br\'{e}zin and S.~Hikami.
	\newblock Universal singularity at the closure of a gap in a random matrix
	theory.
	\newblock {\em Phys. Rev. E (3)}, 57(4):4140--4149, 1998.
	
	\bibitem{DFIM}
	D.~L. Burkholder.
	\newblock Distribution function inequalities for martingales.
	\newblock {\em Ann. Probability}, 1:19--42, 1973.
	
	\bibitem{PTIIM}
	Y.~S. Chow and H.~Teicher.
	\newblock {\em Probability theory}.
	\newblock Springer Texts in Statistics. Springer-Verlag, New York, second
	edition, 1988.
	\newblock Independence, interchangeability, martingales.
	
	\bibitem{BUPM}
	T.~Claeys, T.~Neuschel, and M.~Venker.
	\newblock Boundaries of sine kernel universality for {G}aussian perturbations
	of {H}ermitian matrices.
	\newblock {\em Random Matrices Theory Appl.}, 8(3):1950011, 50, 2019.
	
	\bibitem{LSRT}
	H.~Cohn, N.~Elkies, and J.~Propp.
	\newblock Local statistics for random domino tilings of the {A}ztec diamond.
	\newblock {\em Duke Math. J.}, 85(1):117--166, 1996.
	
	\bibitem{VPT}
	H.~Cohn, R.~Kenyon, and J.~Propp.
	\newblock A variational principle for domino tilings.
	\newblock {\em J. Amer. Math. Soc.}, 14(2):297--346, 2001.
	
	\bibitem{PLE}
	I.~Corwin and A.~Hammond.
	\newblock Brownian {G}ibbs property for {A}iry line ensembles.
	\newblock {\em Invent. Math.}, 195(2):441--508, 2014.
	
	\bibitem{BPLE}
	D.~Dauvergne and B.~Vir\'{a}g.
	\newblock Bulk properties of the {A}iry line ensemble.
	\newblock {\em Ann. Probab.}, 49(4):1738--1777, 2021.
	
	\bibitem{NFTRM}
	D.~S. Dean, P.~Le~Doussal, S.~N. Majumdar, and G.~Schehr.
	\newblock Noninteracting fermions in a trap and random matrix theory.
	\newblock {\em J. Phys. A}, 52(14):144006, 32, 2019.
	
	\bibitem{CBNMT}
	S.~Delvaux, A.~B.~J. Kuijlaars, and L.~Zhang.
	\newblock Critical behavior of nonintersecting {B}rownian motions at a tacnode.
	\newblock {\em Comm. Pure Appl. Math.}, 64(10):1305--1383, 2011.
	
	\bibitem{MERM}
	F.~J. Dyson.
	\newblock A {B}rownian-motion model for the eigenvalues of a random matrix.
	\newblock {\em J. Mathematical Phys.}, 3:1191--1198, 1962.
	
	\bibitem{BUFM}
	L.~Erd\H{o}s, S.~P\'{e}ch\'{e}, J.~A. Ram\'{\i}rez, B.~Schlein, and H.-T. Yau.
	\newblock Bulk universality for {W}igner matrices.
	\newblock {\em Comm. Pure Appl. Math.}, 63(7):895--925, 2010.
	
	\bibitem{URMFD}
	L.~Erd\H{o}s and K.~Schnelli.
	\newblock Universality for random matrix flows with time-dependent density.
	\newblock {\em Ann. Inst. Henri Poincar\'{e} Probab. Stat.}, 53(4):1606--1656,
	2017.
	
	\bibitem{DRM}
	L.~Erd{\H{o}}s and H.-T. Yau.
	\newblock {\em A dynamical approach to random matrix theory}, volume~28.
	\newblock American Mathematical Soc., 2017.
	
	\bibitem{RGM}
	P.~L. Ferrari and H.~Spohn.
	\newblock {Random growth models}.
	\newblock In {\em {The Oxford Handbook of Random Matrix Theory}}. Oxford
	University Press, 09 2015.
	
	\bibitem{NBATP}
	P.~L. Ferrari and B.~Vet\H{o}.
	\newblock Non-colliding {B}rownian bridges and the asymmetric tacnode process.
	\newblock {\em Electron. J. Probab.}, 17:no. 44, 17, 2012.
	
	\bibitem{HTPM}
	P.~L. Ferrari and B.~Vet\H{o}.
	\newblock The hard-edge tacnode process for {B}rownian motion.
	\newblock {\em Electron. J. Probab.}, 22:Paper No. 79, 32, 2017.
	
	\bibitem{WWWM}
	M.~E. Fisher.
	\newblock Walks, walls, wetting, and melting.
	\newblock {\em J. Statist. Phys.}, 34(5-6):667--729, 1984.
	
	\bibitem{RSNEPS}
	A.~Friedman.
	\newblock On the regularity of the solutions of nonlinear elliptic and
	parabolic systems of partial differential equations.
	\newblock {\em J. Math. Mech.}, pages 43--59, 1958.
	
	\bibitem{SE}
	M.~Gaudin.
	\newblock Sur la loi limite de l'espacement des valeurs propres d'une matrice
	al\'{e}atoire.
	\newblock {\em Nuclear Physics}, 25:447--458, 1961.
	
	\bibitem{SMFSSC}
	P.-G.~d. Gennes.
	\newblock Soluble model for fibrous structures with steric constraints.
	\newblock {\em The Journal of Chemical Physics}, 48(5):2257--2259, 1968.
	
	\bibitem{EDSO}
	D.~Gilbarg and N.~S. Trudinger.
	\newblock {\em Elliptic partial differential equations of second order}.
	\newblock Classics in Mathematics. Springer-Verlag, Berlin, 2001.
	\newblock Reprint of the 1998 edition.
	
	\bibitem{MCNP}
	D.~J. Grabiner.
	\newblock Brownian motion in a {W}eyl chamber, non-colliding particles, and
	random matrices.
	\newblock {\em Ann. Inst. H. Poincar\'{e} Probab. Statist.}, 35(2):177--204,
	1999.
	
	\bibitem{FOAMI}
	A.~Guionnet.
	\newblock First order asymptotics of matrix integrals; a rigorous approach
	towards the understanding of matrix models.
	\newblock {\em Comm. Math. Phys.}, 244(3):527--569, 2004.
	
	\bibitem{LDASI}
	A.~Guionnet and O.~Zeitouni.
	\newblock Large deviations asymptotics for spherical integrals.
	\newblock {\em J. Funct. Anal.}, 188(2):461--515, 2002.
	
	\bibitem{ET}
	M.~Hairer and J.~C. Mattingly.
	\newblock Yet another look at {H}arris' ergodic theorem for {M}arkov chains.
	\newblock In {\em Seminar on {S}tochastic {A}nalysis, {R}andom {F}ields and
		{A}pplications {VI}}, volume~63 of {\em Progr. Probab.}, pages 109--117.
	Birkh\"{a}user/Springer Basel AG, Basel, 2011.
	
	\bibitem{ESMP}
	T.~E. Harris.
	\newblock The existence of stationary measures for certain {M}arkov processes.
	\newblock In {\em Proceedings of the {T}hird {B}erkeley {S}ymposium on
		{M}athematical {S}tatistics and {P}robability, 1954--1955, vol. {II}}, pages
	113--124. University of California Press, Berkeley-Los Angeles, Calif., 1956.
	
	\bibitem{DIA}
	J.~M. Holte.
	\newblock Discrete {G}ronwall lemma and its applications.
	\newblock \emph{MAA-NCS {M}eeting at the {U}niversity of {N}orth {D}akota},
	{G}rand {F}orks, 2009.
	
	\bibitem{EUNB}
	J.~Huang.
	\newblock Edge universality for nonintersecting {B}rownian bridges.
	\newblock Preprint, ar{X}iv:2011.01752.
	
	\bibitem{MCTMGP}
	J.~Huang and B.~Landon.
	\newblock Rigidity and a mesoscopic central limit theorem for {D}yson
	{B}rownian motion for general {$\beta$} and potentials.
	\newblock {\em Probab. Theory Related Fields}, 175(1-2):209--253, 2019.
	
	\bibitem{BUSRM}
	J.~Huang, B.~Landon, and H.-T. Yau.
	\newblock Bulk universality of sparse random matrices.
	\newblock {\em J. Math. Phys.}, 56(12):123301, 19, 2015.
	
	\bibitem{CMDWD}
	D.~A. Huse and M.~E. Fisher.
	\newblock Commensurate melting, domain walls, and dislocations.
	\newblock {\em Physical Review B}, 29(1):239, 1984.
	
	\bibitem{ULSDM}
	K.~Johansson.
	\newblock Universality of the local spacing distribution in certain ensembles
	of {H}ermitian {W}igner matrices.
	\newblock {\em Comm. Math. Phys.}, 215(3):683--705, 2001.
	
	\bibitem{DPGDP}
	K.~Johansson.
	\newblock Discrete polynuclear growth and determinantal processes.
	\newblock {\em Comm. Math. Phys.}, 242(1-2):277--329, 2003.
	
	\bibitem{RMDP}
	K.~Johansson.
	\newblock Random matrices and determinantal processes.
	\newblock In {\em Mathematical statistical physics}, pages 1--55. Elsevier B.
	V., Amsterdam, 2006.
	
	\bibitem{NMETP}
	K.~Johansson.
	\newblock Non-colliding {B}rownian motions and the extended tacnode process.
	\newblock {\em Comm. Math. Phys.}, 319(1):231--267, 2013.
	
	\bibitem{MSC}
	I.~Karatzas and S.~E. Shreve.
	\newblock {\em Brownian motion and stochastic calculus}, volume 113 of {\em
		Graduate Texts in Mathematics}.
	\newblock Springer-Verlag, New York, second edition, 1991.
	
	\bibitem{FEUM}
	B.~Landon, P.~Sosoe, and H.-T. Yau.
	\newblock Fixed energy universality of {D}yson {B}rownian motion.
	\newblock {\em Adv. Math.}, 346:1137--1332, 2019.
	
	\bibitem{CLSM}
	B.~Landon and H.-T. Yau.
	\newblock Convergence of local statistics of {D}yson {B}rownian motion.
	\newblock {\em Comm. Math. Phys.}, 355(3):949--1000, 2017.
	
	\bibitem{MTTD}
	B.~Laslier and F.~L. Toninelli.
	\newblock The mixing time of lozenge tiling {G}lauber dynamics.
	\newblock Preprint, ar{X}iv:2207.01444.
	
	\bibitem{LTGDMLS}
	B.~Laslier and F.~L. Toninelli.
	\newblock Lozenge tilings, {G}lauber dynamics and macroscopic shape.
	\newblock {\em Comm. Math. Phys.}, 338(3):1287--1326, 2015.
	
	\bibitem{NBRAW}
	K.~Liechty and D.~Wang.
	\newblock Nonintersecting {B}rownian bridges between reflecting or absorbing
	walls.
	\newblock {\em Adv. Math.}, 309:155--208, 2017.
	
	\bibitem{LODT}
	I.~M. Lifshitz.
	\newblock Kinetics of ordering during second-order phase transitions.
	\newblock {\em Sov. Phys. JETP}, 15:939, 1962.
	
	\bibitem{LLI}
	A.~Matytsin.
	\newblock On the large-{$N$} limit of the {I}tzykson-{Z}uber integral.
	\newblock {\em Nuclear Phys. B}, 411(2-3):805--820, 1994.
	
	\bibitem{DERM}
	M.~L. Mehta and M.~Gaudin.
	\newblock On the density of eigenvalues of a random matrix.
	\newblock {\em Nuclear Phys.}, pages 420--427, 1960.
	
	\bibitem{CSS}
	S.~Meyn and R.~L. Tweedie.
	\newblock {\em Markov chains and stochastic stability}.
	\newblock Cambridge University Press, Cambridge, second edition, 2009.
	\newblock With a prologue by Peter W. Glynn.
	
	\bibitem{FPRM}
	J.~A. Mingo and R.~Speicher.
	\newblock {\em Free probability and random matrices}, volume~35 of {\em Fields
		Institute Monographs}.
	\newblock Springer, New York; Fields Institute for Research in Mathematical
	Sciences, Toronto, ON, 2017.
	
	\bibitem{SIDP}
	M.~Pr\"{a}hofer and H.~Spohn.
	\newblock Scale invariance of the {PNG} droplet and the {A}iry process.
	\newblock volume 108, pages 1071--1106. 2002.
	\newblock Dedicated to David Ruelle and Yasha Sinai on the occasion of their
	65th birthdays.
	
	\bibitem{EDLE}
	H.~Spohn.
	\newblock {K}ardar--{P}arisi--{Z}hang equation in one dimension and line
	ensembles.
	\newblock {\em Pramana}, 64:847--857, 2005.
	
	\bibitem{RMCE}
	T.~Tao and V.~Vu.
	\newblock Random matrices: sharp concentration of eigenvalues.
	\newblock {\em Random Matrices Theory Appl.}, 2(3):1350007, 31, 2013.
	
	\bibitem{TP}
	C.~A. Tracy and H.~Widom.
	\newblock The {P}earcey process.
	\newblock {\em Comm. Math. Phys.}, 263(2):381--400, 2006.
	
	\bibitem{NE}
	C.~A. Tracy and H.~Widom.
	\newblock Nonintersecting {B}rownian excursions.
	\newblock {\em Ann. Appl. Probab.}, 17(3):953--979, 2007.
	
	\bibitem{ACNRV}
	D.~Voiculescu.
	\newblock Addition of certain noncommuting random variables.
	\newblock {\em J. Funct. Anal.}, 66(3):323--346, 1986.
	
\end{thebibliography}
\end{document}